\documentclass[preprint]{imsart}

\usepackage{color}
\usepackage{cancel}
\RequirePackage[OT1]{fontenc}
\RequirePackage{amsthm,amsmath}
\RequirePackage[numbers]{natbib}
\RequirePackage[colorlinks,citecolor=blue,urlcolor=blue]{hyperref}
  
\usepackage{amsfonts}
\usepackage{amsmath,color}
\usepackage{amssymb}
\usepackage{url}
\usepackage{latexsym}
\usepackage{amsthm}  
 \hoffset= - 1.0 cm  \voffset= - 0.1 cm \providecommand{\U}[1]{\protect\rule{.1in}{.1in}}

\newtheorem{theorem}{Theorem}

\newtheorem{hypothesis}{Hypothesis}

\newtheorem{lemma}[theorem]{Lemma}

\newtheorem{proposition}[theorem]{Proposition}
\newtheorem{remark}[theorem]{Remark}

\def\P{\mathbb P}
\def\E{\mathbb E}
\def\N{\mathbb N}
\def\Q{\mathbb Q}
\def\R{\mathbb R}
\def\ds{\displaystyle}
\def\L{\mathcal L}

\def\A{\mathcal A}
\def\H{\mathcal  H}
\def\W{\mathcal  W}

\begin{document}

\begin{frontmatter}
\title{ Correction to  ``An optimal  regularity result
 for 
Kolmogorov equations  
  and weak uniqueness 
  for  some   critical 
  SPDEs'' } 
%  \thankstext{ciao}{Footnote to the title with the ``thankstext'' command.}
\runtitle{Critical SPDEs} 
%\runtitle{} 
\begin{aug}
\author{Enrico Priola  
%\snm{Priola}   
(\thanks{This paper is a correction of  \cite{priolaAOP}  
 which deals with SPDEs like \eqref{qq1} where $F : H \to H$
  is only {\sl continuous  with at most linear growth}. Since there is a mistake in the proof of the regularity lemma   \cite[Lemma 6]{priolaAOP} (see in particular the change of variable at the end of page 1319)
% we do not know if the weak  uniqueness result of \cite{priolaAOP}   holds under the sole hypothesis of continuity  on $F$ (see   \cite[Theorem  1]{priolaAOP}). This remains an open problem (cf.  Remark \ref{peccato}).
  it remains an open problem if the weak  uniqueness result of \cite[Theorem 1]{priolaAOP} holds  under the sole hypothesis of continuity  on $F$ plus  growth condition (cf.  Remark \ref{peccato}).  
To prove  weak uniqueness for \eqref{qq1} we replace   the continuity condition on $F$ with the  stronger assumption  that {\sl $F$ is locally H\"older continuous}.  
 %We replace     
  % the $L^{\infty}$-estimate of 
  Moreover, 
    \cite[Theorem 7]{priolaAOP} is replaced by 
  % which fails to  hold    
   %(this result  does not hold in general) 
      Theorem \ref{ss13}, which is an optimal regularity result in  H\"older spaces and further  \cite[Sections 5.2 and 5.3]{priolaAOP} are replaced by Section 5.2. Then we basically follow the lines of   \cite{priolaAOP}.  
  %We still need to use a localization principle in infinite dimensions.
  %(changing a bit the examples).   
}}) 
%\thanksref{t1,t2,m1}\ead[label=e1]
%{first@somewhere.com}},
%\author{\fnms{Second} \snm{Author}\thanksref{t3,m1,m2}\ead[label=e2]{second@somewhere.com}}
%\and
%\author{\fnms{Third} \snm{Author}\thanksref{t1,m2}
%\ead[label=e3]{third@somewhere.com}
%\ead[label=u1,url]{http://www.foo.com}}

%\thankstext{t1}{Some comment}
%\thankstext{t2}{First supporter of the project}
%\thankstext{t3}{Second supporter of the project}
\runauthor{E. Priola} 

\address{    Dipartimento di  
 Matematica,  
 \\
Universit\`a  di Pavia, Pavia, Italy  
\\
enrico.priola@unipv.it
%\thanksmark{}
}

\end{aug}

\vskip 0.5 cm

\begin{abstract}
   We  show uniqueness in law for  the  critical  SPDE  
 \begin{eqnarray} \label{qq1}
 dX_t = AX_t dt + (-A)^{1/2}F(X(t))dt +  dW_t,\;\;   
 X_0 =x \in H, 
\end{eqnarray} 
%related SDEs 
%written in mild form,  
where   $A$ $ : \text{dom}(A) \subset    H \to H$ is a  negative definite  self-adjoint operator on  a separable Hilbert space $H$ having $A^{-1}$ of trace class  and  $W$ is 
   a cylindrical Wiener process on  $H$. 
 Here 
{  $F: H \to H $  can be   {\sl  locally H\"older continuous 
  with at most linear growth }
  %(possible extensions to  
  (some functions    
  %continuous and locally bounded 
  $F$ which grow more than linearly 
    can  also be considered)}.  This leads to new uniqueness results for 
 % {\bf generalized}  
 % {  stochastic Burgers type} 
 generalized stochastic Burgers equations and for three-dimensional   stochastic Cahn-Hilliard type equations 
  which have interesting  applications.  
  %?? 
  We do not know if  uniqueness  holds under the sole assumption of continuity of $F$ plus  growth condition 
  %as claimed ?? in Theorem 1 of 
   as stated in Priola \cite{priolaAOP}.    
  To get weak uniqueness we use an infinite dimensional localization principle and   an  optimal regularity result for   
the   Kolmogorov equation   $ \lambda u -  L u = f$  
 associated to the SPDE when  $F = z \in H$  is constant   and $\lambda >0$.
 This optimal    result is  
 %extension  
 %  variant  
similar  to a theorem  of   Da Prato \cite{D2}.  
 %$ f: H \to {\mathbb R}$ is  H\"older continuous and  bounded). 
 % \\ 
 %variante di da prato
\end{abstract}


 
\end{frontmatter}

 \startlocaldefs
\numberwithin{equation}{section}
\theoremstyle{plain}
\newtheorem{thm}{Theorem}[section]
\endlocaldefs  
 
\section{Introduction}

%\vskip 1mm 
 
We establish weak uniqueness (or uniqueness in law)    for critical  stochastic evolution equations like 
% in a real separable Hilbert space $H$ of the form 
\begin{equation} \label{sde}
dX_t = AX_t dt + (-A)^{1/2}F(X_t)dt +  dW_t, \;\; X_0 =x \in H.
\end{equation}   
Here $H$ is a separable Hilbert space, $ A: D(A) \subset H \to H$ is a 
self-adjoint operator of negative type 
%with domain $D(A)$  
such that the inverse $A^{-1}$ is of trace class (cf. Section 1.1 and see also Remark \ref{serve}), $W = (W_t)$ is a cylindrical Wiener process on  $H$, cf.  \cite{DZ1},    \cite{DZ},  \cite{hairer}     and the references therein. We    assume that
%$F$ is continuous and has at most linear growth, i.e.,
there exists $\theta \in (0,1)$ such that 
\begin{eqnarray} \label{lin1} 
F: H \to H \;\; \text{is locally $\theta$-H\"older continuous }
%(i.e., $F$ is $\theta$-H\"older continuous on each bounded set of $H$)}
 \text{and verifies}   
\;\; |F(x)|_H \le C_F (1 + |x|_H),  \;\; x \in H,
\end{eqnarray}   
   for some  constant $C_F >0$. The first assumption means that 
 % \text{is locally $\theta$-H\"older continuous 
 {\sl $F$ is $\theta$-H\"older continuous  on each bounded set of $H$.} 
    This allows to prove  both weak existence and weak uniqueness for \eqref{sde}.
  
%Note that in \eqref{lin1} we require there exists $\theta \in (0,1)$ such that on each ball $B$ of $H$, $F: B \to H$ is $\theta$-H\"older continuous (this is the interpretation of locally  $\theta$-H\"older continuous). 
 
   Assumption \eqref{lin1} can be   relaxed if we assume  that  weak existence holds for \eqref{sde}; 
 %  We refer to  
 %for  a possibile generalization of this assumption 
  see Section 7 
  %for possible extensions 
  where we 
  % prove weak uniqueness when
 % prove a more general when
  %consider  cases where 
 % assuming that 
  consider  $F$ which is   only  locally  $\theta$-H\"older continuous   %on each bounded set of $H$. 
 %for a more general result.
 without imposing the  growth condition.

\smallskip  
 Using the analytic  semigroup   $(e^{tA})$   generated by $A$ 
 we consider   mild solutions to \eqref{sde},  i.e., 
\begin{equation*} 
  \label{mq1} 
X_{t}=e^{tA}x+\int_{0}^{t}(-A)^{1/2}e^{(  t-s)  A}F  (
X_{s})
ds
+\int_{0}^{t}e^{(  t-s)  A}dW_{s},\;\;\; t \ge 0
\end{equation*}  
%where $(e^{tA})$ is the analytic  semigroup  generated by $A$ 
(cf. Section 1.1) and   prove
  the following result.
\begin{theorem}
\label{base}  
 Under Hypothesis \ref{d1}  and assuming \eqref{lin1}, for any $x \in H$,  there exists a weak mild solution defined on some filtered probability space.  
 Moreover uniqueness in law (or weak uniqueness) holds for \eqref{sde}, for any $x \in H$.     
\end{theorem}
We do not know if weak uniqueness holds, for any $x \in H$, when $F: H \to H $ is only continuous with at most linear growth. It is an open problem if such more general result 
%of \cite{priolaAOP} 
holds.   This general result is stated in Theorem 1 of \cite{priolaAOP} but there is a mistake in the proof of Lemma 6 of \cite{priolaAOP};  see  Remark \ref{peccato} for more details.  
%The proof given in \cite{priolaAOP} is based on 
  On the other hand 
existence  of weak solutions holds only assuming  continuity of $F$ with at most linear growth; see Section  4. 
 
As in \cite{priolaAOP} examples of SPDEs of the form \eqref{sde} are considered in Section 2 (replacing the continuity of  the coefficients considered in \cite{priolaAOP} with  a  H\"older type condition). In  particular, we can deal with  stochastic Burgers-type equations like  
 $$ 
d u (t, \xi)=   \frac{\partial^2}{\partial   \xi^2}  u(t, \xi) {   dt} +   \frac{\partial }{\partial \xi} {  h( } u(t, \xi)){   dt}  + dW_t(\xi), \;\; u(0, \xi) = u_0(\xi), \;\;\; \xi \in (0,\pi), 
$$ 
 with suitable boundary conditions (cf. \cite{Gy}, \cite{D2} and \cite{RS}) and stochastic Cahn-Hilliard equations (cf. \cite{EM}, \cite{DD}, \cite{NC}, \cite{ES}) like   
$$ 
d u (t, \xi)= -  \triangle^2_{\xi} u(t, \xi){   dt} + \triangle_\xi {  h( }u(t, \xi)){   dt} + dW_t(\xi), \;\;t>0,\;\;  u(0, \xi) = u_0(\xi)\;\; \text{on $G$}, 
$$
 with suitable boundary conditions ($G \subset \R^3$ is a regular bounded open set).
  We  prove {\sl weak well-posedness} for both SPDEs when     $h = h_1 + h_2$  where  $h_1$ is $\theta$-H\"older continuous, for some $\theta \in (0,1)$,  and $h_2$ is Lipschitz continuous   (see also the end of   Section 2.0.1 where we  consider different non-local nonlinearities like $h(u) = u \cdot g(|u|_H)$).  
 Assumptions of Theorem \ref{base}   do not cover classical stochastic  Burgers equations (i.e., ${   h( }u) = \frac{u^2}{2}$) and stochastic Chan-Hiliard equations (i.e., ${   h( }u) = u^3 -u$) for which strong existence and uniqueness can be proved by different methods (cf. \cite{brezniak} and \cite{DD}).
  On the other hand, in Section 7 we  consider   some  locally H\"older continuous  perturbations of classical Burgers equations (cf. Propositions \ref{dap} and \ref{well1}).

 We  mention   \cite{Za}   
 %\cite{bass} and 
  and \cite{bass1}   where  weak uniqueness has been investigated for stochastic evolution equations   
with H\"older continuous  coefficients and 
non-degenerate  multiplicative noise (the diffusion coefficient must be sufficiently close to  
%must be a  ``small perturbation'' of 
a fixed operator).   
%with a special ??? assumption on the diffusion multiplicative ?? coefficient. 
%(the diffusion coefficient has ). lascia 
  Such papers do not cover our main result. Indeed both in \cite{Za} and  in \cite{bass1} the term $(-A)^{1/2} F$ is replaced by $F $ 
   which   is  
   $\theta$-H\"older continuous and bounded (cf.    Hypothesis 2 in \cite{Za} and hypotheses (5.4) and (5.5) in Theorem 5.6 of \cite{bass1})
  %in \cite{Za}; on the other hand such $F$  in \cite{bass1} satisfies  additional %assumptions   (see in particular hypotheses (5.4) and (5.5) in Theorem 5.6 of %\cite{bass1}).  
   On the other hand, 
 weak uniqueness for \eqref{sde} follows by  Section 4 of  \cite{D2} assuming  that $F$ is $\theta$-H\"older continuous and bounded, $\theta \in (0,1)$,  
$
 \text{with} \; \| F\|_{C^{\theta}} \; $ $\text{small enough. }
$
%is a consequence of the results in 
     
     To establish weak uniqueness  for \eqref{sde} we first  prove   
an  optimal regularity result 
 for  the following infinite-dimensional Kolmogorov equation
% This is based on the fact that the same   estimate  \eqref{ss9} holds more generally  if  $u$ is replaced by
% the       solution  $     u^{(z)}$ of the following equation
 \begin{equation} \label{sdd}
  \lambda u -Lu -\langle z, (-A)^{1/2} D u 
\rangle= f,   
\end{equation} 
with $z \in H$; see Theorem \ref{ss13} which  is   similar to a result proved in 
%a variant  of Theorem 3.3 in
\cite{D2}.
%(see the  comments before Theorem \ref{ss13}).
In \eqref{sdd}  $\lambda >0$, $f: H \to \R$ is a given $\theta$-H\"older continuous and  bounded function (i.e., $f \in C_b^{\theta}(H)$)   and 
 $L$ is  an infinite-dimensional Ornstein-Uhlenbeck operator   which is  formally given by  
\begin{equation*} 
%\label{dd}
 L g(x) = \frac{1}{2} \mbox{Tr}(D^2 g(x)) +
 \langle Ax ,  Dg (x)\rangle, \;\; x \in D(A),
\end{equation*}
where $D g(x)$ and $D^2 g(x)$ denote respectively the
first and second Fr\'echet   derivatives of a regular function $g$ at $x \in H$  and $\langle \cdot, \cdot \rangle$ is the inner product in $H $ {  
(for regularity results concerning $L$  when $H = \R^n$ see \cite{lorenzi} and the references therein).} 
According to Chapter 6 in \cite{DZ1} (see also \cite{D2} and     \cite{DFPR})  we investigate properties of  the bounded  solution $u^{(z)} : H \to \R$,
 given by 
 \begin{equation}\label{dd9}  
u^{(z)}(x)= \int_0^{\infty} e^{-\lambda t } P_t^{(z)} f(x)dt, \;\;\; x \in H;
 \end{equation}
here $(P_t^{(z)})$ is an Ornstein-Uhlenbeck type semigroup associated to $L + \langle z, (-A)^{1/2} D u 
\rangle $, see \eqref{gt}.   When $z=0$ we write $P_t^{(0)} = P_t $ and we find the well-known Ornstein-Uhlenbeck  semigroup:

$
P_tf(x) $ $= \E[f(Z_t^x)] $ $ = \E \Big[f(e^{tA} x$ $ + \int_0^t e^{(t-s)A}  dW_s)\Big]
$;    $Z^x$ denotes the Ornstein-Uhlenbeck process which  solves \eqref{sde}
 when $F=0$ (cf. Section 1.2).    

     It easy to prove that  $u^{(z)} \in C^1_b(H) $, i.e., $u^{(z)}$ is continuous and bounded with  the first {   Fr\'echet} derivative $Du^{(z)} : H \to H$ which is continuous and bounded.   
  In  Theorem \ref{ss13}   we prove   that  $D u^{(z)} (x) \in D((-A)^{1/2})$, for any $x \in H,$ $z \in H$,  and there exists  constants $ M_{\theta} >0 $ and $ C_{\theta}(\lambda) >0 $ (independent of $z$ and $f$)
 such that 
\begin{equation}\label{ss9}     
 \sup_{x \in H}\, | (-A)^{1/2} Du^{(z)}(x) |_{H} \, \le  C_{\theta}(\lambda) \, \| f \|_{C^{\theta}}, 
 \;\;\;
  [ (-A)^{1/2} Du^{(z)} ]_{C^{\theta}} \, \le  M_{\theta} \| f \|_{C^{\theta}}, 
\end{equation} 
with $\lim_{\lambda \to \infty} C_{\theta}(\lambda) =0$ (here $ [ \cdot ]_{C^{\theta}} $ stands for the H\"older seminorm, see \eqref{sx}).   
  The fact that 
%precise dependence of the constants 
in \eqref{ss9} the  constant $M_{\theta}$ is independent of $\lambda$ and the fact that
 $C_{\theta}(\lambda) \to 0$ are  important in the proof of Lemma \ref{stima1}; this   allows to perform the localization principle. 
 Estimates like \eqref{ss9} have not been proved in recent papers on Schauder estimates in infinite dimensions, see in particular 
%see Theorem \ref{ss13} with $z=0$ and compare with 
 \cite{bass}, \cite{CL},   \cite{LR} and the references therein.   Bounds similar to \eqref{ss9} are given  in Theorem 3.3 of \cite{D2};  we  improve the estimates in \cite{D2}
 %clarifying that the constants are independent of $z$ and 
  clarifying   the  dependence of the constants on $\lambda$ 
 %and the independence of   
  (see the remarks before  Theorem \ref{ss13}).
  
  {    The  bound 
  \begin{equation} \label{ww5}
  \begin{array}{l}
   \sup_{x \in H} | (-A)^{1/2} D P_t^{(z)} f (x) |
    = 
\| (-A)^{1/2} D P_t^{(z)} f\|_0 \sim \frac{ c_2}{t} \| f\|_0, \;\;\text{as} \;\;  t \to 0^+
\end{array}  
\end{equation}
 (see \eqref{vfg})  
 %it is  independent of $A$)
 containing 
 %the non-integrable function 
 the singular term $\frac{1}{t}$ 
 %gives an indication 
 %indicates 
 suggests  
 that \eqref{ss9}
 %is optimal. 
 cannot be    improved
 replacing $(-A)^{1/2}$ by  $(-A)^{\gamma}$, $\gamma \in (\frac{1}{2}, 1)$; 
  %(on this respect, 
 see  also 
 %the results in 
 %\cite{analityc} and 
 Chapter 6 of \cite{DZ1} and Remark \ref{mai}.}  
  
%  Note that \eqref{ss9}   is a limit case of  known estimates. Indeed  if $\theta \in (0,1)$, and $f : H \to \R $ is $\theta$-H\"older continuous and bounded then 
%\begin{equation}\label{ssp} 
% \begin{array}{l}
%   \| (-A)^{1/2} Du \|_{C^{\theta}_b(H, H)} \le  c_{\theta} \| f\|_{C_b^{\theta}(H)}
%\end{array}    
%   \end{equation} 
%is the main result in \cite{D2}.  

When $z=0$ we can mention  related optimal   regularity results   in $L^p(H, \mu)$-spaces with respect to the Gaussian invariant measure $\mu$ for $(P_t)$ (cf. Section 3 of \cite{CG1}):  
      \begin{equation}  \label{q33}
 %     \label{e11cg} \| u \|_{L^p(\mu)}^p +  \int_H \|D^2 u
%(x)\|_{}^p \, \mu (dx)
%+
\|(-A)^{1/2}D u\|_{L^{p}(\mu)} \le C_p \|f\|_{L^p(\mu)}.    
\;\;\; 1 < p < \infty.
\end{equation}
 When   $f \in L^2(\mu)$ the fact that  the estimate  $\|(-A)^{1/2}D u\|_{L^{2}(\mu)} \le C_2 \|f\|_{L^2(\mu)}$ is sharp follows by   
  Proposition 10.2.5 in \cite{DZ1}.  
%This
%shows that   $\|(-A)^{1/2}D u\|_{L^{2}(\mu)} \le C \| f\|_{L^{2}(\mu)}$ is %%sharp.
  %  The optimality of \eqref{ss9}  is also clear  by  singular   gradient   
% estimate 

We stress that for $p=\infty$ in general  the previous estimate \eqref{q33}  does not hold in infinite dimensions (cf. 
 Remark \ref{peccato}). A counterexample is given in \cite{doleraP} when $(P_t)$ is  associated to a  stochastic heat equation in one dimension.

Concerning the SPDE \eqref{sde}  we first prove   the weak existence in Section 4 (see also Remark \ref{sd}).
To this purpose we adapt a compactness argument already used in \cite{GG}  (see also Chapter 8 in \cite{DZ}).
The proof of the uniqueness part of Theorem \ref{base} is more involved and it is done in various steps (see Sections 5 and 6). In the case when $F \in C_b^{\theta}(H,H)$ we first  consider   equivalence between mild solutions and solutions to the martingale problem  of Stroock and Varadhan \cite{SV79}     (cf. Section 5.1). This allows to use some  uniqueness results available for the martingale problem (cf. Theorems \ref{ria}, \ref{uni1} and \ref{key}).  On this respect we point out that an infinite-dimensional generalization of the martingale problem is   given in Chapter 4 of \cite{EK}. 
  
   In Section 5.2 we prove   weak uniqueness when $F \in C^{\theta}_b(H,H)$  assuming an additional condition. More precisely, we show that there exists a constant ${\tilde C_0}>0$   such that if  $F \in C^{\theta}_b(H,H)$ verifies 
%    $z \in H$  if $F \in C^{\theta}_b(H,H)$ and 
  \begin{gather} \label{aa}
\sup_{x \in H} | F(x)- z |_H =  \| F- z \|_0 < {\tilde C_0}
\end{gather}
 for some $z \in H$ then weak uniqueness holds for  \eqref{sde} for any initial condition $x \in H$. 
  Note that  ${\tilde C_0}$ is a constant   small enough, depending on $\theta $ and $\|F \|_{C^{\theta}}$.  
  
  Estimate  \eqref{ss9} is needed   in order to prove that 
 \begin{gather} \label{qgg}
 \| \langle F - z, (-A)^{1/2} D u^{(z)} 
\rangle \|_{C^{\theta}} \le \frac{1}{2} \| f\|_{C^{\theta}},\;\; \; f \in C_b^{\theta}(H), 
\end{gather}   
    for $\lambda$ large enough if  $F$  verifies \eqref{aa}
    (see Lemma \ref{stima1}).  
    %Differently with respect to Section 5.3 in \cite{priolaAOP}
    We obtain weak uniqueness using \eqref{qgg} and adapting   an argument  
    used  in finite dimension in \cite{SV79} and \cite{IW} (see the proof of Theorem 3.3 in \cite{IW}). 
    %By this  
    This argument is simpler than the other approach to get uniqueness  passing through the study of the  equation 
 %method  we do not have to deal with the equation 
 $ 
  \lambda u -Lu -\langle z, (-A)^{1/2} D u 
\rangle$ $= f  + $ $\langle F - z, (-A)^{1/2} D u 
\rangle $ 
%which is a delicate issue. 
(cf. Sections 5.2 and 5.3 in \cite{priolaAOP}).

  In Section 5.3 we prove uniqueness in law when $F \in C_b^{\theta} (H,H)$ (removing condition \eqref{aa}). To this purpose we  adapt the localization principle which has been introduced  in \cite{SV79} (cf. Theorem \ref{uni1}).  %To perform such method we use  that the constants in the estimates \eqref{ss9}
   %are independent of $z$. 
  In Section 6 we complete the proof of Theorem \ref{base}, showing weak     uniqueness under  \eqref{lin1}. To this purpose 
   we truncate $F$ and 
   prove uniqueness  for the martingale problem up to a stopping time (cf. Theorem  \ref{key}). {  Section 7 considers  the more general case of $F$ which is only locally $\theta$-H\"older continuous    without imposing a growth condition.}

  We finally mention 
  recent papers 
 which investigate  pathwise uniqueness for  SPDEs with additive noise like \eqref{sde} 
%\begin{equation*} \label{s22}
% dX_t = AX_t dt + F(X(t))dt +  dW_t,\;\;  
% X_0 =x \in H,
% \end{equation*}
 when  $(-A)^{1/2}F$ is replaced by a measurable drift   $F$ 
 %and also a measurable drift $F$ can be considered     
 (cf. \cite{DFPR}, \cite{DFRV}  and  see also \cite{mytnik} for the case of semilinear stochastic heat equations). 
%In such papers $F $ can be bounde. 
%In contrast to finite dimension (see..)
% For such equations in infinite dimensions    even if  $F \in C_b(H,H)$   pathwise uniqueness, for any initial  $x \in H$,
 %for  \eqref{s22} 
  In \cite{DFPR} and  \cite{DFRV} pathwise uniqueness holds for $\mu$-a.e.  $x \in H$. It is still not clear if  pathwise uniqueness holds, for any initial  $x \in H$, when   $F \in C_b(H,H)$.   
 On the other hand if $F \in C_b^{\theta}(H,H)$ then pathwise uniqueness holds, for any $x \in H$; see  \cite{DF}. 

 \begin{remark} \label{peccato} {\em As we say before in \cite[Theorem 1]{priolaAOP} it is claimed weak uniqueness for \eqref{sde}, for any  $x \in H$, assuming only that $F: H \to H$ is continuous with at most a linear growth. Actually we do not know if this result holds or not. 
In fact the proof of \cite[Theorem 1]{priolaAOP}  uses    \cite[Theorem 7]{priolaAOP} which 
%states 
 shows in particular that 
%in particular  implies:  
%claims that  if 
 there exists $C >0$, independent of $f$ and $z$, such that 
\begin{equation}\label{magari}         
  %\begin{array}{l}
 \sup_{x \in H}\, | (-A)^{1/2} Du^{(z)} (x)|_H \, \le C\,  \sup_{x \in H}| f(x)|,\;\; f \in C_b^1(H)
 %\end{array}   
\end{equation}
 ($u^{(z)}$ is defined in \eqref{dd9}).  This estimate corresponds to the case $p =\infty$ of \eqref{q33}.  The proof of   \eqref{magari} is  based on \cite[Lemma 6]{priolaAOP} but there is a mistake in the proof of such lemma (see in particular the change of variable at the end of page 1319 in \cite{priolaAOP}). On the other hand, a counterexample given in  \cite{doleraP} shows that in general the  $L^{\infty}$-bound  \eqref{magari}   fails to hold in  infinite dimensions even with $z=0$.  %and for functions $f \in C_b(H)$.
  Theorem 1 in \cite{priolaAOP} could be true with a different proof.
   } 
\end{remark}

\begin{remark} {\em  If we replace $(-A)^{1/2} F$  in \eqref{sde} with $(-A)^{1/2 - \epsilon}F$ with $\epsilon 
\in (0, 1/2]$ then following    Sections 5 and 6  one could  prove uniqueness in law for $F: H \to H$  continuous with at most a linear growth. To this purpose one can use that for any $x \in H$, $f \in B_b(H),$   one   has $ D u^{(z)}(x) \in D( (-A)^{1/2 - \epsilon})$ and 
 $$  
\|   (-A)^{1/2 - \epsilon} D u^{(z)}  \|_{0} \le {c_{\epsilon}}   \, \| f \|_0 
%\;\; \text{with} \;\;  
%{c_{\epsilon}} (\lambda) \to 0  \;\; \text{as } \; \lambda \to \infty ??
 $$
(this follows by the estimate $\|   (-A)^{1/2 - \epsilon} D P_t^{(z)} f \|_{0} $ $\le \frac{C_{\epsilon}} {t^{1-\epsilon}} \| f \|_0,$ $ t>0,$ which can be obtained  in the same way we get     \eqref{vfg}).  However such assumption  excludes the  examples of Sections 2 and 7.} 
\end{remark}

\begin{remark} \label{mai}   {\em      It is not clear if the   uniqueness result holds for \eqref{sde} when $(-A)^{1/2}$ is replaced by $(-A)^{\gamma}$, $\gamma \in (1/2,1)$.
 We believe that  for  $\gamma \in (1/2,1)$ there  should exist a $\theta$-H\"older continuous and bounded drift $F_{\gamma}: H \to H$, $\theta \in (0,1)$,  and $x_{\gamma} \in H$ such that weak uniqueness fails for  
  $dX_t = AX_t dt $ $+ (-A)^{\gamma}F_{\gamma} (X_t)dt +  dW_t,$ $  X_0 =x_{\gamma} $ (on the other hand, weak existence holds, cf. Remark \ref{sd}). 
 In this sense \eqref{sde}  can be considered   as a  critical SPDE.}
\end{remark}

\subsection{ Notations and preliminaries}

Let $H$ be a   real separable Hilbert space.  Denote its
norm and inner product  by $\left\vert \cdot \right\vert_H $
 and 
$\left\langle \cdot , \cdot \right\rangle $ respectively. Moreover  ${\mathcal B}(H)$ indicates its Borel $\sigma$-algebra.  
 Concerning \eqref{sde} 
 as in \cite{D2}, \cite{DFPR} and  \cite{DFRV} 
 we  assume  
 %  (cf. \cite{D2}, \cite{DFPR}, \cite{DFRV})
\begin{hypothesis} \label{d1} 
 $A:D(A)\subset H\to H$ is a negative definite
self-adjoint operator with domain $D(A)$ (i.e., there exists $\omega >0$ such that  $\langle Ax, x\rangle \le - \omega |x|^2_H$, $x \in D(A)$). Moreover
 $A^{-1}$ 
 is a trace class operator.
 \end{hypothesis} 
 In the sequel we will concentrate on an infinite dimensional Hilbert space  $H.$ 
 Since $A^{-1}$ is compact, there exists an orthonormal basis $(e_k)$
in $H$ and an infinite sequence of positive numbers $(\lambda_k)$ such that
 \begin{equation} \label{e1a}
 \begin{array}{l}
  Ae_k=-\lambda_k e_k,\quad k\ge 1,\;\; \text{and } \;\; \sum_{k \ge 1} {\lambda_k^{-1}} < \infty.
\end{array}
\end{equation} 
% Moreover $\lambda_k \to \infty$ 
  Note that $D(A)$ is dense in $H$.
We denote by  ${\mathcal L}(H)$  the Banach space of  bounded and linear operators $T: H \to H$ endowed with the operator norm $\| \cdot \|_{\mathcal  L}.$  
  The operator $A$ generates an analytic semigroup $(e^{tA})$ on $H$
such that $e^{tA} e_k = e^{- \lambda_k t } e_k$, $t \ge 0$.  
 Remark that  
\begin{equation} \label{ewd}
\begin{array}{l}
 \|  (-A)^{1/2} e^{tA} \|_{\cal L} = \sup_{k \ge 1}\, \{  (\lambda_k)^{1/2} e^{-\lambda_k t} \} \le \frac{c}{ \sqrt{t}},\;\;\; t>0,
\end{array}   
 \end{equation}
with  $c= \sup_{u \ge 0} u e^{-u^2}= (2e)^{-1/2}$.  We will also use orthogonal projections with respect to $(e_k)$:
\begin{equation} \label{pp1}
 \begin{array} {l}
\pi_m=\sum_{j=1}^me_j\otimes e_j, \;\;\; 
\pi_m x = \sum_{k=1}^m x^{(k)} e_k, \;\; \text{where $x^{(k)} = \langle x, e_k\rangle $, $x \in H$, $m \ge 1$. }      
 \end{array}
\end{equation} 
   Let $(E, |\cdot |_E)$ be a real separable Banach space. We denote by 
${B}_b(H, E)$
  the Banach space of all real, bounded and Borel functions on
  $H$ with values in $E$, endowed with the supremum norm $\| f \|_0 = \sup_{x \in H} |f(x)|_E$, $f \in {B}_b(H, E).$ Moreover 
  $C_b(H, E) \subset  B_b(H, E)  $ indicates the subspace of all   bounded and continuous  functions.
 We denote by $C^{k}_b (H,E) \subset {B}_b(H, E)$, $k \ge
1$, the  space of all functions $f: H \to E$ which are bounded
and Fr\'echet differentiable on $H$ up to the order $k \ge 1$ with
all the derivatives $D^j f$ bounded and continuous on $H$, $1 \le j \le k$. 

Moreover $C^{\theta}_b(H,E)$, $\theta \in (0,1)$, denotes the Banach space of all functions $f :H \to E$ which are $\theta$-H\"older continuous and bounded endowed with the norm
\begin{equation}\label{sx}
\| f\|_{C^{\theta}} = \| f\|_0 + [f]_{C^{\theta}},
\end{equation}
where $[f]_{C^{\theta}} =  \displaystyle{ \sup_{x\neq x'\in H}
 {(|f(x)-f(x')|_E}\, {|x-x'|^{-\theta}_H}).}$

We also set $B_b(H) = B_b(H, \R), C_b(H) = C_b(H,\R)$, $C_b^{\theta}(H) = C_b^{\theta}(H, \R)$ and $C^{k}_b (H, \R) = C^{k}_b (H)$.  

Let $\theta \in (0,1)$. We say that $F: H \to H$ is locally $\theta$-H\"older continuous if for any bounded set $B \subset H$, we have that
 $
 F: B \to H
 $
 is $\theta$-H\"older continuous. Note that this condition implies that $F$ is  is  bounded on each bounded set of $H$.

\smallskip   
We  will deal with the SPDE \eqref{sde} 
%\begin{equation} \label{uno1}
%dX_{t}=AX_{t}dt +  (-A)^{1/2}F(X_{t})dt+ dW_{t},\qquad X_{0}=x\in H,
%\end{equation}
 where  $W = (W_t)$ $= (W(t))$ is a {\sl cylindrical Wiener} process on $H$.  Thus   
  $W$ is formally given by ``$W_t = \sum_{k \ge 1}
  W^{(k)}_t e_k$'' where  $(W^{(k)})_{k \ge 1}   $   are independent real Wiener processes and   $(e_k)$ is the basis of eigenvectors  of $A$ (cf. \cite{DZ1},  \cite{hairer} and \cite{DZ}).   
 The next definition is meaningful for $F: H \to H$ which is only continuous  because of \eqref{ewd}. 
  
 % Recall that we also assume \eqref{lin1}.   
% We assume that    
%\begin{equation*} \label{d11}
% \text {  $F : H \to H$ is   continuous and } \;\; 
% |F(x)|_H \le C_{F}  (1+  |x|_H),\;\;\; x \in H.
%\end{equation*}
  %Let $x \in H$.
  
 \vskip 1mm 
  A  { \sl  weak mild solution}  to 
(\ref{sde})    is a sequence $( 
\Omega,$ $ {\mathcal F},
 ({\mathcal F}_{t}),$ $ \P, W, X ) $, where $(
\Omega, {\mathcal F},$ $
 ({\mathcal F}_{t}), \P )$ is a  filtered probability space  
  on which it is defined a
cylindrical Wiener process $W$ and
 an ${\cal F}_t$-adapted,   $H$-valued
continuous process $X$ $ = (X_t)$ $ = (X_t)_{t \ge 0}$ such that, $\P$-a.s., 
 \begin{equation}
  \label{mqq}
X_{t}= \,  e^{tA}x \, + \, \int_{0}^{t}(-A)^{1/2}e^{(  t-s)  A}F  (
X_{s})
ds
+\int_{0}^{t}e^{(  t-s)  A}dW_{s},\quad  t \ge 0.
\end{equation}
(hence  $X_0 =x$, $\P$-a.s.). 
 We say that  {\sl uniqueness in law holds for \eqref{sde} for any $x \in H$ } if  given two weak mild solutions   $X$ and $Y$     (possibly defined on different filtered probability spaces and starting at $x \in H$), we have that  $X$ and $Y$ have the same law on ${\cal B}(C([0, \infty); H))$ which is the Borel $\sigma$-algebra of $C([0, \infty); H)$  (this is the  Polish    space of  all continuous functions  from $[0, \infty)$ into $H$ endowed with the metric of the uniform convergence on bounded intervals; cf. \cite{KS} and  \cite{DZ}).
 %\cancel{\color{green} Equation \eqref{mqq} is  meaningful because   of  %\eqref{ewd}.}
  Note that   the stochastic convolution
$$    
  \begin{array} {l}
W_A(t) =\int_{0}^{t}e^{\left(  t-s\right)  A}dW_{s} = \sum_{k \ge 1}
\int_0^t e^{-(t-s) \lambda_k}e_k dW^{(k)}(s)
\end{array} 
$$
is well defined since 
 the series converges in $L^2(\Omega; H)$, for any $t \ge 0$. Moreover $W_A(t)$ is a Gaussian random variable with values in $H$ with distribution $N(0,Q_t)$
 where 
 \begin{equation}\label{qt1}
 \begin{array} {l}
Q_t = \int_0^t e^{2 sA} ds =  (-2 A)^{-1}(I-e^{2tA}),\;\;\; t \ge 0,
\end{array}
\end{equation}
 is   the covariance operator
(see also Chapter 1 in \cite{DZ1}).   
 %%% correzione W_A
Note that $ W_A $ has a continuous version with values in $H$ (see Corollary 2 in  \cite{Tala});   if we assume in addition that $(-A)^{-1 + \delta}$ is of trace class, for some $\delta \in
(0,1)$, then  this fact follows by Theorem 5.11 in \cite{DZ}.

Equivalence between {   different} notions of solutions for \eqref{sde}  
  are clarified  in \cite{DZ1} and \cite{hairer} (see also  \cite{kunze} for a more general setting).    If we write  $ X^{(k)}(t)=  X^{(k)}_t = \langle  X(t), e_k \rangle $, $k \ge 1$, \eqref{sde} is equivalent to the system 
\begin{gather} \label{d33}
   X^{(k)}_t =  x^{(k)} -    \lambda_{k} \int_0^t   X^{(k)}_s ds 
  \, + \, \lambda_{k}^{1/2} \, \int_0^t F^{(k)}( X_s )ds + W_{t}^{(k)},\;\;\; k \ge 1,
\end{gather}     
or to 
 $ \ds  
X^{(k)}_t =  e^{- \lambda_k t}  x^{(k)}    +
  \int_0^t e^{-\lambda_k (t-s)}(\lambda_k)^{1/2}F^{(k)}(X_s)  ds +  \int_0^t e^{-\lambda_k (t-s)} d W^{(k)}_s,
$ 
for $k \ge 1$, $t \ge 0$, with 
$
F(x) = \sum_{k \ge 1}   F^{(k)}(x)e_k, \;\;\; x \in H.   
$

 We will also use the natural filtration of $X$ which is  denoted by  $({\cal F}_t^X)$; 
  ${\cal F}_t^X = \sigma(X_s \, : \, 0 \le s \le t)$ is the  $\sigma$-algebra generated by the r.v. 
 $X_s$, $0 \le s \le t$ (cf. Chapter 2 in \cite{EK}).
\begin{remark} \label{serve} {\em   We point out that Theorem \ref{base}  
   holds under the following  more general hypothesis:
$A:D(A)\subset H\to H$ is self-adjoint, $\langle A x,x  \rangle \le 0$, $x \in D(A)$, and
   $(I -A)^{-1 }$ 
is  of  trace class, with $I = I_H$.
 Indeed in this case  one can  rewrite  equation \eqref{sde} in the form
 $$ 
d X_{t}=(A -I) X_{t}  dt \, 
+ \,  (I - A)^{1/2} [(I -A)^{-1/2} X_t     + (-A)^{1/2} (I -A)^{-1/2}  F(X_{t})] dt \, + \, dW_{t},
$$ 
$ X_{0}=x. $  Now the linear operator   
  $\tilde A = I-A$ and the nonlinear term   $\tilde F(x)=[(I -A)^{-1/2} x         + \, (-A)^{1/2} (I -A)^{-1/2}  F(x)],$
% $$
% \tilde F(x)=[(I -A)^{-1/2} x         + \, (-A)^{1/2} (I -A)^{-1/2}  %F(x)],
% $$
  $ x \in H,$ verify    Hypothesis \ref{d1} and condition    \eqref{lin1} respectively.  
 }
\end{remark}

\subsection {A generalised Ornstein-Uhlenbeck semigroup}    
  
Let us fix $z \in H$.  We will consider  generalised Ornstein-Uhlenbeck operators like  
\begin{equation} \label{ou3}
L^{(z)} g(x) = \frac{1}{2} \mbox{Tr}(D^2 g(x)) +
 \langle x , A Dg (x)\rangle +  \langle z , (-A)^{1/2} Dg (x)\rangle , \; x \in H, \;\; g \in C^2_{cil}(H).
\end{equation}
Here $C^2_{cil}(H)$ denotes the space of {\sl regular 
cylindrical functions.} We say that  $g: H \to \R$ 
   belongs to $C^2_{cil}(H)$ if there exist elements $e_{i_1}, \ldots, e_{i_n}$ of the basis $(e_k)$ of eigenvectors of $A$ and a $C^2$-function  $\tilde g : \R^n \to \R$ with compact support such that 
\begin{gather} \label{cil2}
 g(x) = \tilde g (\langle  x, e_{i_1}\rangle, \ldots, \langle  x, e_{i_n} \rangle),\;\;\; x \in H.
\end{gather}  
By writing the stochastic equation $dX_t = AX_t dt +  (-A)^{1/2} z dt + dW_t,$ $ X_0 =x$ in mild form as  
$
X_t = e^{tA} x $ $+ \int_0^t e^{(t-s)A} dW_s $ $ +  \int_0^t (-A)^{1/2}e^{(t-s)A} z\,  ds, 
$
 one can easily check that the Markov semigroup associated to 
 $L^{(z)}$ is a  generalized Ornstein-Uhlenbeck semigroup $(P_t^{(z)})$:
\begin{gather}\label{gt}
 \begin{array}{l}  \ds 
P_t^{(z)} f (x) \, 
%=\, \int_{H} f(e^{tA} x+ y + (-A)^{-1/2}[z- e^{tA}z] )  \; { N} \big (0 , %
% Q_t) \big)\, (dy),\;\; 
 =\, \int_{H} f(e^{tA} x+ y + \Gamma_t z  )  \; {    N  (0 , Q_t)} \, (dy),\;\; 
f \in { B}_b (H),\; x \in H,
\\   
 \text{ setting $\Gamma_t = (-A)^{1/2} \int_0^t e^{sA} ds, $  } 
 \;\;\;\;\; 
 \Gamma_t z = (-A)^{-1/2}[z- e^{tA}z] 
   = \sum_{k \ge 1} \, \frac {(1- e^{-t \lambda_k })}{(\lambda_k)^{1/2}} z^{(k)} \,  e_k. 
 \end{array} 
\end{gather} 
The  case $z=0$. i.e., $(P_t^{(0)})= (P_t)$ corresponds to the well-known Ornstein-Uhlenbeck semigroup (see, for instance, \cite{DZ1}, \cite{DZ}, \cite{D2}, \cite{DFPR} and \cite{DFRV}) which has  a unique invariant measure $\mu =N(0,S)$
where $S=-\frac12\;A^{-1}$. 
  It is also well-known (see, for instance, \cite{DZ1} and  \cite{DZ}) that under Hypothesis \ref{d1}, $(P_t)$ is strong Feller, i.e, $P_t (B_b(H)) \subset C_b(H)$, $t>0$.  
  Indeed we have $e^{tA}(H) \subset Q^{1/2}_t(H)$, $t>0$, or, equivalently,
\begin{equation}
\label{lam}  
\begin{array} {l} 
 \Lambda_t=Q_t^{-1/2}e^{tA}=\sqrt 2\;
(-A)^{1/2}e^{tA}(I-e^{2tA})^{-1/2} \in {\mathcal L}(H),\;\; t>0.
\end{array} 
\end{equation} 
Moreover $P_t (B_b(H)) \subset C_b^{k}(H)$, $t>0$, for any $k \ge 1$. Following the same proof of Theorem 6.2.2 in \cite{DZ1} one can show that 
 under Hypothesis \ref{d1}, for any $z \in H,$ we have 
$P_t^{(z)} (B_b(H)) \subset  C_b^{k}(H)$, $t>0$, for any  $k \ge 1$.
  Moreover, for any $f \in C_b(H)$, $t>0$, the following formula for the directional derivative along a direction $h$ holds: 
\begin{equation}
\label{e3} 
D_h P_t^{(z)} f(x) = 
\langle D P_t^{(z)} f(x),h
 \rangle  = \int_H \langle
  \Lambda_t h,Q_t^{-\frac12} y\rangle \, f (e^{tA}x+y+ \Gamma_t z) \mu_t(dy),
  \; x,  h \in H,
\end{equation}
 where $\mu_t = N(0,Q_t)$ (cf. \eqref{qt1}) and the mapping: $y \mapsto \langle
  \Lambda_t h,Q_t^{-\frac12} y\rangle$ is a centered
   Gaussian random variable
 on $(H, {\cal B}(H),\mu_t)$  with variance $ |\Lambda_t h|^2$
   (cf. Theorem 6.2.2 in \cite{DZ1}). 
   
   We deduce that,  for $t >0$, $g   \in C_b(H)$, $h,k \in H,$
   \begin{equation}\label{wdc}
 \begin{array}{l}
 \| D_h P_t^{(z)} g\|_0 \le |\Lambda_t  h|_H \,  \| g\|_0, \;\;\; 
 %  \frac{C_1}{\sqrt{t}}|h| \| g\|_0
 \| D^2_{hk} P_t^{(z)} g\|_0 \le 
 %\frac{\sqrt{2}\, C_1^2}{{t}} 
  |  \Lambda_t h|_H \, |\Lambda_t k|_H \| g\|_0, 
\end{array} 
 \end{equation}
where $D_h P_t^{(z)} g= \langle D P_t^{(z)} g(\cdot),h \rangle $, $D^2_{hk} P_t^{(z)} g
 =  \langle D^2 P_t^{(z)} g(\cdot)h,k \rangle $.   
   We have 
 \begin{equation} \label{e5}  
 %\label{ffd}
 \begin{array}{l}
 \ds   \Lambda_t e_k =\sqrt
2\;(\lambda_k)^{1/2}e^{-t\lambda_k}(1-e^{-2t\lambda_k})^{-1/2} e_k,
\\  \ds   
%\text{and so} 
\;\,  \| \Lambda_t \|_{\cal L}\le C_1
t^{-\frac{1}{2}}, \;\; t>0, \; \; C_1 = \sqrt
2 \cdot \sup_{u \ge 0} { [u \, e^{-u^2}}{(1- e^{-2u^2})^{-1/2}}].
\end{array}
\end{equation}    
and so    $ \| D_h P_t^{(z)} g\|_0 \le \frac{C_1}{\sqrt{t}}|h|_H \,  \| g\|_0$, 
$ \| D^2_{hk} P_t^{(z)} g\|_0 \le \frac{\sqrt{2}\, C_1^2}{{t}} \| g\|_0 |h|_H \, |k|_H$. 
 
 \smallskip 
  To study equation \eqref{sdd} we will investigate regularity properties of the continuous function
 \begin{equation}\label{wv}
u^{(z)}(x)=
%R^{(z)}(\lambda)f (x)=
 \int_0^{\infty} e^{-\lambda t }   P_t^{(z)} f(x)dt, \; \;\; x \in H,\; f \in C_b(H) 
\end{equation}
 (we drop the dependence of $u^{(z)}$ on $\lambda $); see  also the remark below. 
  \begin{remark} \label{ss} 
  {\em  Let us fix $z \in H$.   
  % We point out that under Hypothesis \ref{d1}
 %when  $f \in C_b(H)$ and $x \in H$,  the mapping:  $t \mapsto P_t^{(z)} f(x)$ is right-continuous and bounded on $(0, \infty)$ by the semigroup property and   the strong Feller property. Hence, 
 For any $\lambda>0,$ $u^{(z)}: H \to \R$  given in \eqref{wv}   belongs to $C_b(H)$.    
%  we can consider, for any $\lambda>0,$  the continuous and bounded function $u^% 
% {(z)}: H \to \R$ given in \eqref{wv}. 
% \\
Moreover, also 
the mapping: $t \mapsto D_h P_t^{(z)} f(x)$ is right-continuous  on $(0, \infty)$, for $x,h \in H 
$.
%To check this fact let us fix $t>0$. Writing $D_{ h} {P_{t+s}^{(z)}}f(x) =   D_{ h} P_{s+ \frac{t}{2}}^{(z)} [P_{t/2}^{(z)} \, f](x)$, $s \ge 0,$ and using the strong Feller property we get easily the assertion.

Since  $\sup_{x \in H}|D P_t^{(z)} f(x)|_H \le \frac{c \| f\|_0}{\sqrt{t} }$, $t>0,$ differentiating under the integral sign, one  shows that there exists 
the directional derivative $D_h u^{(z)}(x)$ at any point $x \in H$ along any direction $h \in H$. 
    Moreover, it is not difficult to prove that  there exists  
the  first {   Fr\'echet}  derivative  $Du^{(z)}(x)$ at any $x \in H$  and  $Du^{(z)} : H \to H$ is continuous and bounded (cf. the proof of Lemma 9 in \cite{DFPR}).  Finally we have the formula
 \begin{equation}\label{uu1}  
\;\; \; D_h u^{(z)}(x) =   
 \int_0^{\infty} e^{-\lambda t } D_h P_t^{(z)} f(x)dt,\;\;\; x,h \in H
\end{equation} 
and the straightforward estimate $\| Du^{(z)} \|_0 \le c(\lambda) \| f\|_0$ with $c(\lambda)$  independent of $z \in H$. We will prove a  better  regularity result for $Du^{(z)}$ in Section 3. 
} 
\end{remark}   
 
% \begin{remark} \label{maa} {\em 
% In the final part of the  proof of Lemma \ref{ss1} we will  need to use  that
% the following result. For any $t>0$ we have 
%\begin{equation} \label{mas}
% $\Gamma_t (H ) \subset Q_t^{1/2} (H),$ $  t >0$  
%\end{equation}
%(cf. \eqref{gt}). Note that   this is equivalent to {   saying} that   $ Q_t^{-1/2} \Gamma_t \in {\cal L}(H)$, $t>0$, and    we have
 %\begin{gather*}
  %$Q_t^{-1/2} \Gamma_t $ $
 % = \sqrt{2} (-A)^{1/2} (I -e^{2tA})^{-1/2} (-A)^{-1/2}
 %(I -e^{tA}) \\
 %= \sqrt{2} (I -e^{2tA})^{-1/2}  
 %(I -e^{tA}) = \sqrt{2}    (I + e^{tA})^{- 1/2}  
 %(I -e^{tA})^{1/2} $ $ \in {\cal L}(H). $   
  %\end{gather*}   
 %} 
 %\end{remark}

%\vskip 1 mm 
 \section{Examples}  
\subsubsection{  One-dimensional  stochastic Burgers-type equations      
%in 3D  dimensions 1, 2 and 3}
}   We consider 
  \begin{equation} \label{bur0}
d u (t, \xi)=   \frac{\partial^2}{\partial   \xi^2}  u(t, \xi)dt  +   \frac{\partial }{\partial \xi} {   h( }\xi, u(t, \xi))dt + dW_t(\xi), \;\; u(0, \xi) = u_0(\xi), \;\;\; \xi \in (0,\pi),
 \end{equation}
 with Dirichlet boundary condition $u(t,0) = u(t,\pi)=0$, $t>0$ (cf. \cite{Gy} and \cite{D2} and see the references therein). Here $u_0 \in H = L^2(0,\pi)$ and $A = \frac{d^2}{d  \xi^2}$ with Dirichlet boundary conditions, i.e. $D(A) = H^2(0, \pi) \cap H^1_0 (0, \pi)$. It is well-known that $A$ verifies Hypothesis \ref{d1}. The
eigenfunctions are
 $
 e_k (\xi) = \sqrt{2/\pi} \,  \sin (k \xi), $ $ \xi \in \R,\;\; k \ge 1.
 $
 
 The  eigenvalues
 are $-\lambda_k$, where $\lambda_k = k^2 $. 
  The cylindrical noise is $W_t(\xi) = \sum_{k \ge 1 }  
   W_t^{(k)} e_k(\xi)$ (cf. \cite{DZ}).
   {  Classical stochastic Burgers equations with ${   h( }\xi, u) = \frac{u^2}{2}$ are examples of locally monotone SPDEs and strong uniqueness holds (cf. \cite{brezniak}).}
   In \cite{Gy} {  strong uniqueness  is proved}
  %even with non-degenerate multiplicative noise 
  assuming that ${   h( }\xi, \cdot )$ is locally Lipschitz with a linearly growing Lipschitz constant.  
 %In \cite{D2} is assumed that $f: \R \to \R$ is bounded and $\theta $
  
  Here we   assume that {\sl  $h: [0, \pi] \times \R \to \R$ is continuous in both variables; 
  moreover 
   $h = h_1 + h_2$  where  $h_1$ is $\theta$-H\"older continuous in the second variable, for some $\theta \in (0,1)$,  and $h_2$ is Lipschitz continuous in the second variable, uniformly with respect to the first variable.} Hence  we assume that 
   there exists $C_{\theta}>0$ such that   
 $$
 |   h_1( \xi, s) - h_1(\xi, s')| \le C_{\theta} |s- s'|^{\theta},
 $$
 $s , s' \in \R$, $\xi \in [0, \pi]$; $h_2$ verifies a similar condition with $\theta =1$. 
 %(more generally, one could impose Carath\'eodory type conditions on $h$). 
 It is easy to prove that  
 the  Nemiskii operator: 
   $x \in H \mapsto {S (x)=  h( }  \cdot, x(\cdot)) \in H$ is  $\theta$-H\"older continuous from $H$ into $H$  if $h_2=0$. In the general case it is locally $\theta$-H\"older continuous. Moreover  it has at most a linear growth. 
   
 %in order that the Nemiskii operator: $x \in H \mapsto f(\xi, x)  $ is %continuous from $H$ into $H$).
  %Arguing as in Example 4.4 of  \cite{D2},

  To write \eqref{bur0} in the form \eqref{sde} we define $F: H \to H$ as follows
$$ 
F(x)(\xi) = (-A)^{-1/2}\,  \partial_{\xi}  [{   h( }\cdot , x(\cdot))](\xi),\;\;\; x \in  L^2(0,\pi)= H. 
$$
To check that  $F$ verifies \eqref{lin1} it is enough to prove  that 
$T= (-A)^{-1/2} \, \partial_{\xi}$ can be extended to a bounded linear {   operator} from $L^2(0,\pi)$ into $L^2(0,\pi)$. 
We briefly verify this fact.
 Recall that the domain $D((-A)^{1/2})$ coincides with the Sobolev space $ H^1_0(0, \pi)$.
Take $y \in H^1_0(0, \pi)$ and $x \in L^2(0, \pi)$. Define $x_N = \pi_N x$ (cf. \eqref{pp1}). 
 Using that $(-A)^{1/2}$ is self-adjoint and integrating by parts we find (we use inner product in $ L^2(0, \pi)$ and the fact  that $y(0) = y(\pi)=0$)
\begin{gather*} 
 \langle (-A)^{-1/2}  \partial_{\xi} \, y , x_N \rangle = \langle   \partial_{\xi} y , (-A)^{-1/2} x_N \rangle =  - \langle    y , \partial_{\xi} (-A)^{-1/2} x_N \rangle.
\end{gather*}   
 Now   $ \partial_{\xi} (-A)^{-1/2} x_N (\xi)    =  \sqrt{2/\pi} \sum_{k=1}^N  x^{(k)} \cos (k \xi)$ and so  $| \partial_{\xi} (-A)^{-1/2} x_N  |_{L^2(0,\pi)}^2$ $=  |x_N  |_{L^2(0,\pi)}^2$.

 It follows that, for any $N \ge 1$,
$| \langle (-A)^{-1/2}  \partial_{\xi} y , x_N \rangle | $ $\le |y|_{L^2(0,\pi)} \, 
 |x |_{L^2(0,\pi)}$
 and we easily get the assertion. 
 Hence $F = T \circ S$ verifies \eqref{lin1} and {\sl  SPDE \eqref{bur0} is  well-posed in weak sense, for any initial condition $u_0 \in L^2(0,\pi)$.}
   
{ Note that instead of ${   h( }\xi, u)$ one can consider different non-local nonlinearities like, for instance, $u   \, g(|u|_H)$ assuming that $g: \R \to \R$ is bounded and  locally $\theta$-H\"older continuous, for some $\theta \in (0,1)$. }

Indeed let $M>0$; if $u,v \in B =\{ x \in H \, :\, |x|_H \le M \}$ we have 
\begin{gather} \nonumber
\int_0^{ \pi} |u(t) g(|u|_H) - v(t) g(|v|_H)|^2 dt \\ 
\nonumber \le 
2 \int_0^{ \pi} |u(t)|^2 \, |g(|u|_H) - g(|v|_H)|^2 dt +  2 \int_0^{ \pi} |g(|v|_H)|^2
\, |u(t)  - v(t) |^2 dt
\\ \nonumber \le 
2 C_{M, \theta}\int_0^{ \pi} |u(t)|^2 dt \, | u - v |^{2\theta}_H  +  2 \|g \|_0^2 \int_0^{\pi} 
\, |u(t)  - v(t) |^2 dt
\\  \label{bene} 
\le K  \, | u - v |^{2\theta}_H,
\end{gather}
for some constant $K$ possibly depending on $M, g $ and $\theta$, and   assumption \eqref{lin1} follows easily.

\subsubsection{ Three-dimensional stochastic Cahn-Hilliard equations      
}  

 The  Cahn-Hilliard equation is a  model to describe phase separation in a binary alloy and some other media, in the presence of thermal fluctuations; 
  we refer to \cite{NC} for a survey
 on this model. The stochastic Cahn-Hilliard equation
  has been recently much investigated  under monotonicity conditions on $h$ which allow to prove pathwise uniqueness; in one dimension a typical example is ${   h( }s) = s^3 -s$  (see \cite{EM}, \cite{DD}, \cite{NC}, \cite{ES} and the references therein).   
 
 We can treat  such SPDE in one, two or three dimensions. Let us  consider   Neumann boundary conditions in a regular bounded open set $G \subset \R^3$. For the sake of simplicity we concentrate on the cube $G = (0, \pi)^3$.
  The equation has the  form 
 \begin{equation} \label{all}\begin{cases}
d u (t, \xi)= -  \triangle^2_{\xi} u(t, \xi)dt + \triangle_\xi {   h( }u(t, \xi))dt + dW_t(\xi), \;\;t>0,\;\;  u(0, \xi) = u_0(\xi)\;\; \text{on $G$}, 
\\
\frac{\partial}{\partial n} u = \frac{\partial}{\partial n} (\triangle u) =0 \;\; on \; \partial G,
\end{cases}
\end{equation}    
 where $\triangle_{\xi}^2 $ is the bilaplacian and $n$    is the outward unit normal vector on the boundary $\partial G$.  Let us introduce the Sobolev spaces $H^j(G) = W^{j,2}(G)$ and the Hilbert space $H$,
 $$
 \begin{array}{l} 
 H = \big \{ f \in L^2 (G) \, :\, \int_G f(\xi) d\xi =0 \big \}.
 \end{array} 
 $$
 %To be in  the setting of \eqref{sde}, 
 We assume $u_0 \in H$ and  define $ D(A) = \{ f \in H^4 (G) \cap H \, :\,  \frac{\partial}{\partial n} f = \frac{\partial}{\partial n} (\triangle f) =0 $    on $ \partial G \big \}, $ $A f = - \triangle^2_{\xi} f$, $f \in D(A)$. Using also the divergence theorem, we have $A : D(A) \to H$.
 
 The square root has domain $D[(-A)^{1/2}]   = \{ f \in H^2 (G) \cap H \, :\,  \frac{\partial}{\partial n} f  =0 $  on $ \partial G\big \};$ $- (-A)^{1/2} f = \triangle_{\xi} f$, $f \in D[(-A)^{1/2}]$. 
 Note that 
 % the hypotheses of Remark \ref{serve}.  
%(in order to verif Hypothesis \ref{d1} we should consider  
 $A$   is self-adjoint with compact resolvent  and it is   negative definite with $\omega =1$ (cf. Hypothesis \ref{d1}). 
  The
eigenfunctions are
 $$
 e_k (\xi_1, \xi_2, \xi_3) = (\sqrt{2/\pi})^3 \cos (k_1 \xi_1) \cos (k_2\xi_2)
 \cos(k_3 \xi_3), \;\;\; \xi = (\xi_1, \xi_2, \xi_3) \in
 \R^3,  $$
 $k=(k_1, k_2, k_3) \in
 \N^3$, $k \not = (0,0,0)= 0^*$.
 The corresponding  eigenvalues  
 are $-\lambda_k$, where $\lambda_k =  (k_1^2 + k_2^2 + k_3^2 )^2$.
 Since $\sum_{k  \in \N^3, \, k \not = 0^*} \, \lambda_k^{-1} < + \infty 
 $  we see that $A$ verifies Hypothesis \ref{d1}. 
The cylindrical Wiener process  is     $W_t(\xi) = \sum_{k \in \N^3,\; k \not = 0^*}  
   W_t^{(k)} e_k(\xi)$. 
   Note that  $\triangle_\xi {   h( }u(t, \xi)) =  \triangle_\xi \big [ 
  {   h( }u(t, \xi)) - \int_G  {   h( }u(t,\xi))d\xi\big] $.
     
Assuming that {\sl $h = h_1 + h_2$, with $h_1, h_2  : \R \to \R$,  where $h_2$ is Lipschitz continuous and $h_1$ is $\theta$-H\"older continuous, $\theta \in (0,1)$,
%and verifies $|{   h( }s)| \le c(1 + |s|)$, $s \in \R$,} 
}
we can define $F : H \to H$ as follows:
$$
\begin{array}{l} 
F(x)(\xi) = {   h( }x(\xi)) - \int_G  {   h( }x(\xi))d\xi,\;\;\; x \in H,\; \xi \in G.  
\end{array} 
$$
 It is not difficult to prove that  $F$ verifies \eqref{lin1}. Thus {\sl SPDE \eqref{all} is well-posed in weak sense, for any initial condition $u_0 \in H$.  
 }
 
 \section {An   optimal  regularity result  }

Let $f \in C_b^{\theta}(H)$, $\theta \in (0,1),$ and fix $z \in H$. Here we are interested in the regularity property of the function $u^{(z)}: H \to \R$ given in \eqref{wv}. 
  By Remark \ref{ss} we know that  $u^{(z)} \in C^1_b(H)$ and we have a formula for the directional derivative:  
\begin{gather} \label{s55} 
  D_h u^{(z)}(x) = \langle Du^{(z)}(x), h \rangle = \int_0^{\infty} e^{- \lambda t}  D_{ h} P_t^{(z)} f(x) dt,\;\;\; x,h \in H, \; \lambda >0. 
\end{gather}
 %Before stating the optimal regularity result 
% Recall that 
% Section 3 of \cite{D2}  considers a little different generalized OU operator 
% $L^{(z)} $  where the term $\langle z , (-A)^{1/2} Dg (x)\rangle$ in  
%  \eqref{ou3} is replaced by $\langle z , Dg (x)\rangle$.  

 %In the next result the constants $C_{\theta}(\lambda)$ and $M_{\theta}$ depend respectively on $\theta$ and $\lambda$ and on $\theta$. 
 The next assertion (i)    is similar to  Theorem 3.3 in \cite{D2}
  (recall that 
 Section 3 of \cite{D2}  considers a little different generalized OU operator 
 $L^{(z)} $  where the term $\langle z , (-A)^{1/2} Dg (x)\rangle$ in  
  \eqref{ou3} is replaced by $\langle z , Dg (x)\rangle$). Such theorem  shows    that 
  $\| (-A)^{1/2} Du^{(z)} \|_{C_b^{\theta}} $ $= [ (-A)^{1/2} Du^{(z)} ]_{C_b^{\theta}} 
 + \| (-A)^{1/2} Du^{(z)} \|_{0}  $ $\le K_{\theta}(\lambda) \|f\|_{C^{\theta}}$ with
 $K_{\theta}(\lambda)$ which is independent 
 %of $\lambda$ and 
  of $f$ (see also Remark \ref{estendi}). Below we will improve the result
  %the dependence of 
 by clarifying the dependence of  $K_{\theta}(\lambda)$ on $\lambda$.
  %on $\lambda$. 
  % We will improve the dependence of the costant  
   Indeed  the fact that in \eqref{w44} we have $C_{\theta}(\lambda) \to 0$ will be important in the proof of Lemma \ref{stima1}.
    %which is used to perform the localization principle in Section 5.3. 

\begin{theorem}\label{ss13} Let $f \in C_b^{\theta}(H)$, $\lambda >0$, $z \in H$ and consider $u^{(z)} \in C^1_b(H)$ given in \eqref{wv}.
 The following assertions hold.  
\\ 
(i) For any $x \in H$,
$Du^{(z)}(x) \in D((-A)^{1/2})$ and $(-A)^{1/2} Du^{(z)} \in C_b^{\theta}(H, H)$.
%$(-A)^{1/2} Du^{(z)} : H \to H $ belongs to 
  There exist constants $C_{\theta}(\lambda)$ and $M_{\theta} >0$ with $C_{\theta}(\lambda)$ which is decreasing  in $\lambda \in (0, \infty)$, $\lim_{\lambda \to \infty} C_{\theta}(\lambda) =0$, such that  
\begin{gather}\label{w44}    
 \sup_{x \in H} \big [\sum_{k \ge 1} \lambda_k  (D_{e_k} u^{(z)} (x))^2  \big]^{1/2}= \| (-A)^{1/2} Du^{(z)}\|_0  \le C_{\theta}(\lambda)\|f\|_{C^{\theta}};
\\ \label{w441}
  [ (-A)^{1/2} Du^{(z)} ]_{C_b^{\theta}}  \le M_{\theta} \|f\|_{C^{\theta}}.
\end{gather} 
(ii) Let $(f_n) \subset C_b^{\theta}(H)$ be such  that $\sup_{n \ge 1 }\| f_n\|_{C^{\theta}} \le C < \infty$ and  $f_n(x) \to f(x)$, $x \in H$. Define
$$
%\begin{array}{l}
u^{(z)}_n(x) = \int_0^{\infty} e^{- \lambda t}  P_{t}^{(z)}  f_n(x) dt.
%\end{array}  
$$ 
Then 
 \begin{equation} \label{s1144} 
\langle (-A)^{1/2} Du^{(z)}_n(x), h \rangle  
\to  
\langle (-A)^{1/2} Du^{(z)}(x), h \rangle,    \;\; \text{as $n \to \infty$, $h, x, z  \in H$.}    
\end{equation}
\end{theorem}   
\begin{proof} 
 Let us prove (i). The  first preliminary step of the proof uses an interpolation argument similar to the one used in    \cite{D2}.
 %(we simplify some computations in \cite{D2}).
 In the second step we argue  similarly to \cite{PriolaStudia}.  
%different from \cite{D2}. 

\vskip 1mm  
\noindent {\it I step.} We first prove  there there exists $c_{\theta} >0$ such that 
\begin{equation}\label{qhol}
 \| (-A)^{1/2} DP_t^{(z)}  f \|_0  \le  \frac{c_{\theta}}{t^{1 - \frac{\theta}{2}}}\|f\|_{C^{\theta}},\;\; t>0,\;\; f \in C_b^{\theta}(H) 
\end{equation}
%we  
% give a sketch  of the proof 
(cf. Section 2 of \cite{D2}). 
 % We will use 
  Let $g \in C_b(H)$. Note that if $l \in D((-A)^{1/2})$ then by \eqref{e3} and \eqref{wdc} we have, for any $x \in H$,
\begin{gather} \label{256}
|\langle DP_t^{(z)} g(x) , (-A)^{1/2} l\rangle| 
%=  
%  \int_H \langle
%   (-A)^{1/2} \Lambda_t h,Q_t^{-\frac12} y\rangle \, g (e^{tA}x+y+ \Gamma_t z) \mu_t(dy),
%  \; x,  h \in H,
 \le |(-A)^{1/2} \Lambda_t l |_H \| g\|_0 \le {c_2} {t}^{-1} |l|_H \| g\|_0,
\end{gather}
using   that 
 \begin{gather*}
\|  (-A)^{1/2} \Lambda_t \|_{\cal L} \le {c_2} {t}^{-1},\;\;\; t>0,
\end{gather*}
where $c_2 = \sup_{r \ge 0} \sqrt{2 r^2}  e^{-r} \,  (1 - e^{-2r})^{-1/2}$.
%arguing as in the proof of   Proposition 2.1 in \cite{D2} (see also \cite{DZ1}  using 
We get 
 that for $t>0$, $x \in H$, $DP_t^{(z)} g(x) \in D((-A)^{1/2})$ and
%  Moreover, 
  %applying the H\"older  inequality in \eqref{256} (cf. the proof of Theorem 6.2.2 in \cite{DZ1}) we arrive at
  % citarla ??  
 \begin{equation}\label{vfg}
\|   (-A)^{1/2} D P_t^{(z)} g \|_{0} \le \frac{c_2} {t} \| g \|_0, \; \; t>0. 
\end{equation}
If $g \in C^1_b(H)$ then  $ \langle (-A)^{1/2} D P_t^{(z)} g , h\rangle =  
  P_t^{(z)} (\langle Dg(\cdot ) , (-A)^{1/2} e^{tA}  h\rangle $ and 
 \begin{equation}\label{qww}
\|   (-A)^{1/2} D P_t^{(z)} g \|_{0} \le \frac{(2e)^{-1/2}} {t^{1/2}} \| Dg\|_0, \; \; t>0
\end{equation} 
 (see \eqref{ewd}). Now we use an interpolation result proved in Theorem 2.3.3 of \cite{DZ1}:
 \begin{gather*} 
 (C_b(H),C_b^1(H) )_{\theta,\infty} =C_b^{\theta}(H), \;\;\;  \beta\in(0,1).
\end{gather*}
Interpolating between \eqref{vfg} and \eqref{qww} as in page 491 of \cite{D2}   we obtain \eqref{qhol} with $c_{\theta} = c_2^{1 - \theta} (2e)^{-\theta /2 } k_{\theta}$; here $k_{\theta} >0$ verifies
\begin{gather*}
 {k_{\theta}^{-1}} \, \| \varphi \|_{C^{\theta}} \le  \| \varphi  \|_{(C_b(H),C_b^1(H) )_{\theta,\infty} } \le  k_{\theta}\| \varphi \|_{C^{\theta}},\;\;\; \varphi \in C_b^\theta(H).
\end{gather*}
Similarly one can prove that, for any $h \in H$, $l \in D((-A)^{1/2})$, $f \in C_b^{\theta}(H)$,
\begin{equation}\label{qhol1}
\| \langle  D^2 P_t^{(z)}  f (\cdot) (-A)^{1/2} l, h \rangle \|_0  \le  
\frac{\tilde c_{\theta}}{t^{\frac{3}{2} - \frac{\theta}{2}}} \, |  l|_H \,  |  h|_H \, \|f\|_{C^{\theta}},\;\; t>0.   
\end{equation}  
To prove the previous estimate, let us first consider  $g \in C_b(H)$; we have 
\begin{gather} \label{asw} 
\| \langle  D^2 P_t^{(z)}  g (\cdot) (-A)^{1/2} l, h \rangle \|_0  \le |(-A)^{1/2} \Lambda_t l |_H  \, | \Lambda_t h|_H \| g \|_0
\\ \nonumber 
\le c_2 C_1 \, |  l|_H  |  h|_H  \frac{1}{t^{3/2}}\, \| g \|_0
, \; \; t>0 
\end{gather}  
 (cf. \eqref{wdc}). 
 %On the other hand, 
 If $g \in C^1_b(H)$ then $\langle  D^2 P_t^{(z)}  g (x) (-A)^{1/2} l, h \rangle$
   $= \langle D P_t^{(z)} \big (\langle Dg(\cdot ) , (-A)^{1/2} e^{tA}  l \rangle \big) (x), h \rangle $ and 
 \begin{gather} \label{asw1}
\| \langle  D^2 P_t^{(z)}  g (\cdot) (-A)^{1/2} l, h \rangle \|_0  \le |(-A)^{1/2} e^{tA}l |_H  \, | \Lambda_t h|_H \| Dg \|_0
\\ \nonumber  
\le c \, C_1 \, |  l|_H \,  |  h|_H  \frac{1}{t}\,  \| Dg \|_0
, \; \; t>0.  
\end{gather} 
  Interpolating between \eqref{asw} and \eqref{asw1}    we obtain \eqref{qhol1} with $\tilde c_{\theta} = (c_2C_1)^{1 - \theta} (c C_1)^{\theta } k_{\theta}$.
 \\ \\
{\it II step.} We  prove (i).

For any $z \in H$, $\lambda >0$, $l \in D((-A)^{1/2}) $  we have (cf. \eqref{s55})
\begin{gather*}
  \langle D u^{(z)}(x), (-A)^{1/2} l\rangle =  \int_0^{\infty} e^{- \lambda t} \langle   (-A)^{1/2} D_{} P_t^{(z)} f(x),  l \rangle dt,\;\;\; x\in H . 
\end{gather*}
 By estimate \eqref{qhol} we obtain  that $D u^{(z)}(x) \in D((-A)^{1/2})$ and moreover 
 \begin{equation}\label{qcv}
 \| (-A)^{1/2} Du^{(z)} \|_0  \le  c_{\theta}\|f\|_{C^{\theta}}  \int_0^{\infty} e^{- \lambda t} \frac{1} {t^{1- \theta/2} }    dt =   \frac{R_{\theta}}{\lambda^{\theta/2}}\|f\|_{C^{\theta}},\;\; \lambda>0
\end{equation}
 (with $R_{\theta} >0$ only depending on $\theta \in (0,1)$).   This shows  estimate in \eqref{w44}.

To prove  \eqref{w441} we argue as in  the proof of Theorem 4.2 in \cite{PriolaStudia}.   We fix $x, h \in H$.
 We have
 %(we drop the dependence on $z$): 
\begin{gather*}
|  (-A)^{1/2} Du^{(z)}(x + h) -  (-A)^{1/2} D u^{(z)} (x)|_H  \le
 \,  u_h^{(z)} (x) \, +  \, v_h^{(z)}(x), \;\; \text{where} 
 \\
u_h^{(z)}(x) = \int_0^{|h|^{2}_H}  e^{- \lambda t } |(-A)^{1/2} 
 D P_t^{(z)} f (x+ h) -  (-A)^{1/2}  D P_t^{(z)} f (x)|_H dt; 
 \\  v_h^{(z)}(x) =
\int_{|h|^{2}_H}^{\infty} e^{- \lambda t } |  (-A)^{1/2}  D
P_t^{(z)} f (x+  h) -  (-A)^{1/2}  D  P_t^{(z)} f (x)|_H \, dt.
\end{gather*} 
In order to estimate $u_h^{(z)}(x)$ we use  \eqref{qhol}. 
    We find  
$$
\sup_{x \in H} u_h^{(z)}(x) \le c_{\theta}  \| f\|_{C^\theta} \,
\int_0^{|h|^{2}_H} t^{\frac{\theta}{2} - 1} \, dt 
\le
C_{\theta}' \,  \| f\|_{C^{\theta}} \, \, |h|^{ \theta}_H.
$$ 
 Concerning $v_h^{(z)}(x)$ we will use estimate   \eqref{qhol1}.   Let $B_1 = \{ x \in H \, :\, |x|_H \le 1 \}$. Recall that $D((-A)^{1/2}) \cap B_1  
 $ is dense in $B_1$. 
  For $t>0$, we have:
 \begin{gather*}
|  (-A)^{1/2}  D
P_t^{(z)} f (x+  h) -  (-A)^{1/2}  D  P_t^{(z)} f (x)|_H 
\\
%=\sup_{l \in D((-A)^{1/2} ), \, \,  |l|_H \le 1}
% | \langle  (-A)^{1/2}  D
%P_t^{(z)} f (x+  h) -     (-A)^{1/2}  D  P_t^{(z)} f (x),  l \rangle|
= \sup_{l \in D((-A)^{1/2} ), \, \,  |l|_H \le 1}
 | \langle  D
P_t^{(z)} f (x+  h) -     D  P_t^{(z)} f (x), (-A)^{1/2} l \rangle|.
\end{gather*}
 Let us consider $l \in D((-A)^{1/2} ),$ with $ |l|_H \le 1$. We write  
\begin{gather*} 
\big | \langle  D
P_t^{(z)} f (x+  h) -     D  P_t^{(z)} f (x), (-A)^{1/2} l \rangle \big |
= 
\Big |   \int_0^1 \langle D^2
P_t^{(z)} f (x+ s h) h  , (-A)^{1/2} l \rangle ds \Big |
\\
\le {\tilde c_{\theta}} t^{ \frac{\theta}{2} - \frac{3}{2}} \, |  l|_H \,  |  h|_H \|f\|_{C^{\theta}},\;\; t>0.     
\end{gather*}  
 and so  $ |  (-A)^{1/2}  D
P_t^{(z)} f (x+  h) -  (-A)^{1/2}  D  P_t^{(z)} f (x)|_H  \le {\tilde c_{\theta}}{t^{      \frac{\theta}{2}- \frac{3}{2} }}  \,   |  h|_H \|f\|_{C^{\theta}},\;\; t>0.  $ 
  We obtain 
  \begin{gather*} 
\sup_{x \in H}
 v_h^{(z)}(x) \le  \tilde c_{\theta}  \| f\|_{C^{\theta}} \, |h|_H \,
\int_{|h|^{2}_H} ^{\infty}  \big (
 t^{\frac{\theta}{2} - \frac{3}{2}  }    \big ) dt \,
  \le   \, C_{\theta}'' \,  |h|^{\theta -1}_H \, |h|_H 
  \| f\|_{C^{\theta}} = \, C_{\theta}'' \,  |h|^{\theta}_H 
  \| f\|_{C^{\theta}}.
\end{gather*}  
 By the previous  estimates on  $u_h^{(z)}$ and $v_h^{(z)}$ we deduce easily  \eqref{w441}.
 %and this completes  the proof.  
 
  To prove (ii) 
  %\eqref{s1144} 
  we fix $x \in H$ and $l \in D((-A)^{1/2})$. We  write (see \eqref{e3}) for $t>0$
  \begin{gather*}
 D_{ (-A)^{1/2} l} {P_{t}^{(z)}}f(x) = \langle D P_t^{(z)} f_n(x),  (-A)^{1/2} l
 \rangle  = \int_H \langle
  \Lambda_t (-A)^{1/2} l ,Q_t^{-\frac12} y\rangle \, f_n (e^{tA}x+y+ \Gamma_t z) \mu_t(dy).
%  \; x,  h \in H. 
\end{gather*}
  We can  pass to the limit as $n \to \infty$   by the Lebesgue convergence theorem and get $D_{ (-A)^{1/2} l} {P_{t}^{(z)}}f_n(x) \to $  $D_{ (-A)^{1/2} l} {P_{t}^{(z)}}f(x)$ as $n \to \infty$.

  Similarly,  using also the estimate $ |D {{P_{t}^{(z)}}} f_n(x)|_H \le  \frac{c_{\theta}}{t^{1 - \frac{\theta}{2}}}\|f_n\|_{C^{\theta}}  \le 
  \frac{c_{\theta} C}{t^{1 - \frac{\theta}{2}}} $, $t>0$,
     we have, for any $l \in D( (-A)^{1/2} )$, $x \in H, $
$$
 \lim_{n \to \infty} \langle Du^{(z)}_n(x) ,(-A)^{1/2} l \rangle    = 
  \langle Du^{(z)} (x) ,(-A)^{1/2} l \rangle.  
$$
We deduce easily that \eqref{s1144} holds. 
 \end{proof}

\begin{remark} \label{fi12} {\em  This following fact will be useful in the sequel:
 if $G \in C_b(H,H)$, then \eqref{s1144} implies that
\begin{equation}
\label{w12} 
\lim_{n \to \infty} \langle (-A)^{1/2} Du^{(z)}_n(x) , G(x)\rangle =
  \langle (-A)^{1/2}   Du^{(z)}(x) , G(x)\rangle,\;\; x,z \in H. 
\end{equation}
% \qed   
}
\end{remark}

\begin{remark} \label{estendi}{\em  (a) Actually Theorem 3.3 in \cite{D2}  shows    that 
  $\| (-A)^{1/2} Du^{(z)} \|_{C_b^{\theta}} $ $= [ (-A)^{1/2} Du^{(z)} ]_{C_b^{\theta}} 
 + \| (-A)^{1/2} Du^{(z)} \|_{0}  $ $\le K_{\theta}(\lambda) \|f\|_{C^{\theta}}$ with $K_{\theta}(\lambda) $ independent of $f\in C^{\theta}_b(H)$ and $z \in H$.
  The fact that  $K_{\theta}(\lambda)$ in \cite{D2} depends also on $\lambda$ follows from  estimate  (2.15) in \cite{D2}. Indeed in such estimate one has also to consider the supremum norm 
  %of $b_t$, i.e., 
   $\| b_t\|_0$.

%\smallskip
 (b) We do not know if estimates \eqref{w44} and \eqref{w441} hold in  a stronger form  with $\| f\|_{C^{\theta}}$ replaced by $[ f ]_{C^{\theta}}$ (see \eqref{sx}). This happens, for instance, in the finite-dimensional case   considered in \cite{DL}.
}
\end{remark}

\section{Proof of weak existence of Theorem \ref{base} (only assuming continuity of $F$)}
  
In this section we  require that   
\begin{equation} \label{qdd}
 F: H \to H \;\; \text{is continuous and verifies   } \;\; |F(x)|_H \le C_F (1 + |x|_H), \;\; x \in H,
\end{equation}  
   for some  constant $C_F >0$.
   We will prove weak existence  by adapting  a  compactness approach of  \cite{GG}. 
%% corr
This approach   { is inspired by   \cite{DKP}}  (it is also explained  in  Chapter 8 of \cite{DZ}).

Let us fix $x \in H$.  To construct the solution we start with some approximating mild solutions. We introduce, for each $m \ge 1$,  
$$
A_m = A \circ \pi_m ,\;\;\; A_m e_k =  - \lambda_k  \,  e_k,\;\; k = 1, \ldots m, 
$$
 $   A_m e_k =0, $ $k >m;$  here  
 $\pi_m=\sum_{j=1}^me_j\otimes e_j$ ($(e_j)$ is the basis of eigenvectors  of $A$; see \eqref{pp1}).

  For each $m$  there exists a weak mild solution $X_m= (X_m(t))_{t \ge 0}$ on some filtered probability space, possibly depending on $m$  (such solution can also be constructed by the Girsanov theorem, see \cite{GG0}, \cite{DZ} and  \cite{DFPR}). 
  
%\begin{remark} {\em  
Usually the mild solutions $X^m$ are constructed on a time interval $[0,T]$.  However there is a standard procedure based on the Kolmogorov extension theorem to define the solutions  on $[0, \infty)$. On this respect, we refer to  Remark 3.7, page 303, in \cite{KS}.  
%}
%\end{remark}

We know that  
 \begin{equation}\label{2ww}
 %\begin{array}{l}
X_m (t)=e^{tA} x +\int_{0}^{t}e^{\left(  t-s\right)  A} (-A_m)^{1/2}F^{}(X_m(s))
ds+\int_{0}^{t}e^{\left(  t-s\right)  A}dW_{s},\;\;\; t \ge 0.
%\end{array} 
\end{equation}
Recall    that, for any $t \ge 0,$ the stochastic convolution $W_A(t)= \int_{0}^{t}e^{\left(  t-s\right)  A}dW_{s}$ is a Gaussian random variabile with law $N(0,Q_t)$.  Let $p > 2$ and $q = \frac{p}{p-1} <2$.  We find (using also \eqref{ewd} {   and the H\"older inequality) }
\begin{gather*}
|X_m (t)|^p_H \le c_p ( |e^{tA} x|^p_H + \big | \int_{0}^{t}e^{\left(  t-s\right)  A} (-A_m)^{1/2}F^{}(X_m(s))ds |^p_H +
| W_A(t)|^p_H ) 
%\\
%\le c_T| x|^p_H  +  M^p \,  C_F^p \Big ( \int_{0}^{t}|  (t-s)^{-1/2} (1 + |%X_m(s) |_H ds \Big)^p+
%c_p | W_A(t)|^p_H  
\\ \le 
  c_T| x|^p_H +  c_T ( \int_{0}^{t}  (t-s)^{-q/2} ds)^{p/q} \cdot    \int_{0}^{t} (1+ |X_m(s) |_H^p  )  ds  +
c_p | W_A(t)|^p_H 
 \\
\le C_T | x|^p_H + C_T +  C_T  \int_{0}^{t}   |X_m(s) |^p_H ds +
C_T | W_A(t)|^p_H,\;\;  t \in [0,T].  
\end{gather*}
 By the Gronwall lemma we find the bound
% for $t \in [0,T]$
% \begin{gather*}
%|X_m (t)|^p_H \le   C_T (| x|^p_H + 1+ | W_A(t)|^p_H)
% +  k_T  \int_{0}^{T}    | W_A(s)|^p_H ds.  
%\end{gather*}
% We deduce the bound 
\begin{equation} \label{sd1}
\begin{array}{l}
\sup_{m \ge 1} \sup_{t \in [0,T]}\E|X_m (t)|^p_H = {   C_T } < \infty.
\end{array}  
\end{equation} 
{\sl The mild solution $X$ will be a weak limit of solutions $(X_m)$.} To this purpose we need some compactness results.
 The next  result is proved in \cite{GG} 
(the proof  uses that $(e^{tA})$ is a compact semigroup). 
\begin{proposition}
\label{p8.40}
If
$0<\frac{1}{p}<\alpha \leq 1$ then the
operator $G_{\alpha }: L^{p}(0,T;H) \to C([0,T];H) $
\begin{displaymath}
G_{\alpha }f(t)=\frac{\sin \pi \alpha}{\pi}  \int_{0}^{t}(t-s)^{\alpha -
1} e^{(t-s)A}f(s)ds,\quad f \in L^{p}(0,T;H),\;t \in [0,T], \;\; \text{is compact}. 
\end{displaymath}
 \end{proposition}
Below we consider a variant of the previous result. In the proof we use  estimate \eqref{ewd}.  
\begin{proposition}
\label{p8.41}
 Let $p>2$.   Then the
operator ${Q} : L^{p}(0,T;H) \to C([0,T];H)  $,
\begin{displaymath}
{Q} f(t)=\int_{0}^{t} (-A)^{1/2} e^{(t-s)A}f(s)ds,\quad f \in L^{p}(0,T;H),\;t \in [0,T], \;\; \text{is compact}. 
\end{displaymath}
% is compact from $L^{p}(0,T;H)$ into 
% $C([0,T];H)$.  
\end{proposition}
\begin{proof} Since the proof is similar to the one of Proposition \ref{p8.40} we only give a sketch of the proof. 
  Denote by $| \cdot |_{p}$ the norm in
$L^{p}(0,T;H)$. According to the infinite 
dimensional version of the Ascoli-Arzel\'{a}
theorem one has to show that
\begin{enumerate}
\item[(i)] For arbitrary $t \in [0,T]$ the sets
$\{ {Q}f(t): \; |f|_{p}\leq 1\}$
 are relatively
compact in $H$.  

\item[(ii)] For arbitrary $\varepsilon >0$ there exists
$\delta >0$ such that
\begin{equation}
|{Q}f(t)-{Q}f(s)|_H \leq
\varepsilon,\;\; \text { if}\; \;  |f |_{p}\leq
1, \; |t-s|\leq \delta,\quad s,t \in [0,T].
\label{e8.17}
\end{equation}
\end{enumerate}
 To check (i)
let us fix $t \in (0,T]$ and define operators $Q^t $  and  $    Q^{\varepsilon,t }$ from  $L^{p}(0,T;H)$ into $H$, for
$\varepsilon  \in (0,t)$, 
\begin{displaymath}
{Q}^{t}f= Qf (t),\;\; \;\;
{Q}^{\varepsilon ,t}f=\int_{0}^{t-
\varepsilon } (-A)^{1/2} e^{(t-s)A} f(s)ds,\quad f \in
L^{p}(0,T;H).
\end{displaymath}
Since
 $
  {Q}^{\varepsilon ,t}f=e^{\varepsilon
A}\int_{0}^{t-\varepsilon } (-A)^{1/2} e^{(t-
\varepsilon -s) A}f(s)ds
$
and $e^{\varepsilon A},\varepsilon >0,$ is
compact, the  operators
${Q}^{\varepsilon ,t}$ are compact.
 Moreover, 
% setting
%$M= \sup_{t \in [0,T]}\|S(t)\|$, we have, 
by using \eqref{ewd} and the 
H\"{o}lder inequality  (setting $q=\frac{p}{p-1} <2$)
\begin{displaymath}
 \begin{array}{l}
 \ds |{Q}^t f \, - \, {Q}^{\epsilon,t}f|_H =\left|\int_{t-\varepsilon }^{t } (-A)^{1/2} e^{(t-s)A}f(s)ds\right|_H 
%\\  \ds   \leq \left(\int_{t-\varepsilon }^{t }\|(-A)^{1/2} e^{(t-s)A}\|^{q}%_{\cal L}ds
%\right)^{1/q}\left( \int_{t-\varepsilon }^{t}| f(s)|^{p}_H  ds
%\right)^{1/p} \\   
\\ 
 \ds \leq M \left(\int_{t-\varepsilon }^{t }(t-
s)^{ - q/2}ds
\right)^{1/q}\left( \int_{t-\varepsilon }^{t} | f(s)|^{p}_H ds
\right)^{1/p} 
{   \le } \, \,  M_q \varepsilon ^{-1/2 \,  +1/q  } |f|_{p}
\end{array}
\end{displaymath}
with $ -\frac{1}{2} +  \frac{1}{q}  >0 $.  Hence
$Q^{\varepsilon,t  } \rightarrow
{{Q}^t} $, as $\epsilon \to 0^+$, in the operator norm so that
${Q}^t$ is compact and
(i) follows. Let us consider (ii).
For $0\leq t \leq t+u \leq T$ and $|f|_{p}\leq
1$, we have
\begin{displaymath}
 \begin{array}{l}
\ds |{Q}f(t+u)-{Q}f(t)|_H     
 \leq \int_{0}^{t} \|(-A)^{1/2} e^{(t+u-s)A} - 
(-A)^{1/2} e^{(t-s)A}\|_{\cal L}\,   |f(s)|_H ds 
 \\
\ds+\int_{t}^{t+u}| (-A)^{1/2}   e^{(t+u-s)A} f(s)|_H ds 
 \\
 \ds
 \le {   M_p} \Big(   \int_{0}^{u} s^{-q /2}ds
\Big)^{1/q}
+ 
 \Big( \int_{0}^{T} \|(-A)^{1/2} e^{(u+s)A}-  (-A)^{1/2}  e^{sA} \|^{q}_{\cal L} ds
\Big)^{1/q }   = I_1 + I_2.   
\end{array}
\end{displaymath}
It is clear that $I_{1} = M_p' \,   u^{1/2 -
1/p}      \rightarrow 0$ as
$u\rightarrow 0$.

     Moreover, for $s>0$, $(-A)^{1/2} e^{sA}$ is compact; indeed $(-A)^{1/2} e^{sA} e_k$ $= (\lambda_k)^{1/2} e^{-s \lambda_k} e_k$ and 
     $\displaystyle{(\lambda_k)^{1/2} e^{-s \lambda_k}}$ $ \to 0$ as $k \to \infty$). It follows that  
  $\|e^{uA}(-A)^{1/2}e^{sA}-
(-A)^{1/2}e^{sA}\|_{\cal L}$ $ \rightarrow 0$ as $u\rightarrow 0$  for
arbitrary $s>0$.

Since
 $
\ds \|(-A)^{1/2} e^{(u+s)A}- (-A)^{1/2} e^{sA}
\|^{q}\leq \frac{ (2M)^q }{s^{q/2}} , \;   \,       s >0,
 $ $ u\geq 0, 
$
 and $q<2$,
 by the Lebesgue's dominated 
convergence theorem $I_{2}\rightarrow 0$ as   
$u\rightarrow 0$. Thus the proof of
(ii) is complete.
 \end{proof}

\noindent  {\bf Proof of the existence part of Theorem \ref{base}}. Let $x \in H$. We proceed {   in}  two steps. 
\\
{\it I Step.} Let $( X_{m})$ be solutions of \eqref{2ww}.  
 We prove that  their laws $\{ {\cal
L}(X_{m})\}$ form a tight family of 
probability measures on ${\cal B}(C([0,\infty);H))$.

To this purpose it is enough to show that for each $T>0$ the laws $\{ {\cal
L}(X_{m})\}$ form a tight family of 
probability measures on ${\cal B}(C([0,T];H))$.
%(recall that $C([0,\infty);H)$ is a Polish space endowed with the metric %of the uniform convergence on bounded intervals of $[0, \infty)$). 

 Let us fix  $p>2$ and $T>0$.  
 % corr
   We know by \eqref{sd1}
 that there exists a constant $c_{p}>0$ such that $\E|X_{m}(t)|^{p}_H  $ $\leq c_{p},$ $  m\ge 1,$ $t \in [0,T].$ It follows that
\begin{equation}
\sup_{m \ge 1} \E\int_0^T |F_m(X_{m}(t))|^{p}_H < \infty.
\label{e8.18} 
\end{equation} 
with $\pi_m \circ F = F_m$, since   $|F_m(x)|_H \le C_F (1 + |x|_H)$, $m \ge 1$. 

{In order to prove the tightness of $\{ \mathcal L(X_{m})\} $ on  ${\cal B}(C([0,T];H))$ we note  that 
\begin{equation} \label{224c3}
 \begin{array}{lll}
 X_{m}(t) 
&=& e^{tA} x+ Q(F_{m}(X_{m}))(t)+
    W_A(t),\;\; t \in [0,T].
\end{array}
\end{equation} 
 Let $Z_m(t) = e^{tA} x+ Q(F_{m}(X_{m}))(t)$, $t \in [0,T]$.
It is not difficult to prove that the tightness of $\{ \mathcal L(Z_{m})\}$ implies the tightness of $\{ \mathcal L(Z_{m} + W_A)\}$ 
$=\{\mathcal L(X_{m} )\}$   on  ${\cal B}(C([0,T];H))$.

Thus it remains to  show the 
   tightness of $\{ \mathcal L(Z_{m}) \}$.     
 By \eqref{e8.18} and  Chebishev's
inequality, for $\varepsilon >0$ one can   
find $r>0$ such that for all $m \ge 1$
\begin{equation}
\P\Big(   \big(\int_{0}^{T}|F_{m}( X_{m}(s) ) |^{p}_H ds \big)^  
{1/p}  \leq r \Big)> 1-\varepsilon.
\label{e8.19}   
\end{equation}
By Proposition     \ref{p8.41} (recall that $|\cdot|_p $ denotes the norm in
$L^{p}(0,T;H)$) the set
\begin{displaymath}
K=\{ e^{(\cdot )}x+    Qg(\cdot ): \;  \; |g|_{p}\leq r\} \subset C([0,T];H)  
\end{displaymath}
is  relatively  compact. Since $\P (Z_m \in K) = \mathcal  L(Z_{m})(K)> 1-\varepsilon
,$ for any $ m \ge 1,$  the tightness follows. }  
%by the Prokhorov theorem. 

\vskip 1mm       
\noindent 
  {\it II Step. } By the Skorokhod representation theorem, possibly passing to a subsequence of $(X_m)$ still denoted by $(X_m)$,  
 there exists a probability space $(\hat  \Omega, \hat  {\cal F}, \hat  \P )$ and  random variables $\hat X $ and  $\hat X_m$, $m \ge 1$, defined on $\hat \Omega $ with values in $C(0,\infty; H)$  such that the law of $X_m$ coincide with the law of $\hat X_m$, $m \ge 1,$ and moreover 
\begin{gather*}
 \hat X_m \to \hat X,\;\;\; \hat \P-a.s. 
\end{gather*}
Let us fix ${k_0} \ge 1$. Let $\hat X^{(k_0)}_m = \langle \hat X^{}_m,  e_{k_0} \rangle
$.  Recall that $\pi_m \circ F = F_m$.
 It is not difficult to prove that     the processes $(M_{m}^{(k_0)})_{m \ge 1}$  
$$
M_{m}^{(k_0)}(t) = \begin{cases} \hat X^{(k_0)}_m(t) - x^{(k_0)} +   \lambda_{k_0} \int_0^t  \hat X^{(k_0)}_m(s) ds \, - \, \lambda_{k_0}^{1/2} \, \int_0^t F^{(k_0)}(\hat X_m (s))ds,\; \; {k_0} \le m 
\\     \hat X^{(k_0)}_m(t) - x^{(k_0)} +      \lambda_{k_0} \int_0^t \hat X_m^{(k_0)}(s) ds,    \, \, \; \;\; {k_0} > m,\;\; t \ge 0,   
\end{cases}
$$
are square-integrable continuous  ${\cal F}_t^{\hat X_m}$-martingales on $(\hat  \Omega, \hat  {\cal F}, \hat  \P )$ with $M_{m}^{(k_0)}(0) =0$. 
%(cf. the proof in Section 8.4 of \cite{D2}). 
 Moreover the quadratic variation process  $\langle M_{m}^{(k_0)}\rangle_t \, = \, t$, $m \ge 1$ (cf. Section 8.4 in \cite{DZ}).  
 
 Passing to the limit as $m \to \infty$ we find that 
\begin{equation} \label{w11}
M_{}^{(k_0)}(t) =  \hat X^{(k_0)}(t) - x^{(k_0)} +    \lambda_{k_0} \int_0^t  \hat X^{(k_0)}(s) ds \, - \, \lambda_{k_0}^{1/2} \, \int_0^t F^{(k_0)}(\hat X (s))ds, \;  t \ge 0,   
   \end{equation}
 is a square-integrable continuous  ${\cal F}_t^{\hat X}$-martingale with $M_{}^{(k_0)}(0) =0$. To check the martingale property,    
 let us fix $0 < s < t $. We know that $\hat \E [M_{m}^{(k_0)}(t) - M_{m}^{(k_0)}(s) /\,  {\cal F}_s^{\hat X_m}]=0$, $m \ge 1$.   
  
  Consider $0 \le s_1 < \ldots < s_n \le  s$, $n \ge 1$. For  any   $h_j \in C_b (H)$, we have, for $m \ge k_0$,
\begin{gather}\label{w113}  
 \begin{array}{l}
 \hat \E  \Big[ \big ( \hat X^{(k_0)}_m (t) - \hat X^{(k_0)}_m(s) +    \lambda_{k_0} \int_s^t  \hat X_m^{(k_0)}(r) dr  -  \lambda_{k_0}^{1/2} \, \int_s^t F^{(k_0)}(\hat X_m (r))dr    \big )  \\
    \cdot \prod_{j=1}^n 
h_j(\hat X_m ({s_j}))\Big]=0.
\end{array}     
\end{gather}
 Using that $| F^{(k_0)} (x) | \le C_F   (1 + |x|_H)$ and that,  for any  $T>0$,
% \begin{equation*}
% \label{dd}
%\sup_{m \ge 1} \sup_{t \le T} \E [ |X_m(t)|^p_H ]  = 
 $\ds \sup_{m \ge 1} \sup_{t \le T} \hat \E [ |\hat X_m(t)|^p_H ]
 \le C < \infty$
%\end{equation*}
 (cf. \eqref{sd1}) 
by the Vitali   convergence theorem we get easily that \eqref{w113} holds when
 $\hat X_m$ is replaced by $\hat X$ {   (note that  this assertion could   be proved by using only the   dominated convergence theorem).} Then  we obtain that $M_{}^{(k_0)}$ 
 is   a square-integrable continuous  ${\cal F}_t^{\hat X}$-martingale.   
  \\ 
Moreover, by a limiting procedure, arguing as before, we find that $( (M_{}^{(k_0)}(t))^2 -t)$ is a martingale. It follows that $M_{}^{(k_0)}$ is {\it a real Wiener process} on $(\hat \Omega, \hat {\cal F}, \hat \P)$.

Hence, for any $k \ge 1,$ we find that there exists a real Wiener process $M^{(k)}$ such that
\begin{gather*}
\begin{array}{l}
  \hat X^{(k)}(t) =  x^{(k)} -    \lambda_{k} \int_0^t  \hat X^{(k)}(s) ds 
  \, + \, \lambda_{k}^{1/2} \, \int_0^t F^{(k)}(\hat X (s))ds + M_{}^{(k)}(t). 
  \end{array} 
\end{gather*}
We prove now that   $(M^{(k)})_{k \ge 1}$ are independent Wiener processes. 
\\
We fix $N \ge 2$ and
 introduce the processes $(S_{m}^N)_{m \ge 1}$, $S_m^{N}(t) =  
 \big ( M_{m}^{(1)}(t), \ldots,   \ M_{m}^{(N)}(t)\big)_{t \ge 0}$, with values in $\R^N$.     The components of $S_m^{N}$  are square-integrable continuous  ${\cal F}_t^{\hat X_m}$-martingales. Moreover  the    quadratic covariation $\langle M_{m}^{(i)},  M_{m}^{(j)} \rangle_t = \delta_{ij} t $.
 
 Passing to the limit as before we obtain that also the $\R^N$-valued process $(S^N(t))$, 
 $
 S^N (t) = \big ( M^{(1)}(t), \ldots,   \ M^{(N)}(t) \big), \;\; t \ge 0,  
  $
 has  components which are square-integrable continuous  ${\cal F}_t^{\hat X}$-martingales with quadratic covariation $\langle M^{(i)},  M^{(j)} \rangle_t = \delta_{ij} t $. Note that $S^N(0)=0$, $\hat \P$-a.s.
 
  By the L\'evy characterization of the Brownian motion (see Theorem 3.16 in \cite{KS}) we have that  $\big ( M^{(1)}(t), \ldots,   \ M^{(N)}(t) \big)$ is a standard Wiener process with values in $\R^N$. Since $N$ is arbitrary, 
  $(M^{(k)})_{k \ge 1}$ are independent real Wiener processes and  the proof is complete.

\begin {remark} \label{sd} {\em 
%The existence result for \eqref{sde} holds more generally if $F: H \to H$ is %continuous and we have uniform estimates for approximating solutions  like %%
%\eqref{e8.18}. A sufficient condition for this is that $F$ has at most a %linear growth.   
%\\
Following the previous method one can prove existence of weak mild solution even for 
\begin{equation*} 
dX_{t}=AX_{t}dt +  (-A)^{\gamma}F(X_{t})dt+ dW_{t},\qquad X_{0}=x\in H,
\end{equation*}
with $\gamma \in (0,1)$ and $F: H \to H$  continuous and having at most a linear growth.   
}    
\end {remark}

\section {Proof of weak uniqueness  when $F \in C_b^{\theta}(H,H)$, for some $\theta \in (0,1)$}

To get the weak uniqueness of Theorem \ref{base} when 
 $F \in C_b^{\theta}(H,H)$
we  first show  the equivalence between martingale solutions and mild solutions. Indeed  for martingale problems  some useful uniqueness  results  are  available even in infinite dimensions (see, in particular, Theorems \ref{ria}, \ref{uni1} and \ref{key}).

\subsection{Mild solutions and  martingale problem   } 

 We  formulate the martingale problem of Stroock and Varadhan \cite{SV79} for the operator $\L$ given below in \eqref{ll} and associated to \eqref{sde}. We stress that  an infinite-dimensional generalization of the martingale problem is proposed in Chapter 4 of \cite{EK}.  Here we follow 
 % notations and results of
   Appendix of \cite{PrPot}.  In such  appendix  some extensions and modifications of theorems given in
Sections 4.5 and 4.6 of \cite{EK} are proved. 

The results of this section hold more generally when $F \in C_b(H,H)$ in \eqref{sde}.

 \vskip 1mm       
   We use the space $C^2_{cil}(H)$
  of regular
  cylindrical functions (cf. \eqref{cil2}).
  We 
 deal with the following linear operator ${\L}: D(\L) \subset C_b(H) \to C_b(H)$,
with $D(\L) = C^2_{cil}(H)$:   
%(recall that here   $F \in C_b^{\theta}(H,H)$): 
 \begin{eqnarray} \label{ll}
\L f (x) &=& \frac{1}{2} Tr(D^2 f(x)) + \langle x, ADf(x) \rangle +
\langle F(x),  (-A)^{1/2} Df(x) \rangle 
\\ \nonumber & = &
L f (x) +  
\langle F(x),  (-A)^{1/2} Df(x) \rangle,\;\;\; f \in D(\L),\; x \in H.
\end{eqnarray}

\begin{remark} {\em   
 We stress  that the linear operator $(\L, D(\L ))$  in \eqref{ll} is countably pointwise determined,
   i.e., it verifies Hypothesis 17 in \cite{PrPot}. Indeed, arguing  as in Remark 8 of \cite{PrPot}, one shows that there exists  a countable
set ${\mathcal H}_0 \subset D(\L)$ such
 that for any $f \in D(\L) $, there exists 
a sequence $(f_n) \subset {\mathcal H}_0$ satisfying
$$
 \lim_{n \to \infty} ( \| f - f_n \|_{0} +   \|  \L f_n -   \L f \|_{0}) =0. \qed
$$ 
}
\end{remark}

Let $x \in H$. An $H$-valued stochastic process $X = (X_t)= (X_t)_{  t \ge 0 }$
defined on some probability space $(\Omega, {\cal F}, \P)$
with continuous trajectories is a \textsl { solution of the martingale
problem for $(\L, \delta_x)$} if,
 for any $f \in D(\L)$,
\begin{equation}\label{mart}
 \begin{array}{l}
    M_t(f) = f(X_t) - \int_0^t \L f(X_s) ds, \;\; t \ge 0, \;\; \text{is a
    martingale}
\end{array} 
    \end{equation}
(with respect to the natural filtration $({\cal F}_t^X)$) 
% where ${\cal F}_t^X = \sigma(X_s \, : \, 0 \le s \le t)$ is the  
%$\sigma$-algebra generated by the random variables
% is the smallest $\sigma$-algebra that makes  %measurable all
%$X_s$, $0 \le s \le t$),
  and, moreover,     $X_0 =x, \P$-a.s.. 
  
  If we do not assume that $F$ is bounded then in general    
   $M_t(f)$ is only a local martingale because   in general  $\L f$ is not    a bounded function.

 \smallskip 
  We say that \textsl{ the martigale problem for ${ \L}$ is  well-posed} if, for any $x  \in H$,   there exists a martingale solution
 for $({ \L}, \delta_x)$ and, moreover,  uniqueness in law    holds  for  the martingale problem for $({ \L}, \delta_x)$.

 % \smallskip
  Equivalence between 
    mild solutions and martingale solutions  has been proved in a  general setting  in  \cite{kunze} even for SPDEs in Banach spaces.   
  We only give a sketch of the proof of  the next result for the sake of completeness
  (see also Chapter 8 in \cite{DZ}). 
  
  %Clearly, this  holds also if $F \in C_b(H,H)$ and, more %generally,  under % hypotheses in \cite{kunze}. 

\begin{proposition} \label{ser}  Let   $F \in C_b (H,H)$  and $x \in H.$   

(i)  If $X$ is a weak mild solution to \eqref{sde}    with $X_0 = x$, $\P$-a.s.,  then
 $X$ is also 
a  solution  of the martingale
problem for $(\L, \delta_x)$.

(ii) Viceversa, if $X= (X_t)$ is  a  solution of the martingale
problem for $(\L,  \delta_x)$ on some probability space $(\Omega, {\cal F}, \P)$ then there exists a cylindrical Wiener process on $(\Omega, {\cal F},({\cal F}_t^X), \P)$ such that 
$X$ is  a weak mild solution to \eqref{sde}  
  on $(\Omega, {\cal F},({\cal F}_t^X), \P)$
with initial
condition $x$. 
\end{proposition}
\begin{proof}(i) Let   $X$ be  a weak mild solution to \eqref{sde} with $X_0 = x$, $\P$-a.s.
 defined on  a filtered probability space $(
\Omega,$ $ {\mathcal F},
 ({\mathcal F}_{t}), \P) $. Let $f \in D(\L)$. Since  $f$ depends only on a finite number of variables  by the It\^o formula   
  we obtain that
$
%\begin{array}{l}
f(X_t) $ $ - \int_0^t \L f(X_s) ds
%\end{array} 
$
 is an ${\cal F}_t$-martingale, for any $f \in D(\L)$. 
 We get easily  the assertion since ${\cal F}_t^X \subset 
{\cal F}_t$, $t \ge 0$.

\vskip 1mm \noindent  (ii)  Let $X$ be a solution to the martingale problem
 for $(\L, \delta_x)$ defined on $(\Omega, {\cal F}, \P) $.
% The proof is divided into three steps.

\vskip 0.5 mm       
\noindent  \textit{I Step.}  {\it Let $X^{(k)}_t = \langle X_t, e_k \rangle $ and $F(x) = \sum_{k \ge 1} F^{(k)}(x) e_k$.
We show that, for any $k \ge 1$,
$$
 \begin{array}{l}
X^{(k)}_t - x^{(k)} + \lambda_k\int_0^t X^{(k)}_s ds - \int_0^t (\lambda_k)^{1/2}F^{(k)}(X_s) ds
\end{array} 
$$
is a one-dimensional Wiener process $W^{(k)}= (W^{(k)}_t)$.}
 
Let $k \ge 1$. We will  modify   a well known argument
 (see, for instance, the proof of Proposition 5.3.1 in \cite{EK}).
 By the definition of martingale solution,
it follows easily that if
$f(x)= x^{(k)} = \langle x, e_k\rangle $, $x \in H$, the process
 \begin{equation} \label{f8}
\begin{array}{l}
M_t^{(k)} = 
 X^{(k)}_t  -  x^{(k)} + \lambda_k\int_0^t X^{(k)}_s ds - \int_0^t b_k(s) ds \,\; \text{is a continuous  local martingale},
 \end{array} 
 \end{equation}
   which is ${\cal F}_t^X$-adapted,  with    
 $b_k(s) = (\lambda_k)^{1/2}F^{(k)}(X_s)$ {  (to this purpose one has  to approximate the unbounded function $l_k(x)=\langle x, e_k\rangle$ by functions 
  $l_k (x) \eta (\frac{\langle x, e_k\rangle}{n})$, $n \ge 1$, where  
  $\eta \in {\tilde C_0}^{\infty}(\R)$ is such that  $\eta(s)=1$ for $|s| \le 1$).
 }     Then
    using $f(x) = (\langle x, e_k\rangle)^2 $, $x \in H$, we find that
\begin{equation} \label{g5}
\begin{array}{l}     
N_t^k = (X^{(k)}_t)^2 - (x^{(k)})^2 + 2\lambda_k\int_0^t (X^{(k)}_s)^2 ds -
2\int_0^t b_k(s)\, X^{(k)}_s  ds - t,
\end{array} 
\end{equation}
is also a continuous local martingale.  On the other hand, starting from \eqref{f8}   and
applying the It\^o formula (cf. Theorem 5.2.9 in \cite{EK}), we get
$$
\begin{array}{l}
(X^{(k)}_t)^2 = (x^{(k)})^2 - 2\lambda_k\int_0^t (X^{(k)}_s)^2 ds +
2\int_0^t b_k(s) \,X^{(k)}_s  ds + 2 \int_0^t b_k(s)   dM_s^{(k)} +  \langle M^{(k)} \rangle_t,
\end{array} 
$$
where $(\langle M^{(k)} \rangle_t)$      is the variation process of $M^{(k)}$.
Comparing this identity with \eqref{g5} we deduce: $N_t^k - 2 \int_0^t b_k(s)   dM_s^{(k)}= \langle M^{(k)}
\rangle_t -t $ and so 
 $\langle M^{(k)}
\rangle_t =t$  
 (a continuous local martingale of bounded variation is constant).
 By the  L\'evy martingale characterization of the Wiener process (see Theorem 5.2.12 in \cite{EK}) we get that $M^{(k)}$ is a real Wiener process.

\vskip 0.5 mm       
 \noindent \textit{II Step.} \textsl{   We prove that the previous Wiener processes $W^{(k)} = M^{(k)} $   are independent.}

\vskip 0.5 mm  We fix any   $N \ge 2$ and prove that $W^{(k)}$, $k=1, \ldots, N$ are independent. 
  We will argue similarly  to the first step.
By using  functions $f(x)= x^j x^k$, $j,k \in \{ 1, \ldots ,N\}$, we get that
$
\langle W^{(j)}, W^{(k)} \rangle_t
$ $ = \delta_{jk} t.
$
 Again by the  L\'evy martingale characterization of the Wiener process (cf. Theorem 3.16 in \cite{KS})
we get that $(W^{(1)}, \ldots, W^{(N)})$ is an $N$-dimensional standard Wiener process.
It follows that $\{W^{(k)} \}_{k =1, \ldots, N}$ are independent real  Wiener processes.
\end{proof}

For the martingale problem for $\L$ in \eqref{ll}  we have  the following  uniqueness  result (we refer to   Corollary 21 in \cite{PrPot}; see also Theorem 4.4.6 in \cite{EK} and Theorem 2.2 in \cite{kunze}).

\begin{theorem} \label{ria} 
  Suppose the following two conditions:

 (i)  for any $x \in H$, there exists a  martingale
solution
  for $(\L, \delta_x)$;

 (ii)  for any $x \in H$, any two  martingale solutions $X$ and $Y$ for $(\L, \delta_x)$ have the same one dimensional marginal laws (i.e., for    $t \ge 0$, the law of $X_t$ is the same as $Y_t$ on ${\cal B}(H)$).   

 Then the
   martingale
   problem for  $\L$ is well-posed. 
\end{theorem}
Throughout  Section 5 we will apply the previous result and also the next localization principle for $\L$ (cf.  Theorem 26 in \cite{PrPot}). 
\begin{theorem} \label{uni1}  Suppose   
    that for any $x \in
H$ there exists a martingale solution  for  $(\L, \delta_x)$.
   Suppose that
 there exists
a family $\{ U_j\}_{j \in J}$ of open sets $U_j \subset H$  with $\cup_{j \in J} U_j = H$  and  linear operators $\L_j$ with the same domain of $\L$, i.e., 
 $ \L_j: D(\L) \subset C_b (H) \to C_b(H)$, $j \in J $ such that

i)  for any $j \in J$, the martingale problem for $\L_j$ is well-posed.

ii) for any $j \in J$, $f \in D(\L)$, we have
$
 \L_j  f(x) = \L  f(x),\;\; x \in U_j.
$
\\ 
Then the martingale problem for $\L$ is well-posed.  
  \end{theorem}
  
 In {  Sections 6 and 7}  we   treat {\sl possibly unbounded} $F$; we will prove uniqueness 
by truncating $F$ 
% (i.e., we will multiply $F$ by some cut-off function)  and 
 and using uniqueness  for the martingale problem up to a stopping time.   According to  Section 4.6 of \cite{EK} this leads to the  concept of {\sl stopped martigale problem} for  $\cal L$ which we define now.  
  
\smallskip   
Let us fix an open set $U \subset H$ and consider the Kolmogorov operator $\L$ in \eqref{ll} with $F \in C_b (H,H)$. 
%We define the stopped martingale problem for $\L$ on $U$.

 Let $x \in H$.  A  stochastic process $Y = (Y_t)_{t \ge 0}$ with values in $H$
defined on some probability space  $ (\Omega, {\cal F}, \P)$
with continuous paths 
is a  solution of the stopped martingale problem for $(\L, \delta_x, U)$ if  $Y_0 =x$, $\P$-a.s.  and the
 following conditions hold:

 (i) $Y_t = Y_{t \wedge \tau}$, $t \ge 0,$  $\P$-a.s, where
\begin{equation} \label{ta1}
 \begin{array}{l}
\tau = \tau^Y_U= \inf \{ t \ge 0 \; : \; Y_t \not \in U  \}
 \end{array}
 \end{equation}   
($\tau = + \infty$ if the set is empty;  this exit time $\tau$ is an
${\cal F}_t^Y$-stopping time);

 (ii) for any $f \in D(\L) =  C^2_{cil}(H)$, 
%\begin{equation}\label{mart1}
%\begin{array}{l}
  $  \ds  f(Y_t) - \int_0^{t \wedge \tau}
    \L f(Y_s) ds, \;\;\; t \ge 0, \;\; \text{is a ${\cal F}_t^Y$-martingale.}$
%    \end{array}
%\end{equation}
 
 A key result
 says, roughly speaking,  
 % the converse, i.e.
 that if the (global) martingale problem for an operator  is well-posed then also the
 stopped martingale problem for such operator is well-posed for any choice
  of the open set $U$ and for any initial condition $x$ (we refer to  Theorem 22 in \cite{PrPot}; see   also the beginning of Section A.3 for a comparison  between this result and  Theorem 4.6.1 in \cite{EK}).   
   We state this  result for the operator $\L$ in \eqref{ll}.
 %  takes the following form. 
   \begin{theorem} \label{key} Suppose  that
 the  martingale problem for $\L$ is well-posed.

  Then also the
 stopped martingale problem for $(\L, \delta_x, U)$ is well-posed for any  $x \in H$ and for any open set $U$ of $H$. In particular uniqueness in law holds for the
 stopped martingale problem  for $(\L, \delta_x, U)$, for any $x \in H$ and   $U$ open set in $H$.  
 \end{theorem}

%In order to apply Theorems \ref{ria} and \ref{uni1} we need 
% existence  of regular solutions for   related Kolmogorov equations
% and  some convergence results. This will be done in the next section. 

\subsection{ Weak uniqueness  
%nonlinear
 when $F \in C_b^{\theta} (H,H)$ and  $\| F- z \|_{0} < {\tilde C_0} $    }

 Here  we show that there exists a constant ${\tilde C_0}>0$ small enough (depending on $\theta$ and $\|F \|_{C^{\theta}}$) such that if  $F \in C^{\theta}_b(H,H)$ verifies in addition
%    $z \in H$  if $F \in C^{\theta}_b(H,H)$ and 
  \begin{gather} \label{111}
\sup_{x \in H} | F(x)- z |_H =  \| F- z \|_0 < {\tilde C_0}
\end{gather}
 for some $z \in H$,   then uniqueness in law holds for \eqref{sde} for any initial $x \in H$.

To this purpose we will also use Proposition \ref{ser} on equivalence between mild solutions and martingale solutions. 
 We start with a lemma which is a consequence of Theorem \ref{ss13}. We consider the resolvent  
%R^{(z)}_{\lambda}
\begin{gather} \label{2we} 
 R^{(z)}_{\lambda} f(x) = u^{(z)}(x) =
 \int_0^{\infty} e^{-\lambda t }  P_{t}^{(z  )} f(x) dt, \;\;\; f \in B_b(H),\; x \in H,\;\; \lambda>0.
\end{gather}
\begin{lemma}   \label{stima1} 
There exists   $\lambda_0 = \lambda_0 (\| F \|_{C^{\theta}}, \theta ) >0$ large enough and  ${\tilde C_0}= {\tilde C_0}  (\| F \|_{C^{\theta}}, \theta )>0$ small enough such that if  $F \in C_b^{\theta}(H,H)$ verifies 
 \eqref{111}   for some $z \in H$ then
%\end{equation*}
for any $\lambda  \ge \lambda_0 $, $g \in  C_b^{\theta}(H)$,  
\begin{equation}\label{2dv}
  \|  \langle F - z, A^{1/2} D R^{(z)}_{\lambda} g \rangle \|_{C^{\theta}}
   \le \frac{1}{2}\, \| g  \|_{C^{\theta}}.
\end{equation}
\end{lemma} 
\begin{proof}  
    Let $g \in C_b^{\theta}(H)$, $z \in H$ and   
  $\lambda>0$.  We are considering the map: $x \mapsto \langle F(x) - z, A^{1/2} D R^{(z)}_{\lambda} g (x)\rangle$ from $H$ into $\R$.
  By  Theorem \ref{ss13} we know that  $A^{1/2} D R^{(z)}_{\lambda} g  \in C_b^{\theta} (H,H)$. Moreover, 
\begin{gather*}
\| \langle   [  F   - z]  , (-A)^{1/2}    D  R^{(z)}_{\lambda} g \rangle \|_{C^{\theta}} = \| \langle   [  F   - z]  , (-A)^{1/2}    D  R^{(z)}_{\lambda} g \rangle \|_{0} \\ + \,
 [ \langle   [  F   - z]  , (-A)^{1/2}    D  R^{(z)}_{\lambda} g \rangle ]_{C^{\theta}}.
\end{gather*}  
Recall  that  
 given two $\theta$-H\"older continuous and bounded functions $l,m$ we have 
$ [ lm ]_{C^{\theta}}$ $\le \| l\|_0 [m]_{C^{\theta}}$ $+ \| m\|_0 [l]_{C^{\theta}}$.  We get
\begin{gather*}
[ \langle   [  F   - z]  , (-A)^{1/2}     D  R^{(z)}_{\lambda} g  \rangle ]_{C^{\theta}} \le 
 {\tilde C_0}  \,  [ (-A)^{1/2}  D  R^{(z)}_{\lambda} g  ]_{C^{\theta}}  
 \\
 + [F]_{C^{\theta}} \, \| (-A)^{1/2}   D  R^{(z)}_{\lambda} g  \|_0. 
 \end{gather*}
By Theorem \ref{ss13}
there exist $M_{\theta}$  and $C_{\theta}(\lambda) \downarrow 0$ as $\lambda \to \infty$ such that 
\begin{gather*}
 \| (-A)^{1/2} D  R^{(z)}_{\lambda} g \|_0 \le  C_{\theta}(\lambda) \,   \|g \|_{C^{\theta}},\;\; [ (-A)^{1/2} D  R^{(z)}_{\lambda} g ]_{C^{\theta}}
  \le M_{\theta} \| g \|_{C^{\theta}},\;\;\; \lambda>0.
\end{gather*}
It follows that 
\begin{gather*}
[ \langle   [  F   - z]  , (-A)^{1/2}     D  R^{(z)}_{\lambda} g  \rangle ]_{C^{\theta}}
\le ({\tilde C_0} M_{\theta}  \, + C_{\theta}(\lambda)  [F]_{C^{\theta}}) \,   \| g  \|_{C^{\theta}}, \;\; \lambda>0.  
\end{gather*}
On the other hand $\|  \langle   [  F   - z]  , (-A)^{1/2}     D  R^{(z)}_{\lambda} g  \rangle \|_0  \le {\tilde C_0} C_{\theta}(\lambda) \| g  \|_{C^{\theta}}$, $\lambda >0.$
  By choosing $\lambda \ge \lambda_0$ with $\lambda_0$ large enough  and ${\tilde C_0}$ small enough, we obtain the assertion. 
  \end{proof} 

 \begin{lemma}\label{loc}   Let $x \in H$ and consider the  SPDE \eqref{sde}. 
   If there exists   $z \in H$  such that \eqref{111} holds (the constant $\tilde C_0$ is given in Lemma \ref{stima1}) 
 then we have uniqueness in law for \eqref{sde}. 
\end{lemma}  
\begin{proof} By Section 4,  for any $x \in H$, there exists a weak mild solution starting at $x$. Equivalently, by Proposition \ref{ser}, for any $x \in H$, 
 there exists a  solution to the martingale problem for $({\cal L}, \delta_x)$.
 
Taking into account  Proposition \ref{ser}, we will prove   that given  two martingale solutions $X^1$ and $X^2$  for $ ( \L,  \delta_{x})$
%which both solve \eqref{sde} and start at $x_0 \in H$   
 we   have that the law of $X_t^1$ coincides with the law of $X_t^2$ on ${\mathcal B}(H)$, for any $t \ge 0$. By Theorem \ref{ria} we will deduce that $X^1$ a $X^2$ have the same law on ${\mathcal B}(C([0, \infty);H)$.

 % We proceed in some  steps.    
  %The first part of the  proof is inspired by   the proof  of Theorem 4.1 in  \cite{D2}.  
  
  Thanks to Lemma  \ref{stima1} we will  
 adapt an argument used  in finite dimension in \cite{SV79} and \cite{IW} (see the proof of Theorem 3.3 in \cite{IW}). 
 
Let us fix $x_0 \in H$ and let $X = (X_t)  $ be a martingale solution for
 $({\cal L}, \delta_{x_0})$
%$\L$ (see \eqref{ll}) starting at $x_0 \in H$ 
 (defined on some probability space $ (\Omega, {\cal F}, \P)$); recall that
\begin{eqnarray} \label{ll} 
\L f (x)  = 
L f (x) +  
\langle F(x),  (-A)^{1/2} Df(x) \rangle,\;\;\; f \in D(\L),\; x \in H,
\end{eqnarray}
 where $Lf (x) = \frac{1}{2} Tr(D^2 f(x)) + \langle x, ADf(x) \rangle$ 
 and  $D(\L) = C^2_{cil}(H)$ the space of all regular cylindrical functions
  (cf. \eqref{cil2}).
 By the martingale property we have, for $f \in D(\L)$,
\begin{gather*}
 \E [f(X_t)] = f(x_0) + \E \int_0^t {\L} f(X_t) dt.
\end{gather*}
Integrating over $[0, \infty)$ with respect to $e^{-\lambda t}$ and using the  Fubini theorem we get 
\begin{gather*}
 \int_0^{\infty} e^{-\lambda t} \E[f(X_t)] dt = \frac{f(x_0)}{\lambda} +
  \frac{1}{\lambda}  \int_0^{\infty} e^{-\lambda t} \E[\L f(X_t)] dt. 
\end{gather*} 
Hence, introducing the bounded linear operators $G^{\lambda,x_0} : B_b(H) \to \R$:
\begin{equation}\label{rggg}
G^{\lambda,x_0} h =  \int_0^{\infty} e^{-\lambda t} \E[h(X_t)] dt,  \;\; h \in B_b(H), \;\; \lambda > 0,
\end{equation}
we can write 
\begin{equation}\label{aqq3}
\lambda G^{\lambda,x_0} f = f(x_0) + G^{\lambda,x_0} (\L f),\;\; \lambda>0. 
\end{equation}
Next  we proceed in some steps.
 
\vskip 1mm 
\noindent {\sl I step. We  check that \eqref{aqq3} holds  also  for any $f \in C^{2}_{cil, \, b} (H)$   with $D(\L) \subset C^{2}_{cil, \, b} (H)$.}
%which contains $D(\L)$.  
 
 \vskip 1mm
We say that  
  $f: H \to \R$   
   belongs to $C^2_{cil,\, b}(H)$ if there exist elements $e_{i_1}, \ldots, e_{i_m}$ of the basis $(e_k)$ of eigenvectors of $A$ and a bounded $C^2$-function  $\tilde f : \R^m \to \R$ with  all bounded derivatives such that  
\begin{gather} \label{qvv} 
 f(x) = \tilde f (\langle  x, e_{i_1}\rangle, \ldots, \langle  x, e_{i_m} \rangle),\;\;\; x \in H. 
\end{gather}  
(cf. \eqref{cil2}).  To check that \eqref{aqq3} holds for $f \in C^{2}_{cil, \, b} (H)$ we start from \eqref{qvv} and  first  use  a standard  argument to approximate   $\tilde f \in C^2_b(\R^m)$ with a 
 sequence of functions $(\tilde f_n) \subset C^2_b(\R^m)$  having compact support. To this purpose   
 let $\phi \in C^{\infty}(\R^m)$ be such that $0 \le \phi \le 1$, $\phi (x) = 1$,
  $|x| \le 1$ and $\phi (x)=0$, $|x|>2$. Define $\tilde f_n(y) = \tilde f(y) \cdot \phi_n(y)$, $y \in \R^m,$
  where $\phi_n(y) = \phi(\frac{y}{n}) $,
  for $n \ge 1$.  
  
  Set  $f_n(x) = \tilde f_n (\langle  x, e_{i_1}\rangle, \ldots, \langle  x, e_{i_m} \rangle),$ $x \in H.$ We know  that 
  \begin{gather} \label{saa}
\lambda G^{\lambda,x_0} f_n = f_n(x_0) + G^{\lambda,x_0} (\L f_n),\;\; n \ge 1. 
\end{gather}
 By Proposition \ref{ser}
  $X_t$ is also a weak mild solution. Then  we obtain easily that $\E |X_t|_H \le (ct+1)$, $t \ge 0$, for some constant $c>0$. On the other hand, using that first and second derivatives of $(\tilde f_n)$ are uniformly bounded, we have
 \begin{gather*}
\E|\L f_n(X_t)| \le (C + 1 ) \E|X_t|_H \le M t +1, \,\, \;\;  n \ge 1,
\end{gather*}
 where $C$ and $M$ are positive constants independent of $t$ and $n$ ($C$ and $M$ may depend on $m$). Passing to the limit in \eqref{saa} by the dominated convergence theorem we get easily the assertion.
 It follows that   
\begin{equation*}
 G^{\lambda,x_0} (\lambda f - \L f) = f(x_0),\;\;\; \lambda >0,\;\; f \in C^{2}_{cil, \, b} (H).
\end{equation*}
Let  $z \in H$  be such that \eqref{111} holds, we write 
\begin{gather}
 \L f (y) =
L f(y)+\langle (-A)^{1/2}Df(y), z   \rangle  +  \langle (-A)^{1/2}Df(y),    F(y) -z \rangle 
\\ \nonumber =  L^{(z)} f(y) + Sf(y),
 \\ \nonumber   
 Sf(y) = \langle (-A)^{1/2}Df(y),   F(y) -z  \rangle,\;\;\; y \in H, 
\end{gather}
 $L^{(z)} f(y) = L f(y)+\langle (-A)^{1/2}Df(y), z   \rangle$ (cf. \eqref{ou3}). Hence  
 \begin{equation}\label{srr3}
G^{\lambda,x_0} (\lambda f - L^{(z)} f) = f(x_0) + G^{\lambda,x_0} (Sf),\;\;\; \lambda >0,\;\; f \in C^{2}_{cil, \, b} (H).
\end{equation}
  Let us consider the resolvent associated to $L^{(z)}$:
\begin{equation*}\label{erq}
R^{(z)}_{\lambda}  g (x) = \int_0^{\infty} e^{-\lambda t }   P_{t}^{(z)}  g(x) dt,\;\;\; g \in B_b(H),\; x \in H
\end{equation*} 
 (cf. \eqref{2we}). Note that  $f = R^{(z)}_{\lambda}  g \in C^{2}_{cil, \, b} (H)$ if $g \in C^{2}_{cil, \, b} (H).$

{ Inserting in  \eqref{srr3} $f = R^{(z)}_{\lambda}  g $ with $g \in C^{2}_{cil, \, b} (H)$, using that $(\lambda I - L^{(z)} ) (R^{(z)}_{\lambda}  g)$
 $=g$, we obtain 
\begin{gather} \label{s455}
G^{\lambda,x_0} g =  R^{(z)}_{\lambda}  g(x_0) + G^{\lambda,x_0} (S \, R^{(z)}_{\lambda} g), \;\; \; g \in C^{2}_{cil, \, b} (H) , \;\; \lambda >0.
\end{gather}
}
 From now on we consider  $\lambda \ge \lambda_0$ (where $\lambda_0 >0$ is given in Lemma \ref{stima1}). 
 
 \vskip 2mm
 
 \noindent {\sl  II step. We   prove that  \eqref{s455} holds even for $g \in C^{\theta}_b(H)$, $\lambda \ge \lambda_0$.  }
  
\vskip 1mm
Let us first consider $g \in  C^{\theta}_b(H)$ which is cylindrical, i.e.,
 $g(x) = \tilde g (\langle  x, e_{1_1}\rangle, \ldots, \langle  x, e_{i_m} \rangle),$ $x \in H$, with $\tilde g \in C^{\theta}_b(\R^m)$.
  A standard argument based on convolution with mollifiers, shows that there exists $(\tilde g_n) \subset C_b^2(\R^m)$ such that 
  $\tilde g_n (y) \to \tilde g(y)$,  as $n \to \infty$, $y \in \R^m$, and 
  $$
  \| \tilde g_n \|_{C^{\theta}} \le \| \tilde g \|_{C^{\theta}}, \;\;\; 
  n\ge 1. 
$$
 Define $g_n(x) = \tilde g_n (\langle  x, e_{i_1}\rangle, \ldots, \langle  x, e_{i_m} \rangle),$ $x \in H.$
  
  Applying  (ii) of Theorem \ref{ss13}  with $f_n = g_n$ and $f = g$ 
 (see also \eqref{w12}) we deduce that, for $\lambda \ge \lambda_0$,  
 \begin{equation}
\label{w122} 
\lim_{n \to \infty} \langle (-A)^{1/2} D  R^{(z)}_{\lambda} g_n (x) , F(x) -z\rangle =
  \langle (-A)^{1/2}   D R^{(z)}_{\lambda} g (x) , F(x) -z\rangle,\;\; x \in H, 
\end{equation} 
i.e.,  $S R^{(z)}_{\lambda}g_n(x) \to S R^{(z)}_{\lambda}g(x)$, as $n \to \infty$, for any $x \in H$. We write  \eqref{s455}  with  $g_n$:
\begin{gather} \label{qqqq}
\int_0^{\infty} e^{-\lambda t} \E[g_n(X_t)] dt
  =  R^{(z)}_{\lambda}  g_n(x_0) + 
  \int_0^{\infty} e^{-\lambda t} \E[
    (S \, R^{(z)}_{\lambda} g_n) (X_t)] dt.
\end{gather} 
Since by \eqref{w44}
 \begin{gather*}
 \| S R^{(z)}_{\lambda}g_n \|_0 \le \tilde C_0 C_{\theta}(\lambda) \| g_n\|_{C^{\theta}} \le  \tilde C_0 C_{\theta}(\lambda_0) \| g\|_{C^{\theta}},\; \;\; n \ge 1, \; \lambda \ge \lambda_0, 
\end{gather*}
we can pass to the limit as $n \to \infty$ in \eqref{qqqq} 
 by the dominated convergence theorem and  get that   
 \eqref{s455} folds with $g \in C^{\theta}_b(H)$  cylindrical.
   
If $g \in C^{\theta}_b(H)$ we consider the cylindrical functions $g_k(x) = g(\pi_k x)$,  $x \in H$, $k \ge 1$, where $\pi_k$ are the orthogonal projections considered in \eqref{pp1}. We have that $g_k \in C^{\theta}_b(H)$ and 
\begin{gather*}
\lim_{k \to \infty} g_k(x) = g(x),\;\; x \in H,\;\;\;\;\;\; \| g_k\|_{C^{\theta}}
 \le \| g \|_{C^{\theta}},\;\; k \ge 1.
\end{gather*}
Arguing as before, applying   Theorem \ref{ss13}, we pass to the limit as $k \to \infty$ in
\begin{gather*}
 G^{\lambda,x_0} g_k =  R^{(z)}_{\lambda}  g_k(x_0) + G^{\lambda,x_0} (S \, R^{(z)}_{\lambda} g_k); 
\end{gather*}
we finally obtain that  
\eqref{s455} folds for any  $g \in C^{\theta}_b(H)$.
 
\vskip 1mm  
\noindent {\sl  III step. Given  two martingale solutions $X^1$ and $X^2$ %starting at $x_0 \in H$. 
 for $({\cal L}, \delta_{x_0})$.
We show that they have the same  law.}

\vskip 1mm  
  $X^1$ and $X^2$ are defined respectively on the   probability spaces $ (\Omega^1, {\cal F}^1, \P^1)$ and $ (\Omega^2, {\cal F}^2, \P^2)$; we consider 
\begin{gather*}
G^{\lambda,x_0}_i f =  \int_0^{\infty} e^{-\lambda t} \E^{i}  [f(X_t^i)] dt, \;\;
 f \in B_b^{}(H),\;\; i =1,2
\end{gather*}
(using the expectation on each probability space).  By \eqref{s455} we infer with $g \in C_b^{\theta}(H)$, $\lambda \ge \lambda_0$,
\begin{gather*}
(G^{\lambda,x_0}_1 -  G^{\lambda,x_0}_2) g =   (G^{\lambda,x_0}_1 -  G^{\lambda,x_0}_2) \,  (S \, R^{(z)}_{\lambda} g).
\end{gather*}
Define  $T^{\lambda ,x_0}  = G^{\lambda,x_0}_1 -  G^{\lambda,x_0}_2$. 
 Clearly,  $T^{\lambda ,x_0} :  C_b^{\theta}(H) \to \R$ is a bounded linear  functional (we denote by $\| T^{\lambda ,x_0}\|_{{\cal L}(C_b^{\theta}(H), \R )} = \| T^{\lambda ,x_0}\|_{{\cal L}}$ its norm).
 We have,  for $\lambda \ge \lambda _0$,
\begin{gather*}
 \| T^{\lambda ,x_0} \|_{{\cal L}} 
 = \sup_{ g \in  C_b^{\theta}(H),\; \| g\|_{C^{\theta}} \le 1  } \, 
  | T^{\lambda ,x_0} g |  
  = \sup_{ g \in  C_b^{\theta}(H),\; \| g\|_{C^{\theta}} \le 1  } \, 
  | T^{\lambda ,x_0} (S R^{(z)}_{\lambda} g)  | 
 \\ \le \| T^{\lambda ,x_0} \|_{{\cal L}} \,  \sup_{ g \in  C_b^{\theta}(H),\; \| g\|_{C^{\theta}} \le 1  }   \,\| S R^{(z)}_{\lambda} g \|_{C_b^{\theta}}.
%  \\
%  \le C \| T^{\lambda ,x_0} \|_{{\cal L}(C_b^{\theta}(H), \R )}  \,\|  g \|%_{C_b^{\theta}}
\end{gather*}
Using Lemma \ref{stima1} we know that $\| S R^{(z)}_{\lambda} g \|_{C_b^{\theta}} = \|  \langle F - z, (-A)^{1/2} D R^{(z)}_{\lambda} g \rangle \|_{C^{\theta}} \le \frac{1}{2} \, \| g  \|_{C^{\theta}}$. This shows that $\| T^{\lambda ,x_0} \|_{{\cal L}}  \le \frac{1}{2}\| T^{\lambda ,x_0} \|_{{\cal L}}  $  and so $ T^{\lambda ,x_0} =0$ for $\lambda \ge \lambda_0$. 
 
We obtain, for any $g \in C_b^{\theta}(H)$,  $\lambda \ge \lambda_0 >0$, 
  $$ 
  \int_{0}^{\infty}  e^{- \lambda s}  
 \E^1 [g (X_{s}^1)]ds =   \int_{0}^{\infty}  e^{- \lambda s}  
\E^2 [g (X_{s}^2)]ds.
$$  
By the  uniqueness of the Laplace transform and using an approximation argument we find that 
$ \E^1 [g (X_{s}^1)] =  \E^2 [g (X_{s}^2)]$, for any $g \in C_b (H)$, $s \ge 0$.
 Applying   Theorem \ref{ria} we find that $X^1$ and $X^2$ have the same law on  ${\cal B}(C([0, \infty); H))$.  
 This finishes the proof. 
  \end{proof}

\subsection {Weak uniqueness when  $F \in C_b^{\theta}(H,H)$ } 
 
 Here we prove uniqueness using the localization principle (cf. Theorem \ref{uni1})  and  Lemma \ref{loc}.  We will use  the constant ${\tilde C_0}$ introduced in Section 5.2. 
   \begin{lemma}\label{loc1}   Let $x \in H$ and consider the  SPDE \eqref{sde}. 
   If $F \in C_b^{\theta}(H,H)$
  then we have uniqueness in law for \eqref{sde}.
\end{lemma} 
\begin{proof}  By Proposition \ref{ser} it is enough to show that 
 the martingale problem for $\L$  is well-posed (cf. \eqref{ll}).  
 By  Section 4,  for any $x \in H$, there exists a solution to the martingale problem for $(\L, \delta_x)$.

 In order to apply Theorem \ref{uni1}  we proceed into  two  steps.
In the first step we construct a suitable covering of $H$;  in the second step we define  suitable operators ${\L}_j$ such that  according to  Lemma \ref{loc}
 the martingale problem associated to each ${\L}_j$ is well-posed.     

 \smallskip
\noindent
{\it I Step.} There exists
     a  countable set of points $(x_j) \subset H$,
$j \ge 1$, and numbers $r_j >0$   with the following properties:

\smallskip
(i) the open balls $U_j= B(x_j, \frac{ r_j}{2})= \{ x \in H \, :\, |x- x_j|_H < r_j/2 \}$ 
  form a covering for $H$;

(ii)   we have:
$ 
| F(x) - F(x_j) |_H < {\tilde  C_0}, \;\;\;  x \in B(x_j, r_j).
$
\\
\\ 
%In order to  construct the  covering,
Using   the continuity of $F$: for any $x$ we  find $r(x)>0$ such that 
$$
 \begin{array}{l}
|F(y) - F(x)|_H < {\tilde  C_0}, \;\;\;\;  y \in B(x, r(x)).
 \end{array}
 $$   
 We have a covering  $\{ U_{x}\}_{x \in H}$ with $U_x = B(x, \frac{r(x)}{2})$.
Since $H$ is a separable Hilbert space we can choose a countable subcovering $(U_j)_{j \ge 1}, $ with $U_j = B(x_j, \frac{r(x_j)}{2})$ $= B(x_j, \frac{r_j}{2})$.

\smallskip
\noindent
\textit{II Step.} We construct $\L_j$ in order to apply the localization principle.
  
   Let us consider the previous covering $(B(x_j, r_j/2))$.   We take  $\rho \in { C_0}^{\infty}(\R_+)$, 
  $0 \le \rho \le 1$, $\rho(s)=1$, $0 \le s  \le 1$, $\rho(s) =0$ for $s \ge 2$.
   Define 
   $$
   \rho_j (x) = \rho \big ( 4\,  r_j^{-2}\, |x- x_j|^2_H \big), \;\; x \in H.     
 $$ 
Now  $\rho_{j} =1$ in $B(x_j, \frac{r_j}{2})$ and  $\rho_j =0$ outside $B(x_j,   
 \, r_j )$. Set
 $$
 F_j(x) :=  \rho_{j}(x) F(x)  +  (1 - \rho_{j}(x))F(x_{j} ), \;\;  x\in H.
 $$ 
 It is easy to prove that 
 %One can prove  that     
 \begin{equation} \label{qaa}
   F_{j} \in C_b^{\theta}(H,H).
 \end{equation}
%To this purpose we first  note
% that   $(1 - \rho_{j})F(x_{j} ) \in  C_b^{\theta}(H,H)$. 
% 
% Then we consider 
% $\rho_{j} F : H \to H$. Since this is a bounded function it is enough to %check the H\"older condition for points $x, y \in H$ such that $|x-y|_H <1$. 
%   Moreover if $x, y \not \in  B(x_j,   
% \, r_j )$ then $ \rho_{j}(x) F(x) -  \rho_{j}(y) F(y) =0$. 
% Hence 
%  it is enough to check the H\"older condition for points $x, y \in B_j= %B(x_j,   
% \, r_j + 1)$. Using that $F$ is $\theta$-H\"older continuous on $B_j$ we %write 
% \begin{gather*}
%|\rho_{j}(x) F(x) -  \rho_{j}(y) F(y)| \le   |\rho_{j}(x) [F(x) - F(y)]| + 
%|F(y) [\rho_{j}(x)  -  \rho_{j}(y)]|,\;\; x,y \in B_j.  
%\end{gather*}
%It follows that  $|\rho_{j}(x) F(x) -  \rho_{j}(y) F(y)| \le C_j |x-y|%^{\theta}$, $x,y \in B_j$.  
We also have
 \begin{equation*}
 \sup_{x \in H}| F_j(x) - F(x_j) |_H  =  \sup_{x \in B(x_j, r_j)} \, | F(x) - F(x_j) |_H < {\tilde  C_0}.
 \end{equation*}
 Moreover   $F_j(x) = F(x)$, $x \in B(x_j, \frac{r_j}{2}) = U_j$. Define  
$D(\L_j) = C^2_{cil}(H)$, $j \ge 1$, and   
$$
\L f_j (x) = \frac{1}{2} Tr(D^2 f(x)) + \langle x, ADf(x) \rangle +
\langle (-A)^{1/2} F_j(x), Df(x) \rangle,\;\; f \in C^2_{cil}(H),\; x \in H.
$$
  We have 
 $\L_j f(x) =  {\L} f(x),\;\; x \in U_j,\;\; f \in C^2_{cil}(H)$ and 
  the martingale problem  for each $\L_j$,
  is well-posed by Lemma  \ref{loc} (with $F= F_j$ an\ $z = F(x_j)$).  By Theorem \ref{uni1} we find the assertion.   
   \end{proof}

\section{Proof of weak  uniqueness of Theorem \ref{base}
%when $F$ has at most   linear growth
}   

%  

%Let us fix the initial condition $x \in H$ in \eqref{sde}.
Here 
%{\it we finish the  proof of the uniqueness part of Theorem \ref{forse},} i.e.,  
we  
prove uniqueness in law for \eqref{sde} assuming that  $F: H \to H$ 
 is {\sl locally $\theta$-H\"older continuous and has at most linear growth,} i.e., it verifies \eqref{lin1}. 
 To this purpose  we will use Lemma \ref{loc1} and Theorem  \ref{key}.  
% we will truncate $F$ and 
%   show uniqueness  for the martingale problem up to a stopping time.  

\vskip 1 mm 
Let $X= (X_t)_{t \ge 0}$  be a weak mild solution of \eqref{sde} starting at $x \in H$ (under the assumption \eqref{lin1}) defined on some filtered probability space $(
\Omega,$ $ {\cal F},
 ({\cal F}_{t}), \P) $ on which it is defined a
 cylindrical ${\cal F}_{t}$-Wiener process $W$; see Section 4. For a cylindrical function $f \in C^2_{cil}(H) $ in general $\L f$ (see \eqref{ll}) is not a bounded function on $H$ because  $F$ can be unbounded. However
 we  know by a finite-dimensional It\^o's formula that 
 \begin{equation} \label{ffu}
  \begin{array}{l}
   M_t(f) =  f(X_t) - \int_0^t \L f(X_s)ds = f(x) + \int_0^t \langle Df(X_s), dW_s \rangle 
 \end{array} 
 \end{equation}
 is still   a continuous square integrable ${\cal F}_t$-martingale. Note that we can apply It\^o's formula  because there exists $N \ge 1$ such that  $f(x) = f(\pi_N x)$, $x \in H$, with $\pi_N=\sum_{j=1}^Ne_j\otimes e_j$ (cf. \eqref{pp1}) and so  $f(X_t) = f(\pi_N X_t)$.  Moreover  $X_{t, N}
  = \pi_N X_t$ verifies  
% Indeed   we have
 \begin{equation*} 
 X_{t, N}=e^{tA} \pi_N  x +\int_{0}^{t}e^{\left(  t-s\right)  A} (-A_N)^{1/2}F^{}(X(s))
ds+\int_{0}^{t} \pi_N e^{\left(  t-s\right)  A}dW_{s},\;\;\; t \ge 0,
\end{equation*}
where  $A_N=A\pi_N $; writing $\pi_N W_t = \sum_{k=1}^N W_t^{(k)} e_k$ it follows that $X_{t, N}$ is an It\^o process given by 
\begin{equation} \label{maga} 
X_{t, N}
= \pi_Nx+\int_0^tA_N X_{s,N}ds+\int_0^t  (-A_N)^{1/2}F(X_{s})ds+\pi_N W_t.
 \end{equation} 
   Now let us consider  $B(0,n) = \{x \in H \, :\, |x|_H <n \}$ and define   $F_n \in C_b^{\theta} (H,H)$ such that 
$
F_n (y) = F(y),$ $ y \in B(0,n), 
$ $n \ge 1.  $   

\vskip 0.5 mm 
 To this purpose one can take $\eta \in {C_0}^{\infty}(\R)$ such that $0 \le \eta(s) \le 1$, $s \in \R$,  $\eta(s)=1$ for $|s| \le 1$ and  $\eta(s) =0 $ for $|s|\ge 2,$ and set 
$
F_n(y) = F(y)\,  \eta \big(\frac{ |y|_H}{n}\big), $ $ y \in H
$ (one can easily  check that $F_n \in C_b^{\theta} (H,H) $).
%one can argue as in   the proof of \eqref{qaa}).   
Define  
$$ 
 \begin{array}{l}
\L_n f (y) = \frac{1}{2} Tr(D^2 f(y)) + \langle y, A Df(y) \rangle +
\langle F_n (y),  (-A)^{1/2} Df(y) \rangle, \;\; f \in C^2_{cil}(H), \, y \in H.
\end{array}  
$$   
Let us introduce the exit time   $\tau_n^X = \inf \{ t \ge 0 \, : \, |X_t|_H \ge n \}$ ($\tau_n^X = + \infty$ if the set is empty; cf. \eqref{ta1}) for each $n \ge 1$.
 It is an ${\cal F}_t$-stopping time (cf. Proposition II.1.5 in \cite{EK}).
Applying  the optional sampling theorem (cf. Theorem II.2.13 in \cite{EK})  to \eqref{ffu} we deduce  that   
\begin{gather*}
M_{t \wedge \tau_n^X}(f)=  f( X_{t \wedge \tau_n^X}) - \int_0^{t \wedge \tau_n^X} \L f(X_s)ds
= f(X_{t \wedge \tau_n^X}) - \int_0^{t \wedge \tau_n^X} \L_n f(X_{s   \wedge \tau_n^X})ds, \;\; t\ge 0,
\end{gather*}
 is a  martingale with respect to the filtration $({\mathcal F}_{t \wedge \tau_n^X})_{t \ge 0}$; note that the process $(X_{t \wedge \tau_n^X})_{t \ge 0}$ is adapted with respect to $\big ( {\mathcal F}_{t \wedge \tau_n^X} \big )$ (see Proposition II.1.4 in \cite{EK}).

%\smallskip 
Thus  $(X_{ t\wedge \tau_n^X  })_{t \ge 0} $ is a solution to  the  {\sl stopped martingale problem for $(\L_n, \delta_x, B(0,n)$).}
 By Lemma \ref{loc1}   {\sl the martingale problem for each $\L_n$ is well-posed because $F_n \in C_b^{\theta}(H,H)$}.   By Theorem \ref{key}  also the   stopped martingale problem for   $(\L_n, \delta_x, B(0,n)$) is well-posed, $n \ge 1$. 
 
%\smallskip 
\vskip 0.5 mm
Let  $Y$ be another  mild solution starting at $x \in H$. Then $(Y_{ t\wedge \tau_n^Y  })_{t \ge 0} $ also solves the stopped martingale problem for $(\L_n, \delta_x, B(0,n))$. 
  By weak uniqueness of the stopped martingale problem it follows that, for any $n \ge 1$,  $(X_{ t\wedge \tau_n^X  })_{t \ge 0} $ and $(Y_{ t\wedge \tau_n^Y  })_{t \ge 0} $ have the same law. Now it is not difficult to prove that  $X$ and $Y$ have the same law on ${\cal B}(C([0,\infty);  H))$ and this finishes the proof.

%\section{An extension to functions $F: H\to H$ which are locally bounded}
\section{An extension to functions $F: H\to H$  without imposing a growth condition}
 
%only locally  H\"older continuous}
%can growth more than linearly}  
 %By adapting  the previous proof 
 
 Assuming weak existence for \eqref{sde} one can obtain the following extension of  Theorem
  \ref{base}.

  \begin{theorem}
\label{extension} 
 Let us consider \eqref{sde} under Hypothesis \ref{d1}.
 %and  fix  any $x \in H$. 
 Assume   that  
 %there exists $\theta \in (0,1) $ such that 
 
 % \noindent  H1) $F: H \to H$ is continuous and locally bounded (i.e., $F$ is %bounded on bounded sets of $H$);
 
  H1) $F: H \to H$ is $\theta$-H\"older continuous  on each bounded set of $H$, for some $\theta \in (0,1)$;
  %(i.e.,  $F$ is locally $\theta$-H\"older continuous);  
 
 \vskip 1mm
  
 H2) for any $x \in H$, there exists a weak mild solution $(X_t)_{t \ge 0 }$ of \eqref{sde}
 %defined on some filtered probability space and 
 starting at $x$.   
 
\vskip 0,5 mm 
 Under  the previous assumptions weak uniqueness  holds, i.e., for any $x \in H$, all  weak mild solutions starting at  $x $ have the same law on ${\cal B}(C([0, \infty); H)$.      
\end{theorem}
  \begin{proof}  The proof is similar to the one of Section 6.
  %We follow the arguments of Section 6 with the same notations.  
 % Theorem \ref{key}.
   We give
 %  give a sketch of proof 
    some details  
   for the sake of completeness.  
   Let $X= (X_t)_{t \ge 0}$  be a weak mild solution of \eqref{sde} starting at $x \in H$ (defined on some filtered probability space $(
\Omega,$ $ {\cal F},
 ({\cal F}_{t}), \P) $).
  %on which it is defined a
 % cylindrical ${\cal F}_{t}$-Wiener process $W$). 
  %Following the notation of Section 6,
 We   have that $M_t(f)$ in \eqref{ffu} is   a continuous square integrable ${\cal F}_t$-martingale, for any 
   $f \in C^2_{cil}(H)$.
    %Considering balls  $B(0,n) \subset H$, $n \ge 1$,  and 
    Using   that  $F$ is locally $\theta$-H\"older continuous, we obtain that  the bounded  functions $
F_n(y) = F(y)\,  \eta \big(\frac{ |y|_H}{n}\big), $ $ y \in H,
$ belong to $C_b^{\theta}(H,H)$. 
 
 By the optional sampling theorem 
  we find that  $(X_{ t\wedge \tau_n^X  })_{t \ge 0} $ is a solution to  the  { stopped martingale problem for $(\L_n, \delta_x, B(0,n)$).}
  Using Lemma \ref{loc1} and Theorem  \ref{key}     we know that  the   stopped martingale problem for   $(\L_n, \delta_x, B(0,n)$) is well-posed, $n \ge 1$.  Proceeding as in the final part  of Section 6 we obtain the assertion.
 % uniqueness result. 
\end{proof}

\subsection{Singular perturbations of classical stochastic Burgers equations}   

Here we show that  Theorem \ref{extension} can be applied to  SPDEs \eqref{sde} in cases when  $F$ grows more than linearly. As an example we consider   
  \begin{gather} \label{bur}
d u (t, \xi)=   \frac{\partial^2}{\partial   \xi^2}  u(t, \xi)dt +   \, { h( }u(t, \xi) )\cdot  g \big (\,  \big |  u(t, \cdot ) \big |_{H^1_0 } \big) dt     
\nonumber 
+ \frac{1}{2} \frac{\partial }{\partial \xi} \big ( u^2(t, \xi) \big ) dt +
\, \sum_{k \ge 1 }  
\frac{1}{k}   dW_t^{(k)}  {e_k(\xi)},  
% d W_t(\xi), 
\\  \;\; u(0, \xi) = u_0(\xi), \;\;\; \xi \in (0,\pi),
\end{gather}
 $u(t,0) = u(t,\pi)=0$, $t >0$, $u_0 \in H^1_0(0,\pi)$;  {\sl $g: \R \to \R $ is locally $\theta$-H\"older continuous, for some $\theta \in (0,1)$,  and $h : \R \to \R$ is  a $C^1$-function with derivative $h'$ which is locally $\theta$-H\"older continuous. Moreover to get existence of solutions we require  } 
 %the following bound on $g$ and the derivative $h'$:   } 
 %and such that there exists $C>0$ such that $|g(r)| \le C (1+ |r|)$, $r \in %\R$ 
 %(we have fixed $T>0$).
 \begin{equation}\label{er3}
 \sup_{s \in \R} |h'(s)| \, \cdot  \sup_{s \in \R} |g(s) | \le 1. 
\end{equation}
For instance, we can consider   $  h( u(t, \xi) )\cdot  g \big (\,  \big |  u(t, \cdot ) \big |_{H^1_0 } \big) = u(t, \xi) \cdot  $ $\big(\sqrt{ |  u(t, \cdot ) 
|_{H^1_0 } } \,\,\,  \wedge 1 \big)$.

  Recall that $
 e_k (\xi) = \sqrt{2/\pi} \,  \sin (k \xi), $ $ \xi \in [0, \pi],\;\; k \ge 1
 $ (cf. Section 2; note that in \eqref{bur} the  noise  is ``more regular'' than  the one in \eqref{bur0}).

\vskip 1mm 

 We first  establish existence of  mild solutions with values in $H^1_0(0,\pi)$ when $g=0$ (see Proposition \ref{dap}).  This is needed in other to show that the classical Burgers equation can be considered in the form \eqref{sde}
  with  a suitable $F = F_0 : H^1_0(0,\pi) \to H^1_0(0,\pi) $ locally $\theta$-H\"older continuous  (see  \eqref{mq144}).
  To this purpose we follow the approach in  Chapter 14 of \cite{ergodicity}.
  %Note that the noise will be ``colored noise'' in $L^2(0,%%\pi)$, cf. ??? dove (in contrast with the example in %% %\eqref{bur})
 
 Then to get  well-posedness of \eqref{bur} (see Proposition \ref{well1}) we will apply the Girsanov theorem using   an exponential estimate proved in \cite{burgers}. Such  Girsanov theorem  provides existence of weak solutions (cf. Remark \ref{che}). Uniqueness in law is  obtained directly using Theorem  \ref{extension}.

 \vskip 1mm 
 We need to review    basic facts about fractional powers of the  
operator  
$A = \frac{d^2}{d  \xi^2}$ with Dirichlet boundary conditions, i.e., $D(A) = H^2(0, \pi) \cap H^1_0 (0, \pi)$ (cf. Section 2).  The
eigenfunctions are
 $
 e_k (\xi),$ $  k \ge 1,
 $ with 
  eigenvalues $- k^2$ (we set
  $\lambda_k =   k^2 $). 
 For $v \in L^2(0, \pi)$ we write $v_k = \langle v,  e_k\rangle$ $= \int_0^{\pi} v(x) e_k (x) dx $, $k \ge 1$.

We introduce  for $ s >0 $  the Hilbert spaces 
  \begin{equation}\label{s335}
   \begin{array}{l}
{\mathcal H}_{s} = D((-A)^s) = \big \{ u \in L^2(0, \pi) \, : \, \sum_{k \ge 1} \lambda_k^{2s} \,  u_k^2 =\sum_{k \ge 1} k^{4s} \,  u_k^2  < \infty \big \}.
%;\;\;\; {\mathcal H}_0 = L^2(0, \pi).
\end{array} 
\end{equation}
Moreover, for any $u \in {\mathcal H}_s$, $(-A)^s u = $ $
 \sum_{k \ge 1} k^{2s} \,  u_k e_k  $. 
 He also set ${\mathcal H}_0 = L^2(0, \pi)$.
  We have $\langle  u,v \rangle_{{\mathcal H}_s} = \sum_{k \ge 1} k^{4s} \,  u_k v_k$  (note that $|u|_{L^2} \le |(-A)^s u|_{L^2}= |u|_{{\mathcal H}_s}$, $u \in {\mathcal H}_s $, $s >0$).  

 \vskip 1mm  If $u \in H^1_0 (0,\pi)   $,  $|u|_{H^1_0(0,\pi) } = |u'|_{L^2(0,\pi) }$, where $u' $ is  the weak derivative of $u$.   We have   
\begin{gather}\label{h1o}  
{\mathcal H}_{1/2} =  H^1_0(0,\pi) \;\; \text{with equivalence of norms;  }
\\ 
%   
 % ${\mathcal H}_{1/2} =  H^1_0(0,\pi)$ 
%(where $H^1_0$ is the Sobolev space of all absolutely %continuous functions with vanishes at $0$ and $\pi$ and %having square integral derivative)  
%with equivalence of norms. We also have the  embedding  
\label{sob1}   
{\mathcal H}_{1/8} \subset L^4(0,\pi)
\end{gather}
(with continuous inclusion, i.e., there exists $C>0$ such that $|u|_{L^4} \le C |u|_{{\mathcal H}_{1/8}}$, $u \in {\mathcal H}_{1/8}$). 
 Assertion  \eqref{sob1} follows by a classical 
 Sobolev embedding theorem (cf. Theorem 6.16 and Remark 6.17 in \cite{hairer}). We only note that if $u \in {\mathcal H}_{1/8}$ one  can consider the odd extension $\tilde u$ of $u$ to $(- \pi , \pi)$; it is easy to check that $\tilde u$ belongs to the space $ H^{1/4}(-\pi, \pi) $ considered in \cite{hairer}.

 \vskip 1mm 
 We also have with  continuous inclusion (cf. Lemma 6.13 in \cite{hairer})
 \begin{equation}\label{sob2}
{\mathcal H}_{s} \subset \{ u \in C([0, \pi]),\; u(0)=u(\pi)=0\}, \;\;\; s > 1/4.
\end{equation}
 Now let us consider the linear bounded operator $T : {\mathcal H}_{1/2} \to {\mathcal H}_{1/2}$, $T u = (-A)^{-1/2} \partial_{\xi} u $, $u \in {\mathcal H}_{1/2}$;  $T$ can be extended to a linear and bounded operator   $T : {\mathcal H}_0 = L^2(0, \pi)\to {\mathcal H}_0$, see  Section 2.0.1. 
By interpolation  it follows that 
\begin{equation}\label{inter}
 T =(-A)^{-1/2} \partial_{\xi} \;\; \text{is a bounded linear operator from    $\; {\mathcal H}_{s}$ into $ {\mathcal H}_s$, $s \in [0,1/2]$}.
\end{equation}
Indeed by Theorem 4.36 in \cite{interpola} we know that 
  ${\mathcal H}_{s/2} $ can be identified with the real interpolation space $({\mathcal H}_0, {\mathcal H}_{1/2})_{s,2}$, $s \in (0, 1)$. Applying Theorem 1.6 in \cite{interpola}   we deduce \eqref{inter}.   

\vskip 1mm Let $T>0$.  For $g \in C([0,T]; {\mathcal H}_0)$ we define $(Sg)(t) =  \int_{0}^{t} e^{(  t-s)  A} g  (s) ds$, $t \in [0,T]$. 
%It is easy to
One can prove that 
 $Sg \in C([0,T]; {\mathcal H}_s)$, for any  $s \in [0,1)$. More precisely, 
\begin{equation}\label{bou1}
S \; \text{is a bounded linear operator from $C([0,T]; {\mathcal H}_0)$ into $C([0,T]; {\mathcal H}_s)$, $s \in [0,1)$.}
\end{equation}
 This result can be also deduced from Proposition 5.9 in \cite{DZ} with $\alpha =1$, $E_1 = {\mathcal H}_s$ and $E_2 = {\mathcal H}_0$. We only remark that, for any $p>1$, $L^p (0,T; {\mathcal H}_0) \subset C([0,T]; {\mathcal H}_0)$ (with continuous inclusion) and  $|(-A)^{s} e^{tA} x|_{{\mathcal H}_0}$ $=| e^{tA} x|_{{\mathcal H}_s} \le \frac{C}{t^{s}} |x|_{{\mathcal H}_0}$ (see Proposition 4.37 in \cite{hairer}).

\vskip 1mm      
In the next proposition, assertion (i) extends  a result   of \cite{ergodicity} which actually shows the existence of a mild solution to the stochastic Burgers equations with continuous path in ${\mathcal H}_s$, $s \in (0,1/4)$.        
 Assertion (ii) is proved in \cite{burgers}.
 \begin{proposition} \label{dap}
 Let us consider \eqref{bur} with $g=0$.  Then the following assertions hold:
 
 i) for any $u_0 \in {\mathcal H}_{1/2}$ there exists a pathwise unique mild solution $Y = (Y_t) = (Y_t)_{t \ge 0 }$ with continuous paths in ${\mathcal H}_{1/2}$. 
 
 ii)  The following estimate holds, for any $T>0,$
 \begin{equation}\label{esp}
 \E \Big [ \exp \Big ( {\frac{1}{2} \int_0^T \big |Y_s \big |_{{\mathcal H}_{1/2}}^2 ds} \Big )\Big]  < \infty .
\end{equation}
 \end{proposition}
\begin{proof} 
 According to \cite{ergodicity} and  \cite{burgers}, setting   $u(t, \cdot ) = Y_t$ we write \eqref{bur}  with $g=0$ as 
\begin{equation}\label{mil3}
 \begin{array}{l}
  Y_t = e^{tA} u_0 \, + \, \frac{1}{2}\int_{0}^{t} e^{(  t-s)  A} \, \partial_{\xi}  (
Y_{s}^2)
ds
+\int_{0}^{t}e^{(  t-s)  A} \sqrt{C} dW_{s},\;\; t \ge 0,
\end{array} 
\end{equation}
where $W_t = \sum_{k \ge 1 }  
   W_t^{(k)} e_k $ is a cylindrical Wiener process on ${\mathcal H}_0 = L^2(0,\pi)$ and 
   $C= (-A)^{-1}: $ ${\mathcal H}_0 \to {\mathcal H}_0 $ is  symmetric, non-negative
and of trace class,
%with eigenfunctions $(e_k)$, 
 $C e_k = \frac{1}{k^2} e_k$, $k \ge 1$.  
  
 \vskip 1mm 
\noindent {\bf (i)} In Theorem 14.2.4 of \cite{ergodicity} (see also the references therein) it is proved that, for any $T>0,$ there exists a pathwise unique solution $Y$ to \eqref{mil3} on $[0,T]$ such that, $\P$-a.s., $Y \in C([0,T]; {\mathcal H}_0) \cap 
L^2(0,T; {\mathcal H}_{1/2})$ (i.e., $\P$-a.s, the paths of $Y$ are continuous  with values in ${\mathcal H}_0$ and square-integrable with values in ${\mathcal H}_{1/2}$); 
such  result  holds even  if we replace $C $ by identity $I$. 
 By a standard argument based on the pathwise uniqueness, we get a solution $Y$ defined on $[0, \infty)$ which verifies  $Y \in C([0,\infty); {\mathcal H}_0) \cap 
L^2_{loc}(0, \infty; {\mathcal H}_{1/2})$, $\P$-a.s.. 

Let us fix $T>0$. To prove our assertion, we will show that
\begin{equation}\label{s119}
 Y \in C([0,T]; {\mathcal H}_{1/2}), \;\;\; \P\text{-a.s.}.
\end{equation}
Note that the stochastic convolution $W_A(t)= \int_{0}^{t}e^{\left(  t-s\right)  A} \sqrt{C} dW_{s}$ has a modification with continuous paths with values in ${\mathcal H}_{1/2}$
 %% correzione  
(to this purpose one can use Theorem 5.11 in \cite{DZ}).  
  Moreover in  Lemma 14.2.1 of \cite{ergodicity} it is proved that the operator $R$,
\begin{equation}\label{dz1}
\begin{array}{l}
(Rv)(t) = \int_{0}^{t} e^{(  t-s)  A} \, \partial_{\xi}  
v(s)ds ,\;\;\; t \in [0,T],\;   v \in C([0,T]; {\mathcal H}_{1/2}),
\end{array} 
\end{equation}
can be extended to a linear and bounded operator from $C([0,T]; L^1(0, \pi))$ 
 into $C([0,T]; {\mathcal H}_s)$, $s \in (0, 1/4)$. 
 It is straightforward to check that   the mapping: $ h \mapsto h^2$ is 
 %locally Lipschitz  
  continuous from $C([0,T]; {\mathcal H}_0)$ into $C([0,T]; L^1(0, \pi))$. 
  %(recall that locally Lipschitz continuous  means Lipschitz continuous on each bounded set in $H$).  
  It follows that  
\begin{gather*}
 \begin{array}{l}
u \mapsto R (u^2) \;\; \text{
%is locally Lipschitz 
 continuous  from $C([0,T]; {\mathcal H}_0)$ into $C([0,T]; {\mathcal H}_s)$.}
 \end{array}
\end{gather*} 
We deduce from  Lemma 14.2.1 of \cite{ergodicity} that the solution $Y \in C([0,T]; {\mathcal H}_{s}),   \P$-a.s., $s \in (0,1/4)$. To get more spatial regularity for $Y$  we proceed in two steps.

\vskip 1mm 
\noindent {\it I Step.}  We show that, $ \P$-a.s,  $Y \in C([0,T]; {\mathcal H}_{s})$, $s \in (0,1/2)$.

 Let us fix $s = 1/8$.  By \eqref{sob1} we know that the mapping: $ h \mapsto h^2$ is continuous from $C([0,T]; {\mathcal H}_{1/8})$ into $C([0,T]; {\mathcal H}_0)$. Moreover, using \eqref{inter} we can write, for $w \in C([0,T]; {\mathcal H}_0), $
\begin{gather*}
(Rw)(t) = \int_{0}^{t} e^{(  t-s)  A} \, \partial_{\xi}  
w(s) = \int_{0}^{t} e^{(  t-s)  A} \, (-A)^{1/2}  \, [(-A)^{-1/2} \partial_{\xi}]  
w(s) 
ds, \; t \in [0,T]. 
\end{gather*}
Note that    $[(-A)^{-1/2} \partial_{\xi}]  
w  \in C([0,T]; {\mathcal H}_0)$.  
 By \eqref{bou1} we know that, for any $\epsilon \in (0,1)$,  
 $$ t \mapsto (-A)^{1- \epsilon}\int_{0}^{t} e^{(  t-s)  A}  \, [(-A)^{-1/2} \partial_{\xi}]  
w(s)  ds$$ belongs to $C([0,T]; {\mathcal H}_0)$. Hence 
$$
(-A)^{s} Rw \in C([0,T]; {\mathcal H}_0),\;\; s \in (0, 1/2),\; i.e.,\;  
Rw \in C([0,T]; {\mathcal H}_s), \;\; s \in (0, 1/2).
$$ 
%the mapping:
%$$
%\displaystyle{ 
%  (Rw)(t) = \int_{0}^{t} e^{(  t-s)  A} \, (-A)^{1/2}  \, [(-A)^{-1/2} %\partial_{\xi}] w(s) 
%ds \; \; \text{belongs to } \, C([0,T]; {\mathcal H}_s),\; s \in (0, 1/2). 
%}    
%$$
  Using this fact we easily  obtain  that,
  $ \P$-a.s,  $Y \in C([0,T]; {\mathcal H}_{s})$, $s \in (0,1/2)$.
 
\vskip 1mm 

\noindent  {\it II Step.}  We show that $Y \in C([0,T]; {\mathcal H}_{1/2}),  $ $\P$-a.s..  
 
\vskip 1mm 
 Let us fix $s \in (1/4, 1/2)$ and  recall \eqref{sob2}. According to \cite{grisvard} the space  ${\mathcal H}_s$ can be identified with $\{ u \in W^{2s,2}(0, \pi) \, :\,  
 u(0)=u(\pi)=0\}$, where 
 % the Sobolev-Slobodeckij  space $W^{s,2}(0, \pi)$ is defined as follows
\begin{equation*}\label{gris}
 \begin{array}{l}
W^{2s,2}(0, \pi) = \big \{ u \in {\mathcal H}_0 \, :\, [u]_{W^{2s,2}(0, \pi)}^2 
= \int_{0}^{\pi}  \int_{0}^{\pi}  |u(x) - u(y)|^2 \, |x-y|^{-1 - 4s} \, dx dy < \infty  \big \}
\end{array} 
\end{equation*}
 is  
 a Sobolev-Slobodeckij  space; the norm $|u|_{W^{2s,2}(0, \pi)} = |u|_{{\mathcal H}_0} $ $+ [u]_{W^{2s,2}(0, \pi)}$  is equivalent to $|u|_{{\mathcal H}_s}$
 (see also Theorem 3.2.3 in \cite{analityc}, taking into account that 
 ${\mathcal H}_s $ can be identified with the real interpolation space $({\mathcal H}_0, D(A) )_{s,2} $     by Theorem 4.36 in \cite{interpola}). 
 
 Using the previous characterization and \eqref{sob2} it is easy to prove that if  $u \in {\mathcal H}_s$ and $v \in {\mathcal H}_s$ then the pointwise product $u v \in {\mathcal H}_s $. Indeed we have 
 % if  $u \in W^{2s,2}(0, \pi)$ then $u^2 \in W^{2s,2}(0, \pi) $. 
  % Indeed since $W^{2s,2}(0, \pi) \subset C([0, \pi])$ and so 
  $$
   |u(x) v (x) - u(y) v (y)| \le \| u\|_{0} \, |v(x) - v(y)|
   +  \| v\|_{0} \, |u(x) - u(y)|, \;\; x,y \in [0, \pi],
 $$
 and so  $[u v]_{W^{2s,2}(0, \pi)} \le c |u|_{W^{2s,2}(0, \pi)}
  \, |v|_{W^{2s,2}(0, \pi)}
  \le c' |u|_{{\mathcal H}_s}\, |v|_{{\mathcal H}_s}$. It follows that 
   $|uv|_{{\mathcal H}_s} \le C |u|_{{\mathcal H}_s}\, |v|_{{\mathcal H}_s}$. 
 % In particular we have   $ |u^2|_{{\mathcal H}_s} \le C |u|^2_{{\mathcal H}_s}$.
   
 Let now $u  \in C([0,T]; {\mathcal H}_{s})$.  Using that $|u^{2}(t) - u^2(r)|_{{\mathcal H}_s} \le |u(t) - u(r)|_{{\mathcal H}_s} 
  |u^{}(t) + u(r)|_{{\mathcal H}_s} $ $\le  $ $2 |u|_{C([0,T]; {\mathcal H}_{s})} |u(t) - u(r)|_{{\mathcal H}_s}$, $t, r \in [0,T]$, we see that 
    the mapping: 
\begin{equation}
   u \mapsto u^2 \;\; \text{  is continuous from $C([0,T]; {\mathcal H}_{s})$ into $C([0,T]; {\mathcal H}_{s})$}.
\end{equation}
 Hence, 
taking into account I Step, to get the assertion it is enough to prove that 
\begin{gather} \label{da3}
R \eta \in C([0,T]; {\mathcal H}_{1/2}) \;\; \text{if} \; \eta \in C([0,T]; {\mathcal H}_{s}) , \;\; s \in (1/4, 1/2).
\end{gather}
This would imply   $R (\eta^2)  \in C([0,T]; {\mathcal H}_{1/2}) $  $\text{if} \; \eta \in C([0,T]; {\mathcal H}_{s})$ and so  $Y \in C([0,T]; {\mathcal H}_{1/2}),  $ $\P$-a.s..  
   Let us fix $\eta \in C([0,T]; {\mathcal H}_{s}).$
 Using \eqref{inter} we can write 
\begin{gather*}
(R\eta)(t) = \int_{0}^{t} e^{(  t-s)  A} \, \partial_{\xi}  
\eta (s) = \int_{0}^{t} e^{(  t-s)  A} \, (-A)^{1/2}  \, [(-A)^{-1/2} \partial_{\xi}]  
\eta (s) 
ds, \; t \in [0,T], 
\end{gather*}
 where $ [(-A)^{-1/2} \partial_{\xi}]  
\eta   \in C([0,T]; {\mathcal H}_s)$. Hence $
\theta(r)=  (-A)^s [(-A)^{-1/2} \partial_{\xi}]  
\eta(r) \in C([0,T]; {\mathcal H}_0)$. Writing
$$
(R\eta)(t) = \int_{0}^{t} e^{(  t-r)  A} \, (-A)^{1/2 -s}  \, \theta (r) dr, \;\; t \in [0,T],
$$
 and using 
 \eqref{bou1}, we find that $(-A)^{1/2} R \eta \in C([0,T]; {\mathcal H}_0)$
and this shows \eqref{da3}.
 
\vskip 2 mm \noindent {\bf (ii)  } A similar  estimate is proved in  Propositions 2.2 and 2.3 in \cite{burgers}. However in \cite{burgers} equation \eqref{mil3} is considered in $L^2(0,1)$ (instead of $ L^2(0, \pi)$); the authors prove 
that $\E \big [ e^{\epsilon  \int_0^T  |Y_s  |_{H^1_0(0,1)}^2 \,  ds}\, \big] $ $ < \infty$ if $\epsilon \le \epsilon_0 = \pi^2 / 2 \|C \|$ (using the operator   norm  $\|C \|$ of $C$).

The condition $\epsilon \le \epsilon_0$  is used in the proof of Proposition 2.2 in order to get the inequality  $- |x|^2_{H^1_0} + 2 \epsilon |\sqrt{C} x|_{L^2}^2 \le 0$, $x \in H^1_0$. In our case   $\epsilon_0 =1/2$ since $\| C\|=1$.
 % The proof is complete.  
 \end{proof}
In the remaining part   we consider  
$$
\H = H^1_0 (0, \pi) = {\mathcal H}_{1/2}
$$  
as the reference Hilbert space and  study the SPDE \eqref{bur} 
 in $\H$.  
 
 We will  consider  the following restriction of $A:$ 
 \begin{equation} \label{aa2}
  \A = \frac{d^2}{d  \xi^2} \;\; \text{ with   $D(\A) = \big \{ u \in H^3(0, \pi) \; : \; u, \frac{d^2 u}{d  \xi^2}  \in  H^1_0 (0, \pi) \big \}$;} \;\;\; \A : D(\A) \subset \H \to \H.  
\end{equation}
Eigenfunctions of $\A$ are $\tilde e_k(\xi)= \sqrt{2/\pi} \, \frac{1}{k}  \sin (k \xi) =  
 \frac{1}{k} e_k(\xi)
$ with  eigenvalues  $- k^2$, $k \ge 1$.   

It is clear that $\A$ verifies Hypothesis  \ref{d1} when  $H = \H$. Moreover
 $(\frac{e_k}{k}) = (\tilde e_k)$ forms  an orthonormal basis in $\H$. 
 The noise in \eqref{bur} will be indicated by $\W$; it  is a 
   cylindrical Wiener process on $\H$:
   \begin{equation}\label{noise}
   \begin{array}{l} 
\W_t(\xi) = \sum_{k \ge 1 }  \frac{1}{k}
   W_t^{(k)}  {e_k(\xi)} = \sum_{k \ge 1 }   
   W_t^{(k)}  {\tilde e_k(\xi)},\;\;\; t \ge 0,\; \xi \in [0,\pi]. 
\end{array} 
\end{equation} 
Let $D_0$ be the space of infinitely differentiable functions vanishing in a 
neighborhood of $0$ and $\pi$. Such functions are dense in $\H$. The operator
\begin{equation}\label{s33}
(-\A )^{1/2} \partial_{\xi} : D_0 \to \H \; \text{can be extended to a bounded linear operator from $\H$ into $\H$}. 
\end{equation}
To check this fact we consider $y \in D_0$ and $x \in \H$. Define $x_N = \sum_{k = 1 }^N   
   x_k  {\tilde e_k}$, with $x_k = \langle x, \tilde e_k \rangle_{\H}$, $N\ge 1$.
  Using that $(- \A )^{1/2}$ is self-adjoint on $\H$ and integrating by parts we find 
 %(we use inner product in $ \H = H_0^1(0, \pi)$ and )
\begin{gather*} 
\begin{array}{l}
 \langle (-\A )^{-1/2}  \partial_{\xi} \, y , x_N \rangle_{\H} = \langle   \partial_{\xi} y , (- \A)^{-1/2} x_N \rangle_{\H} =  
  \langle \partial_{\xi}^2    y , \partial_{\xi} \, \sum_{k = 1 }^N   
   \frac{x_k}{k}  {\tilde e_k} \rangle_{L^2(0,\pi)}    
  \\ \\
=   - \langle \partial_{\xi}    y , \partial_{\xi}^2 \, \sum_{k = 1 }^N   
   \frac{x_k}{k}  {\tilde e_k} \rangle_{L^2(0,\pi)}
   %= 
  %  - \langle \partial_{\xi}     y , \, \partial_{\xi}  \sum_{k = 1 }^N   
  %   {x_k} \frac{ \cos (k \, \cdot)}{k}  \rangle_{L^2(0,\pi)}
    =
     \langle \partial_{\xi}     y ,   \sum_{k = 1 }^N   
    {x_k} { \sin (k \, \cdot)}  \rangle_{L^2(0,\pi)}. 
    \end{array}
  \end{gather*}   
  Hence $|\langle (-\A )^{-1/2}  \partial_{\xi} \, y , x_N \rangle_{\H}\, |$
  $\le |y|_{\H} \, $ $ (\sum_{k = 1 }^N   
    {x_k}^2)^{1/2}$ $\le |y|_{\H} \,|x|_{\H} $ 
   and 
  we  get the assertion. 
   Let us introduce, for any $x \in \H$,  
   \begin{equation} \label{f02}
   \begin{array}{l}
   F_0 (x) =  \frac{1}{2}  {\A}^{-1/2} \partial_{\xi} [x^2].
   \end{array}
   \end{equation}
   Since the mapping $x \mapsto x^2$ is  locally Lipschitz   from  $\H$ into $\H$   (i.e., it is Lipschitz continuous  on bounded sets of $\H$ with values in $\H$; recall that $|x^2|_{\H} = 2 |x \, \partial_{\xi} x|_{L^2(0,\pi)}$)
it is clear that also
\begin{equation} \label{f0} 
  F_0: \H \to \H  \;\; \text{is locally  Lipschitz.  }
  %% and locally bounded.}  
\end{equation}
The mild solution $Y $ of Proposition \ref{dap} with paths in $C([0,\infty); \H)$ verifies, $\P$-a.s.,
\begin{equation}
  \label{mq144}
Y_{t}=e^{t \A }x +\int_{0}^{t}(- \A)^{1/2}e^{(  t-s)  \A}F_0  (
Y_{s})ds + \int_{0}^{t}e^{(  t-s)  \A}d \W_{s};\;\; t \ge 0, 
\end{equation}
 where $\A$ is defined in \eqref{aa2} and we have set $u_0=x \in \H$.
  
  \smallskip 
  {\sl We consider the following  SPDE which includes 
    \eqref{bur} as a special case:}
 %we consider the following  SPDE 
 %in the  mild form:
 %more general form of equation  \eqref{bur}: 
 % with $f$ and $g$   
 %in the abstract form
%   which  we write as
 \begin{equation}
  \label{cheef}
X_t = e^{t \A }x +\int_{0}^{t}(- \A)^{1/2}e^{(  t-s)  \A}F_0  (
X_{s})ds
 + \int_{0}^{t} e^{(  t-s)  \A} B  (
X_{s})ds 
+
\int_{0}^{t}e^{(  t-s)  \A}d \W_{s},
 \end{equation}
$t \ge 0 $, where  
\begin{equation}\label{nonl}
B: \H \to \H \; \; \text{is locally $\theta$-H\"older  continuous  and } \;\;  |B(x)|_{\H} \le c_0 +  |x|_{\H}, \;\; x \in \H,
\end{equation}
 for some $\theta \in (0,1)$ and $c_0 \ge 0.$ 
 Note that  \eqref{bur} can be written in the form \eqref{cheef} by choosing  
 \begin{equation} \label{s444}
B(x) = h(x)   g(| x|_{\H}), \;\; x \in \H. 
 \end{equation} 
 To check that such $B$ is locally $\theta$-H\"older  continuous we argue similarly to \eqref{bene}.
 Let  $u,v \in B =\{ x \in \H \, :\, |x|_{\H} \le M \}$, for some $M>0$.  Recall  \eqref{sob2}.   
 There exists $C>0$ such that if $u \in B $ then $\sup_{r \in [0, \pi]} |u(r)|  \le CM$.  
We have 
\begin{gather*} \nonumber
\int_0^{\pi} |h'(u(t)) u'(t) g(|u|_{\H}) - h'(v(t)) v'(t) g(|v|_{\H})|^2 dt \\ 
\nonumber \le 
3 \int_0^{ \pi} |h'(u(t)) - h'(v(t))|^2  |u'(t)|^2  |g(|u|_{\H})|^2 dt 
+
 3 \int_0^{ \pi} |h'(v(t))|^2  |u'(t) - v'(t)|^2  |g(|u|_{\H})|^2 dt
\\
+ 3 \int_0^{ \pi} |h'(v(t))|^2  |v'(t)|^2  |g(|u|_{\H}) -  g(|v|_{\H})|^2 dt 
\\
\le 3 c_1 \| u -v \|_0^{2\theta}  \int_0^{ \pi}   |u'(t)|^2   dt
+ 3
  c_1 \int_0^{ \pi}   |u'(t) - v'(t)|^2   dt + 3 c_1  | u - v |^{2\theta}_{\H}
  \int_0^{ \pi}  |v'(t)|^2  dt 
\\ \nonumber \le 
 c_2 
  \, | u - v |^{2\theta}_{\H},
\end{gather*}
for some constants $c_1$ and $c_2$ possibly depending on $M, g, h $ and $\theta$.

 The function $B$ in \eqref{s444}  verifies    \eqref{nonl}  with $c_0 = 0$
 (we only note that  $|B(x)|_{\H}^2 \le \| g\|_{0}^2 \, \int_0^{\pi} |h '(x(\xi)) \cdot\,  \frac{dx}{d \xi}|^2 d \xi $ $\le \| g\|_{0}^2 \, \| h'\|_{0}^2 \, |x|_{\H}^2$ $\le  |x|_{\H}^2$, $x \in \H$). 
% by \eqref{er3} we get \eqref{nonl}  with $c_0=0$). 
%In  \eqref{chee}  $\A$ verifies Hypothesis   \ref{d1} with $H = \H$. %Moreover  
%$$
 % The next result in particular  
%  shows  weak well-posedness for the  SPDE   \eqref{cheef}.
  %  in the form \eqref{cheef} assuming \eqref{nonl}. 
  % Thus in particular it covers the initial SPDE \eqref{bur}. 
  
 We have used  condition \eqref{er3} to guarantee the bound  in \eqref{nonl}. This is   used to check the  Novikov condition \eqref{novikov} and prove the existence part in the following result.
%  We believe that, changing the proof of (i), avoiding the Girsanov theorem, % %one should be possible to assume only that $B$ grows at most linearly.
\begin{proposition} \label{well1} 
 Let us consider \eqref{cheef} on    
 %a real and separable Hilbert space  
 $\H= H^1_0(0, \pi)$ with 
  $\A$ given in \eqref{aa2} and the cylindrical Wiener process $\W$ on $\H$ given in \eqref{noise} ($(W^{(k)})_{k \ge 1
  }$ are independent real Wiener processes). 
   % $\W$ Moreover the cylindrical Wiener
  % verifies Hypothesis   \ref{d1} with $H = \H$. 
% Let $(\tilde e_k)$ be the  orthonormal basis of eigenfunctions of $\A^{-1}$ %and 
% $\W_t = \sum_{k \ge 1}  W_t^{(k)}  \tilde e_k$ be a cylindrical Wiener %process in $\H$ 
  Let  $F_0$ as  in \eqref{f02} and  suppose that   $B: \H \to \H$   verifies  \eqref{nonl}. 
   Then the following assertions hold.
  
 i) For any $x \in \H$,  there exists a weak mild solution $(X_t)_{t \ge 0 }$.
 %defined on some filtered probability space.    
 
ii)  Weak uniqueness holds for \eqref{cheef} for any $x \in \H.$
 \end{proposition}
\begin{proof} {\bf i)} Let us fix $x \in \H$. We will use the Girsanov theorem  as in Appendix A.1 of \cite{DFPR}, using the  reference Hilbert space $\H$.
%(see also Remark 3.7, page 303, in \cite{KS}).   

 Let $Y = (Y_t)$ be the unique  solution to the Burgers equation  \eqref{mq144} with values in $\H$ and such that  $Y_0 =x$. This is defined on  a  filtered probability space 
 $(
\Omega, {\mathcal F},$ $
 ({\mathcal F}_{t}), \P )$  
  on which it is defined the
cylindrical Wiener process $\W$ on $\H$.
 Set 
\begin{gather*}
b(s) = B(Y_s),\;\; s \ge 0,
\end{gather*}
and note that $|b(s)|_{\H} \le c_0  +   |Y_s|_{\H}$, $s \ge 0$,   by \eqref{nonl}. The process   $(b(s))$ is  progressively measurable  and verifies $\E \int_0^T
|b(s)|_{\H}^2    ds < \infty $, $T>0$ (see \eqref{esp}  and recall that $e^r \ge  1 + r$). Moreover, by \eqref{esp} and \eqref{nonl}  it follows  that, for any $T>0,$ 
\begin{gather} \label{novikov}
\E  \big [ e^{\frac{1}{2} \int_0^T  |b(s)  |_{\H }^2 ds}\big]  \le \, C_T \, 
 \E  \big [ e^{\frac{1}{2} \int_0^T  | Y_s |_{\H }^2 ds}\big] < \infty.
\end{gather} 
%(this the usual  Novikov condition). 
Let ${U_t = \sum_{k \ge 1} \int_0^t \langle b(s) , \tilde e_k\rangle_{\H}  \, dW^{(k)}_s}$, $t  \ge  0$, and   
    fix $T>0.$ By Proposition 17 in \cite{DFPR} we know that $\tilde W^{(k)}_t = W^{(k)}_t - \int_0^t \langle \tilde e_k, b(s)\rangle_{\H} ds$, $t \in [0,T]$, $k \ge 1,$ are independent real Wiener processes on $(
\Omega, {\mathcal F},$ $
 ({\mathcal F}_{t}), \tilde \P )$, where 
  the probability measure 
 $$
 \tilde \P = e^{ U_T \,  
 - \, \frac{1}{2} \int_0^T  |b(s)  |_{\H }^2 ds} \, \cdot  \P
 $$ is equivalent to $\P$ (the quadratic variation process $\langle U\rangle_t$ $= \int_0^t \big |b(s) \big |_{\H }^2 ds$, $t \in [0,T]$). 
 
 Hence  $\tilde \W_t = \sum_{k \ge 1} \tilde W^{(k)}_t \tilde e_k$, $t \in [0,T]$, is a cylindrical Wiener process on $\H$ defined on $(
\Omega, {\mathcal F},$ $  
 ({\mathcal F}_{t}), \tilde \P )$. 
  %Using the process $Y$ and 
  Arguing  
as in Proposition 21 of \cite{DFPR} we obtain that 
\begin{gather*}
 \begin{array}{l}
Y_t = e^{t \A }x +\int_{0}^{t}(- \A)^{1/2}e^{(  t-s)  \A}F_0  (
Y_{s})ds
 + \int_{0}^{t} e^{(  t-s)  \A} B  (
Y_{s})ds 
+
\int_{0}^{t}e^{(  t-s)  \A}d \tilde \W_{s}, \;\; t \in [0,T],
 \end{array}
\end{gather*}
 $\P$-a.s.. 
 Thus $Y$ a mild solution on $[0,T]$ to \eqref{cheef} defined on $(
\Omega, {\mathcal F},$ $
 ({\mathcal F}_{t}), \tilde \P )$. 
 
 Since $T>0$ is arbitrary, using   a standard procedure based on the Kolmogorov extension theorem, one can prove the existence of a weak mild solution $X$ to \eqref{cheef} on  $[0, \infty)$. On this respect, we refer to  Remark 3.7, page 303, in \cite{KS} (cf. the beginning of Section 4).

 \vskip 2mm 
 \noindent {\bf ii)} We use Theorem \ref{extension} with $H = \H$, $A = \A$ and $W = \W$. 
 Indeed,  \eqref{cheef} can be rewritten as
 \begin{gather*}
 \label{chee}
 \begin{array}{l}
X_t = e^{t \A }x +\int_{0}^{t}(- \A)^{1/2}e^{(  t-s)  \A}F  (
X_{s})ds
 +
\int_{0}^{t}e^{(  t-s)  \A}d \W_{s}, \;\;\; t \ge 0,
\end{array}
\end{gather*}
 where
 $
 F(x) = F_0(x) + (-\A )^{-1/2} B(x),$ $x \in \H
 $. The function $F : \H \to \H$ is locally $\theta$-H\"older continuous  (cf. \eqref{f0} and \eqref{nonl}). 
\end{proof}

\begin{remark} \label{che} {\rm 
%It is not clear how  to deduce  
Assertion (ii) in Proposition \ref{well1} cannot be deduced directly 
%is not clear if one can 
from  the Girsanov theorem
%. If we try to argue 
%arguing 
as in Appendix A.1 of \cite{DFPR}. To this purpose,
%to get weak uniqueness  we need to establish
%an exponential estimate like  
 one should prove that 
%\begin{equation*}\label{33}
 $\E \Big [ e^{\frac{1}{2} \int_0^T  |B(X_s)  |_{{\mathcal H}}^2 ds}\Big]$ $  < \infty, $
%\end{equation*}
for any weak mild solution $X$  to \eqref{cheef} starting at  $x \in \H$.  A sufficient condition would be   $\E \Big [ e^{\frac{1}{2} \int_0^T  |X_s  |_{{\mathcal H}}^2 ds}\Big] < \infty$. 
  It seems that such estimate does not hold under our assumptions. 
 %On the other hand,
 %Note that 
% in contract to 
%\cite{DFPR}
Note that since the nonlinearity of the Burgers equation grows quadratically 
 one cannot follow the proof of Proposition 22  of \cite{DFPR} to derive %\eqref{33}. 
  $\E \Big [ e^{\frac{1}{2} \int_0^T  |X_s  |_{{\mathcal H}}^2 ds}\Big] < \infty$.
 }
  \end{remark}

 \noindent \textbf{Acknowledgement.}
The author would like to thank D. Bignamini and S. Fornaro  for pointing
out an error in the proof of    Lemma 6 in  \cite{priolaAOP}.

 \def\ciao{
 Probability theory is concerned with random events or
more generally with random variables. Initially
random variables were treated in an intuitive way as
{\it variables which take values with some
probabilities}. They were characterised by their
distribution functions. If $X$ is a random variable then
its distribution function $F$ is defined by the
formula:
$$
F(x) = \P (X\leq x),\qquad x\in \R^1.
$$
Thus $F(x)$ is the probability that $X$ takes values
not greater than $x$. Random functions or, as we say
today, stochastic processes, were defined as families of
random variables $X(t)$, $t\in [0,T]$. They were
characterised by multivariate distribution functions
\begin{equation}\label{findimdistrproc}
F_{t_1,\ldots,t_n} (x_1,\ldots,x_n)
= \P (X(t_1)\leq x_1,\ldots, X(t_n)\leq x_n)
\end{equation}
defined for all $0\leq t_1<t_2<\ldots <t_n\leq T$ and
$x_1,\ldots,x_n\in \R^1$. Mathematical treatment of
stochastic processes was rather  involved and the need to
put the theory on firm mathematical foundations became
apparent. This was done in 1933 by A. Kolmogorov \cite{bib5}.
Probability theory and the theory of stochastic
processes became at that time a part of measure theory
and thus a solid part of mathematics.
\vspace{3mm}

\noindent To see a
need for a formalisation   consider a model
of the Brownian motion proposed  by A. Einstein \cite{einstein}
in 1905. Let $X(t)$ be the projection of the position
of a particle, suspended in a liquid,  onto the first coordinate axis.
It is a random variable in its intuitive sense and
the behaviour of $X(t)$, $t\geq 0$, is very chaotic.
The following postulates seem to be natural:

\begin{enumerate}
  \item[1)] The process $X(t)$, $t\geq 0$,
   has independent increments in the sense that for any sequence $0\leq t_0<t_1<\ldots<t_n$, $n=1,2,\ldots$
   and any numbers $y_1,\ldots,y_n\in\R^1$:
   $$
   \P( X(t_i)-X(t_{i-1})\leq y_i,\; i=1,\ldots,n)
   = \prod_{i=1}^n  \P( X(t_i)-X(t_{i-1})\leq y_i),
   $$
  \item[2)]The trajectories  $X(t)$, $t\geq 0$, should be continuous.
\end{enumerate}
\noindent In fact the postulates  are satisfied by many models and are not
contradictory. If, however, we add the postulate:
\begin{enumerate}
  \item[3)]  The trajectories are differentiable.
\end{enumerate}
then conditions 1)-3) are contradictory. More precisely, only a deterministic, constant speed movement satisfies all the
conditions.  This means
that we cannot construct a realistic and non-contradictory model fulfilling 1)-3). It is
therefore of great importance to deal with existence theorems
and the notes are concerned with the results of that sort.
\vspace{3mm}

\section{Existence questions}

\noindent Let us comment a bit more on the existence problem. Stochastic processes can be defined and characterised
in various
ways, for instance by some postulates like 1)-3) or
by specifying  their finite dimensional distributions
(\ref{findimdistrproc}).
There are several  ways of constructing stochastic processes.
One, very general approach, introduced by A. Kolmogorov \cite{bib5} in 1933,
assumes that finite dimensional distribution functions
$F_{t_1,t_2,\ldots,t_n}$ are given and satisfy some natural
consistency conditions. The approach, called sometimes {\it Kolomogorov's approach},   consists in building a
probability measure $\P$ on the space $\Omega= \R^{[0,T]}$ of all
real functions defined on the time interval $[0,T]$. The measure
$\P$ should be such that for the so called canonical process $X$:
$$
X(t,\omega)=\omega (t), \qquad \omega\in\Omega,\; t\in [0,T],
$$
all the identities (\ref{findimdistrproc})
are true. For the probability that $ X(t_1)\leq x_1,
\ldots, X(t_n)\leq x_n$ we can  now write a correct mathematical
expression:
$$
\P(\{\omega\; :\; X(t_1,\omega)\leq x_1, \ldots
X(t_n,\omega)\leq x_n\}).
$$
\vspace{3mm}

\noindent A variation of Kolmogorov's method is due to Yu. V. Prohorov \cite{prohorov} and is based on the concept of
weak convergence of
measures on metric spaces. The laws of the processes are obtained as weak limits of the laws of directly given simple
processes, see \cite{billingsley1}. The so called martingale method of solving stochastic equations, due to D. Stroock
and S. R. R. Varadhan \cite{stroock} is of the same category.
\vspace{3mm}

\noindent A completely different method is attributed to N. Wiener. It is less general and starts from a specific countable
or uncountable family of  random
variables and builds the required process from the elements of the family in a
constructive manner. For the first time this approach was applied in a joint paper of N. Wiener
and R. E. Paley \cite{bib9}
to  construct a Brownian motion
process. As  the point of
departure one takes  a sequence of independent, normally distributed random
variables $\xi_n$, $n=1,2,\ldots$, defined on $(\Omega,
{\cal F}, \P)$, with
$$
\P(\xi_n\leq x)= \frac{1}{\sqrt{2\pi}} \int_{-\infty}^x
e^{-\frac{y^2}{2}}dy,\qquad x\in\R^1,\;
n=1,2,\ldots.
$$
It turns out that  a process $W$ defined by the infinite series
$$
W(t, \omega)=\sum_{n=1}^{+\infty} \xi_n(\omega)\, e_n(t),\qquad t\in [0,T],
$$
with properly chosen deterministic functions $e_n$, has all the
postulated properties.
\vspace{3mm}

\noindent A powerful version of the method was proposed by K. Ito \cite{ito}.
Processes are constructed as solutions to stochastic
equations which can be solved by iterative methods. To formulate the equations and investigate their properties
K. Ito developed stochastic calculus.

\section{Some history}

\noindent In the language of modern probability, a stochastic process is a
family of random variables $(X(t), t\in {\bf T})$, where ${\bf
T}\subset [0, +\infty)$, defined on a probability space $(\Omega,
{\cal F}, \P)$. Stochastic processes were invented to
articulate and  study random phenomena which evolve in time.
First processes were introduced at the beginning of the twentieth
century. In his 1900 doctoral dissertation L. Bachelier \cite{bib1} regarded
the movement of prices on stock exchange as stochastic processes
and introduced, in an intuitive way, the so called Wiener
process. Several years later in 1909 A. Erlang \cite{bib2} was analysing the
work of telephone exchanges, most modern devices of the time. He
realised that the number of busy lines has a random character and
depends on time in a capricious way. To solve the practical question
of how many operators should work in an exchange he introduced a
special case of  what is now known  as a Markov process with a
finite number of states. In the same year F. Lundberg \cite{bib3} published  a
paper in which he investigated the time evolution of the  capital of
an insurance company and used, for this purpose, compound Poisson
processes. A Wiener process was constructed as a mathematical object in 1923
by N. Wiener \cite{bib7}  and the theory of stochastic processes, as a
part of mathematics, was created by A. Kolmogorov in his 1933 paper on
the axiomatics of probability theory \cite{bib5}. The first monograph on the
subject was published in 1952 by J. Doob \cite{bib6} . Since then the theory of
stochastic processes has been constantly growing and has become the most
important part of  probability theory and one of the most
influential tools of  applied mathematics, see  \cite{bib26}, \cite{bibVII}.
\vspace{3mm}

\noindent The mathematical
theory of stochastic processes became possible due to a rapid
development of  measure theory in the first quarter of the
twentieth century. We will list only the most important events. The
concept of $\sigma$-field was introduced in 1898 by E. Borel \cite{bibI}. Lebesgue's book on integration
\cite{bibII}
 appeared in 1902. Integration on abstract sets
was initiated by M. Fr\'echet \cite{bibIII} in 1915. The fundamental extension theorem by Carath\'eodory \cite{bibIV}
was published in 1918. The general version of the so called Radon-Nikodym theorem was proved by O. Nikodym \cite{bibV} in
 1930. The  first monograph on the general measure theory was written by S. Saks \cite{bibVI} (Polish edition in 1930
 and  English edition in 1937).
\vspace{3mm}

\chapter{Measure theoretic preliminaries}
We collect here  basic definitions and results from measure theoretic foundations of the probability
theory. Special attention is paid to characteristic functions.

\section{Measurable spaces.}

Let $\Omega$ be a set. A collection $\calf$ of subsets
of $\Omega$ is said to be a {\em $\sigma$-field} ({\em
$\sigma$-algebra}) if
\begin{enumerate}
  \item[1)] $\Omega\in\calf$.

  \item[2)] If $A\in\calf$, then $A^c$, the complement
   of $A$, belongs to $\calf$.

  \item[3)] If $A_n\in\calf$, $n=1,2,\ldots$, then
  $\bigcup_{n=1}^\infty A_n\in\calf$.
\end{enumerate}
The pair $(\Omega ,\calf )$ is called a {\em measurable
space}. Let $(\Omega ,\calf )$, $(E ,\cale )$ be two
measurable spaces. A mapping $X$ from $\Omega$ into $E$
such that for all $A\in \cale$ the set $\{\omega :
X(\omega)\in A\}$ belongs to $\calf$ is called a {\em
random variable}, or $E$-valued random variable or a
measurable mapping.

Let $\calm$ be a collection of subsets of $\Omega$. The
smallest $\sigma$-field on $\Omega$, which contains
$\calm$, is denoted by $\sigma (\calm)$ and called the
{\em $\sigma$-field generated by $\calm$}. It is the
intersection of all $\sigma$-fields on $\Omega$
containing $\calm$. Analogously, let $(X_i)_{i\in I}$
be a family of mappings from $\Omega$ into $(E,\cale)$
then the smallest $\sigma$-field on $\Omega$ with
respect to which all functions $X_i$ are measurable is
called the {\em $\sigma$-fields generated by
$(X_i)_{i\in I}$} and is denoted by $\sigma (X_i:i\in
I)$. Given measurable spaces
 $(E_1,\cale_1), (E_2,\cale_2),\ldots ,(E_k,\cale_k)$,
the $\sigma$-field
$\cale_1\times\cale_2\times\ldots\times
\cale_k$ is the smallest $\sigma$-field of subsets of
 $E_1\times
E_2\times\ldots\times E_k$ which contains all sets of
the form  $A_1\times A_2\times\ldots\times A_k$ where
$A_i\in \cale_i$, $i=1,2,\ldots,k$.

Let $(E,\rho)$ be a metric space with metric $\rho$.
The $\sigma$-field on $E$ generated by closed subsets
of $E$ is called {\em Borel} and and is denoted by
$\calb (E)$. Metric spaces will be treated as
measurable spaces with the $\sigma$-field of Borel
sets. In particular random variables taking values in
$\R^1$ will be called real random variables.

\section{Dynkin's $\pi$-$\lambda$ theorem.}

Many arguments in the theory of stochastic processes
depend on Dynkin's $\pi -\lambda$ theorem \cite{Dy}. A collection
$\call$ of subsets of $\Omega$ is said to be a
$\pi$-system if it is closed under the formation of
finite intersections: if $A,B\in\call$ then $A\cap
B\in\call$. A collection $\calm$ of subsets of $\Omega$
is a $\lambda$-system if it contains $\Omega$, is
closed under the formation of complements and of finite
and countable disjoint unions:
\begin{enumerate}
  \item[1)] $\Omega\in\calm$.

  \item[2)] If $A\in\calm$ then $A^c\in\calm$.

  \item[3)] If $A_i\in\calm$ and $A_i\cap A_j=\emptyset$
  for $i\ne j$, $i,j=1,2,\ldots$ then
  $\bigcup_{i=1}^{\infty}A_i \in \calm$.
\end{enumerate}

\begin{Theorem}
  If a $\lambda$-system $\calm$ contains a $\pi$-system
  $\call$, then $\calm\supset\sigma(\call)$.
\end{Theorem}

\noindent {\bf Proof.}
Denote by $\calk$ the smallest $\lambda$-system
containing $\call$, equal to the intersection of all
$\lambda$-systems containing $\call$. Then
$\calk\subset\sigma(\call)$. To prove the opposite
inclusion we show first that $\calk$ is a $\pi$-system.
Let $A\in\calk$ and define $\calk_A=\{B\, :\, B\in\calk
{\rm \; and\;} A\cap B\in\calk\}$. It is easy to check
that $\calk_A$ is closed under the formation of
complements and countable disjoint unions and if
$A\in\call$ then $\calk_A\supset \call$. Thus for
$A\in\call$, $\calk_A=\calk$ and we have shown that if
$A\in\call$ and $B\in\calk$ then $A\cap B\in\calk$. But
this implies that $\calk_B\supset \call$ and,
consequently, $\calk_B=\calk$ for any $B\in\calk$. It
is now an easy exercise to show that if a $\pi$-system
is closed under the formation of complements and
countable disjoint unions then it is a $\sigma$-field.
This completes the proof. \qed

If $E$ is a metric space then the family of all open
(closed) subsets of $E$ is a $\pi$-system. If $E=\R^d$
then the sets $\{y\in\R^d\, : \, y\leq x\}$,
$x\in\R^d$, form a $\pi$-system.

If $(E,\cale)$ is a measurable space and $\mu $ a
nonnegative function  on $\cale$ such that
\begin{enumerate}
  \item[1)] $\mu (A)\in [0, \infty ]$ for all $A\in \cale$,

  \item[2)] $\mu(\emptyset) = 0 $,

  \item[3)] If $A_n\in\calf$, $A_n\cap A_m=\emptyset$
  for $n\ne m$, then
  $$
  \mu (\bigcup_nA_n)=\sum_{n=1}^\infty \mu (A_n).
  $$

\end{enumerate}
then $\mu$ is called a nonnegative measure or shortly a measure.
If $\mu (E) =1$ then the measure is called a {\em probability measure} or
shortly {\em probability}.

The following Proposition will be frequently used.

\begin{Proposition}
  If two probability measures $\mu,\nu$ on $(E,\cale )$
  are identical on a $\pi$-system $\call$ then they
  are identical on $\sigma(\call)$.
\end{Proposition}

\noindent {\bf Proof.}
Let $\calm$ be the collection of all $A\in\cale$ such
that
$$
\mu(A)=\nu(A).
$$
It is clear that $\calm$ is a $\lambda$-system and
therefore $\calm\supset\sigma(\call)$. \qed

\section{Independence}
We recall also the notion of independence playing an essential
role in the theory of stochastic processes.
Let $(\Omega,\calf,\P)$ be a probability space and let
$(\calf_i)_{i\in I}$ be a family of
sub-$\sigma$-fields. These $\sigma$-fields are said to
be independent if, for every finite subset $J\subset I$
and every family $(A_i)_{i\in J}$ such that
$A_i\in\calf_i$, $i\in J$,
\begin{equation}\label{indep}
  \P\left( \bigcap_{i\in J}A_i\right) =
  \prod_{i\in J} \P\left(A_i\right).
\end{equation}
Random variables $(X_i)_{i\in I}$ are independent if
the $\sigma$-fields $(\sigma(X_i))_{i\in I}$ are
independent.

\begin{Proposition}
  If $\call_i$ are $\pi$-systems on $\Omega$ and
  $\calf_i=\sigma(\call_i)$, $i\in I$, then the
  $\sigma$-fields $(\calf_i)_{i\in I}$ are
  independent if for every finite set $J\subset I$
  and sets $A_i\in \call_i$,  $i\in I$,
$$
 \P\left( \bigcap_{i\in J}A_i\right) =
  \prod_{i\in J} \P\left(A_i\right).
  $$
\end{Proposition}

\noindent {\bf Proof.}
Assume, without loss of generality, that
$I=J=\{1,\ldots,n\}$. Let us fix the sets
$A_i\in\call_i$, $i=2,\ldots,n$, and denote by
$\calm_1$ the family of all sets $A_1\in\calf_1$ for
which
\begin{equation}\label{indep1}
   \P\left( \bigcap_{k=1}^nA_k\right) =
  \prod_{k=1}^n \P\left(A_k\right).
\end{equation}
The family contains $\call_1$ and is a
$\lambda$-system, therefore $\calm_1=\calf_1$.
Analogously, fix $A_1\in\calf_1$ and $A_i\in\call_i$,
$i=3,\ldots,n$, and denote by $\calm_2$ the family of
$A_2\in\calf_2$ for which (\ref{indep1}) holds. Then
$\calm_2=\sigma(\call_2)=\calf_2$. Easy induction shows
that (\ref{indep1}) holds for all $A_i\in\calf_i$,
$i=1,2,\ldots,n$. \qed

\begin{Corollary}
  Let $X_1,X_2,\ldots $ be a sequence of real-valued
  random variables. The random variables $(X_i)_{i\in\N}$
  are independent iff for arbitrary $n=1,2,\ldots$ and
  arbitrary real numbers $x_1,\ldots,x_n$:

\begin{equation}\label{indep2}
  \P(X_1\leq x_1,\ldots,X_n\leq x_n)
  =\prod_{i=1}^n\P(X_i\leq x_i).
\end{equation}

\end{Corollary}

\noindent {\bf Proof.} The $\sigma$-fields
$\sigma(X_i)$ are generated by the events $\{X_i\leq
x_i\}$ forming  $\pi$-systems. \qed

\begin{Remark}\begin{em}
  Note that according to the definition $(X_i)$ are
  independent if and only if for arbitrary $n$ and arbitrary
  Borel subsets $\Gamma_1,\ldots,\Gamma_n$ of $\R^1$:
$$
 \P(X_1\in\Gamma_1,\ldots,X_n\in\Gamma_n)
  =\prod_{i=1}^n\P(X_i\in\Gamma_i).
  $$
  So the condition (\ref{indep2}) is a real
  simplification of (\ref{indep}).
\end{em}
\end{Remark}

\section{Expectations and conditioning}

Since $\P$ is a measure on $(\Omega, \calf)$, the Lebesgue
integral,
\begin{equation}\label{defintegrale}
  \int_\Omega X(\omega)\; \P (d\omega),
\end{equation}
of a measurable, non-negative function $X$ is well defined. It is defined also for those functions $X$ for which
either negative or positive parts has a finite integral. If
$$
  \int_\Omega |X(\omega)|\; \P (d\omega)<\infty.
$$
then $X$ is called integrable.
In probability theory,
the Lebesgue integral (\ref{defintegrale}) is called
expected value, or mean value or expectation of $X$ and denoted as
$$
\E\, X.
$$
The law $\mu$ of a random variable $X$ taking values in a measurable space $(E, \cal E)$ is a probability measure on
$(E, \cal E)$ such that
$$
\mu (\Gamma) = \P (\{\omega\,; X(\omega)\in \Gamma\}),\,\,\,\Gamma \in \cal E.
$$
If $\mu$ is the law of a random variable $X$
then for arbitrary Borel, non-negative, or bounded,  function $\phi:E\to\R^1$:
\begin{equation}\label{meanandlaw}
\E(\phi(X))=\int_E\phi(x)\, \mu(dx).
\end{equation}
This identity is true for indicator functions
just by the definition of $\mu$.
Therefore it holds for all measurable functions $\phi$
taking only a finite number of values. But then it
holds for all nonnegative $\phi$ because they are
limits of increasing sequences of functions taking a
finite number of values. Decomposing arbitrary
measurable $\phi$ into its positive and negative parts
one shows that (\ref{meanandlaw}) is true in general.
Identity  (\ref{meanandlaw}) will be frequently used in
the future.

Assume that $\calg$ is a $\sigma$-field contained in $\calf$ and
$X$ is a non-negative or integrable random variable then the
conditional expectation
$$
\E\, (X|\calg)
$$
is defined as  $\calg$-measurable random variable $Z$ such that for each
$A\in\calg$,
$$
\E\, (X\, 1_A) =\E\, (Z\, 1_A).
$$
We take for granted its elementary properties.

\section{Gaussian measures}

A Gaussian measure $\mu$ on $\R^1$ is a probability measure
either concentrated at one point, $\mu=\delta_m$, or having a
density
$$
\frac{1}{\sqrt{2\pi q}}\, e^{\frac{1}{2q}(x-m)^2},
\qquad x\in \R^1,
$$
where $m\in\R^1$, $q>0$.

If $E$ is a Banach space then a probability measure $\mu$ on
$(E,\, \calb(E))$ is said to be Gaussian if and only if the law of an
arbitrary linear functional $h\in E^*$ considered as a random
variable on $(E,\, \calb(E),\, \mu)$ is a Gaussian measure on
$(\R^1,\, \calb(\R^1))$.

In the case when $E=\R^n$, the Gaussian measures are characterized
uniquely by the mean vector $m=(m_1,\ldots ,m_n)$ and the
covariance matrix $Q=(q_{jk})$, and denoted by $N(m,Q)$.
If $m=0$, one writes $N_Q$. Thus
$$
\int_{\R^n} x_j\; N(m,Q)(dx)=m_j, \quad
\int_{\R^n} (x_j-m_j)(x_k-m_k)\; N(m,Q)(dx)=q_{jk},
\qquad j,k=1,\ldots ,n.
$$
If $Q$ is positive definite   $N(m,Q)$ has a density $g_{m,Q}$:
$$
g_{m,Q}= \frac{1}{\sqrt{
(2\pi)^n\det Q }}
\,
e^{-\frac{1}{2} \<Q^{-1}(x-m),(x-m)\>},\qquad x\in\R^n.
$$

A random vector $X$ is Gaussian if its law $\call(X)$ is a
Gaussian measure. If $X=(X_1,\ldots, X_n)$ is Gaussian with values
in $\R^n$, then the random variables $X_i$ are independent if and
only if the matrix $Q$ is diagonal.

\section{Convergence of measures}

A sequence $(\mu_n)$ of probability measures on a
  separable metric space $E$, with metric $\rho$,
   is said to be
{\em weakly convergent} to a measure $\mu$ if for
arbitrary  continuous and bounded $\phi \,
:\, E\to\R^1$:
\begin{equation}\label{weakly1}
\int_E \phi (x)\, \mu_n(dx)\to \int_E \phi (x)\, \mu(dx)
\qquad {\rm as}\; n\to\infty.
\end{equation}
If $(\mu_n)$    converges weakly to  $\mu$    then one
writes $\mu_n\Rightarrow\mu$.
The space $C_b(E)$ of
all bounded and continuous functions of $E$ equipped
with the supremum norm will be denoted by $C_b(E)$.

\begin{Proposition}For a given sequence $(\mu_n)$ there exists
at most one weak limit.
\end{Proposition}
\noindent{\bf Proof.}
It is enough to show that if for two measures $\mu$,
$\nu$:
\begin{equation}\label{weakly2}
 \int_E \phi (x)\, \mu(dx)= \int_E \phi (x)\, \nu(dx)
\end{equation}
for all $\phi\in C_b(E)$ then $\mu=\nu$. Let $A$ be a
closed set and $\phi_k$, $k=1,2,\ldots$, the following
continuous and bounded functions:
$$
\phi_k(x)=\frac{1}{1+k\rho(x,A)},\qquad x\in E,
$$
where $\rho(x,A)=\inf \{\rho(x,y)\, ; \, y\in A\}$.
Then $\phi_k\downarrow\chi_A$ as $k\to\infty$. Since
(\ref{weakly2}) holds for all $\phi_k$ therefore it
holds for the limit. Consequently $\mu(A)=\nu(A)$ for
all closed sets. By Dynkin's $\pi-\lambda$ theorem
$\mu=\nu$. \qed

\begin{Example}\begin{em}
A centred Gaussian measure $\caln_\sigma$ on $\R^d$,
$\sigma >0$, with the covariance matrix $\sigma I$ has
the density
$$
n_\sigma (x)=\frac{1}{\sqrt{(2\pi \sigma)^d}}\,
e^{-\frac{|x|^2} {2\sigma}},\qquad x\in\R^d.
$$
If $\sigma\to 0$, then
$\caln_\sigma\Rightarrow\delta_{\{0\}}$, where
$\delta_{\{0\}}$ is the probability measure
concentrated in $0$.
\end{em}
\end{Example}

The following result is classical attributed to Helly.
It will be proved in Chapter \ref{capprohorov} as a consequence of more general
results.

\begin{Proposition}\label{helly}
  Assume that $(\mu_n)$ is a sequence of probability measures
  on $\R^d$ and that for arbitrary $\epsilon >0$ there exists
  $R>0$ such that
  $$
  \mu_n(B(0,R))\geq 1-\epsilon,
  \qquad n=1,2,\dots.
  $$
Then $(\mu_n)$ contains a subsequence $(\mu_{n_k})$ which
converges weakly to a probability measure $\mu$.
\end{Proposition}

\section{Characteristic functions}

The characteristic function of a
measure $\mu$ on $\R^d$ is defined as follows:
$$
\phi_\mu(\lambda)=\int_E
e^{i\langle\lambda,x\rangle} \mu(dx),\qquad
\lambda\in\R^d.
$$
Characteristic functions are continuous functions, in general, with
complex values. They are denoted also as $\widehat{\mu}$.

\begin{Example}\begin{em}
The characteristic function of  $\caln_\sigma$ is as
follows:
$$
\phi_{\caln_\sigma}(\lambda)=
e^{-\frac {\sigma}{2}|\lambda|^2},\qquad
\lambda\in\R^d.
$$
Assume that $\sigma = 1$  and  $\phi(\lambda )=
\frac{1}{\sqrt{2\pi}}\int_{-\infty}^{+\infty}e^{i\lambda
x}e^{-\frac{x^2}{2}}dx$, $\lambda \in \R^1$. Then
$$\begin{array}{l}\ds
\phi'(\lambda )=
\frac{1}{\sqrt{2\pi}}i\int_{-\infty}^{+\infty}e^{i\lambda
x}xe^{-\frac{x^2}{2}}dx
=
-\frac{i}{\sqrt{2\pi}}\int_{-\infty}^{+\infty}e^{i\lambda
x}\left(e^{-\frac{x^2}{2}}\right)'dx
\eds\\\ds
=
\frac{i^2}{\sqrt{2\pi}}\lambda\int_{-\infty}^{+\infty}e^{i\lambda
x} e^{-\frac{x^2}{2}} dx
=-\lambda \phi(\lambda ).
\eds\end{array}
$$
Therefore $\phi(\lambda )=e^{-\frac{\lambda^2}{2}}$,
$\lambda\in\R^1$.

More generally, the Gaussian distribution
${\cal N}(m,Q)$  on $\R^{d}$ with
 covariance
matrix $Q$  and mean vector $m$  has a characteristic function
$$
e^{i\langle m,\lambda\rangle
-\frac{1}{2}\langle Q\lambda,\lambda\rangle},
\qquad \lambda\in\R^d.
$$

\noindent {\bf Proof.}
For the proof assume that $m=0$ and $Q$ is invertible.
Introduce the new basis given by
the eigenvectors of $Q$. In the
new basis:
$$
\langle Q\lambda,\lambda\rangle =
\sum_{k=1}^d  q_k\lambda_k^2,\qquad
\det \, Q=q_1\ldots q_d,\qquad
\langle Q^{-1}x,x\rangle =
\sum_{k=1}^d  \frac{1}{q_k}x_k^2
$$
and therefore
$$\begin{array}{l}\ds
\frac{1}{\sqrt{(2\pi)^d\det \, Q}}
\int_{\R^d}
e^{i\lambda x}
e^{-\frac{1}{2}\langle Q^{-1}x,x\rangle}
dx
\eds\\\ds
=
\frac{1}{\sqrt{(2\pi)^d\prod_{k=1}^dq_k}}
\int_{\R^d}
e^{i\sum_{k=1}^dx_k\lambda_k}
e^{-\frac{1}{2}\sum_{k=1}^d\frac{1}{q_k}x_k^2}
dx_1\ldots dx_k
\eds\\\ds
=
e^{-\frac{1}{2}\sum_{k=1}^dq_k\lambda_k^2}
=e^{-\frac{1}{2}\langle Q\lambda,\lambda\rangle }. \qed
\eds\end{array}
$$

\end{em}
\end{Example}

\begin{Proposition}
Different probability measures define different
characteristic functions.
\end{Proposition}

\noindent{\bf Proof.}
If $\mu$ and $\nu$ are two measures and $\gamma\in\R^d$
then
$$
e^{-i\langle \gamma,\lambda\rangle} \phi_\mu(\lambda)=
\int_{\R^d}e^{i\langle \lambda, x-\gamma\rangle}
\mu(dx).
$$
Integrating both sides with respect to $\nu$ we have
$$
\int_{\R^d}e^{-i\langle \gamma, \lambda\rangle}
\phi_\mu(\lambda)\nu(d\lambda)=
\int_{\R^d}\phi_\nu( x-\gamma)
\mu(dx),\qquad \gamma\in\R^d.
$$
In particular if $\nu=\caln_\sigma$ one gets:
$$
\frac{1}{\sqrt{(2\pi \sigma)^d}}\,
\int_{\R^d}e^{-i\langle \gamma, \lambda\rangle}
\phi_\mu(\lambda)e^{-\frac{|\lambda|^2} {2\sigma}}
d\lambda =
 \int_{\R^d}e^{-\frac{\sigma}{2}|x-\gamma|^2}
 \mu(dx),\qquad \gamma\in\R^d.
$$
$$
\frac{1}{ (2\pi  )^d}\,
\int_{\R^d}e^{-i\langle \gamma, \lambda\rangle}
\phi_\mu(\lambda)e^{-\frac{|\lambda|^2} {2\sigma}}
d\lambda =
 \int_{\R^d}n_{1/\sigma}(x-\gamma)
 \mu(dx),\qquad \gamma\in\R^d.
$$
The right-hand side is the density of the convolution
$\caln_{1/\sigma}*\mu$ which converges weakly to $\mu$
as $\sigma \to \infty$. Since the left-hand side
depends only on $\mu$ the result follows. \qed

\begin{Theorem}
A sequence $(\mu_n)$ of probability measures
on $\R^d$ converges weakly to a probability
measure $\mu$ if and only if $\phi_{\mu_n}\to \phi_\mu$
pointwise.
\end{Theorem}

\noindent {\bf Proof.}
It is clear that if $\mu_n\Rightarrow\mu$ then
$\phi_{\mu_n}\to \phi_\mu$
pointwise. For the opposite implication we need
the following lemma.

\begin{Lemma}
Assume that in a neighbourhood of $0$ a sequence
$(\phi_{\mu_n})$  of characteristic functions of measures
 $(\mu_n)$  converges  to a function $\phi$ continuous at $0$.
Then for arbitrary $\epsilon >0$ there exists $R>0$ such that
$\mu_n(B(0,R))\geq 1-\epsilon$ for $n=1,2,\ldots$.
\end{Lemma}

\noindent {\bf Proof.}
It is enough to show that the result is true  if $d=1$. If $d>1$ one considers
projections of the measures $(\mu_n)$ on each of the coordinate axes. The
sequences
of projections are all tight by the one dimensional result and therefore
the sequence
$(\mu_n)$ is tight as well.

Let $\mu$ be a probability measure on $\R^1$ and $r$ a positive
number. Then
$$\begin{array}{l}\ds
\frac{1}{2r}
\int_{-r}^r (1-\phi_\mu(\lambda))\, d\lambda
=
\frac{1}{2r}
\int_{-r}^r
\left[
\int_{-\infty}^{+\infty}
 (1-e^{i\lambda x})\, \mu (dx)
\right]\,
d\lambda
\eds\\\ds
=
\int_{-\infty}^{+\infty}
\frac{1}{2r}
\left[
\int_{-r}^r
 (1-\cos \lambda x)\,  d\lambda
\right]\,  \mu (dx)
\geq
\int_{|x|\geq \frac{2}{r}}
\left( 1-\frac{\sin xr}{xr}\right) \,  \mu (dx)
\eds\\\ds
\geq
\int_{|x|\geq \frac{2}{r}}
\left( 1-\frac{1}{ r|x|}\right) \,  \mu (dx)
\geq
\frac{1}{2}\mu \left(x\, ;\, |x|\geq \frac{2}{r}\right).
\eds\end{array}
$$
Assume that $\phi_{\mu_n}$ converges to
$\phi$ on an interval larger than $[-r,r]$.
 Since
$$
\frac{1}{2r}
\int_{-r}^r (1-\phi_{\mu_n}(\lambda))\, d\lambda
\geq
\frac{1}{2}\mu_n \left(x\, ;\, |x|\geq \frac{2}{r}\right)
$$
and  $\lim_{\lambda \to 0}\phi (\lambda)=1$,
$\lim_n \phi_{\mu_n}(\lambda) =\phi (\lambda)$,
$\lambda\in [-r,r]$, therefore for $\epsilon >0$ there exist $n_0$,
$r_0$ such that for $n\geq n_0$ and $0<r<r_0$
$$
\frac{1}{2r}
\int_{-r}^r (1-\phi_{\mu_n}(\lambda))\, d\lambda
<\frac{\epsilon }{2}
$$
and consequently for $n\geq n_0$,
$\mu_n \left(x\, ;\, |x|\geq \frac{2}{r}\right)\leq
\epsilon $. Since an arbitrary finite sequence of
measures on $\R^d$ is tight therefore the result follows.
\qed

By Helly's theorem the sequence $(\mu_n)$ contains
a weakly convergent subsequence.
 Since
all convergent subsequences converge to measures with the same
characteristic functions the sequence is weakly convergent.
\qed

\begin{Proposition}
If $\mu $ is a finite measure on $\R^d$ then
$$
\widehat{\mu}(\lambda) = \int_{\R^d}
e^{i\langle \lambda ,x\rangle}\mu(dx),
\qquad \lambda\in\R^d,
$$
is a uniformly continuous function.
\end{Proposition}

\noindent {\bf Proof.}
One can assume that $\mu$ is a probability measure. If
this is the case then
$$
  \begin{array}{lll}
    | \widehat{\mu}(\lambda) - \widehat{\mu}(\eta)|& =
    & \ds
    \left|  \int_{\R^d}
e^{i\langle \lambda ,x\rangle}\left( 1- e^{i\langle
\eta -\lambda ,x\rangle}\right)
\mu(dx)\right|
\eds\\
      & \leq &\ds
       \int_{\R^d} \left| 1- e^{i\langle
\eta -\lambda ,x\rangle}\right|\mu(dx)
\eds\\
&\leq &\ds
\int_{|x|\leq R} \left| 1- e^{i\langle
\eta -\lambda ,x\rangle}\right|\mu(dx)
+2\mu\{ x\, :\, |x|>R\}
\eds
  \end{array}
$$
For $\epsilon >0$ there exists $R>0$ such that $2\mu\{
x\, :\, |x|>R\}<\frac{\epsilon}{2}$. There exists also
$\delta >0$ such that if $|\eta -\lambda |<\delta$ then
for all $x$, $|x|\leq R$, $\ds\left| 1- e^{i\langle
\eta -\lambda ,x\rangle}\right| \leq
\frac{\epsilon}{2\mu\{ x\, :\, |x|\leq R\}}\eds$.
Consequently, if $|\eta -\lambda |<\delta$ then
$$
| \widehat{\mu}(\lambda) - \widehat{\mu}(\eta)|
\leq \frac{\epsilon}{2}+
\frac{\epsilon}{2\mu\{ x\, :\, |x|\leq R\}}
2\mu\{ x\, :\, |x|\leq R\}< \epsilon .\qed
$$

\begin{Proposition}
If $\mu_n$, $\mu $ are probability measures and
$\mu_n\Rightarrow \mu $ then $\widehat{\mu}_n\to
\widehat{\mu}$ uniformly on bounded sets.
\end{Proposition}

\noindent {\bf Proof.}
It is enough to show that if $\lambda_n\to\lambda$ then
$\widehat{\mu}_n(\lambda_n)\to \widehat{\mu}(\lambda).$
Note that
$$
  \begin{array}{lll}
    | \widehat{\mu}_n(\lambda_n) - \widehat{\mu}(\lambda)|
& \leq  & | \widehat{\mu}_n(\lambda_n) -
\widehat{\mu}_n(\lambda)| +
| \widehat{\mu}_n(\lambda) - \widehat{\mu}(\lambda)|
 \\
    &\leq & \ds
    \int_{\R^d} \left|   e^{i\langle
 \lambda_n ,x\rangle}-  e^{i\langle
 \lambda ,x\rangle}\right|\mu_n(dx)
 +
| \widehat{\mu}_n(\lambda) - \widehat{\mu}(\lambda)|
\eds\\
&\leq & \ds 2 \int_{|x|>R}\mu_n(dx)+ \sup_{|x|\leq R}
 \left| 1-  e^{i\langle
 \lambda_n -\lambda ,x\rangle} \right|
 +
| \widehat{\mu}_n(\lambda) - \widehat{\mu}(\lambda)|
\eds\\
&\leq &I_1+I_2+I_3.
  \end{array}
$$
 The term
$I_3$ tends to zero if $n\to +\infty$. Since the family
$(\mu_n)$ is tight the first term $I_1$ can be made
small by taking $R$ sufficiently large, uniformly in
$n$. Since $\lambda_n\to\lambda$, also the second term
tends to zero if $n\to +\infty$. \qed

\chapter{Kolmogorov's existence theorem}
We prove here the Kolmogorov existence theorem and apply it to establish existence of Gaussian, Markov and L\'evy
processes.

\section{Carath\'eodory's and Ulam's theorems}

A collection $\cale_0$ of subsets of $ E$ is said
to be a {\em field} or {\em algebra} if
\begin{enumerate}
  \item[1)] $E\in \cale_0$.

  \item[2)] If $A\in \cale_0$, then $A^c\in\cale_0$.

  \item[3)] If $A_1,\ldots,A_n\in\cale_0$, then
  $\bigcup_{k=1}^nA_k\in\cale_0$.
\end{enumerate}

We take for granted the following result due to Carath\'eodory \cite{bibIV}.
\begin{Theorem} (Carath\'eodory)
If $\mu$ is a non-negative function on $\cale_0$ such that
\begin{enumerate}
  \item[1)] $\mu (E)< \infty$.
   $\mu (\bigcup_{i=1}^nA_i)
  =\sum_{i=1}^n \mu(A_i)$ if $A_i\cap A_j=\emptyset$,
  $i\ne j$ and $A_i\in\cale_0$.

  \item[2)] If $A_i\in\cale_0$, $A_i\cap A_j=\emptyset$
  for $i\ne j$ and $\bigcup_iA_i\in\cale_0$ then
  $$
  \mu (\bigcup_{i=1}^\infty A_i)=\sum_{i=1}^\infty \mu(A_i).
  $$

\end{enumerate}
Then there exists a unique extension of $\mu$ to a
finite measure on the $\sigma$-field generated by
$\cale_0$.
\end{Theorem}

\begin{Remark}\begin{em}
  If Condition 1) is satisfied $\mu$ is called
  an {\em additive set function}. Condition 2) is equivalent
  to the following {\em continuity condition},

\begin{enumerate}
  \item[2')]  If $A_i\in\cale_0$ and $A_i\supset A_{i+1}$ for
  $i=1,2,\ldots$ and $\bigcap_{i=1}^\infty A_i =\emptyset$ then
  $\mu (A_i)\downarrow 0$ as $i\uparrow \infty$.
\end{enumerate}
\end{em}
\end{Remark}

\noindent {\bf Proof of the Remark.}

2) $\Rightarrow$ 2'). Since $\bigcap_{i=1}^\infty A_i
=\emptyset$,
$A_1 = \bigcup_{i=1}^\infty (A_i \backslash A_{i+1})$
and $A_i \backslash A_{i+1}$, $i=1,2\ldots$, are
disjoint elements of $\cale_0$. Therefore
$$
\mu (A_1)= \sum_{i=1}^\infty [\mu (A_i)-\mu (A_{i+1})]
= \lim_n (\mu (A_1)-\mu (A_{n}))
$$
and $\lim_n  \mu (A_{n}) =0$.

2') $\Rightarrow$ 2). For all $i=1,2,\ldots$, the sets
$A_i'=\bigcup_{n=i}^\infty A_n$ belong to $\cale_0$ and
$\bigcap_iA_i'=\emptyset$. Therefore
$$
\mu \left(\bigcap_{n=1}^\infty A_n\right) =
\mu (A_1)+\ldots +\mu (A_i) +
\mu \left(\bigcap_{n=i+1}^\infty A_n\right)
 =
\mu (A_1)+\ldots +\mu (A_{i}) +
\mu  (A_{i+1}').
$$
Since $\mu (A_{i+1}') \to 0$ as $i\to\infty$, the
result follows. \qed

The following result is due to Ulam \cite{ulam}.
\begin{Theorem} (Ulam)
Assume that $\mu$ is a probability measure on a
separable, complete, metric space. Then for arbitrary
$\epsilon >0$ there exists a compact set $K\subset E$
such that
$$
\mu (K)\geq 1-\epsilon .
$$
\end{Theorem}

\noindent {\bf Proof.}
Let $(a_i)$ be a dense, countable sequence of elements
of $E$. Then, for arbitrary $k=1,2,\ldots$,
$$
\bigcup_{i=1}^\infty B(a_i,\frac{1}{k}) =E,
$$
where $B(a,r)=\{ y\in E\, :\, \rho (a,y)<r\}$, $a\in
E,r>0$. Therefore, for arbitrary $k$, there exists
$n_k\in\N$ such that
$$
\mu\left(\bigcup_{i=1}^{n_k}B(a_i,\frac{1}{k})\right) >
1-\frac{\epsilon}{2^k}.
$$
The set $L= \bigcap_k\bigcup_{i=1}^{n_k}
B(a_i,\frac{1}{k})$ is totally bounded and its closure
$K$ is compact. However
$$
  \begin{array}{l}
    \ds \mu (K)\geq 1-\mu \left(\bigcup_{k=1}^\infty
    \left(\bigcup_{i=1}^{n_k} B(a_i,\frac{1}{k})\right)^c
    \right)\eds \\\ds
    \geq 1-\sum_{k=1}^\infty \mu
     \left(\bigcup_{i=1}^{n_k} B(a_i,\frac{1}{k})\right)^c
     > 1-\sum_{k=1}^\infty\frac{\epsilon}{2^k}=1-\epsilon.
     \eds\qed
  \end{array}
$$

\noindent We have used the following elementary result left as an
exercise.

\begin{Proposition}
  A metric space $E$ is compact iff it is
  totally bounded
\footnote{
A space $E$ is totally bounded if for arbitrary
$\epsilon >0$ there exists a finite set
$\{a_1,\ldots,a_n\}$ such that
$\bigcup_{i=1}^NB(a_i,\epsilon)=E$.
}
   and complete.
\end{Proposition}

\section{Kolmogorov's theorem}

Let us recall that if $(\Omega,\calf)$ is a measurable space and  $\P$ is
a probability mesure on $\calf$ then the collection $(\Omega,\calf ,\P)$
is called a {\em probability space}
and any family $X(t)$, $t\in\calt$, of $E$-valued
random variables defined on $(\Omega,\calf)$ is called
a {\em stochastic process} on $\calt$ or a {\em random
family} on $\calt$. Often $\calt$ is a subset of
$\R^1$. Traditionally if $\calt$ is a subset of $\R^d$
with $d\geq 2$ then a stochastic process is called a
{\em random field}.

As we have already said, a stochastic  process $X$ is usually described in terms of measures it
induces on products of $E$,\,\,see (\ref{findimdistrproc}). Namely if
$X$ is an $E$-valued process on $\calt$ then for each
sequence $(t_1,t_2,\ldots,t_k)$ of distinct elements of
$\calt$, $(X(t_1),\ldots,X(t_k))$ is a random variable
with values in the Cartesian product $E\times
E\times\ldots\times E$ equipped with the product
$\sigma$-field $\cale\times\cale\times\ldots\times \cale$.

The probability
measures on $E^k$:
$$
\mu_{t_1,\ldots,t_k} = \call (X(t_1),\ldots,X(t_k)),
$$
are called {\em finite-dimensional distributions} of
$X$.

Note that the finite-dimensional distributions of a
stochastic process $(X(t))_{t\in\calt}$ satisfy two
{\em consistency conditions}:
\begin{enumerate}
  \item[1)] For any $A_i\in\cale$, $i=1,2,\ldots,k$
  and any permutation $\pi$ of $(1,2,\ldots,k)$:
  $$
  \mu_{t_1,\ldots,t_k}(A_1\times \ldots
  \times A_k) =
   \mu_{t_{\pi 1},\ldots,t_{\pi k}}(A_{\pi 1}\times
   \ldots\times A_{\pi k}).
   $$

  \item[2)] For any $A_i\in\cale$, $i=1,2,\ldots,k-1$,
   $$
  \mu_{t_1,\ldots,t_{k-1}}(A_1\times \ldots
  \times A_{k-1}) =
  \mu_{t_1,\ldots,t_k}(A_1\times\ldots
  \times A_{k-1}\times E).
  $$
\end{enumerate}
The following theorem is due to Kolmogorov.

\begin{Theorem}
  Assume that $E$ is a separable, complete, metric space and
  $\mu_{t_1,\ldots,t_k}$ a family of distributions on
  $E^k$, $k\in \N$, satisfying 1) and 2), then there exists
  on some probability space $(\Omega,\calf,\P)$ a
  stochastic process $(X(t))_{t\in\calt}$ having the
  $\mu_{t_1,\ldots,t_k}$ as its finite-dimensional
  distributions.
\end{Theorem}
The proof will be based on Carath\'eodory and Ulam theorems.
We prove first a version of the Kolmogorov theorem of independent interest.

\begin{Proposition}\label{prekolmogorov}
  Let $E_1,E_2,\ldots$ be a sequence of separable,
  complete, metric spaces and $\mu_1,\mu_2,\ldots$
  probability measures on $\calb(E_1), \calb(E_1)
  \times \calb(E_2),\ldots.$ Then there exists a
  unique probability measure on $(E,\calb (E))$ with
  $E= E_1\times E_2\times\ldots$ such that
  $$
  \mu (a_n\times E_{n+1}\times\ldots)
  =\mu_n(a_n),
  \qquad a_n\in \calb(E_1)
  \times \ldots\calb(E_n),
  $$
  provided the following compatibility condition holds:
  $$
  \mu_{n+1}(a_n\times E_{n+1})
  =
  \mu_n(a_n), \qquad a_n\in \calb(E_1)
  \times \ldots\calb(E_n).
  $$
\end{Proposition}

\noindent {\bf Proof.}
Let $\calf_0$ be the family of all sets of the form $
a_n\times E_{n+1}\times\ldots$, $a_n\in
\calb(E_1)\times \ldots\calb(E_n)$. One checks that
$\calf_0$ is a field and the set function
$$
\mu(a_n\times E_{n+1}\times\ldots)=\mu_n(a_n),
\qquad
a_n\in\calb(E_1)\times \ldots\calb(E_n),\;
n=1,2,\ldots,
$$
satisfies condition 1) of the Carath\'eodory theorem.
We check that also 2') holds. Let $(A_n)$ be a
decreasing sequence of elements of $\calf_0$ such that
$\bigcap_{n=1}^\infty A_n=\emptyset$ and assume that
there exists $\epsilon >0$ such that $\mu (A_n)\geq
\epsilon$, $n=1,2,\ldots$. Without any loss
of generality we can assume that
$$
A_n=a_n\times E_{n+1}\times\ldots,
\qquad {\rm where\;} a_n\in
\calb(E_1)\times \ldots\calb(E_n).
$$
By Ulam's theorem there exists a compact set
$b_n\subset a_n$ such that $\mu_n(a_n\backslash b_n)<
\epsilon /(2^{n+1})$, $n=1,2,\ldots$. Denote
$B_n=b_n\times E_{n+1}\times\ldots$ and define
$C_n=\bigcap_{k=1}^nB_k,$ $n=1,2,\ldots$. The sequence
$C_n$ is decreasing. Since
$$
A_n\backslash C_n = A_1\cap A_2\cap\ldots\cap A_n
\backslash B_1\cap B_2\cap\ldots\cap B_n
\subseteq \bigcup_{k=1}^n(A_k\backslash B_k),
$$
therefore
$$
\mu (A_n\backslash C_n )\leq \sum_{k=1}^n
\mu ( A_k\backslash B_k )<\frac{\epsilon}{2}
$$
and consequently $\mu(C_n)>\frac{\epsilon}{2}$,
$n=1,2,\ldots$. Consequently $C_n,n=1,2,\ldots$ are nonempty sets.
However $C_n=c_n\times
E_{n+1}\times\ldots$ where $c_n$ is a compact subset of
$E_1\times\ldots E_n$. Let $(x_n)$ be any sequence
$x_n=(x_{n1},x_{n2},\ldots)\in C_n$. By diagonal
procedure one can extract a subsequence $x_{n_k}=
(x_{n_k1},x_{n_k2},\ldots)\in C_{n_k}$ such that for
each $m$, $x_{n_km}\to x_{\infty m}$ as $k\to\infty$.
But then $x_\infty =(x_{\infty 1},x_{\infty
2},\ldots)\in \bigcap_kC_{n_k}=\bigcap_nC_{n}\subset
\bigcap_nA_n$ because all sets $C_n$ are closed
in $E$. This contradicts the identity
$\bigcap_nA_n=\emptyset$. \qed

\noindent {\bf Proof of the Kolmogorov theorem.}
We define $\Omega=E^\calt$, $X(t,\omega)=\omega(t)$,
$t\in\calt$ and $\omega\in\Omega$,
$\calf=\sigma(X(t),\, t\in\calt )$. Note that
$$
\calf= \bigcup_{\tau\subset\calt}\sigma
(X(t),\, t\in\tau )
$$
where the union is taken with respect to all countable
subsets $\tau$ of $\calt$. Let
$$
\calf_0= \bigcup_{\eta\subset\calt}\sigma
(X(t),\, t\in\eta )
$$
where the union is taken with respect to all finite
subsets $\eta$ of $\calt$. If $A\in\calf_0$ then for
some $t_1,\ldots,t_k\in\calt$ and $a\in \calb(E^k)$,
$A=\{\omega \, :\,
(X(t_1,\omega),\ldots,X(t_k,\omega))\in a\}$. We define
a set function $\mu$ on $\calf_0$ by the formula
$$
\mu(A) = \mu_{t_1,\ldots,t_k}(a).
$$
It is not difficult to see that Condition 1)  of the
Carath\'eodory theorem is satisfied. To check 2')
assume that $A_1\supset A_2\supset\ldots$ is a
decreasing sequence of $\calf_0$ such that
$\bigcap_nA_n=\emptyset$. Without any loss of
generality one can assume that
$A_n\in\sigma(X(t),t\in\tau)$ where $\tau$ is a
countable subset of $\calt$. Numbering elemnts of
$\tau$ by natural numbers and taking into account that
$A_n$ are cylindrical sets we can assume that
$$
A_n=a_n\times E^{\N},\qquad {\rm where\;} a_n\subset
E^{m_n},n=1,2,\ldots.
$$
We can therefore apply the same reasoning as in the
proof of the proposition to get that $\mu(A_n)\to 0$.
\qed

\section{Some applications}

\subsection{Gaussian processes.}

As we already know a  Gaussian measure on $\R^d$ is any probability measure on $\R^d$
whose characteristic function
$\widehat{\mu}$ is of the form
$$
\widehat{\mu} (\lambda)=e^{i\<\lambda , m\>-\frac{1}{2}
\<Q\lambda,\lambda\>},\qquad \lambda\in\R^d,
$$
where $m$ is a vector form $\R^d$ and $Q$ is a non-negative
definite matrix. The vector $m$ and the matrix $Q$ are the mean vector and the covariance matrix of $\mu$.
A random vector $Z=(Z_1,\ldots,Z_d)$ is Gaussian if its law
$\call(Z)$ is a Gaussian measure. The
random variables $Z_1,\ldots,Z_d$ are independent if and only if
the covariance matrix is diagonal.
\vspace{3mm}

\noindent A real valued process $X(t)$, $t\geq 0$, is called {\it Gaussian} if for
arbitrary numbers $t_1<t_2<\ldots<t_d$, $d=1,2,\ldots$, from the
interval $[0,+\infty)$, the distribution of the random vector
$(X(t_1),\ldots, X(t_d))$ is a Gaussian measure. Let $X$ be
Gaussian and define
\begin{equation}\label{defmediaecov}
  m(t)=\E\, X(t),\qquad q(s,t)= \E\,
  \left((X(s)-m(s))(X(t)-m(t))\right),
  \qquad t,s,\in [0,+\infty).
\end{equation}
The function $q$ is called the {\it covariance} of the process $X$.
Note that for arbitrary real numbers
$\lambda_1,\lambda_2,\ldots\lambda_d$:
$$
 \E\,
  \left(\left(\sum_{j=1}^d(X(t_j)-m(t_j))\lambda_j
  \right)^2\right)=
\sum_{j,k=1}^dq(t_j,t_k)\lambda_j\lambda_k
$$
and therefore
\begin{equation}\label{qdefpos}
  \sum_{j,k=1}^dq(t_j,t_k)\lambda_j\lambda_k\geq 0.
\end{equation}
A function $q$ for which (\ref{qdefpos}) holds
for arbitrary numbers $t_1<t_2<\ldots<t_d$, $d=1,2,\ldots$, from
$[0,+\infty)$
and for arbitrary real numbers
$\lambda_1,\lambda_2,\ldots\lambda_d$, $d=1,2,\ldots$, is called
positive definite. Thus covariances of Gaussian processes are
positive definite functions. We have the following result:

\begin{Theorem}
  For arbitrary function $m(t)$, $t\geq 0$, and arbitrary
  positive definite function $q(t,s)$, $t,s \geq 0$, there exists a
  Gaussian process for which (\ref{defmediaecov}) holds.
\end{Theorem}

\noindent {\bf Proof.}
Let $\mu_{t_1,\ldots,t_d}$ be the Gaussian measure with
$m= (m(t_1),\ldots,m(t_d))$ and
$Q=(q(t_j,t_k))_{j,k=1\ldots d}$. Let $\pi$ be the transformation
from $\R^d$ to $\R^{d-1}$ given by the formula:
$$
\pi (x_1,\ldots, x_d)=(x_1,\ldots , x_{d-1}),
\qquad x=(x_1,\ldots, x_d)\in\R^d.
$$
The consistency condition of Kolmogorov's theorem means that
the image $\mu^\pi_{t_1,\ldots,t_d}$ of the measure
$\mu_{t_1,\ldots,t_d}$ by the transformation $\pi$ is the Gaussian
measure $\mu_{t_1,\ldots,t_{d-1}}$. Let us recall that in general
if $\nu$ is the image of the measure $\mu$ by a trasformation $F$
then for a bounded measurable function $\phi$
$$
\int \phi(F(x))\; \mu(dx)=\int \phi(y)\; \nu(dx).
$$
Thus
$$
\begin{array}{lll}
\widehat{\mu}^\pi_{t_1,\ldots,t_d}(\lambda)
&=&\dis
\int_{\R^{d-1}}e^{i\<\lambda,y\>} \mu^\pi_{t_1,\ldots,t_d}(dy)
\\&=&\dis
\int_{\R^{d}}e^{i\<\lambda,\pi x\>} \mu_{t_1,\ldots,t_d}(dx)
\\&=&\dis
\int_{\R^{d}}e^{i\<\pi^*\lambda, x\>} \mu_{t_1,\ldots,t_d}(dx)
\\&=&\dis
e^{i\<\pi^*\lambda,m\>}
e^{-\frac{1}{2}\<Q\pi^*\lambda,\pi^*\lambda\>}.
\end{array}
$$
Note that
$\pi^* (\lambda_1,\ldots, \lambda_{d-1})=
(\lambda_1,\ldots, \lambda_{d-1}, 0)$, and therefore
$$
\widehat{\mu}^\pi_{t_1,\ldots,t_d}(\lambda)=
e^{i\sum_{j=1}^{d-1}m(t_j)\lambda_j
-\frac{1}{2}\sum_{j,k=1}^{d-1}q(t_j,t_k)\lambda_j\lambda_k}.
$$
Consequently
$\widehat{\mu}^\pi_{t_1,\ldots,t_d}$ is the Gaussian measure on
$\R^{d-1}$ with the mean vector
$(m(t_1)$, $\ldots$, $m(t_{d-1}))$  and the covariance matrix
$(q(t_j,t_k))_{j,k=1\ldots d-1}$, as required.
\qed

\subsection{Wiener processes.}

The following function $q$:
$$
q(t,s)=\min (t,s)=: t\wedge s,\qquad t,s\geq 0,
$$
is positive definite. To see this set
$$
\phi_t(x)= 1_{[0,t]}(x), \qquad x\geq0 .
$$
Then $\phi_t\in L^2(0,+\infty)$ for every $t\geq 0$ and
$$
q(t,s)=\< \phi_t,\phi_s\>_{ L^2(0,+\infty)},\qquad t,s\geq 0.
$$
Consequently
$$\begin{array}{lll}\dis
\sum_{j,k=1}^{d}q(t_j,t_k)\lambda_j\lambda_k
&=&\dis
\sum_{j,k=1}^{d}\<\phi_{t_j},\phi_{t_k}\>\lambda_j\lambda_k
\\
&=&\dis
\left\|\sum_{j=1}^{d}\phi_{t_j}\lambda_j\right\|^2\geq 0.
\end{array}
$$
By the theorem from
the previous subsection there exists a Gaussian process $X(t)$,
$t\geq 0$, such that
$$
\E\,X(t)=0,\qquad\qquad \E\, (X(t)X(s))=t\wedge s,\qquad t,s\geq
0.
$$
A Gaussian process $X$ with continuous trajectories,
mean value $0$ and covariance function $q(t,s)=t\wedge s$,
$t,x,\geq 0$, is called a standard   {\em Wiener process}. Any Wiener process has independent increments
in the sense that for any sequence of non-negative
numbers $t_0<t_1<\ldots<t_d$, $d=1,2,\ldots$, the random
variables
$$
X(t_1)-X(t_0),\ldots ,X(t_d)-X(t_{d-1})
$$
are independent. In fact, by the very definition,  the random vector
$$
Z=(X(t_1)-X(t_0),\ldots ,X(t_d)-X(t_{d-1}))
$$
is Gaussian with the mean vector zero. The elements of its
covariance matrix, out of the diagonal, can be easily calculated.
For $j<k$
$$
\begin{array}{l}
\E\,\left(
(X(t_j)-X(t_{j-1})( X(t_k)-X(t_{k-1}))\right)
\\\qquad
=t_j\wedge t_k-t_j\wedge t_{k-1}
-t_{j-1}\wedge t_k +t_{j-1}\wedge t_{k-1}
\\\qquad
=t_j-t_j
-t_{j-1} +t_{j-1}=0.
\end{array}
$$
Since they vanish the increments are independent.
\vspace{2mm}

\noindent The constructed process is thus a process with independent increments with a prescribed covariance.
Note however that the
continuity of its trajectories does not follow from the Kolmogorov
theorem and requires additional considerations.

\subsection{Markov processes.}

Let $(E,\cale)$ be a complete, separable metric space equipped
with the $\sigma$-field $\cale=\calb (E)$ of Borel subsets of $E$.
Let $\calp (E)$ denote the set of probability measures on
$(E,\cale)$. A {\it transition function}  is a family of transformations
$P^t$, $t\geq 0$, from $E$ into $\calp (E)$ which satisfy the
{\it Chapman-Kolmogorov} equation:
$$
P^{t+s}(x,\Gamma)= \int_E P^{t}(x,dy)P^{s}(y,\Gamma),
\qquad x\in E,\; \Gamma\in \cale,\; t,s,\geq 0.
$$
The value $P^{t}(x,\Gamma)$ can be interpreted as the probability
that a stochastic dynamical system starting from $x$ will be in
the set $\Gamma$ at moment $t$. In a natural way one defines by
induction, for any sequence of non-negative numbers
 $t_1<t_2<\ldots<t_d$, $d=1,2,\ldots$, the {\em probabilities
 of visiting sets } $\Gamma_1,\ldots, \Gamma_d$ {\em at
 moments } $t_1,\ldots ,t_d$ {\em starting  from} $x$ as the
functions
$P^{t_1,\ldots ,t_d}:E\to \calp (E)$:
$$
P^{t_1,\ldots ,t_d}(x, \Gamma_1,\ldots, \Gamma_d)=
\int_{\Gamma_1}P^{t_1}(x, dx_1)
P^{t_2-t_1,\ldots ,t_d-t_{d-1}}(x_1, \Gamma_2,\ldots, \Gamma_d).
$$
The Chapman-Kolmogorov equation implies that
$P^{t_1,\ldots ,t_{d-1}}$  is equal to the projection of
$P^{t_1,\ldots ,t_d}$ on the first $d-1$ coordinates. By
Kolmogorov's theorem there exists an $E$-valued process $X$,
defined on a probability space $(\Omega,\calf,\P)$, such that
\begin{enumerate}
\item[1)] $X(0,\omega)=x,\qquad \omega\in\Omega$;
\item[2)] $\P(\{\omega\; :\; X(t_j,\omega)\in \Gamma_j,\;
j=1,\ldots, d\})= P^{t_1,\ldots ,t_d}(x, \Gamma_1,\ldots,
\Gamma_d)$.
\end{enumerate}
A stochastic process $X$ for which the above two properties 1) and 2) hold is called a {\it Markov process with
transition function} $P^t$. For more information see \cite{bib17}, \cite{Dy}, \cite{BG}, \cite{EK}, \cite{Sh}.

\subsection{L\'evy processes.}

\noindent Let $E=\R^d$. If $X$ is a stochastic process with values in $\R^d$ and for any sequence of non-negative
numbers $t_0<t_1<\ldots<t_d$, $d=1,2,\ldots$, the random
variables
$$
X(t_1)-X(t_0),\ldots ,X(t_d)-X(t_{d-1})
$$
are independent than one says that $X$ has {\it independent increments}. If, for all $t> s\geq 0,$
the distribution of $X(t) - X(s)$ is identical with the distribution of $X(t-s)$, then $X$ has {\it time
homogeneous} increments. A process with independent and time homogeneous increments is called a {\it L\'evy process}
\cite{bib11}, \cite{sato}, \cite{bertoin}.
L\'evy processes are closely related to the so called infinitely divisible families of measures.
\vspace{2mm}

\noindent A family $ \mu_t, t\geq 0$
of probability measures on $\R^d$ is
{\em infinitely divisible} if

1) $\mu_0=\delta_{\{0\}},$

2) $\mu_{t+s}=\mu_t*\mu_s$ for all $t,s\geq 0,$

3) $\mu_t(\{x\; : \; \|x\|<r)\to 1$ as $t\downarrow 0$, for arbitrary $r>0$.

\noindent If $\mu_t$, $t\geq 0$
 is an infinitely divisible family then the formula
\begin{equation}\label{sginvtrans}
  P^{t}(x,\Gamma)= \mu_t (\Gamma -x),
  \qquad \Gamma\in\calb (R^d),\; t\geq 0,\; x\in \R^d,
\end{equation}
defines a transition function. Consequently there exists a Markov
 process $X$, with the transition function
 $P^{t}$. The finite dimensional distributions $\mu_{t_1,t_2,\ldots ,t_n}$ are determined by the identities:
\begin{equation}\label{findimdistrbis}
  \begin{array}{l}\ds
\int_{\R^{nd}}\psi (x_1,\ldots,x_n)\,
 \mu_{t_1,t_2,\ldots ,t_n}(dx_1,\ldots,dx_n)
 \eds\\\ds
 =
\int_{\R^{nd}}\psi (x_1,x_1+x_2,\ldots,x_1+\ldots +x_n)\,
 \mu_{t_1}(dx_1)\mu_{t_2-t_1}(dx_2)\ldots
 \mu_{t_n-t_{n-1}}(dx_n),

 \eds
  \end{array}
\end{equation}
valid for all bounded Borel
$\psi$ and for
all  sequences $0=t_0\leq t_1 < t_2<\ldots <t_n$.

\begin{Proposition}
Let $X(t)$, $t\geq 0$, be an $\R^d$-valued process with
finite dimensional distributions $\mu_{t_1, \ldots
,t_n}$ satisfying (\ref{findimdistrbis}). Then for
arbitrary Borel sets $\Gamma_1,\ldots,\Gamma_n$,
\begin{equation}\label{indincrbis}
\P(X(t_1)\in\Gamma_1,X(t_2)-X(t_1)\in\Gamma_2,\ldots
X(t_n)-X(t_{n-1})\in\Gamma_n)\! =\!
\prod
_{k=1}^{n} \P(X(t_k)-X(t_{k-1})\in\Gamma_k).
\end{equation}
Moreover $\P(X(t)-X(s)\in\Gamma)=\mu_{t-s}(\Gamma)$,
$t\geq s \geq 0$, $\Gamma$, Borel .
\end{Proposition}

\noindent{\bf Proof.}
Let $\phi_1,\ldots,\phi_n$ be bounded continuous
functions on $\R^d$. Define $\psi (x_1,\ldots,x_n)=
\phi_1(x_1)\phi_2(x_2-x_1)\ldots\phi_n(x_n-x_{n-1})$,
 $x_1,\ldots,x_n\in\R^d$. Then
$$
  \begin{array}{l}
    \E \bigg(\psi(X(t_1),\ldots,X(t_n))\bigg)
    =\ds
\int_{\R^{nd}}\psi (x_1,\ldots,x_n)\,
 \mu_{t_1, \ldots ,t_n}(dx_1,\ldots,dx_n)
   \eds  \\\ds
   = \int_{\R^{nd}}\psi (z_1,z_1+z_2\ldots,z_1+\ldots +z_n)\,
 \mu_{t_1}(z_1)\ldots \mu_{t_n-t_{n-1}}(dz_n)
 \eds\\\ds
 =
\int_{\R^{nd}}\phi_1(z_1)\phi_2(z_2)\ldots,\phi_n(z_n)\,
 \mu_{t_1}(z_1)\ldots \mu_{t_n-t_{n-1}}(dz_n)
 \eds\\\ds
=\prod
_{k=1}^{n} \int_{\R^{d}} \phi_k(z_k)\,
\mu_{t_k-t_{k-1}}(dz_k)
=\prod
_{k=1}^{n} \E \bigg(\phi_k(X(t_k-t_{k-1}))\bigg)
\eds\\\ds
=\prod
_{k=1}^{n} \E \bigg(\phi_k(X(t_k)-X(t_{k-1}))\bigg).
\qquad
\qed
\eds
  \end{array}
$$
L\'evy processes have been intensively studied, see \cite{sato} and \cite{bertoin}. The complete description
of their distributions, in terms of their characteristic functions,  will be given in the chapter
on L\'evy processes. However explicit formulae for their distributions are known only in few cases. The most
important ones are the following:
\vspace{2mm}

\noindent 1. {\it Compound Poisson process} on $\R^d$ when
$$
\mu_t  = e^{-t \nu(\R^d)}\sum_{k=0}^{\infty}{\frac{t^k}{k!}}\nu^{*k},
$$
and $\nu$ is a finite measure,
\vspace{2mm}

\noindent 2. {\it Gaussian, L\'evy processes} on $\R^d$. Here
$\mu_t = N( ta, tQ)$,\,\,\,\,$a\in R^d$\,\,\,\, and $Q$ is a symmetric, non-negative, $d\times d$ matrix,
\vspace{2mm}

\noindent 3. {\it Symmetric Cauchy processes} on $\R^d$. Here
$$
\mu_t (dx) = {\frac{\Gamma ((d+1)/2)}{\pi^{(d+1)/2}}}{\frac{t}{(|x|^2 + t^2 )^{(d+1)/2}}}dx,
$$
\vspace{2mm}

\noindent 4. {\it Stable processes of order} $1/2$ on $\R_{+}^1$. Here
$$
\mu_t (dx) = {\frac{t}{\sqrt{2\pi}}}{\frac{1}{\sqrt{}x^3}} e^{-{\frac{t^2}{2x}}}dx.
$$
\subsection{Poisson processes}
A stochastic L\'evy process is called a {\it Poisson process} if its trajectories are increasing functions
with non-negative integer values and with jumps of size $1$. It is a special case of a compound Poisson
process with measure $\nu = \lambda \delta_{\{0\}}$ and it is parametrised by parameter $\lambda > 0$.
\vspace{2mm}

\noindent The value $\pi (t)$ of the Poisson process at moment $t$ can represent,
for instance, the number of customers at a queue at moment $t$, or the number of
car accidents in a given country in the  time interval $[0, t]$.

\chapter{Doob's regularisation theorem}

Kolmogorov's theorem allows to establish existence of a stochastic
process $X$ with prescribed finite dimensional distributions
but does not imply any properties of its
trajectories like continuity or right
continuity. The aim of this chapter is to prove a powerful
regularization result of Doob which, in many cases, allows to
prove existence of a {\it c\`adl\`ag version} of $X$ that is a version which  has right continuous
trajectories with finite left-limits. The theorem and its proof is
also a respectable pretext to introduce {\it martingales}, see \cite{bib6}, a powerful tool of the
theory of stochastic processes. As an application we establish existence of c\`adl\`ag version of L\'evy and
Feller processes.

\section{Martingales and supermartingales.}

Let $T$ be an arbitrary subset of $\R_+$  and let $(\Omega,
\calf,\P)$ be a probability space and $(\calf_t)_{t\in T}$ an
increasing family of sub-$\sigma$-fields of $\calf$. A family
$X_t$, $t\in T$, of finite real-valued random variables
adapted to the family $(\calf_t)_{t\in T}$,
i.e. each $X_t$ is $\calf_t$-measurable, $t\in T$, is said to be a
{\em martingale} \cite{bib14}, respectively a supermartingale, a
submartingale, with respect to  $(\calf_t)_{t\in T}$ if
\begin{enumerate}
\item[1)] $X_t$, $t\in T$, are integrable random variables.
\item[2)] If $s\leq t$, then for every event $A\in\calf_s$
$$
\int_A X_t\; d\P=\int_A X_s\; d\P,
$$
and respectively
$$
\int_A X_t\; d\P\le \int_A X_s\; d\P,
$$
$$
\int_A X_t\; d\P\ge \int_A X_s\; d\P.
$$
\end{enumerate}
In terms of the conditional expectation property 2) can be
rephrased as follows:
\begin{enumerate}
\item[2')] If $s\leq t$, then
$$
\E\; ( X_t|\calf_s) = X_s,\qquad \P-a.s.,
$$
or respectively
$$
\E\; ( X_t|\calf_s) \le  X_s,\qquad \P-a.s.,
$$
$$
\E\; ( X_t|\calf_s) \ge  X_s,\qquad \P-a.s.
$$
\end{enumerate}

Every real constant, respectively decreasing, increasing, function
defined on $T$ is a martingale, respectively a supermartingales, a
submartingale.

\begin{Lemma}\label{martbananli}
  If $(X_t)$ and $(Y_t)$  are supermartingales relative to
  $(\calf_t)$ and $\alpha$, $\beta$ are positive numbers then the
  processes $(\alpha X_t+\beta Y_t)$  and $(X_t\wedge Y_t)$   are
  also supermartingales. If $(X_t)$ is a martingale, then
  $(|X_t|)$ is a submartingale.
\end{Lemma}

\noindent {\bf Proof.} Let $s\leq t$ and $A\in\calf_s$.
Since $\alpha \int_A X_s\; d\P\ge \alpha \int_A X_t\; d\P$ and
$\beta \int_A Y_s\; d\P\ge \beta \int_A Y_t\; d\P$, therefore
$\int_A(\alpha  X_s+\beta Y_s)\; d\P\ge
 \int_A (\alpha X_t+\beta Y_t)\; d\P$.

 Obviously,
$$
\int_A(  X_s\wedge Y_s)\; d\P=
\int_{A\cap \{X_s<Y_s\}}X_s\; d\P+
\int_{A\cap \{X_s\ge Y_s\}}Y_s\; d\P
$$
and
$$
{A\cap \{X_s<Y_s\}}\in \calf_s,\qquad
{A\cap \{X_s\ge Y_s\}}\in \calf_s.
$$
By virtue of the definition of a supermartingale, we obtain
$$
\int_{A\cap \{X_s<Y_s\}}(X_s-X_t)\; d\P\ge 0,\qquad
\int_{A\cap \{X_s\ge Y_s\}}(Y_s-Y_t)\; d\P\ge 0
$$
and, consequently,
$$
\int_A(  X_s\wedge Y_s)\; d\P\ge
\int_{A\cap \{X_s<Y_s\}}X_t\wedge Y_t\; d\P+
\int_{A\cap \{X_s\ge Y_s\}}X_t\wedge Y_t\; d\P=
\int_{A}X_t\wedge Y_t\; d\P.
$$

If $(X_t)$ is a martingale, then $(X_t\vee 0)$ and $(-X_t\vee 0)$
are submartingales, therefore $|X_t| =
(X_t\vee 0)+(-X_t\vee 0)$ is a submartingale, too.
\qed

Let $(\Omega, \calf)$ be a measurable space and $(\calf_t)_{t\in
T}$ an increasing family of $\sigma$-fields, $\calf_t\subset
\calf$, $t\in T$. A function $S:\Omega\to T$ is said to be a {\em
stopping time}, relative to $(\calf_t)$, if for every $t\in T$ the
set $\{\omega\; :\; S(\omega)\leq t\}$ belongs to $\calf_t$, i.e.
the condition $S\le  t$ is a condition involving only what has
happened up to and including time $t$. Let $S$ be a stopping
time. By $\calf_S$ we denote the collection of events $A\in\calf$
such that $A\cap \{S\leq t\}\in \calf_t$ for all $t\in T$. It is
easy to verify that $\calf_S$ is a $\sigma$-field, the so-called
{\em $\sigma$-fields of events prior to S}.

\begin{Proposition}\label{elemstoptimes}
  \begin{enumerate}
  \item[a)] If $T$ is a finite or countable subset of $\R$ then
  $S:\Omega\to T$ is a stopping time if and only if
  $\{\omega\; :\; S(\omega)= t\}\in\calf_t$ for $t\in T$.
    \item[b)] If $S_1$ and $S_2$ are stopping times then
    $S_1\wedge S_2$ and     $S_1\vee S_2$ are again stopping
    times.
      \item[c)] Any stopping time $S$ is $\calf_S$-measurable.
    \item[d)]  If  $S_1\le S_2$ then  $\calf_{S_1}\subset
    \calf_{S_2}$.
    \end{enumerate}
\end{Proposition}

\noindent {\bf Proof.}  The properties $a), b), c)$ follow
directly from the definitions. To prove $d)$ assume that
$A\in \calf_{S_1}$; then $\{S_2\le t\}\cap A=
\{S_2\le t\}\cap \{S_1\le t\}\cap A$ because
$\{S_2\le t\}\cap \{S_1\le t\}= \{S_2\le t\}$.
Since $ \{S_1\le t\}\cap A\in \calf_{t}$ and
$\{S_2\le t\}\in \calf_t$ therefore
$\{S_2\le t\}\cap \{S_1\le t\}\cap A\in \calf_t$.
\qed

\begin{Example}\label{esstopping}\begin{em}
  Let $(X_n)$ be a sequence of random variables adapted to
  $(\calf_n)$
  and let $\calf_\infty=\calf$. Then
  $$
  S=\left\{\begin{array}{l}
  {\rm the \; least\;} n{\rm \; such\; that\;} X_n\ge a,
  \\
  +\infty {\rm \; if \; } X_n<a {\rm \; for\; all\;} n=1,2,\ldots
  \end{array}
  \right.
  $$
  is a stopping time. To see this fix a natural number $k$; then
  $\{S=k\}= \{X_1<a,\ldots ,X_{k-1}<a,X_k\ge a\}\in
  \sigma (X_1,\ldots ,X_k)\subset \calf_k$.
\end{em}
\end{Example}

The following Doob's optional sampling theorem is of fundamental
importance in the whole theory of martingales.

\begin{Theorem}\label{doobsampling}
  Let $(X_n)_{n=1,\ldots ,k}$ be a supermartingale (a martingale)
  relative to
  \\
  $(\calf_n)_{n=1,\ldots ,k}$. Let $S_1,S_2,\ldots
  ,S_m$ be an increasing sequence of $(\calf_n)$-stopping times with values in the set $\{1,\ldots ,k\}$.
  The sequence $(X_{S_i})_{i=1,\ldots ,m}$ is then also a
  supermartingale (a martingale) with respect to the
  $\sigma$-fields $\calf_{S_1},\calf_{S_2},\ldots,\calf_{S_m}$.
\end{Theorem}

\noindent {\bf Proof.} Let $A\in\calf_{S_1}$. We shall prove that
$\int_A(X_{S_1}-X_{S_2})\; d\P\ge 0$ in the supermartingale case
and
$\int_A(X_{S_1}-X_{S_2})\; d\P= 0$ in the martingale case. If, for
every $\omega$, ${S_2}(\omega)-{S_2}(\omega)\le 1$
then
$$
\int_A(X_{S_2}-X_{S_1})\; d\P=
\sum_{r=1}^k \int_{A\cap \{S_1=r\}\cap \{S_2>r\}}
(X_{r+1}-X_{r})\; d\P,
$$
but $A\cap \{S_1=r\}$ and $ \{S_2>r\}$ belong to $\calf_r$,
therefore $A\cap \{S_1=r\}\cap \{S_2>r\}\in\calf_r$. The
definition of a supermartingale (a martingale) implies that the
desired inequality (equality) follows.

Let $r=0,1,\ldots ,k$ and define stopping times $R_r=S_2\wedge
(S_1+r)$. Then $S_1=R_0\leq R_1\le \ldots\le R_k=S_2$ and
$R_{i+1}-R_i\le 1$, $i=0,\ldots,k-1$. By virtue of the first part
of the proof
$$
\int_AX_{S_1}\; d\P\ge \int_AX_{R_1}\; d\P\ge \ldots
\ge \int_AX_{R_k}\; d\P=\int_AX_{S_2}\; d\P,
$$
with equalities in the martingale case. \qed

As the first application of the Doob theorem we establish the
following Doob's inequalities.

\begin{Theorem}\label{doobineq}
  Let $(X_n)_{n=1,\ldots ,k}$ be a supermartingale
  and $c$ a non-negative constant. Then we have
  \begin{enumerate}
\item[1)]
$\dis c\; \P\left(\sup_{ n}X_n\ge c\right)\le \E\,
X_1-\int_{\{\sup_{n}X_n< c\}}X_k\; d\P$.
\item[2)]
$\dis c\; \P\left(\sup_{ n}X_n\ge c\right)\le \E\,
X_1+\E\, X_k^-$.
\item[3)]
$\dis c\; \P\left(\inf_{ n}X_n\le - c\right)\le
-\int_{\{\inf_{n}X_n\le - c\}}X_k\; d\P$.
\item[4)]
$\dis c\; \P\left(\inf_{ n}X_n\le - c\right)\le
\E\, X_k^-$.
  \end{enumerate}
\end{Theorem}

\noindent {\bf Proof.}
To prove the inequalities $1)$, $2)$, define
$S(\omega)=\inf\{n\;:\; X_n\geq c\}$, or $S(\omega)=k$ if
$\sup_nX_n(\omega)<c$. $S$ is a stopping time and $S\ge 1$. Thus,
by Doob's optional sampling theorem, we obtain
$$
\begin{array}{lll}
\E\, X_1&\ge& \E\, X_S
\\&=&\dis
\int_{\{\sup_{n}X_n\ge  c\}}X_S\; d\P+
\int_{\{\sup_{n}X_n< c\}}X_S\; d\P
\\&\ge&\dis
c\; \P\left(\sup_{n}X_n\ge  c\right) +
\int_{\{\sup_{n}X_n< c\}}X_k\; d\P
\end{array}
$$
or equivalently
$$
c\; \P\left(\sup_{n}X_n\ge  c\right) \le \E\, X_1-
\int_{\{\sup_{n}X_n< c\}}X_k\; d\P.
$$
This is exaclty inequality $1)$.

Since $-X_k\le X_k^-$, inequality $2)$ follows, too.

To establish the relations $3)$ and $4)$, we introduce an
analogous stopping time $S$, $S= \inf\{n\;:\; X_n\leq -c\}$
 or $S=k$ if $\inf_nX_n>-c$. Since $S\le k$ therefore
 $$
\begin{array}{lll}
\E\, X_k&\le& \E\, X_S
\\&=&\dis
\int_{\{\inf_{n}X_n\le - c\}}X_S\; d\P+
\int_{\{\inf_{n}X_n>-c\}}X_S\; d\P
\\&\le&\dis
-c\; \P\left(\inf_{n}X_n\le - c\right) +
\int_{\{\inf_{n}X_n>-c\}}X_k\; d\P
\end{array}
$$
These inequalities imply $3)$ and $4)$.
\qed

Let $x=(x_1,\ldots, x_n)$ be a sequence of real numbers and let
$a<b$. Let $R_1$ be the first of the numbers $1,2,\ldots, n$ such
that $x_{R_1}\le a$, or $n$ if there exist no such numbers.
Let $R_k$ be, for every even (respectively odd) integer $k>1$, the
first of the numbers $1,2,\ldots, n$ such
that $R_k>R_{k-1}$ and $x_{R_k}\ge b$ (respectively $x_{R_k}\le
a$), and if no such number exists we set $R_k=n$. In this way a
sequence $R_1,R_2,\ldots $ is defined. The number $U_n(x;a,b)$
of upcrossings
by the sequence $x$  of the interval $(a,b)$ is defined as the
greatest integer $k$ such that one actually has
$x_{R_{2k-1}}\le a$ and $x_{R_{2k}}\ge b$. If no such integer
exists we put $U_n(x;a,b)=0$.

\begin{Theorem}\label{doobupcrossing}
  Let $(X_m)_{m=1,\ldots ,n}$ be a supermartingale
  relative to
  $(\calf_m)_{m=1,\ldots ,n}$ and let $a<b$ be two real numbers.
  Then the following inequality holds:
  $$
  \E\, U_n(X;a,b) \le \frac{1}{b-a}\,\E\, \left((a-X_n)^+\right).
  $$
\end{Theorem}

\noindent {\bf Proof.} Let $R_1,\ldots , R_{2l}$, $2l>n$,
be the stopping times used in the definition of upcrossings and
let $\Sigma_1= (X_{R_2}-X_{R_1})+\ldots
(X_{R_{2l}}-X_{R_{2l-1}})$. Then
$$
\Sigma_1\ge  (b-a)U_n(X;a,b)+(X_n-a)\wedge 0.
$$
Indeed, assume that $U_n(X;a,b)=k$; then
$\Sigma_1= (X_{R_2}-X_{R_1})+\ldots
(X_{R_{2k+2}}-X_{R_{2k+1}})$.
If $R_{2k+1}=n$ then $X_{R_{2k+2}}-X_{R_{2k+1}}=0$, and
if $R_{2k+1}<n$ and $R_{2k+2}=n$ then
$X_{R_{2k+2}}-X_{R_{2k+1}}=X_{n}-X_{R_{2k+1}}\ge X_n-a$. Thus in
both cases
$X_{R_{2k+2}}-X_{R_{2k+1}}\ge (X_n-a)\wedge 0$.

Doob's theorem implies that
$$
0\ge \E\, (\Sigma_1) \ge (b-a)\,\E\,  U_n(X;a,b)
+\E\, ((X_n-a)\wedge 0)
$$
and therefore
$$
(b-a)\,\E\,  U_n(X;a,b)\le - \E\, ((X_n-a)\wedge 0)
=\E\, \left((a-X_n)^+\right).
\qed
$$

\section{Regularization theorem.}

Let $X(t)$, $t\ge 0$, be an $E$-valued stochastic process defined
on $(\Omega,\calf,\P)$. An $E$-valued process
$Y(t)$, $t\ge 0$, is said to be a modification of $X$ if
$$
\P(X(t)=Y(t))=1\quad {\rm for\; all\quad } t\ge0.
$$
It is clear that a modification $Y$ of $X$ has the same finite
dimensional distributions as $X$. If there exists a modification
$Y$ of $X$ which has, $\P$-almost surely, continuous trajectories,
then we say that $X$ has a continuous modification. If there
exists a modification
$Y$ of $X$ whose  trajectories are right-continuous and have
finite left limits
then we say that $X$ has a c\`adl\`ag (continu \`a droite
et pourvu de limites \`a gauche) modification $Y$.

The following Doob's regularization theorem states a condition
under which a supermartingale has a c\`adl\`ag modification.
Although it holds for supermartingales it has been successfully
applied to wide classes of stochastic processes implying existence
of their c\`adl\`ag versions.
We will assume that the filtration $(\calf_t)$ satisfies the
usual conditions, that is
\begin{enumerate}
\item[1)] $\calf$ is $\P$-complete.
\item[2)] $\calf_0 $ contains all $\P$-null sets of $\calf$.
\item[3)] $(\calf_t)$ is right-continuous:
$$
\calf_t= \bigcap_{u>t}\calf_u=: \calf_{t^+},
\qquad t\ge 0.
$$
\end{enumerate}

\begin{Theorem}\label{doobreg}
  Assume that $X(t)$, $t\ge0$, is a supermartingale such that for
  each $t\ge 0$ and $c>0$
\begin{equation}\label{stocconti}
  \lim_{s\to t,\, s>t}\P (|X(t)-X(s)|>c)=0.
\end{equation}
Then $X$ has a c\`adl\`ag version.
\end{Theorem}

\subsection{An application to L\'evy processes}

Before proving the result we deduce a simple application on
regularity of L\'evy processes.

\begin{Theorem}\label{reglevyproc}
An arbitrary  L\'evy process  $X$
 has a c\`adl\`ag version.
\end{Theorem}

\noindent {\bf Proof.}
We can assume that $X$ is a real valued  L\'evy process.

Let $\calg_t=\sigma (X(s)\; :\; s\le t)$, $t\ge 0$,
$\calf$ the completion of $\calg_\infty = \sigma (X(s)\; :\; s\ge
0)$, and $N_0$ all sets of $\calf$ with $\P$-measure zero.
We set
$$
\calf_t = \bigcap_{s>t}\calg_s\vee N_0, \qquad t\ge0.
$$
Then the family $(\calf_t)$ satisfies the usual conditions and for
each $s\le t$ the increment $X(t)-X(s)$ is independent of
$\calf_s$.

For each number $u\in \R$ define a complex valued process
$$
Z(t)=\frac{e^{iuX(t)}}{\E\, e^{iuX(t)}},
\qquad t\ge0.
$$
Then $Z$ is a stochastic  process for which real and imaginary
parts are martingales. In fact
$$
\begin{array}{lll}
  \E\, (Z(t)|\calf_s) & = & \dis
\frac{1}{\E\, e^{iuX(t)}} \E\, \left(
e^{iu(X(t)-X(s))+iuX(s)}|\calf_s\right)
   \\
    & = & \dis
   \frac{1}{\E\, e^{iuX(t)}}
    \E\, \left( e^{iu(X(t)-X(s))}|\calf_s\right)
   e^{iuX(s)},
\end{array}
$$
Since the random variable $X(t)-X(s)$ is independent
of $\calf_s$ we have that
$$
  \E\, (Z(t)|\calf_s)  =
  \frac{1}{\E\, e^{iuX(t)}}
    \E\, \left( e^{iu(X(t)-X(s))}\right)
   e^{iuX(s)}.
   $$
It follows from the definition of L\'evy processes, that
$$
    \E\, \left( e^{iu(X(t)-X(s))}\right)=
     \E\, \left( e^{iuX(t-s)}\right)=
 e^{-(t-s)\psi (u)},
   $$
where $\psi$ is the L\'evy-Khinchin exponent, and the lemma
follows.

The continuity condition (\ref{stocconti}) from the theorem is
satisfied and therefore for arbitrary $u\in\R$ the process
$e^{iuX(t)}$, $t\ge0$, has a c\`adl\`ag modification. It is
possible to show that if, for a sequence $(a_n)$ of real numbers,
the limit $\lim_n e^{iua_n}$ exists for the set of all rational
numbers $u$ then the sequence $(a_n)$ is convergent.

Let $X^u(t)$, $t\ge0$, be a modification of $X(t)$, $t\ge 0$, such
that the process
$$
e^{iuX^u(t)},\qquad t\ge0,
$$
is c\`adl\`ag. For each $u$ we denote by $\Omega_u$ the set of
those $\omega\in\Omega$ for which the trajectory of $X^u$ is
c\`adl\`ag. Denote by $\Q$ and, respectively $\Q_+$, the set of all
rational and, respectively, non-negative rational numbers.
Then the set
$$
\widetilde{\Omega}=\bigcap_{u\in\Q_+}\bigcap_{v\in\Q_+}
\bigcap_{r\in\Q_+} \left(\Omega_u \cap
\{\omega\; :\; X^u(r,\omega)=X^v(r,\omega)\}\right)
$$
is of full measure. If we define
$$
Y(t,\omega)= \left\{
\begin{array}{ll}
  X^u(t,\omega), & \omega\in\widetilde{\Omega}
  {\rm \; and\;} u\in \Q_+, \\
  0, & \omega\in\widetilde{\Omega}^c,
\end{array}
\right.
$$
we get the required modification.
\qed

\subsection{Proof of the regularisation theorem.}
In the proof of Theorem \ref{doobreg}we basically follow  L. C. G. Rogers and D. Williams \cite{rogers}, vol. 1.
To cope with measurability questions it is convenient to
consider, at the beginning, functions defined on $\Q_+$ rather
than on $\R_+$. We say that a function $x:\Q_+\to \R$ has
a c\`adl\`ag regularization if the limits
$$
\lim{s>t,\, s\in \Q_+}x(s),
\qquad
\lim{s<t,\, s\in \Q_+}x(s)
$$
exist for all non-negative, respectively positive $t$.
The following result is a direct consequence of the definition.

\begin{Lemma}\label{regrazionali}
  If $x$ has a c\`adl\`ag regularization  then $y:\R_+\to \R$
  given by the formula
  $$
  y(t)= \lim_{s>t,\, s\in \Q_+}x(s),
\qquad t\ge 0,
$$
is a c\`adl\`ag function.
\end{Lemma}

The next lemma is of basic importance. To formulate it we need to
generalize the definition of the upcrossing function
$U^N(x;a,b)$. Namely the number $U^N(x;a,b)$ of upcrossings by the
function $x$ of the interval $(a,b)$ is the supremum of
$U_n((x(t_1),\ldots, x(t_n));a,b)$ with respect to all
$n=1,2,\ldots$ and arbitrary choice $0\leq t_1<t_2<\ldots <t_n\le
N$, $t_j\in \Q$.

\begin{Lemma}\label{regrazionalidue}
A function $x:\Q_+\to \R$ has a c\`adl\`ag regularization
if and only if for all $N\in\N$ and $a,b\in \Q$, $a<b$,
\begin{enumerate}
\item[i)]
$\sup \{ x(s)\; :\; s\in [0,N]\cap \Q_+\} <+\infty$,
\item[ii)]
$ U^N(x;a,b)<+\infty$.
\end{enumerate}
\end{Lemma}

\noindent {\bf Proof.}
We prove for instance that if $i)$ and $ii)$ hold then $x$ has a
 c\`adl\`ag regularization. The other implication is similar.
 Thus assume
$i)$ and $ii)$ hold but nevertheless for some $t\ge0$
$$
\limsup_{s>t,\, s\in \Q_+} x(s) >
\liminf_{s>t,\, s\in \Q_+} x(s).
$$
There exists rational numbers $a,b$ such that
$$
\liminf_{s>t,\, s\in \Q_+} x(s)<a<b
\limsup_{s>t,\, s\in \Q_+} x(s) .
$$
Consequently for some decreasing sequences of rational numbers,
$s_n>t$, $\overline{s}_n>t$, $n=1,2,\ldots$, converging to $t$:
$$
x(s_n)\le a,\quad x(\overline{s}_n)\ge b,
\qquad n=1,2,\ldots.
$$
This implies that $U^N(x;a,b,)=+\infty$ for $N>t$. Thus
$$
\limsup_{s>t,\, s\in \Q_+} x(s) =
\liminf_{s>t,\, s\in \Q_+} x(s)=
\lim_{s>t,\, s\in \Q_+} x(s).
$$
The condition $i)$ implies that the limit is finite.
\qed

To prove the theorem define
$$
\Gamma =\{ \omega\in \Omega\; :\; X(s,\omega),\; s\in\Q_+,
{\rm \; has\;a\;
c\grave{a}dl\grave{a}g \; regularization}\}.
$$
Then
$$
\Gamma=\bigcup_{N\in\N;\, a,b\in\Q_+;\, a<b}
\left\{
\omega\in\Omega\;:\; U^N(X(\cdot,\omega);a,b)<+\infty,
\quad \sup_{s\in [0,N]\cap \Q_+} |X(s,\omega)| <+\infty\right\}.
$$
It is easy to see that $\Gamma$ is a measurable set
and we define
$$
Y(t,\omega)= \left\{
\begin{array}{ll}\dis
\lim_{s>t,\, s\in \Q_+}  X(s,\omega), & \omega\in\Gamma,
 \\
  0, & \omega\in\Gamma^c.
\end{array}
\right.
$$
We know that $Y(\cdot,\omega)$ is a c\`adl\`ag  function for
$\omega\in\Gamma$. We show now that $\P(\Gamma)=1$. Let us fix $N$
and rational numbers $a<b$. For an arbitrary sequence of rational
numbers $0=s_1<s_2<\ldots <s_{n-1}<s_n=N$ the sequence
$X(s_1),\ldots,X(s_n)$ is a supermartingale. Therefore, by Doob's
inequalites,  for an arbitrary number $c\ge 0$:
$$
\begin{array}{lll}
  \dis
   \P\left(\sup_{k=1,\ldots, n}X(s_k)\ge c\right)
   &\le&\dis
   \frac{1}{c}\,(\E\,
X(0)+\E\, X^-(N)), \\\dis
\P\left(\inf_{k=1,\ldots, n}X(s_k)\le -c\right)
   &\le&\dis
   \frac{1}{c}\,\E\, X^-(N).
\end{array}
$$
>From this, by considering an increasing family of increasing
sequences $(s_1,\ldots,s_n)$, covering $\Q_+\cap [0,N]$, we obtain
that
$$
\lim_{c\uparrow +\infty}\P\left(
\sup_{s\in\Q_+,\, s\leq N}|X(s)|\ge c\right)
\le
\lim_{c\uparrow +\infty}\frac{1}{c}\left(
\E\,X(0)+2\E\, X^-(N)
\right)=0.
$$
In a similar manner, by Doob's upcrossing theorem,
$$
\E\, U^N((X(s_1),\ldots,X(s_n));a,b)\le
\frac{1}{b-a}\, \E\, (a-X(N))^+.
$$
Therefore
$$
\E\, U^N(X;a,b)\le
\frac{1}{b-a}\, \E\, (a-X(N))^+<+\infty
$$
and in particular
$$
\P\,( U^N(X;a,b)<+\infty)=1.
$$
This way we have shown that $\P(\Gamma)=1$.

It is clear that, for each $t\ge 0$, $Y(t)$ is measurable with
respect to
$$
 \bigcap_{s>t}\sigma (X(u)\; :\; u\le s)\vee N_0.
$$
It remains to show that for each $t\ge0$:
\begin{equation}\label{condizmodif}
  \P(Y(t)=X(t))=1.
\end{equation}
Let us fix $t\ge0$ and let $t_n>t$, $n=1,2,\ldots$, be
 a sequence
of rational numbers converging to $t$. Then there exists a
subsequence $t_{n_k}$, $k=1,2,\ldots$, such that
$$
  \P\left(|X(t)-X(t_{n_k})|>\frac{1}{k}\right)\le\frac{1}{k^2}.
  $$
Consequently by the Borel-Cantelli lemma
$$
  \lim_k X(t_{n_k})=X(t),\qquad \P-a.s.
  $$
By the construction
$$
\P\left( Y(t)= \lim_k X(t_{n_k}\right)=1,
  $$
  so (\ref{condizmodif}) holds.
\qed

\subsection{An application to Markov processes.}
We deduce from Theorem \ref{doobreg}
a general regularity result for Markov
processes which implies the regularity of  L\'evy
processes as a very special case.

Let $E$ be a metric, separable, locally compact space and
$E_\partial$ the one-point compactification of $E$. If $E$  is
compact then $\partial$ is an isolated point. Let $P^t$ be a
transition function on $E$. Let $\calb_b(E)$ be the space of all
bounded Borel functions on $E$ and $C_0(E)$ the space of all
continuous functions on $E$ having limit zero at $\partial$.
The formula
$$
P^t\phi(x)=\int_E P^t(x,dy)\, \phi(y),
\qquad t\ge 0,\; x\in E, \; \phi\in\calb_b(E),
$$
defines a semigroup of operators, denoted also by $P^t$, on
$\calb_b(E)$.
The transition function $(P^t)$ is called Feller if
$(P^t)$ is a strongly continuous semigroup on $C_0(E)$.
We can extend $(P^t)$ to $\calb_b(E_\partial)$ by setting
$$
P^t(\partial, \{\partial\})=1,
\quad
P^t(x, \Gamma)=P^t(x, \Gamma\cap E),
\quad {\rm for\;} x\in E \; {\rm and}\;
\Gamma\in \calb(E_\partial).
$$
The extended family is again a transition function. We have the
following result.

\begin{Theorem}\label{regmarkov}
  For arbitrary $x\in E$ there exists a
  c\`adl\`ag  Markov process on $E,$ with a Feller transition function
  $(P^t),$ starting from $x$.
\end{Theorem}

\begin{Remark}\begin{em}
  As we noticed earlier there exists always a Markov process on
  $E$ starting from $x\in E$ and with transition function $P^t$.
  To have a   c\`adl\`ag version we first prove its existence on
  the compact space $E_\partial$ rather than on a locally compact
  space.
  \end{em}
\end{Remark}

\begin{Remark}\begin{em}
  A transition semigroup $(P^t)$ is said to be {\em stochastically
  continuous} if for all $x\in E$ and $\delta>0$
  $$
  \lim_{t\to 0}P^t(x,B(x,\delta))=1.
  $$
If $P^t$ transforms $C_b(E)$ into $C_b(E)$ and is
stochastically continuous, then for all $f\in C_b(E)$ and $x\in E$
  $$
  \lim_{t\to 0}P^tf(x)=f(x).
  $$
Thus the strong Feller property is a slightly stronger condition than
stochastic continuity.
  \end{em}
\end{Remark}

\noindent {\bf Proof of the theorem.}
Let $A$ be the infinitesimal generator of the semigroup $(P^t)$,
regarded on $C_0(E)$, and $R_\lambda$, $\lambda>0$, its resolvent.
Then the domain $D(A)$ of $A$ is dense in $C_0(E)$ and therefore
also the image of $R_1$ is dense in $C_0(E)$. We can easily see
that there exists a sequence $(g_n)$ of non-negative functions
such that the functions $f_n=R_1g_n$, $n=1,2,\ldots$, separate
points in $E$. We extend the functions $f_n$ to $E_\partial$ by
setting $f_n(\partial)=0$. Now the transformation $H:E_\partial\to
\R^\infty$ given by
$$
H(x)=(f_1(x), f_2(x),\ldots )
$$
is continuous onto a compact set and therefore with a continuous
inverse (defined on the image). Consequently $x_n\to x$ in
$E_\partial$ if and only if $f_k(x_n)\to f_k(x)$ for each $k$, as
$n\to+\infty$. Let now $X(t)$, $t\ge0$, be a Markov process
on $E$ starting from $x$. Then for arbitrary $k=1,2,\ldots$ the
process $e^{-t}f_k(X(t))$, $t\ge0$, is a supermartingale, first
with respect to the natural filtration $\sigma (X(s)\,:\,s\le t)$,
$t\ge0$, and then also with respect to the extended versions.
This follows from the formula, valid for all $f\in\calb_b(E)$:
$$
\E\, (f(X(t))\;|\; \sigma (X(u)\,:\,u\le s))
=P^{t-s}f(X(s)),
\qquad s<t,
$$
which can be checked using the $\pi$-systems technique.
By a similar argument as for  L\'evy processes we infer that there
exists an $E_\partial$-valued c\`adl\`ag modification $Y$ of $X$.
To establish existence of an
$E$-valued c\`adl\`ag version we need the following lemma, left as
an exercise.

\begin{Lemma}\label{supermartpos}
  Assume that $Z$ is a c\`adl\`ag, non-negative supermartingale
  and define
  $$
  S(\omega)= \inf\{t\ge 0\;:\; Z(t, \omega )=0\;{\rm or \;} Z(t^-, \omega)=0\}.
  $$
Then
$$
\P(Z(s)=0 {\rm \; for \;all\;} s\ge S)=1.
$$
\end{Lemma}

By the lemma, for arbitrary $t\ge 0$, if either $Y(t^-)=\partial$
or $Y(t)=\partial$, then $Y(s)=\partial$ for $s\ge t$. Since, for
$x\in E$, the transition probabilities $P^t(x,\cdot)$ are
supported by $E$, we see that $\P( Y(t)\in E $ for all $t\geq
0)=1$.
\qed

\chapter{Wiener's approach}

The first rigorous proof of the existence of the Wiener process
was given in 1923 by N. Wiener \cite{bib7}. It was
based on Daniell's method \cite{bib8} of constructing measures on infinite dimensional spaces.
In 1932, N. Wiener together with R. E. Paley \cite{bib9}, gave an explicit construction
of the Wiener process using  Fourier series expansions and assuming only
existence of a sequence of independent, identically distributed
Gaussian random variables. We will present the essence of the method by
describing the so-called L\'evy-Ciesielski construction of the Wiener
process. We describe  also an explicit construction of the Poisson process. At the end of the chapter
direct construction of Steinhaus, of an arbitrary sequence of independent random variables, is given.

\section{Hilbert space expansions}
We start from some elementary facts about Hilbert spaces.
Let $H$ be a separable Hilbert space with scalar product
$\<h,g\>$ and the norm:
$$
\|h\|=\sqrt{\<h,h\>},\qquad h\in H.
$$
Let $h_1,h_2,\ldots$ be an orthonormal complete basis in $H$. Thus
$$
|h_j|=1,\quad \<h_j,h_k\>=0,\qquad  j=1,2,\ldots,\; j\ne k,
$$
and for arbitrary $x\in H$,
$$
  x=\sum_{j=1}^{+\infty} \<x,h_j\>h_j, \qquad x\in H.
$$
The following {\it Parseval identity}  easily follows:
$$
\<x, y\>=\sum_{j=1}^{+\infty} \<x,h_j\>\<y,h_j\>, \qquad x\in H.
$$
Take now $H=L^2(\Omega,\calf,\P)$ and assume that for each
$t\geq 0$, $X(t)\in H$.
If random variables $\xi_1,\xi_2,\ldots$ form in $H$ an
 orthonormal complete basis then we arrive at the {\em first
expansion}:
$$
X(t)=\sum_{j=1}^{+\infty} \xi_j\<X(t),\xi_j\>_{L^2(\Omega)}
=\sum_{j=1}^{+\infty} \xi_j
e_j(t), \qquad t\ge0,
$$
with the convergence, for each $t\geq 0,$ \,in the sense of the space $H$.

If one takes $H=L^2(0,T)$ and if $X(\cdot,\omega)\in H$
and $e_1,e_2,\ldots$ is an orthonormal complete basis  in $L^2(0,T)$,
then we get the {\em second expansion}:
$$
X(t,\omega)=\sum_{j=1}^{+\infty}\<X(\cdot,\omega),e_j\>_{L^2(0,T)}e_j(t)
=\sum_{j=1}^{+\infty}\xi_j(\omega)e_j(t), \qquad t\in[0,T].
$$
Here, for each $\omega \in \Omega $, the convergence is in the $L^2(0,T)$ sense.
Assume, for instance, that
$$
\E X(t)=0,\quad q(t,s)=\E (X(t)X(s)),\qquad t,s\in[0,T].
$$
Define
$$
Qx(t)=\int_0^T q(t,s)x(s)\,ds,\qquad t\in[0,T],\; x\in H= L^2(0,T),
$$
and let $e_j$ be an orthonormal complete basis in $H$, such that
$$
Qe_j=\lambda_je_j,\quad j=1,2,\ldots,\qquad  \lambda_j>0,
\quad j=1,2,\ldots .
$$
Then
$$
\xi_j=\int_0^T X(t)e_j(t)\,dt, \qquad j=1,2,\ldots ,
$$
and
$$
\E(\xi_j\xi_k)=\int_0^T \int_0^T q(t,s)e_j(t)e_k(s)\,dtds=
\lambda_j \delta_{j-k}.
$$
If, in particular, $X(t)=W(t)$, $t\in[0,1]$, is a Wiener process, then
$$
q(t,s)=\E(W(t)W(s))=t\wedge s,
$$
and
$$
e_j(t)=\sqrt2 \sin\Big[\Big(j+\frac12\Big)\pi t\Big],
\quad \lambda_j=\frac{1}{\pi^2\Big(j+\frac12\Big)^2},
\qquad j=0,1,2,\ldots .
$$
Therefore, for each $\omega \in \Omega $,
\begin{equation}\label{espwiener}
W(t, \omega)=\sqrt2 \sum_{j=0}^{+\infty}\xi_j(\omega)
\frac{\sin\Big[\Big(j+\frac12\Big)t\Big]}{\Big(j+\frac12\Big)\pi}
,\qquad t\in[0,1],
\end{equation}
and
$ \xi_0,\xi_1,\xi_2,\ldots $
are independent Gaussian random variables such that
 $\E \xi _j =0$, $\E\xi _j^2=1$.
This expression was found by Wiener.
The difficult thing however is to prove that the series
(\ref{espwiener}) defines a process with continuous paths.
This amounts to the proof that for almost all $\omega\in\Omega$
the series or its  subseries, converges uniformly. Instead of proving this we will
present a similar and an elegant construction due to P. L\'evy and Z. Ciesielski, see \cite{bib11} and \cite{bib12}.

\section{The L\'evy-Ciesielski construction of Wiener's process.}

In the L\'evy-Ciesielski construction, the essential role
is played by a Haar system connected with a dyadic partition
of the interval $[0,1]$.

Namely let $h_0\equiv 1$, and if $2^n\leq k<2^{n+1}$ then
$$
\begin{array}{l}
h_k(t)=\left\{
\begin{array}{lll}
2^{n/2} &{\rm if}&\dis
\frac{k-2^n}{2^n}\leq t < \frac{k-2^n}{2^n}+\frac{1}{2^{n+1}},
\\\\
-2^{n/2} &{\rm if}&\dis
\frac{k-2^n}{2^n}+\frac{1}{2^{n+1}}
\leq t < \frac{k-2^n}{2^n}+\frac{1}{2^{n}},
\end{array}
\right.
\\
h_k(1)=0,
\end{array}
$$
The system $h_k$, $k=0,1,\ldots$ forms an orthonormal and complete
basis in the space $L^2([0,1])$.

\begin{Theorem}
Let $(\xi_k)_{k=0,1,\ldots}$ be a sequence of independent random
variables normally distributed with mean $0$ and covariance $1$,
defined on a probability space $(\Omega, \calf, \P)$. Then,
for $\P$-almost all $\omega \in \Omega $, the series
$$
\sum_{k=0}^{+\infty}\xi_k(\omega)\int_0^th_k(s)\; ds
=W(t,\omega),\qquad t\in [0,1],
$$
is uniformly convergent on $[0,1]$ and defines a Wiener
process on $[0,1]$.
\end{Theorem}

\begin{Lemma}\label{levycieuno}
Let $\epsilon\in (0,1/2)$ and $M>0$. If $|a_k|\leq M\, k^\epsilon$
for $k=1,2,\ldots$ then the series
$\sum_{k=0}^{+\infty}a_k\int_0^th_k(s)\; ds$ is
uniformly convergent on the interval $[0,1]$.
\end{Lemma}

\noindent {\bf Proof.}
If $2^n\le k<2^{n+1}$ then the {\em Schauder functions}
$S_k(t)=\int_0^th_k(s)\; ds$, $t\in [0,1]$, are non-negative, have
disjoint supports and are bounded from above by
$2^{-(n+1)}2^\frac{n}{2}= 2^{-\frac{n}{2}-1}$. Let us denote
$b_n= \max(|a_k|\; :\; 2^n\le k<2^{n+1})$; then
$$
\sum_{2^n\le k<2^{n+1}} |a_k|\, S_k(t)\le
b_n 2^{-\frac{n}{2}-1}
$$
for all $t\in [0,1]$ and $n=0,1,\ldots$. Thus the condition
$$
\sum_{n=0}^{+\infty}b_n 2^{-\frac{n}{2}}<+\infty
$$
is sufficient for the uniform convergence of the series
$$
\sum_{n=0}^{+\infty}\sum_{2^n\le k<2^{n+1}} |a_k|\, S_k(t)
$$
and therefore for the uniform convergence of
$$
\sum_{k=0}^{+\infty}a_k\, S_k(t)
$$
too. From the inequalities $|a_k|\leq M\, k^\epsilon$
it follows that $b_n\le 2^\epsilon M 2^{n\epsilon}$ for all
$n=0,1,\ldots$ and, consequently,
$$
\sum_{n=0}^{+\infty}b_n 2^{-\frac{n}{2}}
\le 2^\epsilon M \sum_{n=0}^{+\infty} 2^{n\left(
-\frac{1}{2}+\epsilon\right)} <+\infty.
\qed
$$

\begin{Lemma}\label{levyciedue}
Let $(\xi_k)_{k=0,1,\ldots}$ be a sequence of normally
distributed random variables with mean $0$ and covariance
$1$. Then, with probability one, the sequence
$$
\left(\frac{|\xi_k|}{\sqrt{\log k}}\right)_{k=2,3,\ldots}
$$
is bounded.
\end{Lemma}

\noindent {\bf Proof.}
Let $c$ be a fixed positive number, then
$$
\P(|\xi_k|\ge c) =\frac{2}{\sqrt{2\pi}}\int_c^{+\infty}
e^{-x^2/2}\; dx \leq
\frac{2}{\sqrt{2\pi}}\int_c^{+\infty}\frac{x}{c}
e^{-x^2/2}\; dx \leq
\frac{2}{c\sqrt{2\pi}}e^{-c^2/2}.
$$
>From this we obtain that for $c>\sqrt{2}$
$$
\sum_{k=2}^{+\infty} \P(|\xi_k|\ge c\sqrt{\log k})
\le
 \frac{2}{\sqrt{2\pi}} \sum_{k=2}^{+\infty}
 \frac{k^{-c^2/2}}{c\sqrt{\log k}}<+\infty.
 $$
Therefore, if $c>\sqrt{2}$ then, with probability one, only for a
finite number of $k,$ we have $|\xi_k|\ge c\sqrt{\log k}$.
\qed

\noindent {\bf Proof of the theorem.}
Lemma \ref{levycieuno} and Lemma \ref{levyciedue} imply that
the series
$$
W(t,\omega):=
\sum_{k=0}^{+\infty} \xi_k(\omega)\, S_k(t)
$$
is for almost all $\omega$ uniformly convergent on the interval
$[0,1]$. Since the functions $S_k$ are continuous and $S_k(0)=0$
for $k=0,1,\ldots$ the constructed process starts from $0$ and has
continuous paths. As the limit of Gaussian processes, $W$ is also
a Gaussian process. Moreover
$$
\E\,(W(t)W(s))=\sum_{k=0}^{+\infty} \int_0^th_k(u)\; du
 \int_0^sh _k(u)\; du
=\sum_{k=0}^{+\infty} \<h_k,\phi_t\>_{L^2(0,1)}
\<h_k,\phi_s\>_{L^2(0,1)},
$$
where $\phi_t=1_{[0,t]}$. Therefore by Parseval's identity
$$\E\,(W(t)W(s))=\<\phi_t,\phi_s\>_{L^2(0,1)}=t\wedge s.
\qed
$$
\subsection{Construction of a  Poisson process}
A non-negative random variable $\xi$ has exponential distribution with parameter $\alpha > 0$ if
$$
\P(\xi > t )=  e^{- \alpha t},\,\,\,\, t\geq 0.
$$
The distribution of $\xi$ has a density $g(t) = \alpha e^{- \alpha t},\,\,t\geq 0.$
\vspace{2mm}

\noindent We have the following result.

\begin{Proposition}
Assume that $(\xi_k)$ is a sequence of independent, exponentially distributed random variables with parameter $\alpha$.
If
$$
  \begin{array}{llll}\label{*}
    \Pi(t) & =  0& {\rm if} & \xi_1>t \\
      & = k      & {\rm if }& \xi_1+\ldots+\xi_k \leq t<
       \xi_1+\ldots+ \xi_{k+1},
  \end{array}
$$
then $\Pi$ is a Poisson process with parameter $\alpha$.
\end{Proposition}
\noindent{\bf Proof.}It is enough to  show that
\begin{enumerate}
  \item[(i)] $\P(\Pi(t) =k) =
   e^{-\alpha t}\ds\frac{(\alpha t)^k}{k!}\eds$.

  \item[(ii)] For arbitrary $k=1,2,\ldots$ and
  $t_0=0<t_1<\ldots <t_k$:
$$
\begin{array}{lll}
\P\left(\Pi(t_i)-\Pi(t_{i-1})=k_i,\, i=1,\ldots,nk\right)
&=&\ds \prod
_{i=1}^{k} \P\left(\Pi(t_i)-\Pi(t_{i-1})=k_i\right)
\eds\\
&=&\ds \prod
_{i=1}^{k} \P\left(\Pi(t_i-t_{i-1})=k_i\right).
\eds
\end{array}
$$
\end{enumerate}
\vspace{2mm}

\noindent Let us recall that if $\xi $ and $\eta$ are independent random variables in a
normed space $E$, say $E=\R^1$ or $E=\R^d$, with the
laws $\mu=\call(\xi)$, $\nu=\call(\eta)$ then their sum
$\zeta = \xi + \eta $ has the law $\sigma=\call(\xi + \eta )$ equal to the
convolution of $\mu$ and $\nu$:
$$
\sigma(\Gamma)=\int_E\mu(\Gamma -y)\, \nu(dy),
\qquad \Gamma\in \calb(E).
$$
In fact if  $\xi$ and $\eta $ are
independent then
$$
\P(\xi + \eta \in\Gamma)=\E (\chi_\Gamma(\xi + \eta))=
\int_E\chi_\Gamma(x+y)\mu(dx)\nu(dy)=
\int_E\left[\int_E\chi_\Gamma(x+y)\mu(dx)\right]\nu(dy).
\footnote{
We used the notation $\chi_\Gamma(z)=1$ if
$z\in\Gamma$, $\chi_\Gamma(z)=0$ if $z\notin \Gamma$.
 }
$$
But $\int_E\chi_\Gamma(x+y)\mu(dx)= \mu(\Gamma -y)$,
$y\in E$, and the required identity holds.

\noindent If random variables  $\xi_1,\ldots,\xi_k$ are independent,
exponentially  distributed with parameter $\alpha$
then, by an easy inductive argument, the distribution
of $\xi_1+\ldots+ \xi_k$ is a measure on $\R^1$ with the
following density $g_k$:
$$
g_k(x)=\alpha\frac{(\alpha
x)^{k-1}}{(k-1)!}e^{-\alpha x},\qquad
x>0,k=1,2,\ldots.
$$
Note that for $k=1,2,\ldots$
$$
  \begin{array}{l}
  \ds
    \P(\Pi(t)=k)= \P\left(X_1+\ldots+X_k \leq t<
      X_1+\ldots+X_{k+1}\right)
      \eds \\
    \ds = \mathop{\int\int}_{
    \begin{array}{c}\scriptstyle{0<y\leq t<y+z}
    \\
    \scriptstyle{0<z}
    \end{array}
    }
    g_k(y)\left(\alpha e^{-\alpha z}\right) dydz=
\frac{\alpha^{k+1}}{(k-1)!}
    \mathop{\int\int}_{
    \begin{array}{c}\scriptstyle{0<y\leq t }
    \\
    \scriptstyle{t-y<z}
    \end{array}
    }
    y^{k-1}e^{-\alpha y}  e^{-\alpha z}  dydz
    \eds
    \\
    \ds
= \frac{\alpha^{k}}{(k-1)!}
\mathop{\int}_{0<y\leq t}
y^{k-1}e^{-\alpha y}  e^{-\alpha (t-y)}  dy
= \frac{\alpha^{k}}{(k-1)!}
 e^{-\alpha t}\int_0^t y^{k-1} dy
= \frac{(\alpha t)^{k}}{k!} e^{-\alpha t}.
\eds
  \end{array}
$$
This proves $(i)$. The identity $(ii)$ is equivalent to the independence of the random variables
$\Pi(t_i)-  \Pi(t_{i- 1}),\,\,i= 1, \ldots, k$ and will  follow from some elementary properties of
exponentially distributed random variables which will be developed now.
\vspace{2mm}

\noindent Let $\xi_1,\xi_2,\ldots$ be a sequence of independent random
variables such that $\P (\xi_m=1)=p$,  $\P (\xi_m=0)=1-p$,
$m=1,2\ldots$. Let $T=\min \{m\,  ;\,  \xi_m=1\}$. Then
$\P (T=k)=(1-p)^{k-1}p$, $k=1,2\ldots$.
Note that the random variable $T$ is the {\it waiting time for the first  success}  in the sequence $\xi_1,\xi_2,\ldots$ .
\begin{Proposition}\label{exponential1}
Denote by $\mu_n$ the distribution  of the random variable $T/n$ where
$p=\alpha /n$,  $n=1,2\ldots$ and $\alpha$ is a positive
constant. Then $(\mu_n)$ converges weakly to the exponential
distribution with parameter $\alpha$.
\end{Proposition}

\noindent {\bf Proof.} The characteristic
function of the distribution of $T$ is
$$
\E\left( e^{i\lambda T}\right) =
\sum_{k=1}^\infty  e^{i\lambda k} (1-p)^{k-1}p
=
pe^{i\lambda }\frac{1}{1-(1-p)e^{i\lambda }},
\qquad \lambda\in\R^1.
$$
Note that
$$\ds
\widehat{\mu}_n (\lambda) =\frac{\alpha}{n}
e^{i\frac{\lambda}{n}}
\frac{1}{1-(1-\frac{\alpha}{n})e^{i\frac{\lambda}{n} }},
\qquad \lambda\in\R^1
\eds
$$
and $\widehat{\mu}_n (\lambda) \to \ds\frac{\alpha}{
\alpha -i\lambda}\eds$,  $\lambda\in\R^1$, and the
characteristic function of the exponential distribution with
parameter $\alpha >0$ is equal to
$$
\alpha\int_0^\infty
e^{i\lambda x } e^{-\alpha  x}dx =
\frac{\alpha}{\alpha -i\lambda},
\qquad  \lambda\in\R^1,
$$
so the result follows. \qed
\vspace{2mm}

\noindent To continue the proof of ii), let, for each natural $n$,\,\,\,\,  $\xi_1^n,\xi_2^n ,\ldots$
be a sequence  of independent random
variables such that
$$
\P (\xi_m^n=1)= \alpha/n\,,  \,\P (\xi_m^n=0)=1-\alpha/n, m=1,2\ldots,
$$
and let
$\pi^n(m)$ be the number of successes in the sequence $\xi_1^n,\xi_2^n,\ldots, \xi_m^n$.
 Define $m_l^n = [nt_l],$ where $[s]$ denotes the integer part of $s$.
By the very definition, for each $n$, the random
variables $\pi^{n}(m_1),\,\pi^{n}(m_2) -\pi^{n}(m_1),\ldots, \pi^{n}(m_{k}) -\pi^{n}(m_{k-1})$, are independent. By a
straightforward generalisation of Proposition \ref{exponential1} the laws of
$$
\pi^{n}(m_1),\,\pi^{n}(m_2) -\pi^{n}(m_1), \ldots, \pi^{n}(m_{k}) -\pi^{n}(m_{k-1})
$$
converge weakly, as $n$ tends to $+\infty$, to the law of
$$
\Pi(t_1),\,\Pi(t_2) -\Pi(t_1),\ldots, \Pi(t_{k}) -\Pi(t_{k-1}),
$$
and the required independence follows. \qed
\vspace{2mm}

\noindent For a more direct proof of ii) see e.g. \cite{billingsley}.
\vspace{2mm}

\noindent The following proposition gives another intuitive characterisation of
exponential random variables.
\begin{Proposition}
Assume that $\xi$ is a positive random variable such that
for all $t,s>0$:
$$
\P(\xi>t+s|\xi>t)=\P(\xi>s).
$$
Then there exists $\alpha >0$ such that
\begin{equation}\label{exponential}
\P(\xi >t)=e^{-\alpha t}.
\end{equation}
\end{Proposition}

\noindent {\bf Proof.} Let $G(s)=\P(alpha >s)$ for $s>0$. Then $G$ satisfies the
functional equation: $G(t+s)=G(s)G(t)$
for all $t,s>0$. The function $G$ is right continuous
and positive, so the functional equation has a unique
solution of the required form. \qed

\section{The Steinhaus construction}

The construction of the Wiener and Poisson processes, presented in the previous section, took for granted existence of
sequence of independent random variables with prescribed  distributions. Existence of such sequences  follows from
the Kolmogorov existence result.  In this section we show
that such sequences can be directly  constructed if the Lebesgue
measure on $[0,1)$ is taken for granted. The construction goes
back to H. Steinhaus \cite{bib10}.

\begin{Theorem}\label{steinh}
  Let $\mu_1, \mu_2,\ldots$ be a sequence of probability measures
on $\R^1$. There exists a sequence $(\xi_n)$ of independent
real-valued random variables defined on $([0,1), \calb ([0,1)),
\P)$, where $\P$ is the Lebesgue measure on $[0,1)$, such that
the distributions of $\xi_n$ are exactly the measures $\mu_n$,
$n=1,2,\ldots$.
\end{Theorem}

\noindent {\bf Proof.}
An arbitrary number $\omega\in [0,1)$, with the exception of a
countable set of numbers of the form $k/2^m$, $k,m=0,1,\ldots$,
can be uniquely represented in the form
$$
\omega= \sum_{n=1}^{+\infty} \frac{\epsilon_n}{2^n},
\qquad {\rm where}\qquad \epsilon_n=0\; {\rm or\;} 1,
\quad n=1,2,\ldots.
$$
The sequence $(\epsilon_1, \epsilon_2,\ldots)$ is called the
dyadic expansion of $\omega$. Define
$$
X_n(\omega)=\epsilon_n,\qquad \omega\in [0,1),\; n=1,2,\ldots.
$$
It is easy to see that if $\epsilon_i=0$ or $1$, $i=1,2,\ldots,
n$,
$n=1,2,\ldots$,
$$
\{\omega\in [0,1)\; :\;
X_1(\omega)=\epsilon_1,\ldots,
X_n(\omega)=\epsilon_n\} =
\left[ \frac{\epsilon_1}{2^1}+\ldots +\frac{\epsilon_n}{2^n},
\frac{\epsilon_1}{2^1}+\ldots +\frac{\epsilon_n}{2^n}
+\frac{1}{2^n}\right).
$$
This implies that
$$\P(X_1=\epsilon_1,\ldots, X_n=\epsilon_n)=
\frac{1}{2^n}=\prod_{i=1}^n \P(X_i=\epsilon_i)
$$
and therefore the random variables $(X_n)$ are independent.

Let $J_i=\{n_{i,j}\; : \; j=1,2,\ldots\}$,
$i=1,2,\ldots$ be disjoint subsets of the set of natural numbers.
Then the random variables
$$
Z_i=\sum_{j=1}^{+\infty} \frac{X_{n_{i,j}}}{2^j},
\qquad i=1,2,\ldots
$$
are independent. We show that they have uniform distribution on
$[0,1)$. Let, for instance, $i=1$ and define
$$
S_n=\sum_{j=1}^{n} \frac{X_{n_{1,j}}}{2^j}.
$$
Then $\P(S_n=k/2^n)= 1/2^n$, $k=0,1,2,\ldots, 2^n-1$, and
therefore, for $t\in [0,1)$, $\P(S_n\leq t)\to t$. On the other
hand, $\P(S_n\leq t)\to \P(Z_1\leq t)$. Thus $\P(Z_1\leq t)=t.$

Now let $\mu$ be a probability measure on $\R^1$ and let $F=F_\mu$
be its distribution function. Define the inverse $F^{-1}$ of $F$ by
$$
F^{-1}(s)=\inf \{ t\; :\; s\leq
F(t)\}, s\in [0,1).
$$
If $Z$  has uniform distribution on $[0,1)$
then the distribution of $F^{-1}(Z)$ is exactly $\mu$. Indeed, from
the definition of $F^{-1}$, for $s\in [0,1)$ and $t\in
(-\infty,+\infty)$, $s\leq F(F^{-1}(s))$ and $F^{-1}(F(t))\leq t$.
Therefore, $\{s\; :\; F^{-1}(s)\leq t\}=[0,F(t)]$ and, consequently,
$\P(\omega\; :\;F^{-1}(Z(\omega))\leq t\}=
\P(\omega\; :\;Z(\omega)\leq F(t)\}= F(t)$.

To finish  the proof of the theorem, it is sufficient to remark
that if $F^{-1}_1, F^{-1}_2,\ldots$ are functions, defined as above,
corresponding to the measures  $\mu_1, \mu_2,\ldots$  and
$Z_1, Z_2,\ldots$ are real-valued random variables uniformly
distributed in $[0,1)$ then the sequence
$\xi_1 = F^{-1}_1(Z_1), \xi_{2} = F^{-1}_2(Z_2)$, $\ldots$ has all the properties required.
\qed

\chapter{Analytic approach to  L\'evy processes}
For  Le\'vy processes
 a complete characterisation of finite dimensional
distributions is possible. It turns out that L\'evy processes are far  reaching generalisation
of the Poisson process. The main result of the chapter is the so called
L\'evy-Khinchin formula. We
derive also some basic properties of the corresponding transition semigroups both in $UC_b(\R^d)$ and in $L^{p}(\R^d)$.
Subordination procedure is discussed as well.
\section{Infinitely divisible families}
We recall  that a family $\mu_t$, $t\geq 0$ of probability measures on $\R^d$ is
infinitely divisible if

1) $\mu_0=\delta_{\{0\}}$,

2) $\mu_{t+s}=\mu_t*\mu_s$ for all $t,s\geq 0$,

3) $\mu_t\Rightarrow \delta_{\{0\}}$ as $t\to 0$.

\noindent The building blocks of the infinite divisible families are the {\it compound Poisson}, the {\it shift}
and the {\it Gaussian}
 families, which we have already met in a slightly different context.
\begin{Proposition} Let $\nu$ be a probability measure on   $\R^d$ and
$\lambda $ a positive constant. Then
$$
\mu_t=e^{-\lambda t} \sum
_{n=0}^{\infty} \frac{(\lambda t)^n}{n!}
\nu^{*n},\; t>0,
\qquad
\mu_0=\delta_{\{0\}},
$$
is an infinitely divisible family, called a compound Poisson family with parameters $\lambda, \nu$.
We  are using the convention $\nu^{*0}=\delta_{\{0\}}$.
\end{Proposition}
\noindent{\bf Proof.}
It is clear that $\mu_t(\R^d)=1$, $t\geq 0$. .
$$
  \begin{array}{lll}
    \mu_t*\mu_s & =&\ds
e^{-\lambda (t+s)} \sum_{n,m}
\frac{(\lambda t)^m}{m!}\frac{(\lambda s)^n}{n!}
\nu^{*(n+m)}
     \eds\\
      & = &\ds
e^{-\lambda (t+s)} \sum_{k=0}^\infty
\nu^{*k}\sum_{n+m=k}
\frac{(\lambda t)^m(\lambda s)^n}{m!\, n!}
\eds\\
      & = &\ds
e^{-\lambda (t+s)} \sum_{k=0}^\infty
\nu^{*k}
\frac{1}{k!} (\lambda t+\lambda s)^k
\eds\\
      &=&\mu_{t+s}.\; \qed
  \end{array}
$$
To check 3) fix a bounded continuous function $\phi$. Then
$$
\int_{\R^d}\phi (x) \mu_{t}(dx) = e^{- \lambda t} \phi (0) + e^{- \lambda t}\sum_{m=1}^
{+\infty}\frac{(\lambda t)^m}{m!}\int_{\R^d}\phi (x) \nu^{*m}(dx).
$$
As $t$ tends to $0$ the sum in the above expression tends to $0$ as well and the required property follows. \qed

\noindent If $(\mu_t)$ are compound Poisson distributions the
corresponding process $X(t)$, $t\geq 0$, is called a
{\it compound Poisson process}. It can be constructed as follows.
\vspace{2mm}

\noindent Let $Z_1,Z_2,\ldots$ be a sequence of independent
random variables with identical laws equal to the
measure $\nu$. Let, in addition, $\Pi(t)$, $t\geq 0$,
be a Poisson process, with parameter $\lambda$,
independent of $Z_1,Z_2,\ldots$. Define
$$
  \begin{array}{llll}
    X(t) & =  0& {\rm if} & \Pi(t)=0 \\
      & = Z_1+\ldots +Z_k  & {\rm if }& \Pi(t)=k.
  \end{array}
$$
Then $X(t)$, $t\geq 0$, is a compound Poisson process
with parameters $\lambda,\nu$. For instance if $\Gamma$
is a Borel subset of $\R^d$ then
$$
  \begin{array}{l}
    \ds
\P(X(t)\in\Gamma )=
\P\left(X(t)\in\Gamma {\rm \; and \;} \Pi(t)=0\right)
+\sum
_{k=1}^{\infty}
\P\left(X(t)\in\Gamma {\rm \; and \;} \Pi(t)=k\right)
    \eds \\
    \ds
= e^{-\lambda t}\delta_{\{0\}}(\Gamma)
+\sum
_{k=1}^{\infty}
\P\left(Z_1+\ldots +Z_k\in\Gamma {\rm \; and \;} \Pi(t)=k\right)
    \eds \\
    \ds
    = e^{-\lambda t}\delta_{\{0\}}(\Gamma)
+\sum
_{k=1}^{\infty}
\P\left(Z_1+\ldots +Z_k\in\Gamma \right)
\P( \Pi(t)=k)
    \eds \\
    \ds
    = e^{-\lambda t}\delta_{\{0\}}(\Gamma)
+\sum
_{k=1}^{\infty}
\nu^{*k}(\Gamma )e^{-\lambda t}
\frac{(\lambda t)^k}{k!}=e^{-\lambda t}
\sum
_{k=0}^{\infty}
\frac{(\lambda t)^k}{k!}\nu^{*k}(\Gamma ). \quad \qed
    \eds
  \end{array}
$$
\vspace{2mm} Fix $a\in\R^d$ and define
$$
\mu_t=\delta_{\{ta\}}\qquad t\geq 0.
$$
Then obviously $(\mu_t)$ is an infinitely divisible family and is called a {\it shift} family.
\vspace{2mm}

\noindent {\it Gaussian infinitely divisible family} on $\R^d$ is parameterised by a vector $a\in \R^d$ and
a symmetric, non-negative, $d\times d$ matrix $Q$ and given by the formula:
$$
\mu_t = N (ta, tQ),\,\,\,t\geq 0.
$$
In particular, if $Q$ is non-degenerate, then $\mu_t$ has a density and
$$
\mu_t(dx)=\frac{1}{\sqrt{(2\pi t )^d detQ}}\,
e^{-\frac{1}{2t}<Q^{-1}(x-a),x-a)>}dx,\qquad t\geq 0.
$$
\vspace{2mm}

\section{The L\'evy-Khinchin formula}

Let $\mu_t$, $t\geq 0$, be the family of compound
Poisson distributions with the intensity parameters $\lambda >0$
and the jump measure $\nu$.
Then it is easy to calculate the
characteristic function of $\mu_t$, $t\geq 0$:
$$
  \begin{array}{l}\ds
    \widehat{\mu}_t(\xi)=\int_{\R^d}e^{i\langle \xi,x\rangle}
    \mu_t(dx) =  e^{-\lambda t}\sum_{n=0}^\infty
    \frac{(\lambda t)^n}{n!}\left( \widehat{\nu}(\xi)\right)^n
    \eds \\\ds
    = e^{-\lambda t}e^{\lambda t\widehat{\nu}(\xi)}
    =e^{-  t\psi (\xi)},\qquad t\geq 0,\xi\in\R^d,
    \eds
  \end{array}
$$
where
$$
\psi(\xi)= \lambda\int_{\R^d}
\left(1-e^{i\langle \xi,y\rangle}\right)\, \nu(dy).
$$
The function $\psi$ is called the exponent of the
family $(\mu_t)$. For the shift family
$\mu_t=\delta_{\{ta\}}$, we have
$$
\widehat{\mu}_t(\xi) =\int_{\R^d}
e^{i\langle \xi,x\rangle}\delta_{\{ta\}}(dx)
=e^{-  t\psi (\xi)}
$$
where $  \psi (\xi)=-i\langle a,\xi\rangle$. For the
Gaussian family $\mu_t=\caln (0,tQ)$,
$\widehat{\mu}_t(\xi)
=e^{-  t\frac{1}{2}\langle Q\xi ,\xi\rangle}$,
so the exponent $\psi$ is of the form
$$
\psi(\xi)=\frac{1}{2}\langle Q\xi ,\xi\rangle,
\qquad \xi\in\R^d.
$$
The following general characterisation, due to L\'evy
and Khinchin,  takes place.

\begin{Theorem}
For arbitrary infinitely divisible family $(\mu_t)$ on
$\R^d$ there exist a vector $a\in\R^d$, a symmetric matrix  $Q\geq 0$
and a nonnegative measure $\nu$ concentrated on
$\R^d\backslash \{ 0\}$  satisfying
\begin{equation}\label{lk1}
\int_{\R^d}\left(|y|^2\wedge 1\right)\, \nu(dy)<\infty,
\end{equation}
such that
\begin{equation}\label{lk2}
\widehat{\mu}_t(\xi)=e^{-t\psi(\xi)},
\end{equation}
where
\begin{equation}\label{lk3}
\psi(\xi)=i\langle a,\xi\rangle
+\frac{1}{2}\langle Q\xi ,\xi\rangle +
\int_{\R^d}\left(1-e^{i\langle \xi,y\rangle}+
\frac{i\langle \xi,y\rangle}{1+|y|^2}\right)\,
\nu(dy).
\end{equation}
Conversely for given $a\in\R^d$, $Q\geq 0$   and a
 measure $\nu$  satisfying (\ref{lk1}), formulae
(\ref{lk2})-(\ref{lk3}) define characteristic functions
of an infinitely divisible family $(\mu_t)$.
\end{Theorem}
Before going to the proof of the theorem we discuss some examples.

\noindent {\bf Examples.}

\noindent {\bf 1)} If $\nu$ is a finite measure
and $\lambda =\nu(\R^d)$ then
$$\psi(\xi)=i\langle a,\xi\rangle
+\frac{1}{2}\langle Q\xi ,\xi\rangle +
\lambda \int_{\R^d}\left(1-e^{i\langle \xi,y\rangle}
\right)\frac{1}{\lambda}\,
\nu(dy)
-i\langle
 \int_{\R^d}
\frac{y}{1+|y|^2}\,
\nu(dy),\xi\rangle,
$$
so $\psi$ is a sum of the exponents introduced at the
beginning.

\noindent {\bf 2)} Assume that $a=0$, $Q=0$ and the measure $\nu$ on $\R^d$ is of the form:
\begin{equation}\label{frazionario}
\nu(dx) =\frac{c}{|x|^{d+\alpha}}dx,
\end{equation}
where $c>0$, $\alpha\in (0,2)$. Then the conditions of
the L\'evy-Khinchin theorem are satisfied and we have the following result.

\begin{Proposition}\label{frazion}
If $a=0$, $Q=0$ and   $\nu$ is given by
(\ref{frazionario}) then
$$
\psi (\xi)=c_1|\xi|^\alpha,\qquad\xi\in\R^d,
$$
where $c_1$ is a positive constant.
\end{Proposition}

\noindent {\bf Proof.}
In the present case
$$
\psi (\xi)=\int_{\R^d}
\left(1-\cos \langle \xi, x\rangle\right)
\frac{c}{|x|^{d+\alpha}}dx.
$$
It is clear that $\psi$ is invariant under rotation
around $0$. Moreover if $\sigma >0$ then
$$
\psi (\sigma\xi)=\int_{\R^d}
\left(1-\cos \langle \xi, \sigma x\rangle\right)
\frac{c}{|x|^{d+\alpha}}dx.
$$
Changing coordinates $y=\sigma x$ one gets the
result.\qed

\noindent The infinitely divisible families described in the
proposition are called $\alpha$-{\it stable}, rotationally invariant families.

\noindent{\bf Proof of the theorem.}
We show first that (\ref{lk1})-(\ref{lk3}) define
infinitely divisible laws. For $r>0$ define
$$
\psi_r(\xi)=
\int_{|y|\geq r}\left(1-e^{i\langle \xi,y\rangle}+
\frac{i\langle \xi,y\rangle}{1+|y|^2}\right)\,
\nu(dy),\qquad \xi\in\R^d,
$$
then $\psi_r$ is the exponent of an infinitely
divisible family. Note that if $|\xi|,|y|\leq R$ then
for some $c>0$:
$$
\left| 1-(\cos \langle\xi,y\rangle)\right|
\leq c |\langle\xi,y\rangle|^2\leq c|y|^2,
$$

$$
\left|  (\sin \langle\xi,y\rangle)
-\frac{\langle\xi,y\rangle}{1+|y|^2}
\right|
\leq c |\sin \langle\xi,y\rangle-\langle\xi,y\rangle|
+\frac{|\langle\xi,y\rangle |\; |y^2|}{1+|y|^2}
\leq c |y|^2.
$$
Consequently
$$
\int_{|y|\leq r}\left|1-e^{i\langle \xi,y\rangle}+
\frac{i\langle \xi,y\rangle}{1+|y|^2}\right|
\frac{1}{|y|^2}|y|^2\,
\nu(dy)\to 0,
$$
as $r\to 0$, uniformly for $\xi$ in a bounded set. Thus
$\psi_r$ converge uniformly on bounded sets to a
continuous function $\psi$. This shows that the
functions defined by (\ref{lk2}) are characteristic
functions of some measures. Since $\widehat{\mu}_t\cdot
\widehat{\mu}_s=
\widehat{\mu}_{t+s}$, and $\widehat{\mu}_t\to 1$ as
$t\downarrow 0$, the family $(\mu_t)$ determined by
(\ref{lk1})-(\ref{lk3}) is infinitely divisible.

\noindent To prove that any infinitely divisible family is of the
form (\ref{lk1})-(\ref{lk3}) we will use a classical
result in the theory of $C_0$-semigroups $(P_t)$ on a
Banach space $X$.

\begin{Proposition}\label{formesponenzclass}
If $(P_t)$ is a $C_0$-semigroup  on   $X$ then for some
$\omega >0$ and $M>0$, $|P_t|\leq e^{\omega t}M$.
Moreover if
$$
R_\lambda = \int_0^{+\infty} e^{-\lambda t}P_tdt,
\qquad \lambda >\omega,
$$
then
$$
P_tf =\lim_{\lambda \to +\infty} e^{tA_\lambda }f,
$$
where
$$
A_\lambda =\lambda \left(\lambda R_\lambda -I\right).
\qquad \lambda >\omega,
$$
\end{Proposition}

\noindent Let now $(\mu_t)$ be an infinitely divisible family.
Define
$$
P_tf(x)=\int_{\R^d} f(x+y)\, \mu_t(dy),
\qquad f\in  UC_b (\R^d),
$$
where $UC_b (\R^d)$ denotes the space of all bounded, uniformly
continuous functions on $R^{d}$, equipped with the supremum norm.
Then, by Proposition \ref{semigroup}, $(P_t)$ is a $C_0$-semigroup on $UC_b (\R^d)$
and for $\lambda >0$
$$
\lambda R_\lambda f(x) =
\lambda \int_0^{+\infty}e^{-\lambda t}\left[
\int_{\R^d}f(x+y)\, \mu_t(dy)\right] \, dt
=\int_{\R^d}f(x+y)\, \eta _\lambda (dy),
$$
where $\eta _\lambda$ is the probability measure
$\eta_\lambda = \lambda \int_0^{+\infty}e^{-\lambda
t}\mu_t\, dt$. (The question of measurability of $t\to \mu_t(\Gamma)$ arises
here. Note that function   $t\to \int f\mu_t(dy)$ is
continuous for $f\in C_b(\R^d)$. Taking a sequence
$f_n\downarrow \chi_\Gamma$  as in a previous proposition we
conclude that measurability of $t\to \mu_t(\Gamma)$ holds
for $\Gamma$ closed. By Dynkin's $\pi -\lambda$ lemma it
holds for $\Gamma$ Borel.)

\noindent The
semigroup
$$
P_t^\lambda = e^{tA_\lambda}
$$

$$
=e^{-\lambda t}e^{t\lambda^2 R_\lambda}
=e^{-\lambda t}\sum_{n=0}^\infty
\frac{(t\lambda)^n}{n!}(\lambda R_\lambda)^n.
$$
 corresponds to the compound Poisson family with the
intensity $\lambda >0$ and the jump measure
$\eta_\lambda$. It follows from the proposition that
these compound Poisson families converge, for each
$t>0$,  weakly to $\mu_t$. Strictly speaking the proposition gives convergence
for every function
belonging to $UC_b (\R^d)$ but the generalization to all functions from
$C_b (\R^d)$ is
immediate. More precisely we have that

\begin{equation}\label{lk4}
\widehat{\mu}_t(\xi) =\lim_{\lambda \to +\infty}
e^{-t\lambda \int_{\R^d} \left( 1- e^{i\langle
\xi,y\rangle}\right)\,\eta_\lambda (dy) }.
\end{equation}
Set
\begin{equation}\label{lk5}
\psi_\lambda (\xi) =
\int_{\R^d} \left( 1- e^{i\langle
\xi,y\rangle}\right)\,\nu_\lambda (dy),
\end{equation}
where
$$ \nu_\lambda =\lambda\eta_\lambda$$
is
a measure of total mass $\lambda$.

\noindent We show that (\ref{lk4}) implies that
\begin{equation}\label{lk6}
\psi_\lambda (\xi)\to \psi(\xi),\qquad
{\rm as}\; \lambda\to +\infty,
\end{equation}
where $\psi$ is a continuous function. To see this note
that $\widehat{\mu}_t\to 1$ uniformly on bounded sets
as $t\to 0$. Therefore, for arbitrary $r>0$, there
exists $t>0$ such that if $|\xi|\leq r$ then,
$$
\left| \widehat{\mu}_t(\xi)-1\right|
\leq \frac{1}{4}.
$$
Let $(\mu_t^\lambda)$ be the family corresponding to
$\lambda$. Since
$\widehat{\mu}_t^\lambda\to\widehat{\mu}_t$, uniformly
on bounded sets as $\lambda\to +\infty$,  therefore for
some $\lambda_0>0$ and $\lambda >\lambda_0$
$$
\left| e^{-t\psi_\lambda(\xi)}-1\right|
\leq \frac{1}{2}.
$$
Consequently
$$
-t \, {\rm Re\;}\psi_\lambda(\xi)=\ln
\left|\widehat{\mu}_t^\lambda(\xi)\right|,
$$
$$
-t\,  {\rm Im\;}\psi_\lambda(\xi)=
{\rm Arg\;}\widehat{\mu}_t^\lambda(\xi) ,
$$
where ${\rm Arg\;}z$ for $z$ such that
$|z-1|<\frac{1}{2}$ is a well defined continuous
function. This shows existence and continuity of the
limit $\psi$ of $\psi_\lambda$. Thus
$$
\widehat{\mu}_t=e^{-t\psi}
$$
and we have to show that $\psi$ is of the form
(\ref{lk1})-(\ref{lk3}).

\noindent Let us start with some heuristic considerations. In
view of (\ref{lk5})-(\ref{lk6}) it seems natural to
investigate convergence of $\nu_\lambda$.
We can not expect convergence for $\lambda\to +\infty$,
since $\nu_\lambda (\R^d)=\lambda$, so the
total mass diverges to $\infty$. However, since
$\nu_\lambda =
\int_0^{+\infty}\lambda^2 e^{-\lambda t}\mu_t\, dt$
and $\mu_0=\delta_{\{0\}}$, we see that
$\nu_\lambda $ accumulates mass near $0$.
So it is reasonable to expect convergence of a measure
of the form $K(y)\nu_\lambda (dy)$, where
$K$ is a function that tends to $0$ as $|y|\to 0$. A
reasonable attempt would be $K(y)=1\wedge |y|^2$
(compare (\ref{lk1})). For technical reasons, we will
take a smoothed version of this function.

Now we come back to precise arguments.
 Let us fix $h>0$ and define a function $K$:
$$
K(y)= \frac{1}{(2h)^d}\int_{[-h,h]^d}
\left( 1-\cos \langle \sigma ,y\rangle\right)\,
d\sigma =
\left(1-\prod_{k=1}^d\frac{\sin hy_k}{hy_k}\right),
\qquad y=(y_1,\ldots ,y_d)\in\R^d.
$$
Then $K(y)>0$ for $y\neq 0$, $\lim_{y\to 0}
\frac{1}{|y|^2}K(y)=\frac{h^2}{3}$,
$\lim_{|y|\to +\infty} K(y)=1$. Denote
$$
\widetilde{\nu}_\lambda(dy)
= K(y)\nu_\lambda(dy), \qquad \lambda >0.
$$
We will show that
$$\widetilde{\nu}_\lambda
\Rightarrow
\widetilde{\nu}
$$
where $\widetilde{\nu}$
is a finite measure on $\R^d$. To this end let us
compute the Fourier transform $\widetilde{\psi}_{\lambda}$ of
$\widetilde{\nu}_\lambda$. For $\xi\in\R^d$,
$$
  \begin{array}{l}
    \ds
    \int_{\R^d}e^{i\langle \xi,y\rangle} \,
    \widetilde{\nu}_\lambda(dy)
    =
\int_{\R^d}e^{i\langle \xi,y\rangle}K(y)\,
\nu_\lambda(dy)
\eds
     \\
    \ds
= \int_{\R^d}e^{i\langle \xi,y\rangle}
\left(1-\frac{1}{(2h)^d}\int_{[-h,h]^d}
e^{i \langle \sigma,y\rangle} \, d\sigma\right)
\,\widetilde{\nu}_\lambda (dy)
=
\frac{1}{(2h)^d}\int_{[-h,h]^d}
\psi_\lambda (\xi+\sigma)\, d\sigma - \psi_\lambda (\xi),
\eds  \end{array}
$$
where $\psi_\lambda $ has been defined in (\ref{lk5}).
Since $\psi_\lambda\to\psi$ uniformly on bounded sets
as $\lambda\to +\infty$, the Fourier transform of
$\widetilde{\nu}_\lambda$ converges to the Fourier
transform $\widetilde{\psi}$ of a bounded measure $\widetilde{\nu}$,
$$
\widetilde{\psi}(\xi) =
\int_{\R^d}e^{i\langle \xi,y\rangle}
\,\widetilde{\nu}(dy).
$$

\begin{Lemma}
If a sequence of nonnegative measures $(\mu_n)$
converges weakly to a measure $\mu$ and a closed
set $\Gamma$ is such that $\mu(\partial
\Gamma)=0$ then the measures $\mu_n^\Gamma$, equal to
the the restrictions of  $\mu_n$ to $\Gamma$,  converge
weakly to $\mu^\Gamma$.
\end{Lemma}

\noindent {\bf Proof.}
The lemma will be proved later, together with other
elementary properties of the weak convergence. Below we will take
$\Gamma=\{ r\leq |y| \leq R\}$ \qed

\noindent We fix numbers $0<r<R<+\infty$ such that the
$\widetilde{\nu}$ measure of the spheres $\{|x|=r\}$,
$\{|x|=R\}$ is zero. Then
$$
  \begin{array}{l}\ds
    \psi_\lambda (y) =
-i\int_{|y|<R} \sin \langle \xi,y\rangle\,
\nu_\lambda (dy)
+
\int_{r\leq |y|<R}\left(1- \cos \langle \xi,y\rangle
\right)\,
\nu_\lambda (dy)
    \eds \\\ds
    +
\int_{|y|<r} \left(1- \cos \langle \xi,y\rangle
\right)\,
\nu_\lambda (dy)
+
\int_{R\leq |y|}
\left(1- e^{i \langle \xi,y\rangle}
\right)\,
\nu_\lambda (dy).
\eds
  \end{array}
$$
The limit of the last integral, as $\lambda\to
+\infty$, is equal to
$$
\int_{R\leq |y|}
\left(1- e^{i \langle \xi,y\rangle}
\right)\,
 \nu (dy),
$$
where $ \nu (dy)=K^{-1}(y)\widetilde{\nu}(dy)$.
Therefore the limit of the first integral exists as
well. Since
$$\begin{array}{l}\ds
\int_{|y|<R} \sin \langle \xi,y\rangle\,
\nu_\lambda (dy)
\eds\\\ds
=
\int_{|y|<R} \frac{\sin \langle \xi,y\rangle
-\langle \xi,y\rangle}{|y|^2}\left(
|y|^2K^{-1}(y)\right)\,
\widetilde{\nu}_\lambda (dy)
+
\langle \xi,\int_{|y|<R} y\,
\nu_\lambda (dy)\rangle
\eds\\ \ds =I_1+I_2,
\eds\end{array}
$$
and the functions $\ds y\to \frac{\sin \langle
\xi,y\rangle
-\langle \xi,y\rangle}{|y|^2}\eds$,
$y\to |y|^2K^{-1}(y)
$
 are continuous and bounded on
$\R^d$ therefore the limit of $I_1$, as $\lambda $ tends to $\infty$,
exists and is equal
to
$$
\int_{|y|<R}\left(\sin \langle \xi,y\rangle
-\langle \xi,y\rangle \right)\,
 \nu  (dy).
$$
Consequently the limit of $I_2$ exists as well and is
 of the form $\langle
\xi,a\rangle$ for some $a\in\R^d$. In a similar
 way one can show that
$$
\lim_{r\to 0}\left[
\lim_{\lambda\to +\infty}
\int_{r\leq |y|<R}\left(1- \cos \langle \xi,y\rangle
\right)\,
\nu_\lambda (dy)\right]
=
\int_{0< |y|<R}\left(1- \cos \langle \xi,y\rangle
\right)\,
 \nu  (dy).
$$
It remains to prove that for some $Q\geq 0$ and all
$\xi\in\R^d$
$$
\frac{1}{2}\langle Q\xi ,\xi\rangle =
\lim_{r\to 0}\left[
\lim_{\lambda\to +\infty}
\int_{  |y|<r}\left(1- \cos \langle \xi,y\rangle
\right)\,
\nu_\lambda (dy)\right].
$$
However
$$\begin{array}{l}\ds
\int_{  |y|<r}\left(1- \cos \langle \xi,y\rangle
\right)\,
\nu_\lambda (dy)
\eds \\\ds
=
\int_{  |y|<r}\frac{1- \cos \langle \xi,y\rangle
-\frac{1}{2} \langle \xi,y\rangle^2
}{|y|^2} |y|^2 K^{-1}(y)\,
\widetilde{\nu}_\lambda (dy)
+
\frac{1}{2}\int_{  |y|<r}
\langle \xi,y\rangle^2 \,
\nu_\lambda (dy)
\eds\\ \ds
=J_\lambda (r,\xi) +\frac{1}{2}
\langle Q_{r,\lambda}\, \xi,\xi\rangle.
\eds\end{array}
$$
It follows, again from the weak convergence of
$\widetilde{\nu}_\lambda\Rightarrow
\widetilde{\nu}$, that
$$
J_\lambda (r,\xi)\mathop{\longrightarrow}_{\lambda\to
+\infty}
\int_{  |y|<r}\left[
1- \cos \langle \xi,y\rangle
-\frac{1}{2} \langle \xi,y\rangle^2\right] \, \nu(dy).
$$
Consequently also $Q_{r,\lambda}\to Q_{r}$ as
$\lambda\to +\infty$. Setting now $r\to 0$
and taking the limit along a
sequence $r_n\to 0$ such that
$\widetilde{\nu}(|x|=r_n)=0$, we get the desired
result. \qed

\section{Infinitely divisible
families  and semigroups}

\subsection{Infinitely divisible families on $[0,+\infty)$}

Similarly as on $\R^d$, a family of probability measures $(\mu_t)_{t\geq 0}$ on
$[0,+\infty)$ is called infinitely divisible if and
only if

$i)$ $\mu_0=\delta_{\{ 0\} }$, $ii)$
$\mu_{t+s}=\mu_t
*
\mu_s$, $t,s\geq 0$, $iii)$ $\mu_t\Rightarrow
 \delta_{\{ 0\}}$ as $t\downarrow 0$.
\vspace{2mm}

\noindent The proper tool to study infinitely divisible families
on $[0,+\infty)$ are Laplace transforms. Denote by
$\widetilde{\mu}$ the Laplace transform of a measure on
$[0,+\infty)$. Thus
$$
\widetilde{\mu}_t(\xi ) =\int_0^{+\infty}
e^{-\xi x}\mu_t(dx),
\qquad \xi \geq 0,\, t\geq 0.
$$
If $\nu$ is a probability measure concentrated on
$[0,+\infty)$
 and $\lambda >0$ a
positive constant then
$$
\mu_t=e^{-\lambda t}\sum_{n=0}^\infty
\frac{(\lambda t)^n}{n!}\nu^{*n},
\qquad t\geq 0,
$$
is an infinitely divisible family on $[0,+\infty)$. It
is called the compound Poisson family. Note that, in
this case,
$$
\widetilde{\mu}_t(\xi ) =e^{-t\psi (\xi)},\qquad
{\rm where}\quad
\psi (\xi)= \lambda \int_0^{+\infty} \left(
1-e^{-\xi x}\right) \, \nu (dx).
$$
In a similar way as for $\R^d$ one can
show that the following result holds.

\begin{Theorem}
Laplace transforms $(\widetilde{\mu}_t)$ of an
infinitely divisible family are of the form:
\begin{equation}\label{laplacelk1}
\widetilde{\mu}_t(\xi ) =e^{-t\psi (\xi)},\qquad
{\rm where}\quad
\psi (\xi)= \lambda \int_0^{+\infty} \left(
1-e^{-\xi x}\right) \, \nu (dx),
\end{equation}
where $\nu$ is a nonnegative measure on $(0,+\infty)$
such that
\begin{equation}\label{laplacelk2}
\nu ([1,+\infty))<\infty,\qquad
\int_0^1x\, \nu(dx)<\infty.
\end{equation}
Conversely, the formula (\ref{laplacelk1}) defines an
infinitely divisible family if $\nu$ satisfies
 (\ref{laplacelk2}).
\end{Theorem}
We also have

\begin{Proposition}
Assume that $\nu(dx) =\frac{c}{x^{1+\beta}}dx$,
$\beta\in (0,1)$, $c>0$, then
$$
\psi (\xi)=c_1\, \xi^\beta,\qquad\xi\geq 0,
$$
for some    constant $c_1>0$.
\end{Proposition}

\noindent {\bf Proof.}
The same as of proposition \ref{frazion}.

\subsection{Subordination}

\begin{Theorem}
Let $(\mu_t)$, $(\nu_t)$ be infinitely divisible
families on $\R^d$ and  $[0,+\infty)$
 with exponents $\psi$
and $\phi$ respectively. Then the formula
$$
\sigma_t=\int_0^\infty \mu_s\, \nu_t(ds),\qquad t\geq 0,
$$
defines an infinitely divisible family on $\R^d$ with
 exponent $\phi (\psi )$.
\end{Theorem}

\noindent {\bf Proof.}
We check first that $\sigma_t*\sigma_u=\sigma_{t+u}$.
Indeed
$$
  \begin{array}{lll}
    \sigma_t*\sigma_u&= &\ds
\left[ \int_0^\infty \mu_s\, \nu_t(ds)\right] *
\left[ \int_0^\infty \mu_r\, \nu_u(dr)\right]
\eds
     \\
    &= & \ds
 \int_0^\infty  \int_0^\infty
 \mu_s * \mu_r\, \, \nu_t(ds) \nu_u(dr)
 \eds
  \\
    &= & \ds
 \int_0^\infty  \int_0^\infty
 \mu_{s +r}\, \, \nu_t(ds) \nu_u(dr)
 \eds
 \\
    &= & \ds
 \int_0^\infty
 \mu_{\sigma}\, ( \nu_t* \nu_u)(d\sigma)
 =\int_0^\infty
 \mu_{\sigma}\,  \nu_{t+u}(d\sigma)=\sigma_{t+u}.
 \eds
  \end{array}
$$
Moreover, by the definition of $\psi$:
$$
\widehat{\sigma}_t(\xi) =\int_0^\infty
\widehat{\mu}_s(\xi)\,\nu_t(ds)=
\int_0^\infty
e^{-s\psi (\xi)} \,\nu_t(ds).
$$
The definition of $\phi$ therefore gives
$$
\widehat{\sigma}_t(\xi) =
e^{-t\phi(\psi (\xi))},\qquad t\geq 0,\,\xi\in\R^d.
$$
Note that $\phi$ is defined on the right half plane of
$\C$ and $\psi (\xi)\geq 0$ for all $\xi\in\R^d$. \qed

\subsection{Semigroups determined by infinitely divisible families}

\noindent Let $(P_t)$ be a semigroup  of linear, continuous  operators on the space
$UC_b (\R^d)$ of continuous where $UC_b (\R^d)$ denotes the space of all bounded, uniformly
continuous functions on $R^{d}$, equipped with the supremum norm.
If $P_t1=1$, $t\geq 0$, and $P_tf\geq 0$
for $f\geq 0$ then $(P_t)$ is called {\it Markovian}. For
each $a\in\R^d$ define the translation $\tau_af$ of a
function $f$ by the formula $\tau_af(x)=f(x+a)$,
$a\in\R^d$, $x\in\R^d$. The semigroup $(P_t)$ is called
{\it translation invariant} if for arbitrary $a\in\R^d$,
$t\geq 0$ and $f\in UC_b (\R^d)$,
$P_t(\tau_af)=\tau_a(P_tf)$.

\noindent We set
\begin{equation}\label{defsg}
P_tf(x)=\int_{\R^d}f(x+y)\, \mu_t(dy),
\end{equation}
where $\mu_t$ is an infinitely divisible family.

\begin{Proposition}\label{semigroup}
If $(P_t)$ is defined on $UC_b(\R^d)$ by the formula
(\ref{defsg}) then $(P_t)$ is a $C_0$-semigroup on
$UC_b(\R^d)$.
\end{Proposition}
\noindent {\bf Proof.}
$i)$ Assume that $f\in UC_b(\R^d)$ and $t>0$. Then
$$
|P_tf(x) -P_tf(z)| = \left| \int_{\R^d} [f(x+y)
-f(z+y)]\,\mu_t(dy)\right| \leq
 \int_{\R^d} \left| f(x+y)
-f(z+y)\right| \,\mu_t(dy).
$$
For $\epsilon >0$ there exists $\delta >0$ such that if
$|x-x'|<\delta$ then $|f(x)-f(x')|<\epsilon$. Thus if
$|x-z|<\delta$ then also $|(x+y)-(z+y)|<\delta$ and
$|f(x+y)-f(z+y)|<\epsilon$. This gives
$|P_tf(x)-P_tf(y)|<\epsilon$.

$ii)$ $\ds |P_tf(x)-f(x)| =\left| \int_{\R^d}\left(
f(x+y)-f(x)\right)\, \mu_t(dy)\right|\eds$. For
$\epsilon
>0$ there exists $\delta >0$ such that
$|f(x+y)-f(x)|<\epsilon$ if $|y|<\delta$ and $x$
arbitrary. Therefore
$$
|P_tf(x)-f(x)| \leq \epsilon \int_{|y|\leq \delta}
\mu_t(dy) +2 \| f\|\int_{|y|> \delta}\mu_t(dy)
\leq \epsilon +2\| f\|\int_{|y|> \delta}\mu_t(dy).
$$
Since $\mu_t\Rightarrow \delta_{\{ 0\}}$ as
$t\downarrow 0$, $P_tf\to f$ uniformly. \qed

\noindent In fact we have the following characterisation of infinite divisible families which proof we
left as an exercise.
\begin{Proposition}
A Markovian semigroup
$(P_t)$ on $UC_b (\R^d)$ is translation invariant if and only if
$$
P_tf(x)=\int_{\R^d}f(x+y)\, \mu_t(dy),
$$
for some  infinitely divisible family $(\mu_t)$.
\end{Proposition}

\begin{Proposition}
There exists a unique extension of the operators
$(P_t)$ given by (\ref{defsg}) from uniformly continuous
functions with bounded supports to the whole
$L^p(\R^d)$ for any  $p\in [1,\infty)$.
The extended family is a $C_0$-semigroup on
$L^p(\R^d)$.
\end{Proposition}

\noindent {\bf Proof.}
$i)$ Extendability.

\noindent Note that
$$
  \begin{array}{lll}
    \|P_tf\|^p_{L^p}&= & \ds
\int_{\R^d}\left| \int_{\R^d} f(x+y)\, \mu_t(dy)
\right|^p dx
    \eds \\
    &\leq & \ds
\int_{\R^d}\left[ \int_{\R^d} |f(x+y)|^p\, \mu_t(dy)
\right]\, dx
=\int_{\R^d}\left[ \int_{\R^d} |f(x+y)|^pdx
\right]\, \mu_t(dy)
\eds\\
 &\leq & \ds
\int_{\R^d}\left( \|f\|^p_{L^p}\right) \, \mu_t(dy)
\leq \|f\|^p_{L^p}.
\eds
  \end{array}
$$

\noindent $ii)$ Continuity of $P_tf$ as $t\downarrow 0$.

\noindent One can
assume that $f$ is continuous with bounded support, as
such functions are dense in $L^p$ and the norms of
$P_t$ are bounded by $i)$.
 Let $f(x)=0$ if $|x|\geq R$. Note that
$$
\begin{array}{lll}\ds
\int_{\R^d}|P_tf(x)-f(x)|^pdx
\eds &=&\ds
\int_{\R^d}\left| \int_{\R^d} \left(
f(x+y)-f(x)\right)\, \mu_t(dy)
\right|^p dx
\eds\\
&\leq&\ds
 \int_{\R^d}\left[ \int_{\R^d} \left|
f(x+y)-f(x)\right|^pdx \, \right] \mu_t(dy).
\eds
\end{array}
$$
For any $r\in (0,R)$,
$$
\begin{array}{lll}\ds
\int_{\R^d}|P_tf(x)-f(x)|^pdx
\eds &\leq&\ds
\int_{|y|\leq r}\left( \int_{|x|\leq 2R} \left|
f(x+y)-f(x)\right|^pdx\right)\, \mu_t(dy)
\eds\\ & +&\ds
  \int_{|y|>r}\mu_t(dy)\left[ 2^{p}
\|f\|^p_{L^p}   \right] \leq I_1+I_2.
\eds
\end{array}
$$
 By taking $r>0$ sufficiently
small the term $I_1$ is smaller than $\epsilon >0$ for
all $t>0$. By taking $t>0$ sufficiently small also the
second term is smaller than $\epsilon$. This completes
the proof.
\qed

\subsection{Generators of $(P_t)$ on $L^2(\R^d)$}

We choose now $p=2$ and extend $P_t$ to the space
$L^2_\C$ of all $2$-summable complex functions.
Denoting as before by $\widehat{g}$ the Fourier
transform (characteristic function)
of $g$ we have for $f$ with compact support
$$
  \begin{array}{lll}
    \widehat{P_tf}(\lambda)&= & \ds
\int_{\R^d}e^{i\langle \lambda ,x\rangle}
[P_tf(x)]\, dx =
\int_{\R^d}e^{i\langle \lambda ,x\rangle}
\left[\int_{\R^d} f(x+y)\, \mu_t(dy)\right]\, dx
    \eds \\
    &= & \ds
\int_{\R^d}e^{-i\langle \lambda ,y\rangle}
\left[\int_{\R^d}
e^{i\langle \lambda ,x+y\rangle} f(x+y)\, dx
\right]\, \mu_t(dy)
\eds \\
&=&\ds
\widehat{f}(\lambda)e^{-t\psi(-\lambda)},
\qquad t\geq 0,\, \lambda\in\R^d.
\eds
  \end{array}
$$
Consequently the Fourier image of $P_t$ acts as a
multiplication by $\exp (-t\psi(-\cdot))$. Therefore
the domain of its generator $A$ is of the form:
$$
D(A)=\{g\in L^2_\C \, ;\,
\psi (-\cdot)\widehat{g}(\cdot)\in L^2_\C\}
$$
and $A$ is given by the formula:
$$
\widehat{Ag}(\lambda)=-\psi(-\lambda)
\widehat{g}(\lambda),
\qquad \lambda\in\R^d.
$$
In particular for the Wiener semigroup with
$\psi(\lambda)=\frac{1}{2}|\lambda|^2$,
$\lambda\in\R^d$:
$$
\widehat{Ag}(\lambda)=-\frac{1}{2}|\lambda|^2
\widehat{g}(\lambda),
$$
therefore $A$ can be identified as the Laplace operator
$\Delta$ on $L^2_\C(\R^d)$.

In the case of stable processes on $\R^d$, with
$\psi(\lambda)=|\lambda|^\alpha$, $\alpha\in (0,2)$,
$$
\widehat{Ag}(\lambda)=-|\lambda|^\alpha
\widehat{g}(\lambda),
$$
and $A$ can be idenitfied with a pseudodifferential
operator with symbol $-|\lambda|^\alpha$. In fact in
this case $A=-(-\Delta)^{\alpha /2}$ and one can write
a formula for the semigroup generated by $A$ in terms
of the heat semigroup only, using the idea of
subordination.

Let $(P_t)$ denote the heat semigroup and
$(P_t^\alpha)$ the semigroup corresponding to
$\psi(\lambda)=|\lambda|^\alpha$, called the
$\alpha$-stable semigroup, $\alpha\in (0,2)$. Note that
$$
|\lambda|^\alpha=2^{-\frac{\alpha}{2}}
\left( \frac{1}{2}|\lambda|^2\right)^{\frac{\alpha}{2}}
= \phi \left( \frac{1}{2}|\lambda|^2\right),
$$
where $\phi(\xi) =2^{-\frac{\alpha}{2}}
\xi^{\frac{\alpha}{2}}$, $\xi\geq 0$. But $\phi $ is the
exponent of an $\frac{\alpha}{2}$-stable semigroup of
measures $\mu_t^\alpha$ on $[0,+\infty)$ and by the
subordination formula
$$
P_t^\alpha = \int_0^\infty P_s\, \mu_t^\alpha(ds).
$$

\chapter{Representation of L\'evy processes}

We show here that Wiener's constructive approach can be applied
to processes with discontinuous paths. In fact starting from a
countable family of random variables we will construct an arbitrary
L\'evy process. The obtained
expression will give an additional interpretation of the
L\'evy-Khinchin formula.

\subsection{Poissonian random measures.}

Let $(E,\cale)$ be a measurable space and $\nu$ a finite measure on $E$.
We recall that a random variable $N$ has Poisson distribution with
parameter $c>0$ if
$$
\P(N=k)=e^{-c}\, \frac{c^k}{k!},\qquad k=1,2,\ldots.
$$
By direct calculation we obtain a formula for the Laplace
transform of the Poisson distribution
$$
\E\, (e^{-\alpha N})= e^{-c(1-e^{-\alpha})},
\qquad \alpha\geq 0.
$$
Let $N$ have a Poisson distribution with parameter $c=\nu(E)$ and
let $\xi_1,\xi_2,\ldots$ be a sequence of independent random
variables with distribution $\frac{1}{\nu(E)}\nu$.

\begin{Theorem}\label{poissonrandommeasure}
  For arbitrary $A\in\cale$ define
  $$
  \pi(A)=\sum_{k=1}^N1_A(\xi_k).
  $$
\begin{enumerate}\item[i)]
The random variable $\pi(A)$ has Poisson distribution with
parameter $\nu(A)$.
\item[ii)] For disjoint sets $A_1,\ldots,A_m$ the corresponding
 random variables are independent and
 $$
 \pi(A_1\cup \ldots\cup A_m)=\pi(A_1)+\ldots+\pi(A_m),
 \qquad \P-a.s.
 $$
 \end{enumerate}
\end{Theorem}

\noindent The family $\pi$ is called a {\it Poissonian measure with intensity
measure} $\nu$.

\noindent {\bf Proof.}
$i)$ Fix $\alpha>0$ and $A\in \cale$. Then,  by   independence,
$$\begin{array}{lll}\dis
\E\, (e^{-\alpha \pi(A)})
&=&\dis
\sum_{m=0}^{+\infty}
\E\, \left(e^{-\alpha \sum_{k=1}^m 1_A(\xi_k)}1_{N=m}\right)
\\&=&\dis
\sum_{m=0}^{+\infty}\P(N=m) \prod_{k=1}^m
\E\, (e^{-\alpha  1_A(\xi_k)})
\\&=&\dis
\sum_{m=0}^{+\infty}e^{-c}\, \frac{c^m}{m!} \left(
\E\, (e^{-\alpha  1_A(\xi_1)})\right)^m.
\end{array}
$$
However
$$
\E\, (e^{-\alpha  1_A(\xi_1)})=
e^{-\alpha}\P(\xi_1\in A)+1-\P(\xi_1\in A)
= \frac{\nu(A)}{c}(e^{-\alpha}-1)+1.
$$
Finally
$$
\E\, (e^{-\alpha \pi(A)})=e^{-c}\; e^{c[
\frac{\nu(A)}{c}(e^{-\alpha}-1)+1]}
= e^{-c\nu(A)(1-e^{-\alpha})},\qquad \alpha\geq 0.
$$

$ii)$ Assume that the sets  $A_1,\ldots,A_m$ are disjoint. For
arbitrary non-negative numbers  $\alpha_1,\ldots,\alpha_m$:
$$\begin{array}{lll}\dis
\E\, \left(e^{-\sum_{j=1}^m\alpha_j \sum_{k=1}^N 1_A(\xi_k)}\right)
&=&\dis
\sum_{l=0}^{+\infty}
e^{-c}\, \frac{c^l}{l!}\E\, \left(
e^{-\sum_{j=1}^m\alpha_j \sum_{k=1}^l 1_{A_j}(\xi_k)}\right)
\\&=&\dis
\sum_{l=0}^{+\infty}
e^{-c}\, \frac{c^l}{l!}\E\, \left(
e^{-\sum_{k=1}^l(\sum_{j=1}^m\alpha_j  1_{A_j}(\xi_k))}\right)
\\&=&\dis
e^{-c}\,\sum_{l=0}^{+\infty} \frac{1}{l!}
 \left[c\,\E\,(
e^{-\sum_{j=1}^m\alpha_j  1_{A_j}(\xi_k)})\right]^l
\\&=&\dis
e^{-c}\,e^{c\,\E\,(
-\sum_{j=1}^m\alpha_j  1_{A_j}(\xi_1))}
\end{array}
$$
Since
$$
\E\,(
-\sum_{j=1}^m\alpha_j  1_{A_j}(\xi_1))
=
\sum_{j=1}^me^{-\alpha_j} \P (\xi_1\in A_j)+1- \P (\xi_1\in A_j)
=\sum_{j=1}^m \frac{\nu(A_j)}{c}(e^{-\alpha_j}-1)+1,
$$
therefore
$$\begin{array}{lll}\dis
\E\, \left(e^{-\sum_{j=1}^m\alpha_j \sum_{k=1}^N 1_A(\xi_k)}\right)
&=&\dis
e^{-c}\,e^{c\,\Big(
 \sum_{j=1}^m \frac{\nu(A_j)}{c}(e^{-\alpha_j}-1)+1
\Big)}
\\&=&\dis
\prod_{j=1}^m e^{-\nu(A_j)(1-e^{-\alpha_j})}
\\&=&\dis
\prod_{j=1}^m \E\, (e^{-\alpha_j\pi(A_j)}).
\end{array}
$$
Since the multivariate Laplace transform of the vector
$(\pi(A_1),\ldots,\pi(A_m))$ is equal to the product of the
Laplace transform of $\pi(A_1),\ldots,\pi(A_m)$ therefore the
required independence follows.
\qed

\begin{Theorem}\label{poissonrandommeasuredue}
Let $\nu$ be a $\sigma$-finite measure on $E$
and $\nu_1,\nu_2,\ldots$ finite measures with
disjoint supports  such that
$\nu=\sum_{m=1}^{+\infty}$. Let $\pi^1,\pi^2,\ldots$
be Poissonian measures constructed for measures
$\nu_1,\nu_2,\ldots$ with independent double system
$(N_m, \xi_1^m,\xi_2^m,\ldots)$. Then the family
  $$
  \pi(A)=\sum_{m=1}^{+\infty} \pi^m(A),
  \qquad A\in\cale,
  $$
  satisfies properties $i)$ and $ii)$ from the previous
  theorem.
\end{Theorem}

 \noindent{\bf Proof.} Since the proof is similar to the previous
 one it will be omitted. \qed

\noindent  It follows from the construction that the family
$\pi(A)$, $ A\in\cale$, has the following structure. There exists
a sequence $(x_k)$ of $E$-valued random variables such that
$$
\pi(A)(\omega) = \sum_{k=1}^{+\infty} \delta_{x_k(\omega)}(A),
\qquad \omega\in\Omega,\; A\in \cale.
$$
Thus the family $\pi(A)$, $ A\in\cale$,  can be identified with a
random distribution of a countable number of points $(x_k)$ and
$\pi(A)$ is equal to the number of  points in the set $A$.
One can also integrate with respect to $\pi$. Assume that $f$ is a
measurable, real function defined on $E$ and set
\begin{equation}\label{intpoisson}
\int_Ef(x)\; \pi (dx)= \sum_{k=1}^{+\infty} f(x_k)
\end{equation}
for all those $f$ for which the series in (\ref{intpoisson}) is
convergent $\P$-a.s. One can also define the integral using the
Lebesgue scheme: first for simple, non-negative functions
$f=\sum_{j=1}^M\alpha_j\, 1_{A_j}$, $A_j\cap A_k=\emptyset$,
$j\neq k$, by setting
$$
\int_Ef(x)\; \pi (dx)= \sum_{j=1}^M \alpha_j\, \pi(A_j),
$$
then by monotone limits for all measurable non-negative $f$
and finally, for arbitrary $f$,  by splitting it into the
difference of positive and negative parts.

\noindent We gather basic properties of  the integral in the following
theorem.

\begin{Theorem}\label{poissonintegral}
\begin{enumerate}\item[i)]
  For arbitrary non-negative measurable $f$
\begin{equation}\label{duestelle}
  \E\, \left( e^{-\alpha \int_Ef(x)\; \pi (dx)}\right)
  = e^{-\int_E(1-e^{-\alpha f(x)})\; \nu (dx)} ,
  \qquad \alpha\geq 0.
\end{equation}

\item[ii)] If $\int_E|f(x)|\; \pi (dx)<+\infty$ then
\begin{equation}\label{trestelle}
  \E\, \left( e^{i\lambda\int_Ef(x)\; \pi (dx)}\right)
  = e^{-\int_E(1-e^{i\lambda f(x)})\; \nu (dx)} ,
  \qquad \lambda\in\R^1.
\end{equation}

\item[iii)] If $\int_E|f(x)|\; \nu (dx)<+\infty$ then
$\E\, \left|\int_Ef(x)\; \pi (dx)\right|<+\infty$ and
$$
\E\, \int_Ef(x)\; \pi (dx)=
\int_Ef(x)\; \nu (dx).
$$

\item[iv)] If $\int_E|f(x)|\; \nu (dx)<+\infty$
and $\int_E|f(x)|^2\; \nu (dx)<+\infty$
then
$$\E\, \left|\int_Ef(x)\; \pi (dx)
- \int_Ef(x)\; \nu (dx)\right|^2=
\int_E|f(x)|^2\; \nu (dx).
$$

\item[v)] If non-negative measurable functions
$f_1,f_2,\ldots , f_M$ have disjoint supports then the random
variables
$\int_Ef_1(x)\; \pi (dx),\dots,\int_Ef_M(x)\; \pi (dx)$ are
independent.

 \end{enumerate}
\end{Theorem}

\noindent {\bf Proof.} We
prove only $i)$ as the proofs of the other points go the same way.

\noindent If $f=\sum_{j=1}^M\alpha_j\, 1_{A_j}$, $A_j\cap A_k=\emptyset$,
$j\neq k$, then
$$
\E\, \left(e^{-\alpha \int_Ef(x)\; \pi (dx)}\right)
=\prod_{j=1}^M \E\, (e^{-\alpha \alpha_j \pi (A_j)})
=\prod_{j=1}^M e^{-\nu (A_j)(1-e^{-\alpha \alpha_j}) }
$$
and the formula holds. By monotone passage to the limit
one gets it for all
non-negative measurable functions.
\qed

\subsection{Representation theorem.}
It is clear that the exponent $\psi$, see formula (\ref{lk3}), of the characteristic  function
$$
\widehat{\mu}_t(\lambda)=e^{-t\psi(\lambda)},\qquad t\geq 0,\;
\lambda\in\R^d,
$$
of an  infinitely divisible law, can be written in an equivalent way
\begin{equation}
  \psi(\lambda)= i<a, \lambda> + <Q \lambda, \lambda> + \psi_{0}(\lambda),
\end{equation}
where
\begin{equation}\label{quattrostelle}
  \psi_{0}(\lambda)= \int_{|x|\geq 1}(1-e^{i\<\lambda,x\>})\;\nu(dx)
  +\int_{|x|< 1}(1-e^{i\<\lambda,x\>}+i\<\lambda,x\>)\;\nu(dx),
\end{equation}
and the measure $\nu,$ concentrated  on $\R^d\backslash \{0\},$ satisfies
$$
\int_{|x|< 1}|x|^2\;\nu(dx) + \nu \{x: |x|\geq 1 \} <+\infty.
$$
The main result of the chapter is the following theorem. In the proof
we follow basically \cite{bertoin}.

\begin{Theorem}\label{rapprlevyprocess}
  Assume that the exponent $\psi_{0}$ of the characteristic function
  $\widehat{\mu}_t$, $t\geq 0$, is of the form
  (\ref{quattrostelle}). Let $\pi$ be the Poissonian measure on
  $[0,+\infty)\times (\R^d\backslash \{0\})$ with the intensity
  $l_1\times \nu$, where $l_1$ denotes the Lebesgue measure.
  \begin{enumerate}\item[i)]
The formula
\begin{equation}\label{cinquestelle}
  X(t)= \int_0^t \int_{|x|\geq 1}x\; \pi(ds,dx)+
  \lim_{\epsilon\downarrow 0}\left[
\int_0^t \int_{\epsilon\leq |x|< 1}x\; \pi(ds,dx)-
\int_0^t \int_{\epsilon\leq |x|< 1}x\; \nu(ds,dx)
  \right]
\end{equation}
defines a process with independent increments having exponent
$\,\,\psi_{0}$. The limit in (\ref{cinquestelle}) exists $\P$-almost
surely uniformly on arbitrary time interval $[0,T]$ if $\epsilon$
tends to $0$ sufficiently fast. In particular (\ref{cinquestelle})
defines a c\`adl\`ag process.
\item[ii)] The process $X$ given by (\ref{cinquestelle}) has
trajectories with bounded variation if and only if
$$
\int_{|x|< 1}|x|\;\nu(dx)<+\infty.
$$
\end{enumerate}
\end{Theorem}

\noindent {\bf Proof.}
$1)$ Consider first the process
$$
X_1(t)=  \int_0^t \int_{|x|\geq 1}x\; \pi(ds,dx),
\qquad t\ge 0.
$$
If $0=t_0<t_1<\ldots <t_M$ then

\begin{equation}\label{increments}
X_1(t_j)-X_1(t_{j-1})=
 \int_{]t_{j-1},t_j]\times \{x\; :\; |x|\geq 1\}}x\; \pi(ds,dx),
\qquad j=1,2,\ldots ,M,
\end{equation}
and since the sets $]t_{j-1},t_j]\times \{x:|x|\geq 1\}$,
$ j=1,2,\ldots ,M$, are disjoint therefore, by Theorem
\ref{poissonrandommeasuredue}, the random variables (\ref{increments}) are
independent. Moreover, the sets $[0,T]\times\{x:|x|\geq 1\}$
contain only a finite number of points $\{(t_k,x_k)\}$ and
therefore the trajectories of $X_1$ are c\`adl\`ag. By Theorem
\ref{poissonintegral}-$ii)$,
$$\begin{array}{lll}\dis
  \E\, \left(e^{i\<\lambda,X_1(t)\>}\right) & = & \dis
 \E\, \left(e^{i\<\lambda, \int_0^t \int_{|x|\geq 1}x\; \pi(ds,dx)\>}
 \right)   \\
    & = & \dis
e^{-  \int_0^t \int_{|x|\geq 1}(1-e^{i\<\lambda,x\>})\; \nu(dx)\,ds}
     \\
    & = & \dis
    e^{-t \int_{|x|\geq 1}(1-e^{i\<\lambda,x\>})\; \nu(dx)},\qquad
    \lambda\in\R^d.
\end{array}
$$
Define, for $\epsilon\in (0,1)$,
$$
X_{2,\epsilon}(t)=
\int_0^t \int_{\epsilon\leq |x|< 1}x\; \pi(ds,dx)-
t \; \nu(\{x\;:\; \epsilon\leq |x|< 1\}),
\qquad  t\ge 0.
$$
In a similar way one shows that the process $X_{2,\epsilon}$ has
independent increments, is independent of $X_1$ and that
$$
 \E\, \left(e^{i\<\lambda,X_{2,\epsilon}(t)\>}\right)  =
  e^{-t \int_{\epsilon \le |x|< 1}(1-e^{i\<\lambda,x\>}
  +i\<\lambda,x\>)\; \nu(dx)}.
  $$
It is therefore enough  to show that the sequence of processes
$X_{2,\epsilon}$ converges  $\P$-a.s.
 uniformly on any fixed time interval $[0,T]$.
It follows from Theorem \ref{poissonintegral} that
the process $X_{2,\epsilon}$ is a martingale with finite
second moment. In fact if $0<\eta<\epsilon <1$ then
$$
\E\, |X_{2,\epsilon}(t)- X_{2,\eta}(t)|^2=
t \int_{\eta\leq |x|<\epsilon}|x|^2\; \nu(dx).
$$
We have the following corollary from the theorem on Doob's
inequalities.

\begin{Lemma}
  If $(Z(t))$ is a martingale with right continuous trajectories
  then for arbitrary $c>0$
  $$
  \P\left(\sup_{0\le t\le T}|Z(t)|\ge c\right)
  \le \frac{1}{c}\, \E\, |Z(T)|.
  $$
\end{Lemma}

\noindent {\bf Proof.} It is enough to prove that for arbitrary
discrete time martingale $(X_n)$ and arbitrary  $c>0$
 $$
  \P\left(\sup_{1\le n\le k}|X_n|\ge c\right)
  \le \frac{1}{c}\, \E\, |X_k|.
  $$
Let us remark that
 $$
  \P\left(\sup_{1\le n\le k}|X_n|\ge c\right)
  \le \P\left(\sup_{1\le n\le k}X_n\ge c\right)
  + \P\left(\inf_{1\le n\le k}X_n\le -c\right)
  $$
  and by inequalities $1)$ and $4)$ of the Doob theorem we have
  that
 $$
  \P\left(\sup_{1\le n\le k}|X_n|\ge c\right)
  \le \frac{1}{c}\, (\E \, X_1+2\E\, X_k^-).
  $$
But $\E\, X_1=\E\, X_k = \E\, X_k^+-\E\, X_k^-$ and
$\E\,|X_k|= \E\, X_k^+ +\E\, X_k^-$, so the result holds.
\qed

\noindent It follows from the lemma that for arbitrary $T>0$
$$
\begin{array}{lll}\dis
\P\left(\sup_{0\le t\le T}|
X_{2,\epsilon}(t)- X_{2,\eta}(t) |\ge c\right)
&\le&\dis
 \frac{1}{c} \,\E\,|
X_{2,\epsilon}(T)- X_{2,\eta}(T) |
\\
&=&\dis
 \frac{1}{c} \left(\E\,
X_{2,\epsilon}(T)- X_{2,\eta}(T) |^2\right)^{1/2}
\\
&=&\dis
 \frac{\sqrt{T}}{c} \left(
\int_{\eta\le |x|<\epsilon}|x|^2\; \nu(dx)
 \right)^{1/2}\to 0
 \qquad {\rm as\;} \eta,\epsilon\to 0.
 \end{array}
$$
Consequently there exists a process $X_2(t)$, $t\in [0,T]$,
and a sequence $\epsilon_n\downarrow 0$ such that
$$
\P\left(\lim_{n\to +\infty}\sup_{0\le t\le T}|
X_{2,\epsilon_n}(t)- X_{2}(t) |=0\right)=1.
$$
This completes the proof of $i)$.

$ii)$ Let $\Delta X_t=X(t)-X(t-)$, $t\in (0,T]$. Then
$$
\E\, \left(e^{-\sum_{0<t\le T}|\Delta X_t|}\right)
=\E\, \left( e^{-\int_0^{+\infty}\int_{\R^d}f(t,x)\;
\pi(dt,dx)}\right),
$$
where $f(t,x) = 1_{[0,T]}(t)|x|$. Consequently
$$
\E\, \left(e^{-\sum_{0<t\le T}|\Delta X_t|}\right)
= e^{-T\int_{\R^d}(1-e^{-|x|})\;\nu(dx)}
$$
Thus if $\int_{\R^d}(1-e^{-|x|})\;\nu(dx)=+\infty$ then
$\P (\sum_{0<t\le T}|\Delta X_t|=+\infty)=1$. If
$\int_{\R^d}(1-e^{-|x|})\;\nu(dx)<+\infty$, equivalently
$\int_{\R^d}|x|\;\nu(dx)<+\infty$, then
$$
\E \,\sum_{0<t\le T}|\Delta X_t|1_{|X(t)|\le M}=
T \int_{|x|\le M} |x|\; \nu(dx)<+\infty),
$$
so $X$ is of bounded variation.
\qed

\chapter{Stochastic integration and Markov processes.}

Wiener's idea to construct stochastic processes starting from
simple probabilistic objects has been extended to all Markov
processes by K. Ito in his seminal paper {\it On stochastic differential
equations} \cite{ito}.
His starting point was a Wiener process and a Poissonian random
measure. He first developed a  stochastic integration theory with
respect to Wiener processes and Poissonian random measures. Then
he introduced stochastic equations and showed that Markov
processes can be obtained as their solutions. In this chapter we
describe the Ito construction and sketch the proofs of basic results.

\section{Constructing Markov chains}

Markov chains are discrete time analogues of Markov processes. We
present here a construction of Markov chains as solutions of
discrete time stochastic equations to motivate better
Ito'approach.

Let $P(x,\Gamma)$,  $x\in E$, $\Gamma\in\cale$, be a transition
kernel on a
measurable space $(E,\cale)$.
Thus for each $x$, $P(x,\cdot)$ is a probability measure on $E$
and for each $\Gamma\in\cale$, $P(\cdot,\Gamma)$ is an
$\cale$-measurable function. By $P^n$, $n=0,1,\ldots$, we define
the iterated kernels
$$
P^{n+1}(x,\Gamma)=\int_E P(x,dy)\; P^n(y,\Gamma),
\qquad x\in E,\, \Gamma \in \cale.
$$
As in continuous time we define by induction, for any sequence of
non-negative integers $n_1<n_2<\ldots < n_d$, the
probabilities of visiting the sets $\Gamma_1,\ldots,\Gamma_d$ at
moments $n_1,\ldots , n_d$ starting from $x$ as the functions
$P^{n_1,\ldots , n_d}:E\to \calp(E^d)$:
$$
P^{n_1,\ldots , n_d}(x,\Gamma_1,\ldots,\Gamma_d)=
\int_{\Gamma_1} P^{n_1}(x,dx_1)\;
P^{n_2-n_1,\ldots , n_d-n_1}(x_1,\Gamma_2,\ldots,\Gamma_d).
$$
A Markov chain $X$ with the transition kernel $P$ and starting
from $x\in E$ is a sequence $X_n$, $n=0,1,\ldots$, of $E$-valued
random variables on a probability space, such that
\begin{enumerate}
\item[1)]
$X_0=x, \qquad \P$-a.s.
\item[2)] $ \P(\{\omega: X_{n_j}(\omega)\in\Gamma_j,\
j=1,\ldots,d\})
=P^{n_1,\ldots,n_d}(x,\Gamma_1,\ldots,\Gamma_d)$.
\end{enumerate}
We have the following result.

\begin{Theorem}
  Let $E$ be a Polish space and $\cale=\calb(E)$ the family of its
  Borel sets. Then for any transition kernel $P$ there exists a
  measurable mapping $F:E\times [0,1)\to E$
  such that, for an arbitrary sequence of independent random
  variables $\xi_1,\xi_2,\ldots $ with uniform distribution on
  $[0,1)$, the inductively defined sequence $X_n$:
\begin{equation}\label{ricorsiva}
  X_0=x,\qquad X_{n+1}=F(X_n, \xi_{n+1}),\qquad
  n=0,1,\ldots
\end{equation}
is a Markov chain with the transition kernel $P$.
\end{Theorem}

\noindent {\bf Proof.} We construct $F$ in the case when $E$ is a
countable set and $E=\R^1$. Let $E=\{1,2,\ldots\}=\N$,
$p_{n,m}=P(n,\{m\})$, $n,m\in\N$. We define
$$
F(n,u)= m\qquad {\rm for}\qquad
u\in [p_{n,1}+\ldots+p_{n,m-1}, p_{n,1}+\ldots+p_{n,m}).
$$
The required measurability is obvious and if $\xi$ has uniform
distribution on $[0,1)$ then
$$
\P(\{\omega\; :\; F(n,\xi(\omega))=m\})= P(n,\{m\}),
$$
as required.

If $E=\R^1$ then we define first a measurable function
$$
G(x,v)  =P(x,(-\infty,v]),
\qquad x\in\R^1,\; v\in\R^1,
$$
and set
$$
F(x,u)= \inf\{v\; :\; u\le G(x,v)\},
\qquad u\in [0,1),
$$
as in the proof of the Steinhaus theorem it is therefore clear
that if $\xi$ has uniform distribution on $[0,1)$ then
\begin{equation}\label{verifsufdue}
\P(\{\omega\; :\; F(x,\xi(\omega))\in\Gamma\})= P(x,\Gamma),
\qquad x\in\R^1,\; \Gamma \in \calb(\R^1).
\end{equation}

A direct construction of $F$ is possible also in the general case
but we will limit ourselves to recalling a general result of
Kuratowski saying (in particular) that if $E$ is Polish then there
exists a one-to-one measurable mapping $S$ from $E$ either on a
finite set $\{ 1,2,\ldots, N\}$ or on $\N$ or on $\R^1$ having
measurable inverse. It is therefore clear that a function $F$ can
be constructed also for the  Polish spaces.

The fact that $(X_n)$
is a Markov chain with the transition kernel $P$ follows from the
independence of $\xi_1,\xi_2,\ldots$ by an induction argument.
\qed

Thus any Markov chain can be generated from a sequence of
independent random variables and Ito's idea was that a similar
result is true for general continuous Markov processes.

\section{The Courr\`ege theorem}

Let  $P^t$ be a transition function on
$(\R^d, \calb(\R^d))$. For each bounded measurable function
$\phi$ on $\R^d$ we set
$$
P^t\phi(x)= \int_{\R^d} P^t(x,dy)\, \phi(y),
\qquad x\in \R^d.
$$
We say that a transition function $P^t$ is Feller if
\begin{enumerate}
\item[1)]
$P^t\phi\in C_0(\R^d)$, for $\phi\in C_0(\R^d)$.
\item[2)] $ P^t\phi\to\phi$ as $t\downarrow 0$, uniformly,
for all $\phi\in C_0(\R^d)$.
\end{enumerate}
The following theorem is due to
Ph. Courr\`ege \cite{curr}. In its formulation
$C_0^\infty (\R^d)$ stands for the
space of infinitely differentiable functions
vanishing at infinity with
all their derivatives
and
$ L_+(\R^d,\R^d)$  for the set of
symmetric, non-negative, $d \times d$
matrices.

\begin{Theorem}\label{courrege}
Let $P^t$ be a Feller transition function such that
for all $\phi$ $C_0^\infty (\R^d)$ and all $x\in\R^d$, the function
$P^t\phi(x)$, of parameter $t\geq 0$, is differentiable.
Then there exist
transformations $F:\R^d\to\R^d$, $Q:\R^d\to L_+(\R^d,\R^d)$ and
a family $\nu(x,\cdot)$, $x\in\R^d$, of non-negative
measures concentrated on
$\R^d\backslash \{0\}$ and satifying
$$
\int_{\R^d} (|y|^2\wedge 1)\; \nu (x,dy)<+\infty,
\qquad x\in\R^d,
$$
such that for all $ x\in\R^d$, $\phi\in C_0^\infty (\R^d)$
$$
\lim_{t\to 0} \frac{P^t\phi(x)-\phi(x)}{t}
= A\phi(x)
$$
where
$$\begin{array}{l}\dis
 A\phi(x)= \<F(x),D\phi(x)\> +\frac{1}{2}{\rm \ Trace}\
 [Q(x)D^2\phi(x)]
 \\\dis
 \qquad\qquad+ \int_{\R^d} \left(
 \phi(x+y)-\phi(x)-\frac{\<y,D\phi(x)\>}{|y|^2+ 1}
 \right)\; \nu(x,dy).
 \end{array}
 $$
\end{Theorem}

The functions $F$, $Q$ and the measure $\nu$ are called
{\it characteristics} of $P^t$ or of the process.
They have interpretations as {\it local drift, local diffusion} and
{\it local jump measure} of the corresponding Markov process. The operator $A$ is, the
so called, {\it generator} of $P^t$ or generator of the corresponding Markov process.

We will not prove this theorem but will make only some comments.

Let us notice that Proposition
\ref{formesponenzclass} states that
the transition function $(P^t)$ is a
limit of transition functions $(P^t_\lambda)$ of the following
form:
$$
P^t_\lambda\phi = e^{t\lambda \left(\lambda R_\lambda
-I\right)}\phi=
e^{-\lambda t} \sum_{k=0}^{+\infty}
\frac{(\lambda t)^k}{k!}\, \widetilde{P}_\lambda^k\phi
$$
where $\widetilde{P}_\lambda$ is a transition kernel defined for
all bounded measurable $\phi$:
$$
\widetilde{P}_\lambda\phi
= \lambda\int_0^{+\infty} e^{-\lambda t}P^t\phi\; dt.
$$
The operator $\widetilde{P}_\lambda$ is bounded on
$C_0(\R^d)$.
Notice that for the approximating transition semigroup
$(P^t_\lambda)$, called also {\it compound Poissonian semigroup},
the operator $A$ from the Courr\`ege theorem is of the form
$$
A_\lambda \phi(x)=
\lambda \int_{\R^d} \left(
 \phi(x+y)-\phi(x) \right)\; \widetilde{P}_\lambda(x,dy).
$$
Thus for the compound Poissonian semigroup, the drift  and
diffusion vanish and the jump measure is equal to
$\lambda\widetilde{P}_\lambda(x,\cdot)$, $x\in\R^d$.
As for L\'evy's processes the structure of the limiting semigroup
$(P^t)$ is more complex. From the proof of the
L\'evy-Khinchin formula we could deduce the following result.

\begin{Theorem}
  Let $(P^t)$ be a transition function of the form
$$
P^t\phi(x)= \int_{\R^d} \phi(x+y)\; \mu_t(dy),
$$
where $(\mu_t)$ is an infinitely divisible family of measures with
the representation (\ref{lk2})-(\ref{lk3}). Then
$(P^t)$ satisfies the assumptions of the Courr\`ege theorem and
for $\phi\in C_0^\infty (\R^d)$:
$$\begin{array}{l}\dis
 A\phi(x)= \<a,D\phi(x)\> +\frac{1}{2}{\rm \ Trace}\
 [QD^2\phi(x)]
 \\\dis
 \qquad\qquad+ \int_{\R^d} \left(
 \phi(x+y)-\phi(x)-\frac{\<y,D\phi(x)\>}{|y|^2+ 1}
 \right)\; \nu(dy).
 \end{array}
 $$
\end{Theorem}
The operator
$$
 A_0\phi(x)= \<a,D\phi(x)\> +\frac{1}{2}{\rm \ Trace}\
 [QD^2\phi(x)]
 $$
corresponds to the  process:
$$
X(t) = a\, t + Q^{1/2}W(t),\qquad t\ge 0,
$$
where $W$ is a Wiener process on $\R^d$ with identity covariance.
The L\'evy process with the integral part of $A$ was constructed,
in an earlier chapter, from  Poissonian random measures.

\section{Diffusion with additive noise}
Ito, in the already quoted paper \cite{ito}, constructed Markov processes with
rather general  characteristics $\,(F(\cdot),Q(\cdot),\nu(\cdot))\,$. In the simpler case,
when $\nu(\cdot)=0$ and the diffusion matrix is constant, the construction
does not require any stochastic calculus as we will show now.

Let $R=(r_{kl})$ be a non-negative $m\times m$ matrix, $W$ an
$m$-dimensional continuous Wiener process with components
$W_1,W_2,\ldots ,W_m$ such that
$$
\E\, W_k(t)=0,\quad
\E\,( W_k(t)W_l(t))=(t\wedge s)r_{kl},\quad
t,s\ge 0,\; k,l=1,\ldots m.
$$
Let in addition $F(x)=(F_1(x),\ldots ,F_d(x))$, $x\in\R^d$, and
$B$ be a $d\times m$ matrix.
Ito's equation, formally written as
$$
dX(t)=F(X(t))\; dt+ B\;dW(t),\qquad X(0)=x,
$$
should be understood as an integral equation
\begin{equation}\label{eqito}
  X(t,\omega)=x+\int_0^t
  F(X(s,\omega))\; ds+ B\;W(t,\omega),\qquad t\ge0.
\end{equation}
We have the following result.

\begin{Theorem}
  Assume that the mapping $F$ satisfies a Lipschitz condition.
  Then for all $\omega\in\Omega$, $x\in\R^d$, the equation
  (\ref{eqito}) has a unique solution $X^x$. It is a Markov
  process starting from $x$ with respect to a transition function
  having the generator
\begin{equation}\label{genaddnoise}
 A\phi(x)= \<F(x),D\phi(x)\> +\frac{1}{2}{\rm \ Trace}\
 [QD^2\phi(x)]
\end{equation}
 where $Q=BRB^*$.
\end{Theorem}

The existence can be proved by a fixed point method in the
space of continuous functions $C([0,T],\R^d)$, for any $T>0$.
Note also that the form of $A$ is obvious if $Q=0$.

We close this
section by presenting examples.

The stochastic harmonic oscillator equation
\begin{equation}\label{stocarmeq}
  \frac{d^2X}{dt^2}=-\alpha X+\frac{dW}{dt},
  \qquad X(0)=x_0,\quad \frac{dX}{dt}(0)=v_0,
\end{equation}
can be written in a vector form:
\begin{equation}\label{stocarmsist}
d\left(\begin{array}{l}{X(t)}\\{v(t)}
\end{array}\right)
=\pmatrix{0&1\cr-\alpha&0}
\left(\begin{array}{l}X(t)\\{v(t)}
\end{array}\right)\,dt
+\left(\begin{array}{l}{0}\\{1}
\end{array}\right)
\,dW(t)
\end{equation}
where $W$ is a $1$-dimensional Wiener process and $\alpha\ge 0$.
It has the following solution:
$$
\begin{array}{lll}
X(t)&=&\dis
(\cos\sqrt{\alpha} t)x_0+ \frac{1}{\sqrt{\alpha}}
(\sin\sqrt{\alpha} t)v_0+
\frac{1}{\sqrt{\alpha}}\int_0^t \sin(\sqrt{\alpha} (t-s))\,dW(s)\\
v(t)&=&\dis
-\sqrt{\alpha}(\sin\sqrt{\alpha} t)
+(\cos\sqrt{\alpha} t)v_0+\int_0^t
\cos(\sqrt{\alpha}(t-s))\,dW(s).
\end{array}
$$
The integrals with respect to $W$ are the usual Stieltjes
integrals.

The equation (\ref{stocarmsist}) is a special case of the linear
equation
\begin{equation}\label{itolineare}
dX=AX\;dt + B\;dW, \qquad X(0)=x\in\R^d,
\end{equation}
where $A$ and $B$ are respectively  $d\times d $ and
$d\times m$  matrices and
$W$ is an $m$-dimensional Wiener process with covariance $R$.
Again the equation  (\ref{itolineare})
should be understood in the integral
form
$$
X(t)=x+\int_0^t AX(s)\;ds+B\; W(t) ,\qquad t\ge0.
$$
The solution to  (\ref{itolineare}) can be written explicitely
\begin{equation}\label{esplsolitolin}
 X(t)=e^{At}x+ \int_0^t
e^{A(t-s)}B\;dW(s)= e^{At}x+ B\;W(t)+ \int_0^t
Ae^{A(t-s)}BW(s)\;ds,
\end{equation}
and this can be easily checked by substitution.
Note also that
$$ \call (X(t))=N(e^{At}x,Q_t),
$$
where
$$
Q_t=\int_0^t e^{As}BB^*e^{A^*s}\;ds
$$
and
$$
P^t\phi(x)= \int_{\R^d} \phi (e^{At}x+y)\;
N(0,Q_t)(dy),\qquad t\ge0.
$$
Assuming that $\phi$ has continuous and bounded Fr\'echet
derivatives up to order $2$ one checkes, by applying Taylor's
formula to $\phi$, that
$$
 A\phi(x)= \lim_{t\downarrow 0} \frac{P^t\phi (x)-\phi(x)}{t}=
 \<Ax,D\phi(x)\> +\frac{1}{2}{\rm \ Trace}\
 [QD^2\phi(x)],\qquad x\in\R^d,
 $$
where $Q=BRB^*$, which confirms the formula
(\ref{genaddnoise}).

\section{Stochastic integrals}

We pass now to the general case when the diffusion $Q$ may depend
on the state and the jump measure does not vanish.
To introduce stochastic equations which determine Markov processes
with given characteristics we need stochastic
integrals with respect to a Wiener process $W$ on $\R^m$ and with
respect to a Poissonian random measure $\pi$ on $[0,+\infty)\times
\R^m$. Let  $W$ be  a $\R^m$-valued Wiener process
with covariance $R=(r_{kl})$ and  $\pi$  a
Poissonian measure with the intensity $\mu$ on $[0,+\infty)\times
\R^m$. We can assume that $W$ and $\pi$ are defined on a probability
space $(\Omega,\calf,\P)$ equipped with a filtration $(\calf_t)$
satisfying the usual conditions and that they satisfy the
following conditions:
\begin{enumerate}
  \item[1)] $W$ has continuous paths, $W(0)=0$.
  \item[2)] $W(t)$ is $\calf_t$-measurable and $\sigma(W(t)-W(s))$
  is independent of $\calf_s$ for all $t\ge s\ge0$.
  \item[3) ] $\pi(\Gamma)$ is $\calf_t$-measurable for each Borel
  set $\Gamma\subset [0,t]\times \R^m$
  and $\sigma(\pi(\Delta))$
  is independent of $\calf_t$ for each Borel set
  $\Delta\subset (t,+\infty)\times \R^m$.
  \item[4) ] The $\sigma$-fields generated by $W$ and $\pi$ are
  independent.
\end{enumerate}

\subsection{Integration  with respect to a Wiener process $W$.}

The aim is to define
$$
\int_0^t \Phi(s)\; dW(s),\qquad t\ge0,
$$
where $\Phi$ is a stochastic process whose values are in the space
$L(\R^m,\R^d)$ of linear operators from
$\R^m$ to $\R^d$. One starts with the simple integrands which are
of the form
\begin{equation}\label{integrandisempl}
  \Phi(s)=\sum_{j=1}^k 1_{(t_j,t_{j+1}]}(s)\, \Phi_j
\end{equation}
where $0\le t_1<t_2<\ldots<t_{k+1}$ is any finite sequence
and $\Phi_j$ is an $L(\R^m,\R^d)$-valued random variable,
$\calf_{t_j}$-measurable, such that $\E\, |\Phi_j|^2<+\infty$.
For $\Phi$ of the form
(\ref{integrandisempl}) one sets
$$
\int_0^t \Phi(s)\; dW(s)=
\sum_{j=1}^k\Phi_j\, ( W(t\wedge t_{j+1})-W(t\wedge t_j)).
$$
It is clear that the stochastic integral is a square-integrable
martingale with continuous trajectories. We have the following
fundamental {\em  isometric formula}:

\begin{Proposition}
  For an arbitrary simple process $\Phi$ and $t\ge0$ one has
\begin{equation}\label{isomitodimfin}
  \E\,\left|\int_0^t \Phi(s)\; dW(s)\right|^2
  =\E\,\int_0^t \|\Phi(s)R^{1/2}\|_{HS}^2\; ds,
\end{equation}
  where, for an arbitrary operator $C\in L(\R^m,\R^d)$,
$  \|C\|_{HS}= \left(\sum_{j=1}^m |Ce_j|^2\right)^{1/2}$ and
$e_1,\ldots e_m$ is an orthonormal basis of $\R^m$.
\end{Proposition}

\noindent{\bf Proof.}
Let $\Phi_j$ be represented as a matrix $(\Phi_j(k,l))$ and let
$W_1(t),\ldots ,W_m(t)$ be the components of $W(t)$. Then for
$t\ge s$ and $\Phi$, $\calf_s$-measurable and $W(t)-W(s)$
independent of $\calf_s$ we have
$$
\begin{array}{l}
  \E\, |\Phi_j(W(t)-W(s))|^2 \\
  \qquad =  \dis
\sum_{k=1}^d\E\, \left|\sum_{l=1}^m
\Phi_j(k,l)(W_l(t)-W_l(s))\right|^2
   \\
\qquad = \dis
\sum_{k=1}^d\; \sum_{l_1=1}^m\sum_{l_2=1}^m\E\, \left[
\Phi_j(k,l_1)\Phi_j(k,l_2)(W_{l_1}(t)-W_{l_1}(s))
(W_{l_2}(t)-W_{l_2}(s))\right]
\end{array}
$$
But
$$
\begin{array}{l}
\dis
\E\, \left[
\Phi_j(k,l_1)\Phi_j(k,l_2)(W_{l_1}(t)-W_{l_1}(s))
(W_{l_2}(t)-W_{l_2}(s))\; \vert \calf_s\right]
\\\dis
\qquad =
\Phi_j(k,l_1)\Phi_j(k,l_2)\,
\E\, \left[
(W_{l_1}(t)-W_{l_1}(s))
(W_{l_2}(t)-W_{l_2}(s))\right]
\\\dis
\qquad =
\Phi_j(k,l_1)\Phi_j(k,l_2)\,
r_{l_1,l_2} \, (t-s),
\end{array}
$$
and therefore
$$
 \E\, |\Phi_j(W(t)-W(s))|^2 =
 (t-s)\, {\rm \ Trace}\
 [\Phi_jR\Phi_j^*] =
  (t-s)\, \|\Phi_jR^{1/2}\|^2_{HS}.
$$
Thus the formula (\ref{isomitodimfin}) follows for $k=1$. The
general case can be obtained by induction and proper conditioning.
\qed

The isometric formula is of fundamental importance. It allows to
extend the integral to processes $\Phi$ which are limits of simple
processes in the norms
$$
\E\,\int_0^T \|\Phi(s)R^{1/2}\|_{HS}^2\; ds
$$
for each $T$.

These processes are again $(\calf_t)$-adapted, the integral is a
square integrable martingale and
the formula (\ref{isomitodimfin}) holds.
The integration can be extended to even larger classes of
integrands which are $(\calf_t)$-adapted and for which
$$
\P\left( \int_0^t \|\Phi(s)R^{1/2}\|_{HS}^2\; ds<+\infty,
\quad t\ge 0\right)=1.
$$

\subsection{Integration  with respect to a Poissonian measure $\pi$.}

Let $\pi$ be a Poissonian measure satifying
$3)$ with the intensity $\mu$.
We will denote by $\cale_0$ the family of all sets $\Gamma$ from
$\cale=\calb ([0,+\infty)\times \R^d)$ for which
$\mu(\Gamma)<+\infty$. The co called {\em compensated Poissonian
measure} $\check{\pi} $ is given by the formula
$$
\check{\pi}(\Gamma)=\pi(\Gamma)-\mu(\Gamma),\qquad
\Gamma\in\cale_0.
$$
Our aim is to define first stochastic integrals:
\begin{equation}\label{stelladue}
  \int_0^t\int_{\R^d}\phi(s,x)\; \check{\pi}(ds,dx),
\end{equation}
for a class of stochastic random fields $\phi$. Let us recall that
stochastic integrals of the type (\ref{stelladue}), with
deterministic $\phi$, were introduced in the section on random
measures. We will call a real-valued field $\phi$ {\em simple}
if there exist a finite sequence of
non-negative numbers $0\le t_1<t_2<\ldots<t_{k+1}$, a sequence
$\phi_j$ of square integrable  $\calf_{t_j}$-measurable random
variables and a sequence $\Gamma_j$ of subsets of $\R^d$ such that
$(  t_j,t_{j+1}]\times \Gamma_j\in\cale_0$. For simple fields $\phi$, we
define
$$
 \int_0^t\int_{\R^d}\phi(s,x)\; \check{\pi}(ds,dx)=
\sum_{j=1}^k\phi_j\,
\check{\pi}((  t_j \wedge t, t_{j+1}\wedge t]\times \Gamma_j).
$$
We have the following {\em isometric formula}:

\begin{Proposition}
  For simple random fields $\phi$:
\begin{equation}\label{isomrandomfield}
  \E\,\left| \int_0^t\int_{\R^d}\phi(s,x)\; \check{\pi}(ds,dx)
  \right|^2
  =\E\,\int_0^t \int_{\R^d}\phi^2(s,x)\;\nu( ds,dx).
\end{equation}
\end{Proposition}

\noindent{\bf Proof.}
As in the proof of the isometric formula for the Wiener process we
limit our calculations to $k=1$.
Let $0\le s<t$, $\phi$ be $\calf_s$-measurable random variable  and
$\mu((  s,t]\times \Gamma)<+\infty$. Then,
$$
  \E\,\left(| \phi\; \check{\pi}((  s,t]\times \Gamma)|^2
  \right)
  = \E\,\left(\phi^2\, \E\,(  |\check{\pi}((  s,t]\times \Gamma)|^2 \right)
  =(\E\,\phi^2)\;\mu((  s,t]\times \Gamma),
  $$
so (\ref{isomrandomfield}) holds in this case.
\qed

The isometric formula allows to extend the concept of integral to
all those random fields which are limits, in the norm defined
by the right hand side of (\ref{isomrandomfield}), of simple
random fields $\phi$. The formula (\ref{isomrandomfield}) remains
true for them and the stochastic integral is again a square
integrable martingale.  Doob's regularisation theorem allows to
claim that the stochastic integral
$$
 \int_0^t\int_{\R^d}\phi(s,x)\; \check{\pi}(ds,dx)
$$
has  a c\`adl\`ag version.

If $\phi$ is a simple field and
\begin{equation}\label{stellaquattro}
  \E\,\int_0^t \int_{\R^d}|\phi(s,x)|\;\mu( ds,dx)<+\infty,
  \qquad t\ge0,
\end{equation}
then the stochastic integral
\begin{equation}\label{stellacinque}
 \int_0^t\int_{\R^d}\phi(s,x)\; {\pi}(ds,dx),
  \qquad t\ge0,
\end{equation}
can be defined, starting again from simple fields $\phi$, but it
is not a martingale. Thus if $\phi$ is a limit, in
all the norms (\ref{stellaquattro}), of simple fields  then the stochastic integral
is a well defined process as the limit of stochastic integrals.

\subsection{Determining equations}

 To obtain Markov processes with general characteristics
$(F,Q,\nu)$ one needs Ito equations of the form:
\begin{equation}\label{stellona}
\begin{array}{lll}
  dX(t) & = &
  \dis f(X(t))\; dt+B(X(t))\; dW(t) \\
   &  &\dis  +
  \int_{\R^d}G_0(X(t-),y)\;\check{\pi}(dt,dy)
  +\int_{\R^d}G_1(X(t-),y)\;{\pi}(dt,dy), \\
  X(0) & = & x
\end{array}
\end{equation}
with properly chosen coefficients and random measure.
The equation (\ref{stellona}) should be understood as an integral
equation
\begin{equation}\label{stellonadue}
\begin{array}{l}\dis
X(t) = x+\int_0^t f(X(s))\; ds+
\int_0^tB(X(s))\; dW(s) \\
\dis\qquad+\int_0^t
  \int_{\R^d}G_0(X(s-),y)\;\check{\pi}(ds,dy)
  +\int_0^t\int_{\R^d}G_1(X(s-),y)\;{\pi}(ds,dy),
\end{array}
\end{equation}
and a solution $X$ should be a c\`adl\`ag,
$(\calf_t)$-adapted process.

Assume now that characteristics $(f,Q,\nu)$ are given. How the
coefficients of the equation (\ref{stellona}) should be defined to
obtain a solution $X$ with the given characteristics?
Following Ito's 1951 paper we write the operator $A$, from Theorem \ref{courrege}, in a more
convenient form:
$$\begin{array}{l}\dis
 A\phi(x)= \<f(x),D\phi(x)\> +\frac{1}{2}{\rm \ Trace}\
 [Q(x)D^2\phi(x)]
 \\\dis
 \qquad\qquad+ \int_{|x|<1} \left(
 \phi(x+y)-\phi(x)-\<y,D\phi(x)\> \right)\; \nu(x,dy)
  \\\dis
 \qquad\qquad
 +\int_{|x|\ge1} \left(
 \phi(x+y)-\phi(x) \right)\; \nu(x,dy),
 \end{array}
 $$
where
$$
f(x)= F(x)+
 \int_{|y|<1}\frac{y\,|y|^2}{1+|y|^2}\; \nu(x,dy)
 -\int_{|y|\ge1}\frac{y}{1+|y|^2}\; \nu(x,dy).
$$
Let $\pi$ be the Poissonian measure corresponding to the
intensity measure
$$
\mu(ds,dx)=ds\; \frac{1}{|x|^{d+1}}\,dx.
$$
Assume that for each $x$ there exists a transformation
$G(x,y)$ from $\R^d\backslash \{0\}$ into $\R^d$ such that
$$
\begin{array}{lll}
  |G(x,y)|<1 & {\rm for} & |y|<1, \\
  |G(x,y)|\ge 1 & {\rm for} & |y|\ge1,
\end{array}
$$
and that the image of the measure $|y|^{-d-1}dy$ by the
transformation $G(x,\cdot)$ is precisely $\nu(x,\cdot)$.
Define
$$
G_0(x,y)=\left\{
\begin{array}{ll}
  G(x,y), & {\rm if\;}  |y|< 1, \\
0 & {\rm otherwise},
\end{array}
\right.
$$
$$
G_1(x,y)=\left\{
\begin{array}{ll}
  G(x,y), & {\rm if\;}  |y|\ge 1, \\
0 & {\rm otherwise}.
\end{array}
\right.
$$
Let in addition, for a given matrix valued function $B$,
$B(x)RB^*(x)=Q(x)$, $x\in\R^d$. Then, under appropriate regularity
conditions on $f$, $B$, $G_0$, $G_1$, the equation
(\ref{stellona}) has a unique Markovian solution with the
characteristics $(F, Q, \nu)$.
We do not specify conditions under which the described
representation holds but  refer to  Ito's  papers and to the
third  volume of {\it The Theory of Stochastic Processes} by
I.I. Gihman and A.V. Skorohod \cite{skorochod}.
\vspace{2mm}

\noindent Let us finally notice that the Poissonian measure, used in the
construction, is exactly the one which defines a L\'evy process
with the exponent function $\psi(\lambda)=|\lambda|$,
$\lambda\in\R^d$. This process is called a {\it symmetric Cauchy
process}. Therefore, with some abuse of language, we can say that
Markov processes can be effectively constructed from  uniform,
Brownian and Cauchy motions.

\chapter{Stochastic integration in infinite dimensions}
Of great interest in applications are  dynamical systems for which the state
space is infinite dimensional. Specific examples are provided by heat and reaction-diffusion processes or by
electromagnetic waves.
Deterministic dynamical systems of this type are described as solutions of partial differential equations of
evolutionary types. They can be  regarded as solutions of differential equations in infinite dimensional spaces.
Stochastic dynamical systems in infinite dimensions, or equivalently, Markov processes in infinite dimensions, can
be constructed  using appropriate generalisation of the theory of stochastic integration in Hilbert spaces.
To treat stochastic evolution equations with
respect to both continuous and discontinuous noise processes it is convenient to develop  the theory
of   stochastic integration   with respect to square
integrable Hilbert-valued martingales. To motivate better the material we start from more classical results when the
martingale is a Wiener process.
\section{Integration with respect to a Wiener process}
\subsection{Wiener process on a Hilbert space}
Wiener processes which are used
in applications evolve on  a Hilbert or a Banach space $U$ or also on linear topological spaces like  the space of Schwartz distributions.
The general  definition is as follows.  A
stochastic process $W(t)$, $t\geq 0$, with values in $E\,,$ is called a Wiener
process if
\begin{enumerate}
\item $W(t)$, $t\geq0$ has continuous paths and starts from $0$.
\item The increments $W(t_2) - W(t_1), \ldots, W(t_{n}) - W(t_{n-1})$,
for $0 \leq t_1 < t_2 < \ldots < t_{n}$ are independent random variables:
\begin{equation}
    \P \left(W(t_2) - W(t_1) \in \Gamma_1, \ldots,
    W(t_n) - W(t_{n-1}) \in \Gamma_{n-1}
    \right) = \prod_{i=1}^{n-1} \Pr \left\{ W(t_{i+1}) - W(t_i) \in \Gamma_i \right\}.
\end{equation}
\item For each measurable set  $\Gamma \subset E$ and all $t, h \geq 0$
$$
    \P\left(W(t+h) - W(t) \in \Gamma \right) =
    \P\left(W(h) \in \Gamma \right).
$$
\end{enumerate}
In other words, an $E$-valued, continuous, stochastic process is a Wiener
process if and only if for each $x^*$ in the space  $E^*$ of all continuous linear functionals on $E$, the real valued process $\langle
W(t), x^*\rangle$ is a 1-dimensional Wiener process.  All processes and random
variables are always assumed to be defined on a fixed probability space
$(\Omega, {\cal F}, \Pr)$.

\noindent It is well known that 1-dimensional Wiener processes have Gaussian
distributions.  A probability measure $\mu$ on $\R^1$ is said to be Gaussian
if it is either concentrated at a point $m\in \R^1$ or has a density
$$
\frac{1}{\sqrt{2\pi q}} e^{-\frac{1}{2q}(x-m)^2}, \quad x \in \R^1.
$$
It is denoted by $N(m,q)$.  If $q = 0$, $N(m,0) = \delta_{\{m\} }$.  More
generally a probability measure $\mu$ on a linear topological space is
Gaussian if every functional  $x^* \in E^*$, considered as a random variable on $(E,{\cal
B}(E),\mu)$, is a real Gaussian random variable. Equivalently if linear functionals transport
$\mu$ onto some $N(m,q))$.

We will additionally require that random variables $W(t)$ are  symmetric that is,
$$
\E < W(t), x^* > = 0, \quad t \geq 0,\, x^* \in
E^*.
$$
By far the most important Wiener processes are those which evolve on Hilbert
spaces.  In a sense all other Wiener processes are variants of Hilbert space
valued Wiener processes.  Basic properties of Gaussian Hilbert space valued
random variables are gathered in the following theorem.

\begin{Theorem}
Assume that $\xi$ is a symmetric Gaussian random variable with values in a separable
Hilbert space $H$.  Then for sufficiently small $s > 0$,
$$
\E \left( e^{s |\xi|^2 } \right) < +\infty.
$$
In particular all moments $\E |\xi|^k < +\infty$, $k = 1,2,\ldots$ are finite.
\end{Theorem}
It follows from the theorem that $\E |\xi| \< +\infty$ and $E\xi =0$ and
there exists a non-negative operator $Q$ such that
$$
    \E \langle \xi , a \rangle \langle \xi , b \rangle =
    \langle Qa, b \rangle.
$$
A non-negative operator $Q$ is said to be trace class if and only if
for an arbitrary orthonormal complete sequence $(e_n)$:
$$
\sum_{n=1}^{+\infty} \langle Qe_n, e_n \rangle < +\infty
$$

We have also
\begin{Theorem}
\label{thrm:Gaussian-moments} The covariance operator $Q$ is compact and trace
class. For arbitrary $m \in H$ and arbitrary trace
class non-negative  operator $Q$ there exists a unique Gaussian
 measure corresponding to
$(m,Q)$ and denoted by  $N(m,Q)$ or $N_{m,Q}$.
\end{Theorem}

\noindent{\bf Proofs.}
\begin{Lemma}
Assume that $\xi$ takes values in $\R^d$, is Gaussian and such that $\E\xi =
0$.  Then for all
$$
    s < {\frac{1}{2\E|\xi|^2}}, \quad \rm{we \; have}\quad
    \E\left(e^{s|\xi|^2}\right) \leq \frac{1}{\sqrt{1 - 2s \E |\xi|^2}}.
$$
\end{Lemma}

\noindent{\bf Proof.}
We consider first the case $s < 0$ and denote $s = -\lambda$, $\eta =
\sqrt{\lambda} \xi$.  Then $\E e^{-\lambda|\xi|^2} = \E e^{-|\eta|^2}$.  Let
the matrix $Q$ be the covariance of $\eta$, and $e_1, \ldots , e_d$ an
orthonormal basis such that
$$
Qe_j = \lambda_j e_j, \quad j = 1, \ldots, d
$$
where $\lambda_1, \ldots, \lambda_d$ are the eigenvalues of $Q$.  Then
$$
\eta = \sum_{j=1}^d \langle \eta,e_j \rangle e_j = \sum_{j=1}^d \eta_j e_j
$$
where $\eta_1,\ldots,\eta_d$ are independent with
$N0,\lambda_1),\ldots,N(0,\lambda_d)$ distributions.  Consequently
\begin{equation}
    \E \left( e^{-|\eta|^2} \right) =
    \E \left( e^{-\sum_{j=1}^d \eta_j^2 }\right) =
    \prod_{j=1}^d \E \left( e^{-\eta_j} \right)=
    \prod_{j=1}^d \frac{1}{\sqrt{1 + 2\lambda_j}}
\end{equation}
$$=
    \frac{1}{\sqrt{(1 + 2\lambda_1) \cdots (1 + 2\lambda_d) }} \leq
    \frac{1}{\sqrt{1 + 2(\lambda_1 + \cdots \lambda_d)}}.
$$
Since $\E |\eta|^2 = \lambda_1 + \cdots + \lambda_d$ the result follows. The
proof for $s \in (0,\frac{1}{2\E|\xi|^2})$ is similar.
\qed

\noindent{\bf Proof }
of Theorem \ref{thrm:Gaussian-moments}.
\newcommand{\xidsq}{|\xi_d|^2}
Choose an arbitrary
orthonormal basis $(e_n)$.  Then
$$
\xi = \sum_{n=1}^{+\infty} < \xi , e_n > e_n
$$
and $\xi = \lim_d \xi_d$ where $\xi_d = \sum_{n=1}^d < \xi,e_n > e_n$.
Moreover $\xi_d$ is a $d$-dimensional Gaussian random variable and by the lemma
$$
\E \left( e^{-|\xi_d|^2} \right) \leq \frac{1}{\sqrt{1 + 2 \E|\xi_d|^2}}.
$$
Note that $|\xi_d|^2 \uparrow |\xi|^2$ so $\E |\xi_d|^2 \uparrow \E|\xi|^2$
and therefore
$$
\E\left(e^{-|\xi|^2} \right) \leq \frac{1}{\sqrt{1 + 2\E |\xi|^2 }}
$$
If $\E |\xi|^2 = +\infty$ then $|\xi|^2 = +\infty$ with probability $1$, which
is not the case, so $\E |\xi|^2 < +\infty$.  But then the lemma is true also for
infinite dimensional Gaussian random variables using the fact that $|\xi_d|^2
\uparrow |\xi|^2$.  This completes the proof of Theorem \ref{thrm:Gaussian-moments}.
\qed
\vspace{2mm}

\noindent Since
$$
    \E |\xi|^2 = \E \left( \sum_{n}^{+\infty} < \xi , e_n \>^2 \right)
    = \sum_{n}^{+\infty} \E < \xi, e_n \>^2 = \sum_{n}^{+\infty} < Qe_n,
    e_n >,
$$
Trace $Q$ $ < +\infty.$ \qed
\vspace{3mm}

Thus distributions of any Hilbert space valued Wiener process $W$ are
uniquely determined by a non-negative trace class operator $Q$ being the
covariance of $W(1)$.  Moreover
$$
W(t) = \sum_{n}^{+\infty} \sqrt{\lambda_n}\beta_n(t) e_n \quad t \geq 0
$$
where $\beta_1, \beta_2, \ldots$ are independent real Wiener processes and
$(\lambda_n,e_n)$ is the eigensequence of $Q$.

We assume now that $H = L^2({\cal O},\mu)$ where $\mu$ is a finite,
non-negative measure.  Then the operator $Q$ is an integral operator
$$
Q\varphi(x) = \int\limits_{{\cal O}} q(x,y) \varphi(y) \mu(dy)
$$
and the (positive definite) function $q$ is the spatial correlation.  If
$W(t,x)$, $t \geq 0, x \in {\cal O}$ is the $Q$-Wiener process then
$$
\E W(t,x) W(s,y) = (t \wedge s) q(x,y).
$$
This heuristic formula can be justified in many specific cases.
\vspace{2mm}

\noindent {\bf Examples} Define
$$
    W(t,x) = \sum_{n}^{+\infty} \frac{\beta_n(t)}{n} \sqrt{\frac{2}{\pi}} \sin nx
    \quad x \in (0,\pi).
$$
Then $q(x,y) = x \wedge y - \frac{1}{\pi} xy$, $x,y \in (0,\pi)$ and for each
fixed $t$, $W(t,\cdot)$ is a Brownian bridge process.
\vspace{2mm}

\noindent Define
$$
    \widetilde{W}(t,x) = \sum_{n}^{+\infty} \frac{\beta_n(t)}{n+\frac12}
    \sqrt{\frac{2}{\pi}} \sin \left[ \left(n+\frac12\right) x \right]
    \quad x \in (0,\pi).
$$
Then $\widetilde{q}(x,y) = x \wedge y$.  For each fixed $t$,
$\widetilde{W}(t,\cdot)$ is a real valued Brownian motion.  The process
$\widetilde{W}$ is also called Brownian sheet on $[0,+\infty) \times [0,\pi]$.

\subsection{Stochastic integration with respect to a Wiener process}
%\subsection{Stochastic integration and reproducing kernels of Wiener processes}
To deal with stochastic equations one needs the concept of stochastic
integrals:
$$
\int_0^t \Phi(s) dW(s)
$$
where $\Phi(s,\omega)$ are operators from $U$ into another Hilbert space $H$.
\vspace{2mm}

\noindent Let us assume that the process $W$ is defined on a probability space
$(\Omega, \cal F, \P)$ and that ${\cal F}_t$ is an increasing family of $\sigma $ fields contained in $\cal F$ and such
that $W(t)$ is ${\cal F}_t$ measurable, for any $t\geq 0,$ and that for any $t\geq s\geq 0$, $W(t) - W(s)$ is independent of
${\cal F}_s$. If this is the case one says that $W$ is a {\it Wiener process with respect} to ${\cal F}_t$. The following
proposition follows directly from the definitions.
\begin{Proposition}\label{midentities}
Let $W$ be a {\it Wiener process with respect} to ${\cal F}_t$ with the covariance operator $Q$. Then
for arbitrary vectors $a, b \in U$ and arbitrary $0\leq s \leq t \leq u \leq v$
\begin{equation}
\E (<W(t) - W(s), a> <W(t) - W(s), b> |{\cal F}_s) = (t- s)<Q a, b>,
\end{equation}
\begin{equation}
E (<W(t) - W(s), a> <W(v) - W(u), b> |{\cal F}_u)= 0.
\end{equation}
\end{Proposition}
Let $L(U, H)$ be the Banach  space of linear continuous operators from $U$ into $H$.  An $L(U, H)$stochastic process $\Phi $
 is said to be simple if there
exist a   sequence of nonnegative numbers  $t_{0}= 0 <t_1 < \ldots < t_m$,  a sequence of operators $\Phi_{j}\in L(U, H),
j= 1, \ldots, m,$
and  a sequence of events $ A_j\in {\cal F}_{t_j}, j= 0, \ldots, m-1, $ such that
$$
\Phi (s) = \sum_{k= 0}^{m- 1} {\bf 1}_{A_k}\,\,{\bf 1}_{]t_k, t_{k+ 1}]}(s),\,\,\,s\in [0, +\infty [,
$$
where  ${\bf 1}_{B}$ denotes the indicator function of the set $B$. For simple processes $\Phi$ we set:
$$
\int_0^t \Phi(s) dW(s)= \sum_{k= 0}^{m- 1} \Phi_{k} (W(t_{k+ 1}\wedge t)- W(t_{k }\wedge t)),\,\,\,t\in ]0,+\infty [.
$$
Let $L_{HS}(U,H) $ be the space of all Hilbert- Schmidt operators, from $U$ into $H$, equipped with
the Hilbert- Schmidt norm $||\cdot ||_{L_{HS}(U, H) }$.

\begin{Proposition}\label{isometricW}
For arbitrary simple process $\Phi$:
$$
\E \left| \int_0^t \Phi(s)dW(s) \right|^2 =
    \E \int_0^t \|\Phi(s)Q^{1/2} \|^2_{L_{HS}(U,H)} ds,
    \quad t \geq 0.
$$
\end{Proposition}
\noindent{\bf Proof}
By the very definition
$$
\E \left| \int_0^t \Phi(s)dW(s) \right|^2 = \E |\sum_{k= 0}^{m- 1} {\bf 1}_{A_k}
\Phi_{k} (W(t_{k+ 1}\wedge t)- W(t_{k }\wedge t))|^2
$$
$$
= \E \sum_{k, l= 0}^{m- 1}{\bf 1}_{A_k}{\bf 1}_{A_l}<\Phi_{k}(W(t_{k+ 1}\wedge t)- W(t_{k }\wedge t)), \Phi_{l}
(W(t_{l+ 1}\wedge t)- W(t_{l }\wedge t))>.
$$
By formulas (\ref{midentities}),
\begin{equation}\label{midentities1}
\E (<W(t_{k+ 1}\wedge t)- W(t_{k }\wedge t), a> <W(t_{l+ 1}\wedge t)- W(t_{l}\wedge t), b>|{\cal F}_{t_{l}})=0,
\end{equation}
if $k\neq l$ and is equal
\begin{equation}\label{midentities2}
(t_{k+ 1}\wedge t - t_{k }\wedge t) <Q a, b>,
\end{equation}
if $k= l$.

\noindent Let $e_{j}$ be  an orthonormal basis in $U$ and $k\leq l$. We have
$$
I_{klj}= \E (<\Phi_{k}(W(t_{k+ 1}\wedge t)- W(t_{k }\wedge t)), e_j> <\Phi_{l}
(W(t_{l+ 1}\wedge t)- W(t_{l }\wedge t)), e_j >|{\cal F}_{t_{l}})
$$
$$
=\E (< (W(t_{k+ 1}\wedge t)- W(t_{k }\wedge t)), \Phi_{k}^{*} e_j> <
(W(t_{l+ 1}\wedge t)- W(t_{l }\wedge t)), \Phi_{l}^{*}e_j >|{\cal F}_{t_{l}}).
$$
By formulas  $(\ref{midentities})$ if $k < l$,
$$
\E I_{klj}{\bf 1}_{A_k}{\bf 1}_{A_l}= 0
$$
and if $k=l$,
$$
\E I_{kkj}{\bf 1}_{A_k}= \E {\bf 1}_{A_k}(t_{k+ 1}\wedge t - t_{k }\wedge t) <Q \Phi_{k}^{*}e_{j} ,  \Phi_{k}^{*}e_{j}>.
$$
But
$$
\sum_{j} <Q \Phi_{k}^{*}e_{j} ,  \Phi_{k}^{*}e_{j}> = \sum_{j} |Q^{1/2} \Phi_{k}^{*}e_{j}|^2 =
|Q^{1/2} \Phi_{k}^{*}|^{2}_{L_{HS}(U,H)},
$$
and since
$$
|Q^{1/2} \Phi_{k}^{*}|^{2}_{L_{HS}(U, H)}=|\Phi_{k}Q^{1/2}|^{2}_{L_{HS}(U, H)},
$$
the required identity follows by elementary calculations.\qed

\noindent  Note that the space
$$
\widetilde{U} = Q^{1/2}(U)
$$
equipped with the scalar product
$$
< a, b >{\widetilde{U}} = < Q^{-1/2}a, Q^{-1/2}b >_U,
$$
is a Hilbert space.
\vspace{2mm}

\noindent Let $L_{HS}(\widetilde{U},H) $ be the space of all Hilbert- Schmidt operators equipped with
the Hilbert- Schmidt norm $||\cdot||_{HS}=  ||\cdot ||_{L_{HS}(\widetilde{U},H) }$. By Proposition \ref{isometricW}, for simple $\Phi$,
\begin{equation}\label{isometricW1}
    \E \left| \int_0^t \Phi(s)dW(s) \right|^2 =
    \E \int_0^t \|\Phi(s) \|^2_{HS} ds,
    \quad t \in [0,T].
\end{equation}
The definition of the stochastic integral and the above formula \ref {isometricW1} extend to all adapted,
$L_{HS}(\widetilde{U},H)$ valued
processes for which the right hand side of (\ref{isometricW1}) is finite. By the same localisation procedure as in finite
dimensions the concept of the stochastic integral and all its basic properties can be extended to all adapted
$L_{HS}(\widetilde{U},H) $- processes
$\Phi(s)$ for which
$$
    \P \left(
    \int_0^t \| \phi(s) \|_{HS}^{2} ds <  +\infty,\,\,
    \ t \geq 0 \right) = 1.
$$

\section{Stochastic integration with respect to martingales}
\subsection{Introduction}
If $U$ is  a separable Hilbert space
and  $W(t)$, $t\ge0$, a $U$-valued Wiener process, with covariance
$Q$, then the   isometric formula holds,
$$
  \E\,\left|\int_0^t \Phi(s)\; dW(s)\right|^2
  =\E\,\int_0^t \|\Phi(s)Q^{1/2}\|_{HS}^2\; ds.
  $$
 To prove  the isometric formula in the general case, when the integrator is an $U$- valued square integrable
 martingale $M$, one needs to introduce bracket processes
$$
\l M, M \r _t,\,\,\, \ll M, M \rr _{t}.
$$
The former  is a real process  and the
latter an  operator valued. It turns out that there
exists an operator valued stochastic process $Q_s, s\geq 0$ \, such that,
\begin{equation}\label{mp1}
\ll M, M \rr _t = \int_{0}^{t} Q_{s} d \l M, M \r _s,\,\,\,t\geq 0.
\end{equation}
The values of the process $Q_s, s\geq 0,$ are non-negative, trace class operators.
In the case when $M$ is the Wiener process $W$, then
$$
\l W, W \r _t = t\,\, {\rm Trace}\, Q,\,\,\,\,\,  \ll W, W \rr _t = Q/{\rm Trace}\, Q,\,\,t\geq 0.
$$
The isometric formula in the general case was discovered by  M\'etivier and G. Pistone \cite{metivier1} and is the form
\begin{equation}\label{mp2}
  \E\,\left|\int_0^t \Phi(s)\; dM(s)\right|^2
  =\E\,\int_0^t \|\Phi(s)Q_s^{1/2}\|_{L_{HS}(U, H)}^2\; d\l M, M \r _s.
\end{equation}

\subsection{Doob-Meyer decomposition}

Assume that $U$ is a separable Hilbert space,
with norm $|\cdot|$ and scalar product $(\cdot,\cdot)$,
and $(\Omega, \calf, \P)$ a probability space equipped with an
increasing family of $\sigma$-fields $\calf_t\subset\calf$,
$t\in I$. The set $I$ will be either a bounded interval $[0,T]$ or
$[0,\infty)$. A $U$-valued family of random variables $X(t)$,
$t\in I$ is called an integrable process if
\begin{equation}\label{integrproc}
  \E\, |X(t)|<+\infty,\qquad t\in I.
\end{equation}
If for all $t\in I$, $X(t)$ is an $\calf_t$-measurable random
variable, then the family $X$ is called an adapted process. An
integrable and adapted process $X(t)$, $t\in I$, is said to be a
martingale if
\begin{equation}\label{defmartingale}
  \E\, (X(t)\, |\, \calf_s)=X(s),\qquad
  \P-a.s.
\end{equation}
for arbitrary $t,s\in I$, $t\geq s$. The identity
(\ref{defmartingale}) is equivalent to the statement
\begin{equation}\label{defmartingaleequiv}
  \int_FX(t)\;d\P=  \int_FX(s)\;d\P,\qquad
  F\in \calf_s,\; s\leq t, \; s,t\in I.
\end{equation}
A real valued integrable and adapted process
$X(t)$, $t\in I$ is said to be a submartingale (resp. a
supermartingale) if
$$
  \E\, (X(t)\, |\, \calf_s)\geq X(s),
\qquad {\rm (resp.\;\;}
  \E\, (X(t)\, |\, \calf_s)\leq X(s)
  {\rm )},\qquad
  \P-a.s.
$$

\begin{Lemma}
\begin{enumerate}
  \item[i)] If $M(t)$, $t\in [0,T]$ is a $U$-valued martingale,
  then $|M(t)|$, $t\in [0,T]$ is a submartingale.
  \item[ii)] If $g$ is an increasing, convex function from
  $[0,+\infty)$ into   $[0,+\infty)$ and $\E\, g(|M(t)|)<+\infty$
  for $t\in [0,T]$, then $g(|M(t)|)$, $t\in [0,T]$ is a
  submartingale.
\end{enumerate}
\end{Lemma}

\noindent{\bf Proof.}
$i)$ Let $t,s\in [0,T]$, $t>s$, then
$$
|M(s)|=|\E\, (M(t)\,|\,\calf_s)|\leq
\E\, (|M(t)|\,|\,\calf_s)
$$
as required.

$ii)$ Since $|M(s)|\leq\E\, (|M(t)|\,|\,\calf_s)$,
$\P$-a.s., $s<t$, then monotonicity and convexity
of $g$ together with Jensen's inequality imply:
\begin{equation}\label{convmart}
  g(|M(s)|)\leq   g(\E\, (|M(t)|\,|\,\calf_s))\leq
\E\,(  g (|M(t)|)\,|\,\calf_s)
\end{equation}
as required. \qed

>From now on we will assume that the martingale
$M(t)$, $t\in [0,T]$ is right-continuous
and square-integrable:
$$
\E\, |M(t)|^2<+\infty,\qquad t\in [0,T].
$$
Then by the lemma,
$|M(t)|^2$, $t\in [0,T]$ is a submartingale.
Denote by $\calp_T$ the smallest $\sigma$-field of subsets of
$[0,T]\times \Omega$ containing all sets of the form:
$(s,t]\times \Gamma$, where $0\leq s<t\leq T$ and $\Gamma\in
\calf_s$, and also $\{0\}\times \Gamma$ where $\Gamma\in\calf_0$.
A stochastic process $X(t)$, $t\in [0,T]$ with values in a
measurable space is called predictable if the function
$X(t,\omega)$, $t\in [0,T]$, $\omega\in \Omega$ is measurable
with respect to the $\sigma$-field $\calp_T$. We will need the
following fundamental result, see e.g. \cite{rogers}.

\begin{Theorem}
{\em (Doob-Meyer decomposition).}
  For arbitrary c\`adl\`ag real valued submar\-tingale $Y$
  there exists a unique predictable and right-continuous
  increasing process $A(t)$, $t\in [0,T]$ such that $A(0)=0$ and
\begin{equation}\label{meyerdecomp}
  M(t)=Y(t)-A(t),
  \qquad t\in [0,T]
\end{equation}
is a martingale.
\end{Theorem}

\noindent We give only a proof of the discrete time version of the theorem.
\begin{Proposition}
  Assume that $Y_0,Y_1,\ldots,Y_N$ is a submartingale with respect
  to $\calf_0$, $\calf_1$, $\ldots,\calf_N$. Then
\begin{equation}\label{meyerdecompdiscr}
  Y_n=M_n+A_n,\qquad n=0,1,\ldots ,N
\end{equation}
where $(M_n)$ is an $(\calf_n)$ martingale and $(A_n)$ is an
increasing sequence such that
\begin{equation}\label{incprocdiscr}
  A_0=0,\qquad
  A_n {\rm \; is \;} \calf_{n-1} {\rm\; measurable
  \; for\;} n=1,2,\ldots, N.
\end{equation}
The decomposition (\ref{meyerdecompdiscr}) with the
specified properties is unique.
\end{Proposition}

\noindent{\bf Proof.}
Define
$$
A_0=0\qquad {\rm and}\qquad A_n=\sum_{k=0}^{n}\left(
\E\, (Y_k\,|\, \calf_{k-1})-Y_{k-1}\right),
\quad n=1,2,\ldots, N.
$$
Since $(Y_k)$ is a submartingale we have
$$
\E\, (Y_k\,|\, \calf_{k-1})\geq Y_{k-1},
\qquad k=1,2,\ldots, N
$$
and therefore the sequence $(A_n)$ is increasing. It is also clear
that $A_n$ is $\calf_{n-1}$ measurable. If now
$$
M_n=Y_n-A_n,\qquad n=0,1,\ldots, N,
$$
then, for $n=1,2,\ldots, N$,
$$
M_n-M_{n-1}=Y_n-Y_{n-1} -(A_n-A_{n-1})=
Y_n-\E\, (Y_n\,|\,\calf_{n-1}).
$$
Consequently
$$
\E\, (M_n-M_{n-1}\,|\,\calf_{n-1})
=\E\, (Y_n\,|\,\calf_{n-1})-\E\, (Y_n\,|\,\calf_{n-1})=0,
$$
and therefore the sequence $M_n$, $n=0,1,\ldots, N$ is a
martingale as required. To show uniqueness assume that
$$
Y_n=M'_n+A'_n,\qquad
n=0,1,\ldots, N,
$$
for a martingale $(M')$ and an increasing sequence
$(A'_n)$ satisfying (\ref{incprocdiscr}). Then
\begin{equation}\label{unicmeyer}
  M_n-M_n'=A_n'-A_n,\qquad
n=0,1,\ldots, N
\end{equation}
and
$$
  M_{n-1}-M_{n-1}'= \E\, (  M_{n}-M_{n}'\,|\,\calf_{n-1})
=\E\, (  A_{n}'-A_n\,|\,\calf_{n-1})=A_{n}'-A_n,
$$
because $A_{n}'-A_n$ is $\calf_{n-1}$ measurable.
By (\ref{unicmeyer}) we therefore have
$$
 M_{n-1}-M_{n-1}'= M_{n}-M_{n}',
 \qquad n=1,2,\ldots, N,
 $$
 since $M_0=M_0'=Y_0$ the uniqueness follows.
\qed

If $M(t)$, $t\in [0,T]$ is a square-integrable
and right-continuous martingale, with values in a separable
Hilbert space $U$, then  $|M(t)|^2$, $t\in [0,T]$
is a right-continuous real valued submartingale
and the corresponding predictable process is denoted by
$\l M,M \r _t$, $t\in [0,T]$ and called the (angle) bracket process.
If $M$ and $N$ are two square-integrable right-continuous
martingales then there exists a unique right-continuous, predictable
process denoted by $\<M,N\>_t$, $t\in [0,T],$ with bounded
variation, such that
$$
<M(t),N(t)>-\<M,N\>_t, \qquad t\in [0,T]
$$
is a martingale. It is given by a polarization identity,
$$
\<M,N\>_t=\frac{1}{2}\left( \<M+N,M+N\>_t
-\<M,M\>_t -\<N,N\>_t\right), \qquad t\in [0,T].
$$

\begin{Example}\begin{em}
  Assume that $M(t)=W(t)$, $t\geq 0$ is a $U$-valued Wiener
  process. Then
  $$
  \<W,W\>_t=t\; {\rm Trace \;} Q,
  $$
where $Q$ is the covariance operator of $W(1)$.

In fact, for $0\leq s\leq t\leq T$, by the independence of the
increments,
$$
\begin{array}{lll}
  \E\, (|W(t)|^2\, |\, \calf_s) & = &
    \E\, (|(W(t)-W(s))+W(s)|^2\, |\, \calf_s) \\
    & = &     \E\, ((|W(t)-W(s)|^2+|W(s)|^2
    +2 <W(t)-W(s),W(s) >)\, |\, \calf_s) \\
    & = &  \E\, |W(t)-W(s)|^2+|W(s)|^2 \\
    & = & (t-s)\, {\rm Trace \;}Q+|W(s)|^2.
\end{array}
$$
Consequently
$$
 \E\, (|W(t)|^2-t\, {\rm Trace \;} Q\, |\, \calf_s)
 =|W(s)|^2-s\, {\rm Tr\;} Q,
 $$
 as required.
\end{em}
\end{Example}

With the same proof we have a more general result.

\begin{Proposition}
Assume that $M(t)$, $t\in [0,T]$ is a square-integrable,
right-continuous process, having independent, time-homogeneous
increments, with respect to the family $(\calf_t)$, $t\in [0,T]$,
and starting from $0$. Then
  $$
  \<M,M\>_t=t\, {\rm Trace  \;} Q,
  $$
where $Q$ is a trace class, non-negative operator on $U$ such that
\begin{equation}\label{defqindip}
  <Qa,b>=\E\, [<M(1),a> <M(1),b>],\qquad a,b\in U.
\end{equation}
\end{Proposition}

Note that the relation (\ref{defqindip}) defines the operator $Q$
uniquely. Its symmetricity and non-negativity is obvious. Since,
for an arbitrary, complete orthonormal basis $(e_n)$ in $U$,
$$
\begin{array}{lll}
  {\rm Trace \;} Q & = &
  \dis \sum_n <Qe_n,e_n>=\sum_n \E\, (<M(1),e_n>^2)  \\
    & = & \dis \E\, (\sum_n  <M(1),e_n>^2)=
    \E\,|M(1)|^2<+\infty,
\end{array}
$$
the trace of $Q$ is finite.
As for the Wiener process one shows that
$\E |M(t)|^2=t \E |M(1)|^2$ for $t\ge 0$.

\subsection{Operator valued angle bracket process}
For stochastic integration we need a generalisation of the angle
bracket. Denote by $L_1(U)$ the space of all nuclear operators on
$U$ equipped with the standard nuclear norm. Then $L_1(U)$ is a
separable Banach space. By $a\otimes b$ we denote a rank one
operator on $U$ given by:
$$
a\otimes b\; (u)=a\, <b,u>,\qquad u\in U.
$$
It is easy to show that $\|a\otimes b\|_{L_1(U)}=|a|\, |b|$,
$a,b\in U$.
By $L_1^+(U)$ we denote the subspace of $L_1(U)$ consisting of all
self-adjoint, non-negative, nuclear operators.
If $M(t)$, $t\in [0,T]$ is a right-continuous,
square-integrable martingale then the process
$$
M(t)\otimes M(t),\qquad t\in [0,T],
$$
is an $L_1(U)$ valued, right-continuous process such that
$$
\E\, \|M(t)\otimes M(t)\|_{L_1(U)}=
\E\, |M(t)|^2\leq \E\, |M(T)|^2<+\infty,
\qquad t\in [0,T].
$$
We have the following basic result, mentioned in the introduction to the section, from which the
isometric formula will follow.
\begin{Theorem}\label{MAIN}
  There exists a unique right-continuous, $L_1^+(U)$ valued,
  increasing, predictable process $\<\<M,M\>\>_t$, $\in [0,T]$,
  $\<\<M,M\>\>_0=0$ such that the process
  $$
  M(t)\otimes M(t)-\<\<M,M\>\>_t,\qquad t\in [0,T],
  $$
  is an $L_1(U)$ martingale. Moreover there exists a predictable
  $L_1^+(U)$ valued process $Q_t$, $t\in [0,T]$ such that
\begin{equation}\label{defqtdim}
  \<\<M,M\>\>_t=\int_0^tQ_s\; d\<M,M\>_s
\end{equation}
\end{Theorem}

\noindent{\bf Proof.}
Let $(e_k)$ be an orthonormal and complete basis in $U$. Then
$$
M(t)=\sum_{k=1}^{+\infty} <M(t),e_k>e_k,\qquad t\in [0,T].
$$
Denote the process
$<M(t),e_k>$, $t\in [0,T]$ by $M^k(t)$, $t\in [0,T]$. Since
$$
\E\, |M(t)|^2=\E\, \sum_{k=1}^{+\infty}
(M^k(t))^2
=\sum_{k=1}^{+\infty}\E\, (M^k(t))^2<+\infty
,\qquad t\in [0,T],
$$
the processes $M^k$ are right-continuous, square-integrable, real
martingales. Let $\<M^k,M^j\>$, $k,j=1,2,\ldots$ be the
corresponding angle bracket processes. Set
\begin{equation}\label{defanglebrakvect}
  \<\<M,M\>\>_t=\sum_{k,l=1}^{+\infty}e_k\otimes e_l
  \; \<M^k,M^l\>_t,\qquad t\in [0,T].
\end{equation}
Thus the infinite matrix $(\<M^k,M^l\>_t)$ is a representation of
the operator $\<\<M,M\>\>_t$ in the basis $(e_k)$.
To see that (\ref{defanglebrakvect}) defines an $L_1^+(U)$ valued
process we identify $U$ with $l^2$ and for each $a=(a^j)\in l^2$
consider a process
$$
<a,M(t)>=\sum_{j=1}^{+\infty} a^jM^j(t),\qquad t\in [0,T].
$$
It is a square-integrable right-continuous martingale such that
for $a,b\in l^2$
$$
\< <a,M>, <b,M>\>_t= \lim_N \sum_{i,j=1}^N a^ib^j \<M^i,M^j\>_t
= (\<\<M,M\>\>_ta,b).
$$
with  the limit existing in $L^1(\Omega)$. The inequality
$|\<M^i,M^j\>_t|^2\leq \<M^i,M^i\>_t\<M^j,M^j\>_t$, needed for the prof of the convergence, follows
from the proof of the lemma below.
This proves that the operator-valued function $\<\<M,M\>\>$ is a
well defined $L(U)$ valued function. The operators
$\<\<M,M\>\>_t$, $t\in [0,T]$ are symmetric and non-negative
because
$$
\<<a,M>, <b,M>\>_t=\< <b,M>, <a,M>\>_t
$$
and
$$
\< <a,M>, <a,M> \>_t\geq 0, \qquad t\in [0,T].
$$
For $0\leq s\leq t\leq T$ we have
$$
 (\{\<\<M,M\>\>_t-\<\<M,M\>\>_s\}a,a) =
\< <a,M>, <a,M> \>_t-\< <b, M> <a,M> \>_s\geq 0,
$$
so the operators $\<\<M,M\>\>_t-\<\<M,M\>\>_s$ are non-negative.
Consequently
$$\begin{array}{lll}\dis
\|\<\<M,M\>\>_t-\<\<M,M\>\>_s\|_{L_1(U)}
&=&\dis {\rm Tr\; }\{\<\<M,M\>\>_t-\<\<M,M\>\>_s\}\\
&
=&\dis\sum_j\{\<M^j,M^j\>_t-\<M^j,M^j\>_s\},
\end{array}
$$
$$
\E\, \|\<\<M,M\>\>_t-\<\<M,M\>\>_s\|_{L_1(U)}=
\E\, (|M(t)|^2 - |M(s)|^2 )<+\infty.
$$
This shows that $\<\<M,M\>\>_t$, $t\in [0,T]$ is an $L_1^+(U)$
valued, predictable, increasing function. To show that it is also
right-continuous and can be represented in the form
(\ref{defqtdim}) we establish first  a lemma.

\begin{Lemma}
  Assume that $M$, $N$, are right-continuous, square-integrable
  martingales on $[0,T]$. There exists a predictable process
  $q(s)$, $s\in [0,T]$ such that
  $$
  \<M,N\>_t=\int_0^tq(s)\; d[\<M,M\>_s+\<N,N\>_s].
  $$
\end{Lemma}

\noindent{\bf Proof.}
It is enough to show that for almost all random outcome
the measure in $[0,T]$ corresponding to the function
$\<M,N\>$ of bounded variation is absolutely continuous with
respect to the sum of the measures induces by $<M,M>$ and $<N,N>$.
For fixed $s\in [0,T]$ and arbitrary real $x$, the process
$$
\<M+xN, M+xN\>_t - \<M+xN, M+xN\>_s,\qquad
t\in [s,T]
$$
is the angle bracket corresponding to
$M(t)+xN(t)$, $t\in [s,T]$. Consequently, for all $x\in \R^1$
$$
\begin{array}{l}
  \<M+xN, M+xN\>_t-\<M+xN, M+xN\>_s\\
  \qquad =x^2(  \<N,N\>_t-\<N, N\>_s)
  +2x(  \<M, N\>_t-\<M, N\>_s)
  +(  \<M, M\>_t-\<M, M\>_s)
  \\
  \qquad \geq 0
\end{array}
$$
and thus
$$
(  \<M, N\>_t-\<M, N\>_s)^2\leq
(  \<M, M\>_t-\<M, M\>_s)
(  \<N,N\>_t-\<N, N\>_s)
$$
or equivalently
\begin{equation}\label{schwartsforvarq}
|  \<M, N\>_t-\<M, N\>_s|\leq
(  \<M, M\>_t-\<M, M\>_s)^{1/2}
(  \<N,N\>_t-\<N, N\>_s)^{1/2}.
\end{equation}
>From (\ref{schwartsforvarq}),
$$
|  \<M, N\>_t-\<M, N\>_s|\leq\frac{1}{2}\{
(  \<M, M\>_t+\<N, N\>_t)-
(  \<M,M\>_s+\<N, N\>_s)\}.
$$
This way we have shown that on each subinterval of $[0,T]$ the
 total
variation of the measure corresponding to $\<M,N\>$ is smaller
than the total variation corresponding to $\<M,M\>+\<N,N\>$. In
particular if a Borel set $\Gamma\subset [0,T]$ is of measure zero
with respect to $d(\<M,M\>+\<N,N\>)$ it is of measure zero
with respect to $d\<M,N\>$ and this implies the required absolute
continuity.

To prove predictability of $q$ we use a real analysis result. If
$\mu$ and $\nu$ are two finite non-negative measures on
$[0,+\infty)$ and $\mu$ is absolutely continuous with respect to
$\nu$ then for $\nu$-almost all $t>0$
$$
\frac{d\mu}{d\nu}(t) = \liminf_{r\uparrow t}
\frac{\mu((r,t])}{\nu((r,t])},
$$
where the limit is taken with respect to $r<t$, $r$ rational.
Since the limes inferior of predictable processes is predictable
the result follows.
 \qed

It follows from the lemma that for each $i,j\in\N$ there exists a
predictable process $q^{i,j}(t)$, $t\in [0,T]$ such that
$$
\<M^i,M^j\>_t=\int_0^tq^{i,j}(s)\; d\<M,M\>_s
=\int_0^tq^{i,j}(s)\; d\sum_{k=1}^{+\infty}\<M^k,M^k\>_s.
$$
Thus, if the space $U$ is finite dimensional we have the
representation
$$
 \<\<M,M\>\>_t=\int_0^tQ_s\; d\<M,M\>_s
 $$
 where $Q_s$, $s\in [0,T]$ is a predictable function with values
 in the space $L_1^+(U)$. In general we have for arbitrary
 $a\in l^2$ with only a finite number of coordinates different
 from zero
 $$
(\<\<M,M\>\>_ta,a) = \< <a,M>, <a,M> \>_t=
\int_0^t <Q_sa,a> \; d\<M,M\>_s.
$$
In particular
$$
{\rm Trace \;} \<\<M,M\>\>_t =\int_0^t{\rm Trace\;}Q_s\; d\<M,M\>_s
<+\infty
$$
$\P$-a.s.,
where  $Q_s$ is given by the formula
$ Q_s=\sum_{k,l=1}^{+\infty}e_k\otimes e_l
  \; q^{k,l}(s)$.
 This implies that $Q_s$ takes values in $L_1^+(U)$ even
if $\dim U=+\infty$.
 \qed
\subsection{Final comments}
The proof of the isometric formula (\ref{mp2}) is now similar to its proof for the Wiener process. The main
ingredient  is  the following generalisation of Proposition \ref{midentities} which follows directly from
Theorem \ref{MAIN},
\begin{Proposition}\label{midentities3}
Let $M$ be a {\it Wiener process with respect} to ${\cal F}_t$ with the covariance operator $Q$. Then
for arbitrary vectors $a, b \in U$ and arbitrary $0\leq s \leq t \leq u \leq v$
\begin{equation}
\E (<M(t)- M(s), a> <M(t)- M(s)), b> |{\cal F}_s) = \int_{s}^{t} <Q_{\sigma}a, b> d<M, M>(\sigma)
\end{equation}
\begin{equation}
E (<M(t)- M(s) , a> <M(u)- M(v), b> |{\cal F}_u)= 0.
\end{equation}
\end{Proposition}
\vspace{3mm}

\noindent As an introduction to {\it stochastic evolution equations} with respect to Hilbert valued
martingales we comment on the so called {\it weak solution}, playing a fundamental role in the applications.
\vspace{3mm}

\noindent Let $H$ and $U$ be separable Hilbert spaces, $A,$  the infinitesimal generator
of a $C_0$-semigroup $S(t)$, $t\geq 0$ on $H$ and $M(t)$, $t \geq0$ a square integrable
martingale on $U$.   Consider an
equation
\begin{equation} \label{eq:dy}
    dy(t) = Ay(t) dt + f(t) dt + \phi(t) dM(t), \quad t \geq 0,\,\,y(0) = x,
\end{equation}
where $f$ is an $H$-valued, adapted process and $\phi(t)$,
$t\geq0$ an adapted process of,
possibly unbounded, linear operators from $U$ into $H$. Both processes $f, \phi$ can be functionals of the unknown process $y$.

\noindent A linear subspace $D \subset H$ is said to be a core for $S(t)$, $t \geq 0$ if
$D \subset D(A)$ and for every $x\in D(A)$ there exists a sequence $(x_n)$ of
elements from $D$ such that $x_n \to x$ and $Ax_n \to Ax$ as $n \to +\infty$.
It is well known, see [2], that if a dense linear subspace $D \subset H\,$
is invariant for
the semigroup then it is a core.
Let $D^*$ be a core of the operator $A^*$ adjoint to $A$.  A continuous,
adapted process $y(t)$, $t\geq 0$ is said to be {\it a weak  solution} of
\ref{eq:dy} if for each $x^* \in D^*$ and $t \geq 0$, $\Pr$-a.s.,
\begin{equation}
    < x^*, y(t) > =
    < x^*, x >
    \int_0^t < A^*x^*, y(s) > ds
\end{equation}
\begin{equation}
    +\int_0^t < x^*, f(s) > ds +
    \int_0^t < \phi^*(s)x^*, dM(s) >_U.
\end{equation}
We have the following important proposition which allows to replace differential equations by integral
equations.

\begin{Proposition}
Process
\begin{equation}
\label{eq:VOC}
    y(t) = S(t)x + \int_0^t S(t-s)f(s) ds + \int_0^t S(t-s) \phi(s) dM(s)
\end{equation}
is a weak solution of the equation \ref{eq:dy}.
\end{Proposition}
The relation (\ref{eq:VOC})
is the so called  {\it variation of constant formula}.

\chapter{Prohorov's theorem}\label{capprohorov}
The chapter is devoted to an  exposition of  Prohorov's tightness theorem in metric spaces.

\section{Motivations}
The Kolmogorov theorem is of great theoretical
importance. In practice however only seldom finite
dimensional distributions are known. Two stochastic
processes $X$ and $X'$ are called {\em equivalent} if
they have identical finite dimensional distributions.
For each $\omega\in\Omega$, the function $t\to
X(t,\omega)$ is called {\em a path}
of the process $X$. It is
easy to construct two equivalent processes such that
one of them has continuous paths and the paths  of the
other one are discontinuous.

\begin{Example}\begin{em}
Define $X(t)=t$, for $t\in [0,1]$. Let $X'(t)$, $t\in
[0,1]$, be defined on the probability space
$((0,1],\calb((0,1]), \P )$, where $\P$ is the Lebesgue
measure on $(0,1]$, by the formula: $X(t,\omega)=0 $ if
$t=\omega$ and $X'(t,\omega)=t$ if $t\ne \omega$. Then
the processes $X$ and $X'$ are equivalent although the
trajectories of $X$ are continuous and all the
trajectories of $X'$ are discontinuous. \qed
\end{em}
\end{Example}

A priori any function from $\calt$ into $E$ can be a
path of a process constructed in Kolmogorov's way. To
prove that there exists an equivalent version of the
process with special paths properties requires an
additional work.

A powerful way of constructing good versions of
stochastic processes is based on the concept of weak
convergence of measures.

\begin{Example}\begin{em}
Let $\xi_1,\xi_2,\ldots$ be a sequence of independent,
real, random variables with identical Gaussian laws
with mean (expectation) $0$ and second moment $1$. Such
a sequence exists on a probability space
$(\Omega,\calf,\P)$. For each $n=1,2,\ldots$ and $t\in
[0,1]$ define
$$
S_n=\xi_1+\ldots+\xi_n,\qquad S_0=0,
$$
$$
X_n(t)= \frac{1}{\sqrt{n}}S_{[nt]} + (nt-[nt])
 \frac{1}{\sqrt{n}}\xi_{[nt]+1}.
$$
For each $n=1,2,\ldots$ and $\omega\in\Omega$,
$X_n(\cdot,\omega)$ is a continuous function and $X_n$
can be regarded as a random variable with values in
$C[0,1]$. Let $\mu_n$ be the law of $X_n$ treated as
$C[0,1]$-valued random variable. One can show that the
sequence $\mu_n=\call(X_n)$ weakly converges to a
measure $\mu$ on $C[0,1]$ which can be identified with
the so called Wiener measure. The measure $\mu$ is
automatically concentrated on $C[0,1]$. If one takes
now as a new probability space $(C[0,1],\calb (C[0,1]),
\mu)$ and defines $X(t,\omega)=\omega (t)$ for
$\omega\in C[0,1]$, then $X$ is the so called Wiener
process with continuous pahts.
\end{em}
\end{Example}
Let $(X_n)$ be a sequence of $E$-valued random variables. One
says that  $(X_n)$  converges weakly to $X$ if the corresponding
distributions ${\cal L}(X_n) = \mu_n$  converge weakly to the law
 ${\cal L}(X) = \mu$.   If $\mu_n\Rightarrow\mu$ and $\phi$ is a
continuous function $\phi:E\to \R^1$ then also  $(\phi(X_n))$
converges weakly to $\phi(X)$.

{\bf Proof.}
For arbitrary $\lambda\in\R^1$,
$$
\E \left( e^{i\lambda\phi(X_n)}\right)
=\int_E   e^{i\lambda\phi(x)}\mu_n(dx)
\to \int_E   e^{i\lambda\phi(x)}\mu(dx),
$$
because $x\to e^{i\lambda\phi(x)}$ is a bounded continuous
function. This implies that characteristic functions of the laws of
real valued random variables $\phi(X_n)$ converge to the
characteristic function of the law of $\phi (X)$. This proves the
result. \qed

In particular if we know that continuous processes $X_n(t)$, $t\in
[0,T]$, converge weakly in $C([0,T];E)$ to a continuous process
 $X(t)$, $t\in [0,T]$,  and $\phi$ is a continuous function on
 $C([0,T];E)$ then $\phi(X_n)$  converges weakly to $\phi(X)$.
Interesting functions are for instance the following ($E=\R^1$):
$$
\phi (X)=\max_{0\leq t\leq T}X(t),
\qquad
\phi (X)=\int_0^TX(t)\, dt.
$$

\section{Weak topology}

$E$ is Polish,
that is separable, complete, metric space,
 with metric $\rho$, $\cale$ is the Borel
$\sigma$-field of subsets of $E$ and $\calm_1(E)$ is
the set of all probability measures on $E$. Denote by
$C_b(E)$ the space of all bounded continuous functions
on $E$. Then $\calm_1(E)$ can be regarded as a subset
of $C_b(E)^*$, the space of all bounded linear
functionals on $C_b(E)$, equipped with the
weakest topology
such that all mappings $\phi\to\phi (f)$ are continuous
for all $f\in C_b(E)$.
We say that a sequence
$(\mu_n)$ converges weakly to $\mu$ if an arbitrary
$C_b(E)^*$ neighbourhood of $\mu$ contains all but a
finite number of elements of $(\mu_n)$. This can be
equivalently stated as follows:
for arbitrary $f\in C_b(E)$
\begin{equation}\label{weakconv}
(\mu_n,f)\to (\mu,f)\quad {\rm as}\quad n\to +\infty.
\end{equation}

\begin{Proposition}
Let $(\mu_n)$ and $\mu $ be elements of $\calm_1(E)$.
The following conditions are equivalent.
\begin{enumerate}
  \item[(i)] $(\mu_n,f)\to (\mu,f)$ for all $f\in C_b(E)$.

  \item[(ii)] $(\mu_n,f)\to (\mu,f)$ for all $f\in UC_b(E)$.

  \item[(iii)] $\overline{\lim}_{n\to\infty}
  \mu_n(F)\leq \mu(F)$ for all closed sets $F\subset E$.

\item[(iv)] $\underline{\lim}_{n\to\infty}
  \mu_n(G)\geq \mu(G)$ for any open set $G\subset E$.

\item[(v)] $ {\lim}_{n\to\infty}
  \mu_n(A)=\mu(A)$ for any   $A\in\cale$
   such that $\mu(\partial A)=0$.
\end{enumerate}
\end{Proposition}

\noindent {\bf Proof.}
$(i)\Rightarrow (ii)$ is obvious.

$(ii)\Rightarrow (iii)$. Let $F$ be a given closed set
and $F_\epsilon$ its $\epsilon$-neighbourhood. Then
there exists a uniformly continuous function
$f_\epsilon$ on $E$ such that $0\leq f_\epsilon \leq
1$, $f_\epsilon
=1$ on $F$ and $f_\epsilon =0$ on $F_\epsilon^c$.
Define
$f_\epsilon(x)=\phi(\frac{1}{\epsilon}\rho(x,F))$,
where $\phi :\R^1\to\R^1$ is piecewise linear,
$\phi(\xi)=1$ for $\xi\leq 0$, $\phi(\xi)=0$ for
$\xi\geq 1$; also recall that $\rho(x,F)-\rho(y,F)\leq
\rho(x,y)$, $x,y\in E$. Then
$$
\mu_n(F)=\int_F f_\epsilon\mu_n\leq
\int_Ef_\epsilon\mu_n\to\int f_\epsilon\mu,
$$
so $\overline{\lim}_{n\to\infty}
  \mu_n(F)\leq \int_E f_\epsilon(x)\mu(dx)$, but
$f_\epsilon\to \chi_F$ as $\epsilon\to 0$, so $\int_E
f_\epsilon \mu\to \mu(F)$.

$(iii)\Leftrightarrow (iv)$. By taking into account
that $G^c$ is a closed set and $\nu(G^c)=1-\nu(G)$ for
arbitrary $\nu\in\calm_1(E)$.

$(iii)\Rightarrow (i)$. We will show that
$\overline{\lim}_{n\to\infty}
  (\mu_n,f)\leq (\mu, f)$. We can assume
that $0<f<1$ (adding a constant and multiplying by a
positive number). Fix $k\in\N$ and define closed sets
$$
F_i=\{ x\, :\, f(x)\geq \frac{i}{k}\},\qquad
i=0,1,\ldots,k.
$$
Then $F_i$ are closed sets and for arbitrary
$\mu\in\calm_1(E)$
$$
\sum_{i=1}^k\frac{i-1}{k}
\mu\left(x\, ;\, \frac{i-1}{k}\leq f<\frac{i}{k}\right)
\leq \int f\, \mu\leq
\sum_{i=1}^k\frac{i}{k}
\mu\left(x\, ;\, \frac{i-1}{k}\leq f<\frac{i}{k}\right),
$$
or equivalently
$$
\sum_{i=1}^k\frac{i-1}{k}
 \left[ \mu(F_{i-1})-\mu(F_i)\right]
\leq \int f\, \mu\leq
\sum_{i=1}^k\frac{i}{k}
 \left[ \mu(F_{i-1})-\mu(F_i)\right].
$$
Taking into account cancellation:
$$
\frac{1}{k}\sum_{i=1}^k \mu(F_i) \leq \int f\, \mu\leq
\frac{1}{k}+\frac{1}{k}\sum_{i=1}^k\mu(F_i).
$$
Consequently
$$
\mathop{\overline{\lim}}_{n\to\infty}\int f\, \mu_n\leq
\frac{1}{k}+\frac{1}{k}\sum_{i=1}^k\mu(F_i)
\leq \frac{1}{k}+\int f\, \mu.
$$
Since $\overline{\lim}_{n\to\infty}(\mu_n,-f)
\leq (\mu,-f)$ one gets also that
$\underline{\lim}_{n\to\infty}(\mu_n,f)
\geq (\mu,f)$. This proves $(i)$.

Equivalence with $(v)$. Recall that $\partial A$
consists of all those points in $E$ which have in any
neighbourhood points in $A$ and points in $A^c$. Then
$\partial A=\overline{A}\cap \overline{A^c}$. Note also
that $\overline{A} = \stackrel{\circ}{A}\cup \partial
A$ and that $\partial A = \overline{A}\backslash
\stackrel{\circ}{A}$. Assume that $\mu(\partial A)=0$.
Then
$$
\mu(\overline{A})
\geq\mathop{\overline{\lim}}_{n\to\infty} \mu_n(\overline{A})
\geq\mathop{\overline{\lim}}_{n\to\infty}\mu_n( {A})
\geq\mathop{\underline{\lim}}_{n\to\infty}\mu_n( {A})
\geq\mathop{\underline{\lim}}_{n\to\infty} \mu_n(\stackrel{\circ}{A})
\geq \mu( \stackrel{\circ}{A}).
$$
But if $\mu(\partial A)=0$ then $
\mu(\overline{A})=\mu(
\stackrel{\circ}{A})$ and from the above inequalities
$$
\mathop{\overline{\lim}}_{n\to\infty} \mu_n( {A})=
\mathop{\underline{\lim}}_{n\to\infty} \mu_n( {A}).
$$

Assume that for all Borel sets $A$ such that
$\mu(\partial A)=0$ $ {\lim}_{n\to\infty}
\mu_n( {A})=\mu(A)$. Now for arbitrary closed
set $F$ and its closed $\delta$-neighbourhood
$F_\delta=\{ x\, :\, \rho(x,F)\leq \delta \}$ one has
that $\partial F_\delta \subset \{x\, :\, \rho(x,F)=
\delta\}$. For different $\delta >0$ the sets
$\{x\, :\, \rho(x,F)=
\delta\}$ are disjoint and there exists a
sequence $\delta_k\downarrow 0$ such that $\mu(x\, :\,
\rho(x,F)=\delta_k)=0$. Consequently
$\lim_{n\to\infty}\mu_n(F_{\delta_k})=\mu(F_{\delta_k})$.
However $\mu_n(F)\leq \mu_n(F_{\delta_k})$ and
therefore $\overline{\lim}_{n\to\infty} \mu_n(F)\leq
\mu(F_{\delta_k})$ for each $k$. But
$\mu(F_{\delta_k})\downarrow \mu(F)$ so
 $\overline{\lim}_{n\to\infty} \mu_n(F)\leq\mu(F)$.
  This proves the result. \qed

\section{Metrics on $\calm_1(E)$}

A metric space $E$ equipped with a metric $d$ will be denoted by
$E_d$.

\begin{Lemma}\label{ucbedseparable}
Let $E_\rho$ be a separable metric space. Then
there exists a metric $d$ equivalent to $\rho$ such that
the space $UC_b(E_d)$ is separable.
\end{Lemma}

\noindent {\bf Proof.}
Let $(a_k)$ be a dense countable subset of $E$. Define
a mapping $h\, :\, E\to [0,1]^\N$ by the formula:
$$
h(x)=\left(\frac{\rho(x,a_k)}{1+\rho(x,a_k)}\right).
$$
Then $h$ is continuous,
one to one and the inverse is continuous.
Thus $E$  is homeomorphic to a subset  $h(E)$ of the compact
metric space $[0,1]^\N$. Identifying $E$ with $h(E)$,
a metric that $E$ inherits
from $[0,1]^\N$ is
$$
d(x,y)= \sum_{k=1}^{+\infty}\frac{1}{2^k}
\frac{\rho(x,a_k)}{1+\rho(x,a_k)},
$$
which is
therefore equivalent to $\rho$.

Now
let $\widehat{E}_d$ be the completion of $E_d$ with respect
to $d$. Then $\widehat{E}_d$ is compact and
$UC_b(\widehat{E}_d) = C(\widehat{E}_d)$.
But $C(\widehat{E}_d)$ is separable,
by Lemma \ref{ceseparable} below, and consequently
$UC_b(\widehat{E}_d)$ is separable as well.
Since $UC_b(\widehat{E}_d)$ is isomorphic to
$UC_b(E_d)$ the result follows. \qed

\begin{Lemma}\label{ceseparable}
If $F$ is a compact metric space then the space
$C(F)$ is separable.
\end{Lemma}

\noindent {\bf Proof.}
We first note that the space $C([0,1]^\N)$ is separable
because, for arbitrary natural number $k$,
$C([0,1]^k)$ is separable, and continuous functions
on $[0,1]^\N$, which depend on a finite number
of coordinates, are dense in $C([0,1]^\N)$
by the Stone-Weierstrass theorem.

Next we construct a homeomorphism $h$ of $F$ onto a
subset of $[0,1]^\N$, proceeding as in the proof of Lemma
\ref{ucbedseparable}. Then $C(F)$ and $C(h(F))$ are isometric.
Since $F$ is compact, $h(F)$ is a compact subset of
$[0,1]^\N$. Since every
continuous function on $h(F)$ extends
to a continuous function on $[0,1]^\N$,
it is easily seen that separability of
$C([0,1]^\N)$ implies that
$C(h(F))$ is separable as well, and the result follows.
\qed

\begin{Proposition}
Let $E$ be a Polish space. Then there
exists a
countable, dense subset of $C_b(E)$, denoted by
$(f_k)$, such that
the following metric
$$
\Delta (\mu,\nu)=
\sum_{k=1}^{+\infty}\frac{1}{2^k}
\frac{|(\mu,f_k)-(\nu ,f_k)|}{1+|(\mu,f_k)-(\nu ,f_k)|}
$$
defines the weak topology on  $\calm_1(E)$.
\end{Proposition}

\noindent {\bf Proof.}
Let $d$ be the metric constructed in
Lemma \ref{ucbedseparable}. We take as
$(f_k)$ a
countable, dense subset of $UC_b(E_d)$.

Let $ {\cal O} $ be a $C_b(E)^*$ neighbourhood of
$\mu$. One has to show that there exists $r>0$ such
that $\{\nu\, :\, \Delta (\mu,\nu)<r\}\subset {\cal O}
$. One can assume that ${\cal O} $ is determined by a
finite number of continuous functions
$g_1,\ldots,g_k\in C_b(E)$ and numbers
$r_1,\ldots,r_k$:
$$
{\cal O} =\{ \nu\, :\, |(g_i,\mu)-(g_i,\nu )|<r_i,\;
i=1,2,\ldots ,k\}.
$$
Assume to the contrary that there exists a sequence of
measures $\nu_n$ such that $\Delta (\mu,\nu_n)\to 0$
but nevertheless $\nu_n\notin {\cal O}$. Without any
loss of generality one can assume that
$$
|(g_1,\mu)-(g_1,\nu_n )|\geq r_1
$$
for all $n=1,2,\ldots$. But
$(f_k,\nu_n)\stackrel{n\to\infty}{\longrightarrow}
(f_k,\mu)$ for a dense set in $UC_b(E_d)$ and therefore
the convergence takes place for all $f$ in $UC_b(E_d)$
and therefore for all $f$ in $C_b(E)$, in particular
for $f=g_1$, a contradiction. Fix now $r>0$ and
consider a ball of radius $r$ with center $\mu$:
$
B(\mu,r)=\{ \nu\, :\, \Delta (\mu,\nu)<r\}.
$
Take
$$
{\cal O} =\{ \nu\, :\, |(f_k,\mu)-(f_k,\nu )|<r_k,\;
k=1,2,\ldots ,N\}
$$
and $N$ such that $\sum_{k>N}2^{-k}< r/2$. If
$r_1,\ldots,r_N$ are chosen such that
$$
\sum_{k=1}^N \frac{1}{2^k}\frac{r_k}{1+r_k}<
\frac{r}{2}
$$
then ${\cal O}\subset B(\mu,r)$. This completes the
proof. \qed

\begin{Remark}\begin{em}
One can introduce an equivalent metric which is even
complete. An example of such a metric was considered by
Prohorov:
$$
L(\mu,\nu)=\inf \{\delta\, :\, \mu(F)\leq
\nu(F_\delta)+\delta \; {\rm and}\;
\nu(F)\leq
\mu(F_\delta)+\delta\; {\rm for \; all \; closed}\;
F\subset E\}.
$$
Another possibility is to use the so called Wasserstein
metric, called also Fortet-Mourier metric
$$\begin{array}{ll}
W(\mu,\nu)=\sup&\ds
\bigg\{\left| \int_E\phi(x)\,\mu(dx)
- \int_E\phi(x)\,\nu(dx)\right|\, :\,
\eds\\&\ds
\sup_x|\phi(x)|\leq 1,\;
|\phi(x)-\phi(y)|\leq \rho(x,y),\,x,y\in E\bigg\}.
\eds  \end{array}
$$
\end{em}
\end{Remark}

\section{Prohorov's theorem}

\begin{Theorem}
\begin{enumerate}
  \item[(i)] Let $\Gamma\subset \calm_1(E)$
be a compact set then for arbitrary $\epsilon >0$ there
exists a compact set $K\subset E$ such that
\begin{equation}\label{tight}
\mu(K)\geq 1-\epsilon\quad {\rm for\; all}\; \mu\in\Gamma.
\end{equation}
  \item[(ii)] If $\Gamma\subset \calm_1(E)$ is a
 set such that for each $\epsilon >0$ there exists a
compact set $K\subset E$ such that (\ref{tight}) holds
then $\overline{\Gamma}$ is a compact subset of
$\calm_1(E)$.
\end{enumerate}
\end{Theorem}

\noindent {\bf Proof.}
Let $\{a_k\}$ be a dense set in $E$. Define
$$
G_{n,k}=\bigcup_{j=1}^nB\left(a_j,\frac{1}{k}\right).
$$
The maps $\mu\to\mu(G_{n,k})$ are lower semicontinuous
and for each $k$, $\mu(G_{n,k})\uparrow 1$. By Dini's
theorem for each $\epsilon >0$ and $k$ there exists an
$n_k$ such that for all $\mu\in\Gamma$
$$
\mu(G_{n_k,k})\geq 1-\frac{\epsilon}{2^k}.
$$
Define $K=\bigcap_{k=1}^\infty\overline{G_{n_k,k}}$.
Then for all $\mu\in\Gamma$
$$
\mu(K)=\mu\left(\bigcap_{k=1}^\infty\overline{G_{n_k,k}}\right)
=1-\mu\left(\left(\bigcap_{k=1}^\infty\overline{G_{n_k,k}}\right)^c
\right)
= 1-
\mu\left(\bigcup_{k=1}^\infty\overline{G_{n_k,k}}^c\right)
\geq 1- \sum_{k=1}^\infty
\mu\left(\overline{G_{n_k,k}}^c\right),
$$
$$
\mu\left(\overline{G_{n_k,k}}^c\right)=
1- \mu\left(\overline{G_{n_k,k}}\right)\leq
\frac{\epsilon}{2^k}.
$$
Consequently $\mu(K)\geq 1-\epsilon$. The set is
closed, and totally bounded because for each $k$
$$
K\subset
\bigcup_{j=1}^{n_k}B\left(a_j,\frac{2}{k}\right).
$$
\begin{Lemma}
If a metric
 space $E$ is compact then it is complete and
totally bounded. Conversely, a complete and totally
bounded space is compact.
\end{Lemma}
This completes the proof of $(i)$.

To show $(ii)$ we use Riesz's theorem.
\begin{Theorem}
If a metric
space $K$ is compact then an arbitrary linear
functional $\psi$ on $C(K)$ such that $\psi(f)\geq 0$
for all $f\geq 0$ is of the form
$$
\psi(f)=\int_Ef(x)\,\mu(dx)
$$
where $\mu$ is a nonnegative finite measure and the
representation is unique.
\end{Theorem}

Proof of $(ii)$. Assume that $E$ is compact and define
$\psi_\mu(f)=\int_Ef(x)\,\mu(dx)$. Note that $C(E)$ is
a separable space, by Lemma \ref{ceseparable},
 and
therefore, given an arbitrary sequence
in $\Gamma$, by a diagonal procedure
 one can extract a subsequence
$\psi_{\mu_n}$ such that
$\psi_{\mu_n}(f)$ is convergent for all $f$
in a dense subset of $C(E)$. But then  $\psi_{\mu_n}$
is convergent (by $3\epsilon$-method) on the whole
$C(E)$ and its limit is a nonnegative functional on
$C(E)$. By Riesz's theorem the result follows if $E$ is
compact.

Next assume that
$E$ is $\sigma$-compact, i.e. a countable union of
compact subsets.
We saw in the proof of
Lemma \ref{ucbedseparable} that
there exists a
homeomorphism $h$ of $E$ onto a subset
 $h(E)$ of a compact
metric space $F$ (and one can take $F=[0,1]^\N$).
Let us identify $E$ with $h(E)$ and think of
$E$ as a topological subspace of $F$.
Since compact sets of $E$ are compact subsets of $F$
 as well, it follows that $E$ is a $\sigma$-compact subset
 of $F$, hence a Borel subset of $F$.
Thus $\calm_1(F)\supseteq \calm_1(E)$. Assume
that $(\mu_n)\subset \Gamma$. Then there exists a
subsequence $(\mu_{n_k})$ such that
$\mu_{n_k}\Rightarrow \mu$ in $\calm_1(F)$.
For arbitrary $\epsilon >0$ there exists a compact
subset $K_\epsilon \subset E$ such that
$$
\mu_{n_k}(K_\epsilon)\geq 1-\epsilon\quad {\rm for\; all}\;
k=1,2,\ldots.
$$
Since $K_\epsilon$ is closed in $F$ as well
therefore
$$
\overline{\lim}_k \mu_{n_k}(K_\epsilon)\leq \mu(K_\epsilon)
$$
and we see that $\mu(K_\epsilon)\geq 1-\epsilon$. Thus
$\mu(F\backslash E)=0$. So $\mu$ is
supported by $E$. An arbitrary $f\in UC_b(E)$ can be
uniquely extended to $f\in C(F)$ and
$$
\int_Ef(x)\,\mu_{n_k}(dx) =
\int_{F}f(x)\,\mu_{n_k}(dx) \to
\int_{F}f(x)\,\mu (dx)=
\int_Ef(x)\,\mu(dx).
$$
This proves the result if $E$ is $\sigma$-compact.

In the general case, it follows from the tightness
 condition that there exists a
$\sigma$-compact subset $E_0$ of $E$ such that
$\mu (E\backslash E_0)=0$ for all
 $\mu\in\Gamma$. We can replace $E$ by
 $E_0$ and the general case
  follows from the result proved above.
This completes the proof.
\qed

\chapter{Invariance principle and Kolmogorov's test}

We develop  some criteria of weak convergence and tightness  in the space $C([0,T); E)$ of continuous functions
with values in a metric space $E$
The convergence of random walks to the Wiener process is given as application.
We introduce also the factorisation method, due to L. Schwartz,
to establish tightness and prove Donsker's invariance
principle. Finally we prove Kolmogorov's continuity criteria
in the context of weak convergence.
\section{Weak convergence in $C([0,T];E)$}

Let $(E,\rho)$ be a complete, separable metric space
and  $C([0,T];E)$ the space of all continuous functions
defined on $[0,T]$ with values in $E$ equipped with
metric $R$:
\begin{equation}\label{defdiR}
R(x,y)=\sup_{t\in [0,T]} \rho (x(t),y(t)).
\end{equation}
The space  $C([0,T];E)$ is also a complete metric
space. To find out whether a given sequence of measures
on  $C([0,T];E)$ is tight it is necessary to know
characterizations of compact subsets of  $C([0,T];E)$.
We have the following generalization of the
Arzel\`a-Ascoli theorem.

\begin{Proposition}
A set $K\subset  C([0,T];E)$ has a compact closure in
$C([0,T];E)$ if and only if
\begin{enumerate}
  \item[(i)] There exists a compact set $L\subset E$ such that
\begin{equation}\label{compprima}
K\subset \{ x\, ;\, x(t)\in L {\rm \; for \; all\; }
t\in [0,T]\}.
\end{equation}

  \item[(ii)] For arbitrary $\epsilon >0$ there
  exists $\delta >0$ such that
\begin{equation}\label{compseconda}
K\subset \{ x\, ;\,
\sup_{
\begin{array}{c}
\begin{scriptstyle}
|t-s|\leq \delta
\end{scriptstyle}
 \\ \begin{scriptstyle}
 t,s\in [0,T]
 \end{scriptstyle}
\end{array}
 }
\rho(x(t),x(s))\leq \epsilon
\}.
\end{equation}
\end{enumerate}
\end{Proposition}

For each $x\in C([0,T];E)$ define the modulus of
continuity $\omega_\delta (x)$, $\delta >0$, by the
formula
\begin{equation}\label{modcontin}
\omega_\delta (x)=
\sup_{
\begin{array}{c}
\begin{scriptstyle}
|t-s|\leq \delta
\end{scriptstyle}
 \\ \begin{scriptstyle}
 t,s\in [0,T]
 \end{scriptstyle}
\end{array}
 }
\rho(x(t),y(t)),\qquad \delta >0.
\end{equation}
Then the inclusion (\ref{compseconda}) can be written
$$
K\subset \{ x\, ;\,
\omega_\delta (x)
\leq \epsilon
\}.
$$

\begin{Theorem}\label{compinc}
A sequence $(\mu_n)$ of probability measures on
$C([0,T]; E)$ is tight if and only if
\begin{enumerate}
  \item[(i)] For arbitrary $\epsilon >0$ there exists
  a compact set $L\subset E$ such that for all
  $n=1,2,\ldots$

\begin{equation}\label{compmisprima}
\mu_n (x\, ; \, x(t)\in L {\rm \; for \; all \;}
t\in [0,T] ) \geq 1-\epsilon.
\end{equation}

  \item[(ii)] For arbitrary $\epsilon >0$ and arbitrary
  $\eta >0$  there exists
  $\delta >0$ such that  for all
  $n=1,2,\ldots$

\begin{equation}\label{compmisseconda}
\mu_n (x\, ; \, \omega_\delta (x)\leq \eta )
 \geq 1-\epsilon.
\end{equation}

\end{enumerate}
\end{Theorem}

\noindent {\bf Proof.}
Assume that $(\mu_n)$ is tight and let $K_\epsilon$ be
a compact set such that
$$
\mu_n(K_\epsilon)\geq 1-\epsilon,
\qquad {\rm for\; all\; }n=1,2,\ldots.
$$
Since $K_\epsilon$ is compact in $C([0,T];E)$, there
exists a compact set $L_\epsilon\subset E$ such that
$$
K_\epsilon\subset \{ x\, ;\, x(t)\in L_\epsilon {\rm \;
for
\; all\; } t\in [0,T]\}.
$$
Therefore
$$
\mu_n (x\, ; \, x(t)\in L_\epsilon {\rm \; for \; all \;}
t\in [0,T] ) \geq 1-\epsilon
\qquad {\rm for\; all\; }n=1,2,\ldots.
$$
In a similar way, by the compactness of $K_\epsilon$,
for arbitrary $\eta >0$ there exists $\delta >0$ such
that if $x\in K_\epsilon$ then $\omega_\delta
(x)\leq\eta$. Therefore for each $n\in\N$,
$$
\mu_n (x\, ; \, \omega_\delta (x)\leq \eta )
\geq \mu_n(K_\epsilon) \geq 1-\epsilon .
$$
We show now that conditions $(i),(ii)$ imply tightness
of $(\mu_n)$. Fix $\epsilon >0$ and choose first a
compact set $L_\epsilon \subset E$ such that
(\ref{compmisprima}) holds for all $n$ with $\epsilon$
replaced by $\epsilon /2$. Let $(\eta_m)$ be a sequence
of positive numbers converging to zero. There exists a
sequence of numbers $\delta_m >0$ such that for all
$n$, and all $m$:
$$
\mu_n( x\, ;\, \omega_{\delta_m}(x)>\eta_m)
<\frac{\epsilon}{2^{m+1}}.
$$
Define
$$
K_\epsilon = \{ x\, ;\, x(t)\in L_\epsilon {\rm
\; for \; all\;} t\in [0,T]\}
\cap
\bigcap_{m=1}^\infty \{ x\, ;\,
\omega_{\delta_m}(x)\leq\eta_m\}
= B_0 \cap \bigcap_{m=1}^\infty B_m.
$$
It is easy to see that $K_\epsilon$ is closed and by
the Arzel\`a-Ascoli theorem it is compact. Now for each
$n$:
$$
  \begin{array}{lll}
    \mu_n(K_\epsilon) & = &
    \ds
1-\mu_n\left(B_0^c \cup\bigcup_{m=1}^\infty
B_m^c\right)
\geq 1-\frac{\epsilon}{2}-\sum_{m=1}^\infty
\mu_n(B_m^c)
    \eds \\
    &\geq & \ds 1-\frac{\epsilon}{2}-\sum_{m=1}^\infty
\frac{\epsilon}{2^{m+1}}\geq 1-\epsilon.\eds
  \end{array}
$$
as required. \qed

\begin{Remark}\begin{em}
Condition $(i)$ in Theorem \ref{compinc} is equivalent
to the requirement that there exists an increasing
sequence of compact sets $L_m\subset E$ such that
\begin{equation}\label{equivcompprima}
\lim_m\, \sup_n\, \mu_n(x\, ;\,
x(t)\in L_m {\rm \; for \; all\;} t\in [0,T])=0
\end{equation}
Similarly condition $(ii)$ can be replaced by the
following one:
\begin{equation}\label{equivcompseconda}
{\rm For \; arbitrary \; }\eta >0,\qquad
\lim_{\delta\to 0}\left[
\overline{\lim_n}\, \mu_n(x\, ;\,
\omega_\delta (x)>\eta)\right]=0.
\end{equation}
We show for instance that (\ref{equivcompseconda})
implies (\ref{compmisseconda}).  If
(\ref{equivcompseconda}) holds then for $\epsilon >0$
there exists $\delta_0>0$ such that for all $\delta\in
(0,\delta_0]$:
$$
\overline{\lim_n}\, \mu_n(x\, ;\,
\omega_\delta (x)>\eta)<\epsilon .
$$
But then there exists $n_0$ such that for all $n\geq
n_0$,
$$
 \mu_n(x\, ;\,
\omega_{\delta_0} (x)>\eta)<\epsilon .
$$
 Since, by Ulam's theorem, an arbitrary finite
set of measures if tight, one can find
$\delta_1<\delta_0$ such that for $n\leq n_0$
$$
\mu_n(x\, ;\, \omega_{\delta_1}(x)>\eta)<\epsilon,
$$
consequently for all $n=1,2,\ldots$
$$
\mu_n(x\, ;\, \omega_{\delta_1}(x)>\eta)<\epsilon,
$$
as required.
\end{em}
\end{Remark}

\begin{Remark}\begin{em}
Assume that $E=\R^1$. Then condition $(i)$ in Theorem
\ref{compinc} can be replaced by the requirement that for
arbitrary $\epsilon >0$ there exists
  a compact set $L\subset E$ such that for all
  $n=1,2,\ldots$
$$
\mu_n (x\, ; \, x(0)\in L ) \geq 1-\epsilon.
$$
Indeed, it is easy to check that, together with
condition $(ii)$, this implies that
(\ref{compmisprima}) holds.
\end{em}
\end{Remark}

If $X_n$, $n=1,2,\ldots$ and $X$ are random variables
with values in a metric space $E$ and defined on
possibly different probability spaces
$(\Omega_n,\calf_n,\P_n)$, $n=1,2,\ldots$,
$(\Omega_0,\calf_0,\P_0)$ respectively then we say that
$(X_n)$ converges weakly to $X$ if and only if
$\call(X_n)\Rightarrow\call(X)$ where $\call (Z)$
denotes the distribution of the random variable $Z$. If
$(X_n)$ converges weakly to $X$ then we write
$(X_n)\Rightarrow X$.

It is instructive at this moment to recall the so
called Skorokhod imbedding theorem although we will not
need it.

\begin{Theorem} (Skorokhod)
Assume that $E$ is a complete,
separable  metric space. If
$(X_n)\Rightarrow X$ then there exists a probability
space $(\Omega ,\calf,\P)$ and $E$-valued random
variables $X_n'$, $n=1,2,\ldots$, $X'$ defined on
$(\Omega ,\calf,\P)$ such that
\begin{enumerate}
  \item[(i)] $\call(X_n')=\call(X_n)$,
  $\call(X')=\call(X)$.

  \item[(ii)] For almost all $\omega\in \Omega$,
   $X_n'(\omega) \to X(\omega)$ as $n\to\infty$.

\end{enumerate}

\end{Theorem}

\begin{Remark}\begin{em}
It is clear that if $(ii)$ holds then
$(X_n')\Rightarrow X'$.
\end{em}
\end{Remark}

We will need the following elementary properties to
identify the limiting measure.

\begin{Proposition}\label{sommaprob}
Assume that $E$ is a separable normed space and $(X_n)$
are $E$-valued random variables. If $X_n\Rightarrow X$
and
$$
X_n=Y_n+\xi_n
$$
where $\xi_n\to 0$ in probability then
$$
Y_n\Rightarrow X.
$$
\end{Proposition}

\noindent {\bf Proof.}
Let $\phi\in UC_b(E)$. Then
$\E(\phi(Y_n))=\E(\phi(X_n-\xi_n))$ and
$$
\left| \E(\phi(Y_n)) -\E(\phi(X_n))\right| \leq
\E \left| \phi(X_n-\xi_n) -\phi(X_n)\right| .
$$
For arbitrary $\epsilon >0$ there exists $\delta >0$
such that if $\|x-y\|\leq\delta$ then
$|\phi(x)-\phi(y)|\leq \epsilon$. Therefore
$$
  \begin{array}{lll}
    \ds \E \left| \phi(X_n-\xi_n) -\phi(X_n)\right|\eds
    &\leq & \ds
\E \left( \chi_{\|\xi_n\|\leq \delta}
\left| \phi(X_n-\xi_n) -\phi(X_n)\right|\right)
+ 2 \|\phi\|\, \P\left( \|\xi_n\|>\delta\right)
    \eds \\
     &\leq & \ds
     \epsilon +
      2 \|\phi\|\, \P\left( \|\xi_n\|>\delta\right)
      \eds
      \\
      & \leq & 2\epsilon
  \end{array}
$$
if $n$ is sufficiently large. \qed

\begin{Proposition}
If $F$ is a continuous mapping from a metric space $E$
into a metric space $E_1$ and $E$-valued random
variables $(X_n)$ converge weakly to $X$ then the
random variables $F(X_n)$ converge weakly to $F(X)$.
\end{Proposition}

\noindent{\bf Proof.}
If $\phi : E_1\to\R^1$ is bounded and continuous then
obviously $\phi(F)$ is a bounded and continuous
function from $E$ into $\R^1$ and the result follows.
\qed

\begin{Theorem}
Let $E$ be a  Polish space.
A sequence of probability measures $(\mu_m)$ on
$C([0,T];E)$
 converges weakly to a measure $\mu$ if and only if
\begin{enumerate}
  \item[(i)] $(\mu_n)$ is tight.

  \item[(ii)] For any sequence $0=t_0\leq t_1\leq t_2
  \leq\ldots\leq t_k\leq T$ the finite dimensional
  distributions $\mu_m^{(t_1,\ldots ,t_k)}$ converge
  weakly to $\mu^{(t_1,\ldots ,t_k)}$.
\end{enumerate}
\end{Theorem}

\begin{Remark}\begin{em}
As we already know, if $\mu$ is a measure on
$C([0,T];E)$ then $\mu^{(t_1,\ldots ,t_k)}$ is a
measure on $E^k$ given by the formula
$$
\mu^{(t_1,\ldots ,t_k)}(\Gamma)=
\mu \{x\, ;\, (x(t_1),\ldots,x(t_k))\in\Gamma\}
\qquad {\rm for \; any \; }\Gamma\in \calb(E^k).
$$
\end{em}
\end{Remark}

\noindent {\bf Proof.}
If $\mu_m\Rightarrow \mu$ then $(\mu_m)$ is tight so
$(i)$ holds. If $\psi\in C_b(E^k)$ then
$$\phi(x)=\psi (x(t_1),\ldots,x(t_k)) ,
\qquad x\in C([0,T];E),
$$
 is a
bounded continuous function on $C([0,T];E)$ and
therefore
$$
\int_{C([0,T];E)}\phi(x)\, \mu_m(dx)
\to \int_{C([0,T];E)}\phi(x)\, \mu(dx).
$$
But
$$
\int_{C([0,T];E)}\phi(x)\, \mu_m(dx)=
\int_{E^k}\psi (y)\, \mu_m^{(t_1,\ldots,t_k)}(dy),
$$
so $\mu_m^{(t_1,\ldots,t_k)}\Rightarrow
\mu^{(t_1,\ldots,t_k)}$ on $E^k$.

Conversely, assume that $(\mu_n)$ is tight and $(ii)$
holds. It is enough to show that arbitrary weakly
convergent subsequences of $(\mu_m)$ converge to the
same limit. If $\mu^1$ and $\mu^2$ are two such limits
then, for arbitrary $k=1,2,\ldots$ and $\Gamma\in
\calb(E^k)$,
$$
\mu^1\left( x\, ;\, (x(t_1),\ldots,x(t_k))\in\Gamma \right)
=
\mu^2\left( x\, ;\, (x(t_1),\ldots,x(t_k))\in\Gamma \right).
$$
But the cylindrical sets $\{x\, ;\,
(x(t_1),\ldots,x(t_k))\in\Gamma\}$ generate the Borel
$\sigma$-field $\calb(C([0,T];E))$ so by Dynkin's
$\pi-\lambda$ theorem, $\mu^1=\mu^2$. \qed

\section{Classical proof of the invariance principle.}

Let $\xi_n^m$ be independent real valued random
variables defined on a probability space $(\Omega,
\calf, \P)$,
$$
\E\, \xi_n^m=0,\quad \E\, (\xi_n^m)^2 =\frac{1}{m},
\qquad n,m=1,2,\ldots,
$$
which, in addition, have Gaussian distributions. We
define, by induction, the sequences
$(B_n^m)_{n=1,\ldots}$:
$$
B_{n+1}^m= B_n^m+\xi_{n+1}^m,\quad B_0^m=0,
\qquad n=0,1,2,\ldots.
$$
Let $(X_t^m)$ be the following stochastic processes:
\begin{equation}\label{defdixt}
X_t^m=B_n^m+(mt-n)(B_{n+1}^m-B_n^m)
= B_n^m+(mt-n)\xi_{n+1}^m,\qquad {\rm for\;}
t\in \left[\frac{n}{m},\frac{n+1}{m}\right].
\end{equation}
Let us recall that a {\em Wiener process} is a
stochastic process $W$ such that: $i)$ $W(0)=0$; $ii)$
for arbitrary $0=t_0 < t_1<\ldots<t_k$ the random
variables
$$
W(t_1),W(t_2)-W(t_1),\ldots,W(t_k)-W(t_{k-1})
$$
are centred
 Gaussian and independent; $iii)$ $\E|W(t)-W(s)|^2 =
t-s$, $t\geq s\geq0$; $iv)$ the trajectories of $W$ are
continuous with probability one.

On the measurable space $(E, \calb(E))$ where
$E=C([0,T])$ define a canonical process $X$ as follows:
$$
X_t(f)=f(t),\qquad f\in E, t\in [0,T].
$$

\begin{Theorem}\label{11.2.1}
The distributions $\mu_m$ of the processes $X^m$
converge weakly to a measure $\mu$ and the canonical
process on $(E,\calb(E),\mu)$ is a Wiener process.
Consequently the Wiener process $W$ exists and
$X_m\Rightarrow W$.
\end{Theorem}

\noindent {\bf Proof. Step 1.}
We will show first that the sequence $(\mu_m)$ is tight
on $E$. Since for all $m=1,2,\ldots$, $X_0^m=0$ it is
enough to prove that
\begin{equation}\label{perdonsker2}
\lim_{\delta\to 0}\, \overline{\lim_m}\,
\P\bigg( \sup_{
\begin{array}{c}
\begin{scriptstyle}
|t-s|\leq \delta
\end{scriptstyle}
 \\ \begin{scriptstyle}
 t,s\in [0,T]
 \end{scriptstyle}
\end{array}
} |X_t^m-X_s^m|>\epsilon\bigg)=0
\end{equation}
for arbitrary fixed $\epsilon >0$. Choose $\delta\in
(0,T]$ and assume that $1/m<\delta$. Define
$$
  \begin{array}{lll}
    k_m&= & \ds
    \max \{k\, :\, \frac{k}{m}\leq\delta\}+1,
    \qquad k_m=[m\delta]+1
    \eds \\
    l_m&= & \ds
    \inf \{ l \, ;\, l\frac{k_m}{m}\geq T\} -1,
    \; m=1,2,\ldots .
    \eds
  \end{array}
$$
Thus $k_m$ denotes the minimal number of intervals of
length $1/m$ covering the interval $[0,\delta]$ and
$l_m$ measures how many nonintersecting intervals of
length $k_m/m$ (equal approximately to $\delta$) can be
included in $[0,T]$. Let
$$
I^m_l =\left[ l\frac{k_m}{m}, (l+1)\frac{k_m}{m}
\right],\qquad l=0,1,2,\ldots .
$$
We show first that
\begin{equation}\label{perdonsker3}
\sup_{
\begin{array}{c}
\begin{scriptstyle}
|t-s|\leq \delta
\end{scriptstyle}
 \\ \begin{scriptstyle}
 t,s\in [0,T]
 \end{scriptstyle}
\end{array}
} |X_t^m-X_s^m|
\leq
3\max_{0\leq l\leq l_m-1}\left[
\max_{lk_m<j\leq (l+1)k_m}
|B_j^m - B^m_{lk_m}|\right].
\end{equation}
If $|t-s|\leq\delta$ and $t,s\in I^m_l$ for some $l$
then
$$
|X_t^m-X_s^m|\leq \left| X^m_{j\frac{k_m}{m}} -
 X^m_{k\frac{k_m}{m}}\right|
$$
for some $j,k\in\{ lk_m,lk_m+1,\ldots,(l+1)k_m\}$, and
therefore
$$
  \begin{array}{lll}
    |X_t^m-X_s^m|&\leq & \ds
\left| X^m_{j\frac{k_m}{m}} -
 X^m_{l\frac{k_m}{m}}\right| +
 \left| X^m_{k\frac{k_m}{m}} -
 X^m_{l\frac{k_m}{m}}\right|
    \eds \\
    &\leq & \ds
\left| B^m_{jk_m} - B^m_{lk_m}\right| +
\left| B^m_{kk_m} - B^m_{lk_m}\right| .
\eds
  \end{array}
$$
If $|t-s|\leq\delta$ and $t\in I^m_l$, $s\in I^m_{l+1}$
then in a similar way:
$$
|X_t^m-X_s^m|\leq \left| X^m_{j\frac{k_m}{m}} -
 X^m_{k\frac{k_m}{m}}\right|
$$
where $j\in\{ lk_m,\ldots,(l+1)k_m\}$, $k\in\{
(l+1)k_m,\ldots,(l+2)k_m\}$, and
$$
  \begin{array}{lll}
    |X_t^m-X_s^m|&\leq & \ds
\left| X^m_{j\frac{k_m}{m}} -
 X^m_{l\frac{k_m}{m}}\right| +
 \left| X^m_{(l+1)\frac{k_m}{m}} -
 X^m_{l\frac{k_m}{m}}\right| +
\left| X^m_{k\frac{k_m}{m}} -
 X^m_{(l+1)\frac{k_m}{m}}\right|
    \eds \\
    &\leq & \ds
\left| B^m_{jk_m} - B^m_{lk_m}\right| +
\left| B^m_{(l+1)k_m} - B^m_{lk_m}\right| +
\left| B^m_{kk_m} - B^m_{(l+1)k_m}\right| .
\eds
  \end{array}
$$
This way the estimate (\ref{perdonsker3}) has been
proved.

It follows from (\ref{perdonsker3}) that
$$
  \begin{array}{lll}
   \ds
\P\bigg(
\sup_{
\begin{array}{c}
\begin{scriptstyle}
|t-s|\leq \delta
\end{scriptstyle}
 \\ \begin{scriptstyle}
 t,s\in [0,T]
 \end{scriptstyle}
\end{array}
} |X_t^m-X_s^m| >\epsilon \bigg)
\eds & \leq& \ds
\P \bigg(
\max_{0\leq l\leq l_m-1}\;
\max_{lk_m<j\leq(l+1)k_m}
|B_j^m- B_{lk_m}^m| >\frac{\epsilon}{3}\bigg)
\eds \\
&\leq &\ds
\sum_{l=0}^{l_m-1} \P\bigg(
\max_{lk_m<j\leq(l+1)k_m}
|B_j^m- B_{lk_m}^m| >\frac{\epsilon}{3}\bigg).
\eds
\end{array}
$$
By Ottaviani's inequality, see Subsection
\ref{Ottaviani},
$$\begin{array}{l}\ds
\P\bigg(
\max_{lk_m<j\leq(l+1)k_m}
|B_j^m- B_{lk_m}^m| >2\frac{\epsilon}{6}\bigg)
\left(
1-
\max_{lk_m<j\leq(l+1)k_m}
\P\bigg(|B_{(l+1)k_m}^m- B_{j}^m| >\frac{\epsilon}{6}\bigg)
\right)
\eds\\\ds
\leq \P\bigg(
|B_{(l+1)k_m}^m- B_{lk_m}^m| >\frac{\epsilon}{3}\bigg).
\eds\end{array}
$$
However by Chebyshev's inequality and the definition of
$B^m_j$:
$$
\P\bigg(
|B_{(l+1)k_m}^m- B_{lk_m}^m| >\frac{\epsilon}{6}\bigg)
\leq
\frac{\E\left|B_{(l+1)k_m}^m- B_{lk_m}^m\right|^2}
{\left(\frac{\epsilon}{6}\right)^2}
\leq \frac{k_m}{m}\frac{36}{\epsilon^2}.
$$
Consequently,
$$
\P\bigg(
\max_{lk_m<j\leq(l+1)k_m}
|B_j^m- B_{lk_m}^m| >\frac{\epsilon}{3}\bigg)
\leq
\frac{1}{1-\frac{k_m}{m}\frac{36}{\epsilon^2}}
\P\left(|B_{k_m}^m|>\frac{\epsilon}{3}\right).
$$
Taking into account that $\frac{k_m}{m}\to\delta$,
$1-\frac{k_m}{m}\frac{36}{\epsilon^2}\to
1-\delta\frac{36}{\epsilon^2}$ and that $l_m\leq
\frac{T}{\delta}+1\leq \frac{2T}{\delta}$ we have that
$$
\overline{\lim_m}\, \P\bigg(
\sup_{
\begin{array}{c}
\begin{scriptstyle}
|t-s|\leq \delta
\end{scriptstyle}
 \\ \begin{scriptstyle}
 t,s\in [0,T]
 \end{scriptstyle}
\end{array}
} |X_t^m-X_s^m|
\bigg) \leq
\overline{\lim_m}\,
\frac{l_m}{1-\frac{k_m}{m}\frac{36}{\epsilon^2}}
\P\left(|B_{k_m}^m|>\frac{\epsilon}{3}\right)
\leq
\frac{
2T\frac{1}{\delta}
\P\left(|\zeta_\delta |>\frac{\epsilon}{3}\right)
}
 { 1-\delta \frac{36}{\epsilon^2}}
 ,
$$
where $\call(\zeta_\delta)=N(0,\delta)$. It is
therefore enough to show that
\begin{equation}\label{perdonsker4}
\lim_{\delta\to 0} \frac{1}{\delta}\frac{1}
{\sqrt{2\pi\delta}}\int_{|x|>\frac{\epsilon}{3}}
e^{-\frac{|x|^2}{2\delta}}\, dx =0.
\end{equation}
However, for any $c>0$, taking
$\frac{x}{\sqrt{\delta}}=y$,
$$\begin{array}{l}\ds
\frac{1}{\sqrt{2\pi\delta}}
\int_{|x|>c}
e^{-\frac{|x|^2}{2\delta}}\, dx =
\frac{1}{\sqrt{2\pi }}
\int_{|y|>\frac{c}{\sqrt{\delta}}}
e^{-\frac{|y|^2}{2}}\, dy
\eds\\\ds
\leq
\frac{1}{\sqrt{2\pi }}
\int_{|y|>\frac{c}{\sqrt{\delta}}}
\left( |y|\frac{\sqrt{\delta}}{c}\right)^2
e^{-\frac{|y|^2}{2}}\, dy =
\frac{\delta}{c^2\sqrt{2\pi }}
\int_{|y|>\frac{c}{\sqrt{\delta}}}
 y^2
e^{-\frac{|y|^2}{2}}\, dy.
\eds\end{array}
$$
Since
$$
\lim_{\delta\to 0} \; \int_{|y|>\frac{c}{\sqrt{\delta}}}
 y^2
e^{-\frac{|y|^2}{2}}\, dy =0,
$$
the identity (\ref{perdonsker4}) holds.

{\bf Step 2.} We show that the finite dimensional
distributions of $X^m$ converge weakly to the finite
dimensional distributions of $W$. Although we do not
know whether $\call(W)$ is supported by
$C([0,T];\R^1)$, the statement is meaningful.

Let $0=t_0<t_1<t_2<\ldots<t_k\leq T$. It is enough to
show that the laws of
$$
\left(X^m_{t_1}, X^m_{t_2}-X^m_{t_1},\ldots ,
X^m_{t_k}-X^m_{t_{k-1}}\right)
$$
converge to the law of
$$
\left(W(t_1), W(t_2)-W(t_1),\ldots ,
W(t_k)-W(t_{k-1})\right) .
$$
For each $m=1,2,\ldots$, $j=1,2,\ldots$ set
$$
t_{j,m}=\max \, \left\{
\frac{k}{m}\, ; \, \frac{k}{m}\leq t_j\right\},
\qquad j=1,\ldots ,k.
$$
It is enough to show that the random variables
$$
\left(X^m_{t_{1,m}}, X^m_{t_{2,m}}-X^m_{t_{1,m}},\ldots ,
X^m_{t_{k,m}}-X^m_{t_{k-1,m}}\right)
$$
converge weakly to
$$
\left(W(t_1), W(t_2)-W(t_1),\ldots ,
W(t_k)-W(t_{k-1})\right) .
$$

\footnote{
Indeed we have, for instance,
$$
X^m_{t_1}-X^m_{t_{1,m}}=\theta_m^{(1)}\xi_m^{(1)},
$$
with $\theta_m^{(1)} $ a number in $[0,1]$ and
$\xi_m^{(1)} $ a random variable with $\E\,
(\xi_m^{(1)})^2=1/m$, hence converging to zero in
probability; by Proposition \ref{sommaprob} the random
variables $X^m_{t_1}$ and $X^m_{t_{1,m}}$ have the same
weak limit.
}
 However $X^m_{t_{1,m}},
X^m_{t_{2,m}}-X^m_{t_{1,m}},\ldots ,
X^m_{t_{k,m}}-X^m_{t_{k-1,m}}$ are independent Gaussian
random variables identical with
$$
\left(B^m_{mt_{1,m}}, B^m_{mt_{2,m}}-B^m_{mt_{1,m}},\ldots ,
B^m_{mt_{k,m}}-B^m_{mt_{k-1,m}}\right).
$$
The second moments of these random variables are
$$
 t_{1,m}, t_{2,m} - t_{1,m},\ldots ,
 t_{k,m} - t_{k-1,m}
$$
and they converge to
$$
 t_{1}, t_{2} - t_{1},\ldots ,
 t_{k} - t_{k-1}.
$$
This easily implies the required convergence. \qed

\section{The Ottaviani inequality} \label{Ottaviani}

\begin{Theorem}
Let $\xi_1,\xi_2,\ldots$ be independent random
variables with values in a normed space $(E,|\cdot |)$.
Define $S_k=\xi_1+\ldots+\xi_k$, $k=1,2\ldots$. Then
for arbitrary $r\geq 0$ and arbitrary natural numbers
$n>m\geq 1$
$$
\P\bigg( \max_{m<j\leq n}|S_j-S_m|>2r\bigg)
\left( 1-\max_{m<j\leq n}\P\bigg(|S_n-S_j|>r\bigg)\right)
\leq\P\bigg(  |S_n-S_m|>r\bigg).
$$
\end{Theorem}

\noindent {\bf Proof.}
 Let
$$
  \begin{array}{lll}
    \tau &= & \ds
    \min\, (k\in\{m+1,\ldots,n\}\, ;\,
    |S_k-S_m|>2r)
    \eds \\
    &= & \ds
    +\infty {\rm \; if\;} |S_k-S_m|\leq 2r {\rm \;
    for\;} k=m+1,\ldots,n.
    \eds
  \end{array}
$$
Note that if $\tau=k$ and $|S_n-S_k|\leq r$ then
$$
|S_n-S_m|\geq |S_k-S_m| -|S_n-S_k|>r,
\qquad k=m+1,\ldots,n.
$$
Consequently
\begin{equation}\label{Ottprima}
\P ( \tau=k {\rm \; and\;} |S_n-S_k|\leq r )\leq
\P ( \tau=k {\rm \; and\;} |S_n-S_m|> r ).
\end{equation}
The event $\{\tau=k\}$ is $\sigma\{\xi_{m+1},\ldots
,\xi_k\}$ measurable and the event $\{|S_n-S_k|\leq r
\}$
 is $\sigma\{\xi_{k+1},\ldots
,\xi_n\}$ measurable so they are independent. Therefore
$$
\P ( \tau=k {\rm \; and\;} |S_n-S_k|\leq r )
= \P ( \tau=k)\,\P (  |S_n-S_k|\leq r )
$$
and by (\ref{Ottprima})
\begin{equation}\label{Ottseconda}
\P ( \tau=k)\,\P (  |S_n-S_k|\leq r )\leq
\P ( \tau=k {\rm \; and\;} |S_n-S_m|>r ).
\end{equation}
Adding inequalities in (\ref{Ottseconda}) with respect
to $k=m+1,\ldots,n$ one gets
$$
\left( \sum_{k=m+1}^n\P (\tau=k)\right)
\min_{m+1\leq k\leq n}\,\P (  |S_n-S_k|\leq r )
\leq
\P (|S_n-S_m|>r ).
$$
But $\sum_{k=m+1}^n\P (\tau=k) =
\P ( \max_{m<j\leq n}|S_j-S_m|>2r )$ and the result follows.
\qed

\section{Tightness by the factorisation method}

We prove now the tightness of the sequence $(\mu_n)$
using the so called factorisation method, which goes back to L. Schwartz, see a discussion in \cite{factorization}. Both methods,
the classical one and the factorisation, have much
wider range of applications.

For arbitrary integrable function $f$ defined on
$[0,T]$ denote by $I_1f$ its integral:
\begin{equation}\label{defi1}
I_1f(t)= \int_0^tf(s)\, ds,\qquad t\in [0,T],
\end{equation}
and, more generally, by $I_\alpha f$, $\alpha >0$, the
$\alpha$-fractional integral of $f$,
whenever defined:
\begin{equation}\label{defi2}
I_\alpha f(t)=\frac{1}{\Gamma(\alpha)}
\int_0^t (t-s)^{\alpha -1}\, f(s)\,
ds,\qquad t\in [0,T].
\end{equation}

The $\alpha$-fractional integral of $f$ is called also
the Riemann-Liouville integral of $f$, see \cite{oldham}
The operators $I_\alpha $, $\alpha >0$, have the
semigroup property
$$
I_{\alpha +\beta}f = I_\alpha (I_\beta f)
$$
and in particular, for $\alpha\in (0,1)$,
\begin{equation}\label{semigrdii}
I_1 f= I_\alpha (I_{1-\alpha } f).
\end{equation}
To check (\ref{semigrdii}) notice that
$$
  \begin{array}{lll}
    \ds I_\alpha (I_{1-\alpha } f)(t)\eds & =&
    \ds
\frac{1}{\Gamma(\alpha)}\frac{1}{\Gamma (1-\alpha)}
\int_0^t(t-s)^{\alpha -1} \left[
\int_0^s(s-\sigma)^{-\alpha}f(\sigma)\, d\sigma
\right]\,ds
    \eds \\
    &= & \ds
\frac{1}{\Gamma(\alpha) \,\Gamma (1-\alpha)}
\int_0^t \left[
\int_\sigma^t
(t-s)^{\alpha -1}(s-\sigma)^{-\alpha}\,ds
\right] f(\sigma)\, d\sigma .
\eds
  \end{array}
$$
But changing variables $u=(t-\sigma)v$
$$
  \begin{array}{l}
    \ds
\int_\sigma^t
(t-s)^{\alpha -1}(s-\sigma)^{-\alpha}\,ds
=
\int_0^{t-\sigma}
(t-\sigma -u)^{\alpha -1}u^{-\alpha}\,du
    \eds \\
    \ds
=\int_0^1
(t-\sigma)^{\alpha -1} (1-v)^{\alpha
-1}(t-\sigma)^{-\alpha}v^{-\alpha}(t-\sigma)\,dv
= \int_0^1
 (1-v)^{\alpha -1} v^{-\alpha} \,dv.
\eds
  \end{array}
$$

 We will call (\ref{semigrdii}) the factorisation
formula. For its infinite dimensional generalisation see \cite{DZ1} and \cite{DZ2}. We will need the following functional analytic
result:

\begin{Proposition}\label{compfraz}
If $1/p<\alpha \leq 1$ then the operators $I_\alpha$
are compact from $L^p[0,T]$ into $C[0,T]$.
\end{Proposition}

We prove the tightness of $(\mu_m)$ in the following
way.

Note first that
\begin{equation}\label{fact4}
X^m_t =\int_0^t \dot{X}^m_s \, ds = I_1(\dot{X}^m)(t)
= I_\alpha (I_{1-\alpha } (\dot{X}^m))(t),
\end{equation}
where we denoted by $\dot{X}^m_s $ the derivative
$\frac{d}{ds}X^m_s $ (which is a piecewise constant
function for each $\omega$). One can treat the random
processes $I_{1-\alpha } (\dot{X}^m)$ as
$L^p[0,T]$-valued random variables. We will show that:

\begin{Lemma}\label{lpestimate}
If $2\alpha <1$  then for all $m=1,\ldots$ and
$p>0$
\begin{equation}\label{fact5}
\E \| I_{1-\alpha } (\dot{X}^m)\|_{L^p[0,T]}\leq
\frac{c_p}{\Gamma^p(1-\alpha)}
\frac{T^{1+(1-2\alpha){(p/2)}}}
{ (1-2\alpha)^{p/2}},
\end{equation}
 where $c_p$ is the $p$-th
moment of a normalized Gaussian random variable.
\end{Lemma}

By Chebishev's inequality, for arbitrary $r>0$,
$$
\P\bigg(  \| I_{1-\alpha } (\dot{X}^m)\|_{L^p[0,T]}
\geq r\bigg)\leq
\frac{\E  \| I_{1-\alpha } (\dot{X}^m)\|_{L^p[0,T]}^p}{r^p},
$$
so by (\ref{fact5})
$$
\lim_{r\to\infty} \left[
\sup_m \, \P\bigg(  \| I_{1-\alpha } (\dot{X}^m)\|_{L^p[0,T]}
> r\bigg)\right] =0.
$$
If $p$ and $\alpha$ are chosen in such a way that
$$
\frac{1}{p}<\alpha<\frac{1}{2}
$$
then by Proposition \ref{compfraz} and by the semigroup
property (\ref{semigrdii}), the laws $\call(X^m)$ are
tight on $C[0,T]$.

It remains to prove Lemma \ref{lpestimate}.

Note that, by the very definition,
$$
\dot{X}^m_\sigma =m\, \xi^m_{n+1}
\qquad {\rm \; for\;\;} \frac{n}{m}<\sigma <\frac{n+1}{m}.
$$
Consequently,
$$
\dot{X}^m_\sigma =m\sum_{n=0}^\infty
\chi_{(\frac{n}{m},\frac{n+1}{m})}(\sigma)
\xi^m_{n+1},
\qquad  \sigma \in [0,T],
$$
and
$$
\E \| I_{1-\alpha } (\dot{X}^m)\|_{L^p[0,T]}^p=
m^p \, \E\left[
\int_0^T\left(
\sum_n (
 I_{1-\alpha } \chi_{(\frac{n}{m},\frac{n+1}{m})}
 )(t)\xi^m_{n+1}
 \right)^p
 \, dt\right].
 $$
If $\zeta$ is a Gaussian random variable, $\E\zeta =0$,
then $\E |\zeta |^p =c_p (\E |\zeta |^2)^{p/2}$, where
$c_p$ is the $p$-th moment of a normalised Gaussian
random variable. Therefore
$$
\E \| I_{1-\alpha } (\dot{X}^m)\|_{L^p[0,T]}^p=
m^p c_p \int_0^T \left[ \E
\left(
\sum_n (
 I_{1-\alpha } \chi_{(\frac{n}{m},\frac{n+1}{m})}
 )(t)\xi^m_{n+1}
 \right)^2
 \right]^{p/2} \, dt.
 $$
But, for each $t\in [0,T]$,
$$
\E
\left(
\sum_n (
 I_{1-\alpha } \chi_{(\frac{n}{m},\frac{n+1}{m})}
 )(t)\xi^m_{n+1}
 \right)^2=
 \frac{1}{m} \sum_n (
 I_{1-\alpha } \chi_{(\frac{n}{m},\frac{n+1}{m})}
 )(t)^2
$$
and
$$
(
 I_{1-\alpha } \chi_{(\frac{n}{m},\frac{n+1}{m})}
 )(t)^2 = \frac{1}{\Gamma^2(1-\alpha)}
 \left< \chi_{[0,t]}(t-\cdot)^{-\alpha},
 \chi_{(\frac{n}{m},\frac{n+1}{m})}
 \right>^2_{L^2[0,T]}.
$$
Since
$$
 \left\|\chi_{(\frac{n}{m},\frac{n+1}{m})}
 \right\| ^2_{L^2[0,T]} =\frac{1}{m}
$$
and, for different $n$,
$\chi_{(\frac{n}{m},\frac{n+1}{m})}$ are orthogonal
functions, consequently, by Parseval's inequality:
$$
  \begin{array}{lll}
    \ds
    \sum_n  \left< \chi_{[0,t]}(t-\cdot)^{-\alpha},
 \chi_{(\frac{n}{m},\frac{n+1}{m})}

 \right>^2
 \eds &=& \ds
\frac{1}{m} \sum_n
\left< \chi_{[0,t]}(t-\cdot)^{-\alpha},
 \sqrt{m}\chi_{(\frac{n}{m},\frac{n+1}{m})}
 \right>^2
 \eds \\
    &\leq & \ds
\frac{1}{m} \int_0^t(t-\sigma)^{-2\alpha}d\sigma =
\frac{1}{m} \int_0^t \sigma^{-2\alpha}d\sigma,
\quad t\in [0,T].
\eds
  \end{array}
$$
This way we arrive at the estimate:
$$
  \begin{array}{lll}
    \ds \E \| I_{1-\alpha } (\dot{X}^m)\|_{L^p[0,T]}^p\eds
  &  \leq & \ds
  m^p c_p \frac{1}{m^{p/2}}\frac{1}{\Gamma^p(1-\alpha)}
  \int_0^T\left( \frac{1}{m}
  \int_0^T  \sigma^{-2\alpha}d\sigma\right)^{p/2}dt\eds
  \\
    &\leq & \ds
\frac{c_p}{\Gamma^p(1-\alpha)} T
\frac{T^{(1-2\alpha){(p/2)}}}
{ (1-2\alpha)^{p/2}},
\eds
  \end{array}
$$
which proves the lemma. \qed

\section{Donsker's theorem}

\begin{Theorem}\label{donsker}
Assume that for each $m$,
 $\xi_1^m, \xi^m_2,\ldots $
are independent, identically distribu\-ted such that
$$
\E\, \xi_n^m=0,\quad \E\, (\xi_n^m)^2 =\frac{1}{m},
\qquad  m=1,2,\ldots,\; n=1,2,\ldots,
$$
then the processes $X^m$ converge weakly to a Wiener
process.
\end{Theorem}

\noindent {\bf Proof.}
The proof is the same as for Gaussian $\xi_n^m$. Only
in the final part one has to use the following central
limit theorem.

{\bf Central limit theorem.} \begin{em} Under the
conditions of Theorem \ref{donsker}
$$
\call (B^m_m) \Rightarrow N(0,1)
\qquad {\; as \;} m\to\infty.
$$
More generally, if $\frac{k_m}{m}\to \sigma^2$ as
$m\to\infty$ then
$$
\call (B^m_{k_m}) \Rightarrow N(0,\sigma^2)
\qquad {\; as \;} m\to\infty.
$$
\end{em}

\section{Tightness by Kolmogorov's test}

Let $(E,\rho)$ be a metric space and $X_n$ a sequence
of $E$-valued, continuous processes defined on possibly
different probability spaces $(\Omega_n, \calf_n,
\P^n)$. We have the following version of
 the Kolmogorov's continuity test, see \cite{bib13}.

\begin{Theorem}\label{kolmtest}
Assume that there exist positive numbers $\alpha$, $r$,
$c$, such that for all $s,t\in [0,T]$
\begin{equation}\label{kolmogtest1}
\E^n\bigg(
\rho^r(X_n(t),X_n(s))\bigg) \leq c \,
|t-s|^{1+\alpha}.
\end{equation}
If the laws of $X_n(0)$ are tight on $E$ then also the
laws of $X_n(\cdot)$ on $C([0,T];E)$ are tight.
\end{Theorem}

\noindent {\bf Proof.}
To simplify notation we assume that $T=1$. Let $\cald$
be the set of dyadic numbers $t_{mk}=k/2^m$,
$k=0,1,2,\ldots,2^m$, $m=0,1,2,\ldots $. We show that
for arbitrary $\epsilon >0$
$$
\lim_{\delta\to 0}\, \sup_n\, \P^n \bigg(
\sup_{
\begin{array}{c}
\begin{scriptstyle}
 t,s\in \cald
\end{scriptstyle}
 \\ \begin{scriptstyle}
|t-s|\leq \delta
 \end{scriptstyle}
\end{array}
}
\rho(X_n(t),X_n(s)) > \epsilon\bigg) =0.
$$
Let us fix $k\in\N$ and assume that $\frac{1}{2^k}\leq
\delta <\frac{1}{2^{k-1}}$. Then
$$
\sup_{
\begin{array}{c}
\begin{scriptstyle}
 t,s\in \cald
\end{scriptstyle}
 \\ \begin{scriptstyle}
|t-s|\leq \delta
 \end{scriptstyle}
\end{array}
}
\rho(X_n(t),X_n(s)) \leq 3 \sup_j\,
\bigg[ \sup_{\frac{j}{2^k} <\frac{l}{2^m} <\frac{j+1}{2^k}}
\rho\left(X_n\left(\frac{l}{2^m} \right),
X_n\left(\frac{j}{2^k} \right)\right) \bigg],
$$
where $m>k$.

We have the following elementary lemma.
\begin{Lemma}\label{dyadic}
Any dyadic number $x=\frac{l}{2^m}<1$ where
$l=0,1,\ldots 2^m-1$ has a unique representation in the
form:
$$
x= \sum_{j=1}^m \frac{\epsilon_j}{2^j},
\qquad {\; \rm where\;} \epsilon_j =0 {\rm \; or\;} 1.
$$
\end{Lemma}
\noindent {\bf Proof of the Lemma.}
If $m=0$ and $m=1$ the result is true. Let the result
be true for $m-1$ and $\frac{l}{2^m}<1$. If $l=2l'$
then  $\frac{l'}{2^{m-1}}<1$ and by induction
$$
 \frac{l}{2^m}=\frac{l'}{2^{m-1}} =
 \sum_{j=1}^{m-1}\frac{\epsilon_j}{2^j}.
$$
If $l=2l'+1$ then
$$
 \frac{l}{2^m}=\frac{l'}{2^{m-1}} +
  \frac{1}{2^m}
$$
and again we have representation.
Uniqueness: assume
$$
\sum_{j=1}^{m}\frac{\epsilon_j}{2^{j}}
=\sum_{j=1}^{m}\frac{\epsilon_j'}{2^{j}}.
$$
Let $k$ be the smallest value of the index $j$ such
that $\epsilon_j \neq \epsilon_j'$. Without loss of
generality we may suppose $\epsilon_k =1,$
$\epsilon_k'=0$. Then
$$
\frac{1}{2^{k}}+\sum_{j=k+1}^{m}\frac{\epsilon_j}{2^{j}}
=\sum_{j=k+1}^{m}\frac{\epsilon_j'}{2^{j}}.
$$
This is impossible, since
$$
\frac{1}{2^{k}} >
\sum_{j=k+1}^{m}\frac{1}{2^{j}}\geq
\sum_{j=k+1}^{m}\frac{\epsilon_j'}{2^{j}}.
\qed
$$

\noindent {\bf Proof of the Theorem.}
By the lemma
$$
\frac{l}{2^m}= \frac{j}{2^k}+
 \sum_{r=1}^s \frac{1}{2^{m_r}},
$$
where $k< m_1 <\ldots <m_s\leq m$. Consequently
$$
\rho\left(X_n\left(\frac{l}{2^m} \right),
X_n\left(\frac{j}{2^k} \right)\right) \leq
\sum_{i=1}^s
\rho\left(X_n\left(\frac{j}{2^k}+\sum_{r=1}^i
 \frac{1}{2^{m_r}}\right),
X_n\left(\frac{1}{2^k} +\sum_{r=1}^{i-1}
 \frac{1}{2^{m_r}}\right)\right),
$$
and
$$
\sup_{
\begin{array}{c}
\begin{scriptstyle}
 t,s\in \cald
\end{scriptstyle}
 \\ \begin{scriptstyle}
|t-s|\leq \delta
 \end{scriptstyle}
\end{array}
}
\rho(X_n(t),X_n(s)) \leq 2 \sum_{m=k+1}^\infty
\bigg[
\sup_{l\leq 2^m-1}
\rho\left(X_n\left(\frac{l+1}{2^m} \right),
X_n\left(\frac{l}{2^m} \right)\right)
\bigg].
$$
However
$$\begin{array}{l}\ds
 \P^n \bigg(
\sup_{l\leq 2^m-1}
\rho\left(X_n\left(\frac{l+1}{2^m} \right),
X_n\left(\frac{l}{2^m} \right)\right)
>\frac{1}{m^2}
\bigg)
\eds\\\ds
\leq
\sum_{l\leq 2^m-1} \P^n \bigg(
\rho\left(X_n\left(\frac{l+1}{2^m} \right),
X_n\left(\frac{l}{2^m} \right)\right)
>\frac{1}{m^2}
\bigg)
\eds\end{array}
$$
and by Chebishev's inequality and (\ref{kolmogtest1})
$$
\leq m^{2r}\sum_{l\leq 2^m-1} \E^n
\rho^r\left(X_n\left(\frac{l+1}{2^m} \right),
X_n\left(\frac{l}{2^m} \right)\right)
\leq
m^{2r} 2^m c\, \left(\frac{1}{2^m}\right)^{1+\alpha}
= c\, \frac{m^{2r}}{2^{m\alpha}}.
$$
Choose $k$ such that
$$
\sum_{m=k+1}^\infty \frac{1}{m^2} <\frac{\epsilon}{2},
$$
then
$$
  \begin{array}{l}
    \ds
    \P^n \bigg(
\sup_{
\begin{array}{c}
\begin{scriptstyle}
 t,s\in \cald
\end{scriptstyle}
 \\ \begin{scriptstyle}
|t-s|\leq \delta
 \end{scriptstyle}
\end{array}
}
\rho(X_n(t),X_n(s)) > \epsilon\bigg)
\eds \\\ds\leq
 \P^n \bigg(
2 \sum_{m=k+1}^\infty
\sup_{l\leq 2^m-1}
\rho\left(X_n\left(\frac{l+1}{2^m} \right),
X_n\left(\frac{l}{2^m} \right)\right)
>\epsilon \bigg)
\eds \\
    \ds\leq
 \P^n \bigg(
 \sum_{m=k+1}^\infty
\sup_{l\leq 2^m-1}
\rho\left(X_n\left(\frac{l+1}{2^m} \right),
X_n\left(\frac{l}{2^m} \right)\right)
>\sum_{m=k+1}^\infty \frac{1}{m^2} \bigg)
\eds
\\ \ds\leq
\sum_{m=k+1}^\infty
\P^n \bigg(
\sup_{l\leq 2^m-1}
\rho\left(X_n\left(\frac{l+1}{2^m} \right),
X_n\left(\frac{l}{2^m} \right)\right)
>  \frac{1}{m^2} \bigg)
\eds \\ \ds
\leq
 c\, \sum_{m\geq k+1}
 \frac{m^{2r}}{2^{m\alpha}}.
 \eds
  \end{array}
$$
Since the series $\sum_{m}
 \frac{m^{2r}}{2^{m\alpha}}$ is convergent, and
$\delta\in \Big[\frac{1}{2^k},\frac{1}{2^{k-1}}\Big)$,
$$
\sum_{m\geq k+1}
 \frac{m^{2r}}{2^{m\alpha}} \leq
 \sum_{m\geq [\log_2\frac{1}{\delta}]+2}
 \frac{m^{2r}}{2^{m\alpha}}\to 0
 \qquad {\;\rm as\;} \delta\to 0. \; \qed
$$

Theorem \ref{kolmtest} implies Theorem \ref{11.2.1}. Note first the following result.
from the following lemma.

\begin{Lemma}\label{perkolmotest}
Assume that $d >0$ is a positive number and
$\xi_1,\xi_2,\ldots $ are independent, real random
variables such that for some $\sigma >0$ and all
$k=1,2,\ldots$,
$$
\E\, \xi_k=0,\qquad \E\, \xi_k^2=\sigma\, d.
$$
Define $X_{kd}= \xi_1+\ldots+\xi_k$, $k=1,2,\ldots$ and
$$
X_t=X_{kd} +\frac{t-kd}{d}\, \xi_{k+1},\qquad {\rm \;
for
\;} t\in [kd,(k+1)d].
$$
Then
$$
\E \, |X_t-X_s|^2\leq \sigma \, |t-s|,\qquad {\rm \; for
\;all\; } t,s\geq 0.
$$
\end{Lemma}

\noindent {\bf Proof.}
Assume that $s\geq t$ and $ s\in [ld,(l+1)d]$. Then
$$
  \begin{array}{lll}
    X_s-X_t&= & \ds
    X_{ld}-X_{kd} +
\frac{s-ld}{d}\, \xi_{l+1}
- \frac{t-kd}{d}\, \xi_{k+1}
    \eds\\
    &= & \ds
\xi_l+\ldots+\xi_{k+1} +
\frac{s-ld}{d}\, \xi_{l+1}
- \frac{t-kd}{d}\, \xi_{k+1}
\eds\\
    &= & \ds
\frac{s-ld}{d}\, \xi_{l+1}
+\xi_l +\xi_{l-1}+\ldots +\xi_{k+2}+
\left(
1- \frac{t-kd}{d}\right)\, \xi_{k+1}.
\eds
  \end{array}
$$
So
$$
\E \, |X_t-X_s|^2 = \left( \frac{s-ld}{d}\right)^2
\sigma\, d + \sigma \, d \, (l-(k+1)) +
\left(
1- \frac{t-kd}{d}\right)^2\sigma \, d.
$$
Since
$$
\sigma\, (s-t) = \sigma \, (s-ld)+\sigma \, d \, (l-(k+1))
+\sigma \,   ((k+1)d-t),
$$
$$
\sigma\, d \,\left( \frac{s-ld}{d}\right)^2
=
  \frac{s-ld}{d} \sigma\, (s-ld) \leq \sigma\, (s-ld),
$$
$$
\sigma \, d\, \left(
1- \frac{t-kd}{d}\right)^2 =
\frac{(k+1)d-t}{d} \sigma\, ((k+1)d-t)\leq
\sigma \, ((k+1)d-t),
$$
the result follows. \qed
\vspace{2mm}

\noindent If now $\zeta$ is a Gaussian random variable, $\E\zeta =0$,
$\E \zeta^2 =\gamma^2$, then for arbitrary $p>0$
$$
\E\, |\zeta|^p =
\left( \E\, \left| \frac{\zeta}{\gamma}\right|^p\right)
\gamma^p
$$
and consequently for any $p>0$ there exists $c_p >0$
such that
$$
\E\, |\zeta|^p \leq c_p\,
\left( \E\, \left|  \zeta \right|^2\right)^{p/2}.
$$
Thus if, in addition to the assumptions of Lemma
\ref{perkolmotest}, one requires that $\xi_1, \xi_2, \ldots, $ are
Gaussian then
$$
\E \, |X_t-X_s|^p\leq
c_p\, \left( \E \, |X_t-X_s|^2\right)^{p/2} \leq c_p \,
\sigma^{p/2} |t-s|^{p/2} .
$$
So if $p>2$ the conditions of the Kolomogorov test are
satisfied.
}

\end{document}